\documentclass{article}

\usepackage{comment}
\usepackage{amsfonts}
\usepackage{amssymb}
\usepackage{amsthm}
\usepackage{mathtools}   
\usepackage{csquotes}
\usepackage{footmisc}
\usepackage{cite}
\usepackage[alphabetic]{amsrefs}
\usepackage{xpatch}
\xpatchcmd{\bibsection}{*}{}{}{} 
\usepackage{longtable}
\usepackage[all,cmtip]{xy}
\usepackage{fancyhdr}

\usepackage{graphicx}
\usepackage{hyperref}

\usepackage{tikz-cd}
\usepackage{tikz}
\usepackage{pifont}

\usepackage{aurical}
\usepackage{stmaryrd}
\usepackage[T1]{fontenc}
\usepackage[outline]{contour}
\usepackage{xcolor}
\usepackage{etoolbox}
\usepackage{array}
\usepackage{lipsum}

\usepackage{enumitem}

\usepackage{parskip}

\newcommand{\git}{\mathbin{
  \mathchoice{\mkern-3mu/\mkern-6mu/\mkern-3mu}
    {\mkern-3mu/\mkern-6mu/\mkern-3mu}
    {/\mkern-5mu/}
    {/\mkern-5mu/}}}

\newcommand\blfootnote[1]{%
  \begingroup
  \renewcommand\thefootnote{}\footnote{#1}%
  \addtocounter{footnote}{-1}%
  \endgroup
}

\newcommand{\genlegendre}[4]{%
  \genfrac{(}{)}{}{#1}{#3}{#4}%
  \if\relax\detokenize{#2}\relax\else_{\!#2}\fi
}

\newcommand{\WIP}[1]{\scalebox{#1}{\begin{tikzpicture}[limb/.style={line cap=round,line width=1.5mm,line join=bevel}]
\draw[line width=2mm,rounded corners,fill=yellow] (-2,0) -- (0,-2) -- (2,0) -- (0,2) -- cycle;
\fill (1.5mm,7mm) circle (1.5mm);
\fill(0,-7.5mm) -- ++(10mm,0mm) -- ++(120:2mm)--++(100:1mm)--++(150:2mm) arc (70:170:2.5mm and 1mm);
\draw[limb] (-7.5mm,-6.5mm)--++(70:4mm)--++(85:4mm) coordinate(a)--++(-45:5mm)--(-2.5mm,-6.5mm);
\fill[rotate around={45:(a)}] ([shift={(-0.5mm,0.55mm)}]a) --++(0mm,-3mm)--++
        (7mm,-0.5mm)coordinate(b)--++(0mm,4mm)coordinate(c)--cycle;
\draw[limb] ([shift={(-0.6mm,-0.4mm)}]b) --++(-120:5mm) ([shift={(-0.5mm,-0.5mm)}]c) --++
        (-3mm,0mm)--++(-100:3mm)coordinate (d);
\draw[ultra thick] (d) -- ++(-45:1.25cm);
\end{tikzpicture}}}

\robustify{\WIP}

\definecolor{cec1d24}{RGB}{236,29,36}
\definecolor{cffffff}{RGB}{255,255,255}

\newcommand{\what}[1]{{\widehat{#1}}}

\newcommand{\bs}{\backslash}

\newcommand{\lra}{\longrightarrow}
\newcommand{\ra}{\rightarrow}

\newcommand{\bA}{\mathbb{A}}

\newcommand{\bC}{\mathbb{C}}
\newcommand{\bD}{\mathbb{D}}
\newcommand{\bE}{\mathbb{E}} 
\newcommand{\bF}{\mathbb{F}}
\newcommand{\bG}{\mathbb{G}}

\newcommand{\bI}{\mathbb{I}}

\newcommand{\bL}{\mathbb{L}}
\newcommand{\bM}{\mathbb{M}}
\newcommand{\bN}{\mathbb{N}}
\newcommand{\bP}{\mathbb{P}}
\newcommand{\bQ}{\mathbb{Q}}

\newcommand{\bS}{\mathbb{S}}
\newcommand{\bT}{\mathbb{T}}
\newcommand{\bX}{\mathbb{X}}
\newcommand{\bY}{\mathbb{Y}}
\newcommand{\bZ}{\mathbb{Z}}

\newcommand{\BMP}{\textbf{MP}}

\newcommand{\Qbar}{{\overline{\mathbb{Q}}}}

\newcommand{\cB}{\mathcal{B}}
\newcommand{\cC}{\mathcal{C}}
\newcommand{\cD}{\mathcal{D}}
\newcommand{\cE}{\mathcal{E}}
\newcommand{\cF}{\mathcal{F}}
\newcommand{\cG}{\mathcal{G}}
\newcommand{\cH}{\mathcal{H}}
\newcommand{\cI}{\mathcal{I}}
\newcommand{\cJ}{\mathcal{J}}

\newcommand{\cL}{\mathcal{L}}
\newcommand{\cM}{\mathcal{M}}
\newcommand{\cN}{\mathcal{N}}
\newcommand{\cO}{\mathcal{O}}
\newcommand{\cP}{\mathcal{P}}

\newcommand{\cR}{\mathcal{R}}

\newcommand{\cT}{\mathcal{T}}
\newcommand{\cU}{\mathcal{U}}
\newcommand{\cV}{\mathcal{V}}
\newcommand{\cX}{\mathcal{X}}
\newcommand{\cY}{\mathcal{Y}}
\newcommand{\cZ}{\mathcal{Z}}

\newcommand{\Zhat}{{\widehat{\mathbb{Z}}}}

\newcommand{\la}[1]{\,^{#1}\!} 
\newcommand{\floor}[1]{\lfloor #1 \rfloor}
\newcommand{\ceil}[1]{\lceil #1 \rceil}

\newcommand{\mf}[1]{\mathfrak{#1}} 

\newcommand{\fc}{\mathfrak{c}}
\newcommand{\ff}{\mathfrak{f}}

\newcommand{\fm}{\mathfrak{m}}
\newcommand{\fR}{\mathfrak{R}}

\newcommand{\fs}{\!\!\fatslash}

\newcommand{\et}{{\acute{e}t}}

\newcommand{\fppf}{{fppf}}

\newcommand{\Aut}{\operatorname{Aut}}

\newcommand{\Inn}{\operatorname{Inn}}
\newcommand{\Out}{\operatorname{Out}}

\newcommand{\Ker}{\operatorname{Ker}}

\newcommand {\spmatrix}[4]{\left[\begin{smallmatrix}#1&#2\\#3&#4\end{smallmatrix}\right]}

\newcommand{\slc}[1]{{(\ol{#1})}}

\newcommand{\rot}{\operatorname{rot}}
\newcommand{\ch}{\operatorname{char}}

\newcommand{\cAdm}{\mathcal{A}dm}
\newcommand{\Adm}{Adm}

\newcommand{\colim}{\operatorname{colim}}
\newcommand{\Frac}{\operatorname{Frac}}

\newcommand{\Gal}{\operatorname{Gal}}
\newcommand{\id}{\operatorname{id}}
\newcommand{\Hom}{\operatorname{Hom}}

\newcommand{\cHom}{\mathcal{H}om}

\newcommand{\Epi}{\operatorname{Epi}}

\newcommand{\GL}{\operatorname{GL}}
\newcommand{\SL}{\operatorname{SL}}
\newcommand{\PSL}{\operatorname{PSL}}
\newcommand{\PGL}{\operatorname{PGL}}
\newcommand{\PSU}{\operatorname{PSU}}

\newcommand{\Sz}{\operatorname{Sz}}

\newcommand{\Stab}{\operatorname{Stab}}

\newcommand{\disc}{\operatorname{disc}}
\newcommand{\Tr}{\operatorname{Tr}}

\newcommand{\ev}{\operatorname{ev}}

\newcommand{\bal}{{\text{bal}}}

\newcommand{\tr}{\operatorname{tr}}

\newcommand{\Inv}{\operatorname{Inv}}
\newcommand{\Hig}{\operatorname{Hig}}
\newcommand{\ord}{\operatorname{ord}}

\newcommand{\rightiso}{\stackrel{\sim}{\longrightarrow}}

\newcommand{\abs}{{\operatorname{abs}}}

\newcommand{\Div}{\operatorname{Div}}

\newcommand{\Cl}{\operatorname{Cl}}
\newcommand{\ol}[1]{{\overline{#1}}}
\newcommand{\ul}[1]{{\underline{#1}}}

\newcommand{\an}{\text{an}}

\newcommand{\genus}{\operatorname{genus}}
\newcommand{\Spec}{\operatorname{Spec}}

\newcommand{\red}{\text{red}}

\newcommand{\pr}{\operatorname{pr}}

\newcommand{\ext}{\text{ext}}

\newcommand{\ai}{\text{ai}}

\newcommand{\conn}{\text{conn}}

\newcommand{\pre}{{\text{pre}}}
\newcommand{\bgs}{{\text{bgs}}}
\newcommand{\MP}{{\text{MP}}}

\newcommand{\ls}[1]{(\!(#1)\!)}
\newcommand{\ps}[1]{[\![#1]\!]}

\newcommand{\tp}{\text{top}}

\newcommand{\gr}{\operatorname{gr}}

\newcommand{\blt}{\bullet}
\newcommand{\gen}{\text{gen}}
\newcommand{\sm}{\text{sm}}

\newcommand{\Sch}{\underline{\textbf{Sch}}}
\newcommand{\AffSch}{\underline{\textbf{AffSch}}}

\newcommand{\Mod}{\underline{\textbf{Mod}}}

\newcommand{\FinSets}{\underline{\textbf{FinSets}}}
\newcommand{\Sets}{\underline{\textbf{Sets}}}

\newcommand{\Coh}{\underline{\textbf{Coh}}}

\newcommand{\FEt}{\operatorname{FEt}}

\newcommand{\Graphs}{\underline{\textbf{Graphs}}}

\newcommand{\Glue}{\operatorname{Glue}}
\newcommand{\Ext}{\operatorname{Ext}}

\theoremstyle{definition}\newtheorem{defn}{Definition}[subsection]
\theoremstyle{remark}
\theoremstyle{remark}
\theoremstyle{remark}\newtheorem{remark}[defn]{Remark}
\theoremstyle{plain}\newtheorem{question}[defn]{Question}
\theoremstyle{remark}\newtheorem*{remark*}{Remark}
\theoremstyle{remark}
\theoremstyle{remark}
\theoremstyle{remark}
\theoremstyle{remark}
\theoremstyle{definition}\newtheorem{sit}[defn]{Situation}

\theoremstyle{plain}\newtheorem{prop}[defn]{Proposition}
\theoremstyle{plain}\newtheorem{thm}[defn]{Theorem}
\theoremstyle{plain}
\theoremstyle{plain}\newtheorem{lemma}[defn]{Lemma}
\theoremstyle{plain}\newtheorem{cor}[defn]{Corollary}
\theoremstyle{plain}\newtheorem{conj}[defn]{Conjecture}
\theoremstyle{plain}\newtheorem*{thm*}{Theorem}
\theoremstyle{plain}\newtheorem*{conj*}{Conjecture}
\theoremstyle{plain}\newtheorem*{prop*}{Proposition}

\title{Nonabelian level structures, Nielsen equivalence, and Markoff triples}

\author{William Chen}

\begin{document}
\maketitle

\blfootnote{The author's research was partly supported by National Science Foundation Award No. DMS-1803357.}

\begin{abstract} In this paper we establish a congruence on the degree of the map from a component of a Hurwitz space of covers of elliptic curves to the moduli stack of elliptic curves. Combinatorially, this can be expressed as a congruence on the cardinalities of Nielsen equivalence classes of generating pairs of finite groups. Building on the work of Bourgain, Gamburd, and Sarnak \cite{BGS16, BGS16arxiv}, we apply this congruence to show that for all but finitely many primes $p$, the group of Markoff automorphisms acts transitively on the nonzero $\bF_p$-points of the Markoff equation $x^2 + y^2 + z^2 - 3xyz = 0$. This yields a strong approximation property for the Markoff equation, the finiteness of congruence conditions satisfied by Markoff numbers, and the connectivity of a certain infinite family of Hurwitz spaces of $\SL_2(\bF_p)$-covers of elliptic curves. With possibly finitely many exceptions, this resolves a conjecture of Bourgain, Gamburd, and Sarnak, first posed by Baragar \cite{Bar91} in 1991. Since their methods are effective, this reduces the conjecture to a finite computation.
\end{abstract}



\tableofcontents

\section{Introduction}

Let $k$ be a field. Let $\cM_{g,n}$ denote the moduli stack over $k$ of smooth curves of genus $g$ with $n$ marked points. For an integer $d$, there is a Hurwitz stack $\cH_{g,n,d}$ which classifies connected degree $d$ covers of smooth $(g,n)$-curves, only ramified above the $n$ marked points. While the stacks $\cM_{g,n}$ are irreducible, the Hurwitz stacks $\cH_{g,n,d}$ are typically disconnected. It is a classical problem to find combinatorial invariants of covers that can distinguish their connected components, or equivalently, to determine whether an open and closed substack of $\cH_{g,n,d}$ corresponding to some fixed values of various combinatorial invariants is connected.


For example, when $k = \bC$, Clebsch, Luroth, and Hurwitz used a combinatorial argument to show that the substack of $\cH_{0,n,d}$ corresponding to covers which are simply branched over the $n$ marked points is connected; taking $d$ large enough so that every curve of genus $g$ admits a degree $d$ map to $\bP^1$ with simple branching, this gave the first proof of the connectedness of $\cM_g$. When $(g,n) = (0,3)$ and $k = \bQ$, the connected components of $\bigsqcup_{d\in\bZ_{\ge 1}}\cH_{0,3,d}$ are in bijection with the $\Gal(\Qbar/\bQ)$-orbits of dessins d'enfants. Describing these orbits in terms of combinatorial data remains a deep and very open problem \cite{HS97, Schn97, GGA1, GGA2}.


When a combinatorial type is fixed and $g$ or $n$ are allowed to be large, there are a number of connectedness results which go under the name of \emph{genus stabilization} ($g\gg 0$ see \cite[\S6]{DT06} or \cite{CLP16}), or \emph{branch stabilization} ($g = 0$, $n\gg 0$), which first appeared as a theorem of Conway and Parker\footnote{see \cite[Appendix]{FV91}; also see \cite{Lonne18}, \cite[Theorem 6.1]{EVW16}, \cite[Corollary 12.5]{LWZB19}.}. Recently such connectedness results have found notable applications to number theory. In \cite{EVW16} and \cite{LWZB19}, Ellenberg, Venkatesh, Westerland, Liu, Wood, and Zureick-Brown proved branch stabilization results and used them to study the function field analogs of Cohen-Lenstra-Martinet heuristics on class groups. In \cite{RV15}, Roberts and Venkatesh also used branch stabilization to study an open problem of Malle and Roberts on the infinitude of number fields with alternating or symmetric monodromy group and bounded ramification.


When $(g,n)$ are fixed and the combinatorial type is allowed to vary, less is known, and it is unclear what one should even expect. In this paper, we will prove a connectedness result of this type, and apply it to study a conjecture of Bourgain, Gamburd, and Sarnak regarding the Markoff equation. We will consider the case $(g,n) = (1,1)$, and will consider only $G$-\emph{Galois} covers, where $G$ is a finite group. We will write $\cM(1) := \cM_{1,1}$, and use $M(1) := M_{1,1}$ to denote its coarse scheme. Thus we are interested in classifying $G$-covers of elliptic curves only ramified above the origin. In this case, an important combinatorial invariant is the \emph{Higman invariant} (also called Nielsen type), which analytically is described by the conjugacy class in $G$ given by monodromy around the origin of the elliptic curve. In this situation the coarse schemes of our Hurwitz stacks over $\bC$ are disjoint unions of (possibly noncongruence) modular curves\footnote{By this, we mean a quotient of the upper half plane $\cH$ by a finite index subgroup of $\SL_2(\bZ)$, possibly noncongruence.}. The arithmetic of these Hurwitz spaces are linked to the arithmetic of the Fourier coefficients of noncongruence modular forms \cite[\S5]{Chen18}. These components are also \emph{Teichmuller curves} generated by square-tiled surfaces as first studied by Veech \cite{Vee89}, and are interesting from the perspective of billiards and dynamics \cite{Chen17,HS06,Loc03,MT02,Zor06}. 



The central result of this paper is a congruence on the degree of the ``forgetful map'' from any component of the Hurwitz space of $G$-covers of elliptic curves to $M(1)$. If $\Pi$ denotes the fundamental group of a once-punctured surface of genus 1 (this is a free group of rank 2), then via Galois theory this can also be interpreted combinatorially as a divisibility theorem on the cardinalities of $\Out(\Pi)$-orbits on the set
\begin{equation}\label{eq_epi_ext_intro}
\Epi^\ext(\Pi,G) := \Epi(\Pi,G)/\Inn(G)	
\end{equation}
where $\Epi(A,B)$ denotes the set of surjective homomorphisms $A\rightarrow B$. When $G = \cG(\bF_q)$ is the group of $\bF_q$-points of an algebraic group $\cG$, the set $\Epi^\ext(\Pi,G)$ can be related to the $\bF_q$-points of the character variety of $\cG$-representations of $\Pi$, and the $\Out(\Pi)$-action on $\Epi^\ext(\Pi,G)$ induces an action on the character variety. When $\cG = \SL_2$, the associated character variety is the affine space $\bA^3$, and the subvariety $\bX\subset\bA^3$ defined by the Markoff equation $x^2 + y^2 + z^2 - xyz = 0$ is stable under $\Out(\Pi)$. In \cite{BGS16,BGS16arxiv}, Bourgain, Gamburd, and Sarnak studied $\Out(\Pi)$-action on $\bX$, and conjectured that the action is always transitive on $\bX(\bF_p) - \{(0,0,0)\}$. They were almost able to prove the conjecture, but their analytical methods fell short when $p^2-1$ contains many prime factors. Our congruence provides sufficient rigidity to resolve their conjecture for all primes outside an explicitly computable finite set $\bE_\bgs$. This implies a strong approximation property for the Markoff equation, and the finiteness of congruence conditions satisfied by Markoff numbers. In terms of Hurwitz spaces, this implies that for any prime $p\notin\bE_\bgs$ and any non-central conjugacy class $\fc$ of $\SL_2(\bF_p)$ of trace $-2$, the Hurwitz space $M_p$ of $\SL_2(\bF_p)$-covers of elliptic curves only branched over the origin with Higman invariant $\fc$ is \emph{connected}.



This connectedness yields in some sense a complete geometric description of the curves $M_p$. In particular, by working with the character variety we are able to completely describe the ramification behavior of $M_p$ over $M(1)$, from which we will derive an exact genus formula. In particular, the genus is asymptotic to $\frac{1}{12}p^2$ and has genus at least 2 for $p\ge 11, p\notin\bE_\bgs$. By Faltings' theorem this implies that for large $p$ and any number field $K$, there exist only finitely many elliptic curves which admit an $\SL_2(\bF_p)$ cover only branched over the origin with ramification index $2p$.


In this introduction we will begin in \S\ref{ss_admissible_covers_intro} by describing the main congruence, both geometrically and combinatorially. Next in \S\ref{ss_applications_intro} we will describe the applications to the Markoff equation and the connectedness of certain Hurwitz stacks. In \S\ref{ss_summary_of_paper_intro} we will summarize the paper and sketch the proof of the congruence. In \S\ref{ss_further_directions} we will describe some further directions and related work.

\subsection{Admissible covers, Nielsen equivalence, and the main congruence}\label{ss_admissible_covers_intro}

We will need to work with a natural compactification of the Hurwitz stacks by ``admissible covers'' as introduced by Harris and Mumford \cite{HM82} and reinterpreted by Abramovich, Corti, and Vistoli \cite{ACV03}. We will work with the stacks $\cAdm(G)$ of admissible $G$-covers of stable pointed curves of genus 1 (henceforth called \emph{1-generalized elliptic curves}), where $G$ is a finite group. These stacks are well behaved over $\bZ[1/|G|]$, but for our purposes it will suffice to work over an algebraically closed field $k$ of characteristic 0. The following discussion takes place over such a field $k$.


Let $\ff : \cAdm(G)\rightarrow\ol{\cM(1)}$ be the ``forgetful map'', where $\ol{\cM(1)}$ denotes the compactification of $\cM(1)$ by stable pointed curves. The stack $\cAdm(G)$ is a smooth proper Deligne-Mumford stack, and the substack $\cAdm^0(G)$ classifying smooth $G$-covers of elliptic curves only branched over the origin is open and dense. In particular, the inclusion $\cAdm^0(G)\subset\cAdm(G)$ induces a bijection on connected components. If $\fc$ is a conjugacy class of $G$, then the substack $\cAdm(G)_\fc\subset\cAdm(G)$ classifying covers with Higman invariant $\fc$ is open and closed. In this paper we will establish, for any connected component $\cX\subset\cAdm(G)_\fc$ with coarse scheme $X$, a divisibility theorem on the degree of the finite flat forgetful map $\ff : X\rightarrow \ol{M(1)}$. For example, we will show:

\begin{thm}[See Corollary \ref{cor_vdovin_2}]\label{thm_vdovin_intro} Let $G$ be a finite group. Let $c\in G$ and let $\fc$ be its conjugacy class. Let $\cX\subset\cAdm(G)_\fc$ be a connected component with coarse scheme $X$, parametrizing covers with Higman invariant $\fc$. For a prime $\ell$, write $r := \ord_\ell(|c|)$. Then we have
\begin{itemize}
\item[(a)] Write $\ord_\ell(|G|) = r+s$, and let $j\ge 0$ be an integer such that $G$ does not contain any proper normal subgroup of order divisible by $\ell^{j+1}$. If $r \ge 3s + j$, then
$$\deg(X\rightarrow \ol{M(1}))\equiv 0\mod \ell^{\ceil{\frac{r-3s-j}{2}}}$$
\item[(b)] Suppose $G$ is nonabelian and simple. If $\ell^{r+1}\ge |G|^{1/3}$ and $G$ is not isomorphic to $\PSL_2(\bF_q)$ for any $q$, then
$$\deg(X\rightarrow \ol{M(1)})\equiv 0\mod \ell^{\ceil{\frac{r}{2}}}$$
\end{itemize}
\end{thm}
For a more precise version of this theorem, see Theorems \ref{thm_congruence}, \ref{thm_combinatorial_congruence}. Note that the $G$-stabilizer of any ramified point is generated by a conjugate of $c$, so $|c|$ is the ramification index at any ramified point. Also note that in case (a), if $\ord_\ell(|c|) = \ord_\ell(|G|)$ and $G$ does not contain any proper normal subgroups of order divisible by $\ell$, then we obtain a congruence $\equiv 0\mod \ell$. This is the case for example when $c = \spmatrix{-1}{1}{0}{-1}\in\SL_2(\bF_\ell)$ for $\ell\ge 3$. This case leads to our main result on Markoff triples (see \S\ref{sss_bgs_conjecture_intro}).


The forgetful map $\ff : \cAdm(G)\rightarrow\ol{\cM(1)}$ restricts to an \'{e}tale map $\cAdm^0(G)\rightarrow\cM(1)$, which in turn factors as
$$\cAdm^0(G)\rightarrow\cM(G)\rightarrow\cM(1)$$
where $\cM(G)$ is the moduli stack of elliptic curves with $G$-structures (see \S\ref{ss_comparison_with_G_structures} below). Here the first map is an \'{e}tale gerbe (and hence induces a homeomorphism on topological spaces) and the second map is finite \'{e}tale. Thus the sets of connected components of $\cAdm(G),\cAdm^0(G)$, and $\cM(G)$ are in natural bijection with each other. We will use this fact repeatedly without mention. Since $\cM(G)\rightarrow\cM(1)$ is finite \'{e}tale, the components of $\cAdm(G)$ (equivalently $\cM(G)$, $\cAdm^0(G)$) can be characterized combinatorially using Galois theory. For example, if $\Pi$ denotes the fundamental group of a once-punctured orientable surface of genus 1, then the geometric fibers of $\cM(G)$ above $\cM(1)$ are in bijection with the set $\Epi^\ext(\Pi,G)$ (see \eqref{eq_epi_ext_intro}); note that $\Pi$ is free of rank 2. Let $\Out^+(\Pi)$ be the index 2 subgroup of $\Out(\Pi) := \Aut(\Pi)/\Inn(\Pi)$ represented by automorphisms which act on $\Pi/[\Pi,\Pi] \cong \bZ^2$ with determinant 1. Then the fundamental group of $\cM(1)$ is the profinite completion of $\Out^+(\Pi)$, the connected components of $\cM(G)$ are in bijection with the orbits of the natural action of $\Out^+(\Pi)$ on $\Epi^\ext(\Pi,G)$, and the degree of the induced map on coarse schemes $\Adm(G)\rightarrow M(1)$ is either equal to the cardinality of the orbit, or twice the cardinality (Theorem \ref{thm_basic_properties}). From this we deduce the following combinatorial interpretation of Theorem \ref{thm_vdovin_intro}:

\begin{thm}\label{thm_combinatorial_vdovin_intro} Let $G$ be a finite group.	 Let $c\in G$, and let $\fc$ be its conjugacy class. Fix a basis $(a,b)$ for $\Pi$. Let $\varphi : \Pi\rightarrow G$ be a surjection satisfying $\varphi([a,b])\in\fc$. Let $\Out^+(\Pi)\cdot\varphi$ denote the orbit of $\varphi$ in $\Epi^\ext(\Pi,G)$. For a prime $\ell$, write $r := \ord_\ell(|c|)$. Then we have
\begin{itemize}
\item[(a)] Write $\ord_\ell(|G|) = r+s$, and let $j\ge 0$ be an integer such that $G$ does not contain any proper normal subgroup of order divisible by $\ell^{j+1}$. If $r\ge 3s+j$, then
$$|\Out^+(\Pi)\cdot\varphi|\equiv 0\mod \ell^{\ceil{\frac{r-3s-j}{2}}}$$
\item[(b)] Suppose $G$ is nonabelian and simple. If $\ell^{r+1}\ge |G|^{1/3}$ and $G$ is not isomorphic to $\PSL_2(\bF_q)$ for any $q$, then
$$|\Out^+(\Pi)\cdot\varphi|\equiv 0\mod \ell^{\ceil{\frac{r}{2}}}$$
\end{itemize}
Let $\gamma_{-I}\in\Aut(\Pi)$ be the automorphism $(a,b)\mapsto (a^{-1},b^{-1})$. If $\varphi\circ\gamma_{-I} = \varphi$, then the congruences mod $\ell^k$ can be strengthened to a congruence mod $2\ell^k$.
\end{thm}

Since the congruences give at best $\equiv 0\mod |c|$, we find that the results above are only interesting when $G$ is nonabelian and generated by two elements. Note that when $G$ is abelian, the sizes of the $\Out^+(\Pi)$-orbits of $\varphi : \Pi\rightarrow G$ correspond to the indices of certain congruence subgroups inside $\SL_2(\bZ)$, which are well-understood. When $G = \bZ/n\bZ$, the action is transitive and the size of the single orbit is the index $[\SL_2(\bZ):\Gamma_1(n)]$; when $G = (\bZ/n\bZ)^2$, the action has $\phi(n)$ components, and the size of each orbit is $[\SL_2(\bZ) : \Gamma(n)] = |\SL_2(\bZ/n\bZ)|$ \cite[\S3.9]{DS06}.


Having fixed a basis $(a,b)$ for $\Pi$, we can view $\Epi^\ext(\Pi,G)$ as the set of conjugacy classes of generating pairs of $G$. From this point of view, in combinatorial group theory the $\Out(\Pi)$-orbits on $\Epi^\ext(\Pi,G)$ are called \emph{Nielsen equivalence classes of generating pairs of $G$}. In general, if $F_r$ denotes a free group of rank $r$, then the problem of Nielsen equivalence asks for a description of the orbits of $\Out(F_r)$ on $\Epi^\ext(F_r,G)$. Similarly, one can ask for a description of the orbits of $\Out(F_r)$ on $\Epi(F_r,G)/\Aut(G)$, in which case the orbits are called $T_r$-systems. Such problems arose in the 1950's in the study of group presentations, but have recently garnered renewed interest due to their relevance to the \emph{product replacement algorithm} for generating random elements of finite groups \cite{Pak00, Lub11}. Let $d(G)$ denote the minimum cardinality of a generating set of $G$. When $r\ge d(G)+1$, the expectation is

\begin{conj}[{\cite[Conjecture 1]{Gar08}}]\label{conj_wiegold} Let $G$ be a finite group. If $r\ge d(G)+1$, then $\Out(F_r)$ acts transitively on $\Epi^\ext(F_r,G)$.	
\end{conj}

When $G$ is finite simple, this conjecture is attributed to Wiegold, and dates back to the 1970s. The conjecture is known if $G$ is solvable \cite{Dun70}, or if $r\ge \log_2(|G|)$ \cite[Corollary 3.3]{Lub11} (this can be viewed as an analog of branch/genus stabilization). However when $r = d(G) = 2$, the action is rarely transitive, and the first conjectural complete description of $T_2$-systems for a family of perfect groups was not given until 2013 by McCullough and Wanderley \cite{MW13}. For $G = \SL_2(\bF_q)$, they define a map ``the trace invariant'', denoted $\tau$
$$\begin{array}{rrcl}
\tau : & \Epi(\Pi,\SL_2(\bF_q)) & \lra & \bF_q \\
 & \varphi & \mapsto & \tr\varphi([a,b])
\end{array}$$
and they conjecture that the trace invariant completely describes $T_2$-systems for $\SL_2(\bF_q)$:
\begin{conj}[{``$T$-classification conjecture'' \cite{MW13}}]\label{conj_MW_intro} The trace invariant map $\tau$ induces a bijection from the set of $T_2$-systems on $\SL_2(\bF_q)$ onto its image in $\bF_q$.
\end{conj}
Moreover, they give a complete description of the image of $\tau$, and computationally verify their conjecture for all $q\le 131$. When $q = p$ is prime and $t = -2$, we will be able to resolve this conjecture for all but finitely many primes $p$ (Theorem \ref{thm_bgs_conj_intro}).

\subsection{Applications}\label{ss_applications_intro}

From now on we will let $\Pi$ be a free group of rank 2, with basis $a,b$. For a field $k$ and $t\in k$, write $\Hom(\Pi,\SL_2(k))_t$ to denote the set of homomorphisms $\varphi : \Pi\rightarrow\SL_2(k)$ satisfying $\tr\varphi([a,b]) = t$. There is a beautiful theory of the character variety for $\SL_2$-representations of $\Pi$, which we review in \S\ref{ss_character_variety}. A consequence of this theory is that for $t\in\bF_q - \{2\}$, the map
\begin{equation}\label{eq_character_variety_intro}
\begin{array}{rcl}
\Hom(\Pi,\SL_2(\bF_q))_t/\GL_2(\bF_q) & \lra & \{(x,y,z)\in\bF_q^3\;|\; x^2+y^2+z^2-xyz-2 = t\}	\\
\varphi & \mapsto & (\tr\varphi(a),\tr\varphi(b),\tr\varphi(ab))
\end{array}
\end{equation}
where $\GL_2(\bF_q)$ acts by conjugation on $\SL_2(\bF_q)$, is a \emph{bijection}. Moreover, the natural action of $\Aut(\Pi)$ on the source induces an action on the target by \emph{polynomial automorphisms} of $\bA^3$ preserving the hypersurface $x^2+y^2+z^2-xyz-2=t$. Let $\Gamma\subset\Aut(\bA^3)$ be the subgroup induced by the action of $\Aut(\Pi)$. Then $\Gamma$ is generated by permutations of the coordinates and the ``Vieta'' involution $(x,y,z)\mapsto (x,y,xy-z)$. By Galois theory the bijection \eqref{eq_character_variety_intro} provides a dictionary between the geometric properties of the moduli stacks $\cM(G)$ for subgroups $G\le\SL_2(\bF_q)$ and the properties of the $\Gamma$-action on the $\bF_q$-points of the surface $x^2 + y^2 + z^2 - xyz -2 = t$.

\subsubsection{Markoff triples and the conjecture of Bourgain, Gamburd, and Sarnak}\label{sss_bgs_conjecture_intro}

When $t = -2$, the surface appearing in \eqref{eq_character_variety_intro} is the \emph{Markoff surface} $\bX$ given by\footnote{Typically the Markoff surface is given by $\bM : x^2+y^2+z^2-3xyz = 0$. However the map $(x,y,z)\mapsto (3x,3y,3z)$ defines a map $\bM\rightarrow\bX$ which is an isomorphism away from 3, and gives a bijection on $\bZ$-points, so for our purposes it is harmless (and more convenient) to work with $\bX$. In \S\ref{ss_bgs_conjecture}, we will make this distinction precise.}
$$\bX : x^2 + y^2 + z^2 - xyz = 0.$$
which we view as an affine hypersurface in $\bA^3_\bZ$. For primes $p$, let $\bX^*(p) := \bX(\bF_p) - \{(0,0,0)\}$. It follows from the classification of subgroups of $\SL_2(\bF_p)$ that for odd $p$, the bijection \eqref{eq_character_variety_intro} restricts to give a bijection
\begin{equation}\label{eq_character_variety_intro_2}
\Epi(\Pi,\SL_2(\bF_p))_{-2}/\GL_2(\bF_p)\rightiso\bX^*(p)
\end{equation}

This surface first appeared in the work of Markoff \cite{Mar79, Mar80} where by a descent argument he showed
\begin{thm}[Markoff \cite{Mar79,Mar80}]\label{thm_markoff_transitivity} The positive integer solutions to $\bX$ form a single orbit under the action of $\Gamma$.
\end{thm}

In \cite{BGS16,BGS16arxiv}, Bourgain, Gamburd, and Sarnak studied the action of $\Gamma$ on $\bX^*(p)$ (for $p$ prime) and conjectured that an analogue of Markoff's result should also hold for its $\bF_p$-points:

\begin{conj}[Bourgain, Gamburd, Sarnak \cite{BGS16}]\label{conj_bgs_intro} For all primes $p$, $\Gamma$ acts transitively on $\bX^*(p)$.
\end{conj}

In other words, they conjecture that $\bX(\bF_p)$ is the union of at most two $\Gamma$-orbits: the singleton orbit $\{(0,0,0)\}$, and $\bX^*(p)$ (this is empty if and only if $p = 3$). Since $\bX(\bZ)$ contains the solutions $\{(0,0,0),(3,3,3)\}$ and $\bX(\bF_3) = \{(0,0,0)\}$, a positive solution to this conjecture would imply that for any prime $p$, the reduction mod $p$ map $\bX(\bZ)\rightarrow\bX(\bF_p)$ is surjective, in which case, following \cite{BGS16}, we say that $\bX$ satisfies \emph{strong approximation} at all primes $p$.


To the author's knowledge this conjecture was first posed by Baragar in his 1991 PhD thesis \cite[\S V.3]{Bar91}. Since the conjugation action of $\GL_2(\bF_p)$ on $\SL_2(\bF_p)$ generates the full automorphism group $\Aut(\SL_2(\bF_p))$, it follows that this is also a special case of McCullough and Wanderley's Conjecture \ref{conj_MW_intro} with $(q,t) = (p,-2)$. However, the first major progress on Conjecture \ref{conj_bgs_intro} did not come until the work of Bourgain, Gamburd, and Sarnak in 2016 \cite{BGS16arxiv}, where they were able to prove

\begin{thm}[{Bourgain, Gamburd, Sarnak \cite{BGS16arxiv}}]\label{thm_bgs_intro} The following are true.
\begin{itemize}
\item[(a)] Let $\bE_\bgs$ denote the set of primes for which $\Gamma$ does not act transitively on $\bX^*(p)$. For any $\epsilon > 0$, we have
$$\#\{p\in\bE_\bgs\;|\;p\le x\} = O(x^\epsilon)$$
\item[(b)] Let $\cC(p)$ be the largest orbit of $\Gamma$ on $\bX^*(p)$. Then for any $\epsilon > 0$, we have
$$|\bX^*(p) - \cC(p)|\le p^\epsilon\qquad\text{for large $p$}$$
\end{itemize}
\end{thm}

Thus, part (a) says that their conjecture holds for all but a sparse (but possibly infinite) set of primes, and part (b) says that even if it were to fail, it cannot fail too horribly. By the above discussion, the bijection \eqref{eq_character_variety_intro_2} yields the following equivalent group-theoretic translation of their result:
\begin{thm}[{Bourgain, Gamburd, Sarnak}]\label{thm_combinatorial_bgs_intro} For $t\in\bF_p$, let $\Epi(\Pi,\SL_2(\bF_p))_t$ denote the subset of surjections which satisfy $\tr\varphi([a,b]) = t$.
\begin{itemize}
\item[(a)] Let $\bE_\bgs$ denote the set of primes for which $\Out(\Pi)$ does not act transitively on $\Epi(\Pi,\SL_2(\bF_p))_{-2}/\GL_2(\bF_p)$. For any $\epsilon > 0$, we have
$$\#\{p\in\bE_\bgs\;|\;p\le x\} = O(x^\epsilon)$$
\item[(b)] Let $\cC(p)$ be the largest orbit of $\Out(\Pi)$ on $\Epi(\Pi,\SL_2(\bF_p))_{-2}/\GL_2(\bF_p)$. Then for any $\epsilon > 0$, we have
$$|\Epi(\Pi,\SL_2(\bF_p))_{-2}/\GL_2(\bF_p) - \cC(p)|\le p^\epsilon\qquad\text{for large $p$}$$
\end{itemize}
\end{thm}

By Proposition \ref{prop_Higman_vs_trace}, for any surjection $\varphi : \Pi\rightarrow\SL_2(\bF_p)$ satisfying $\tr\varphi([a,b]) = -2$, $\varphi([a,b])$ must be conjugate to a matrix of the form $\spmatrix{-1}{u}{0}{-1}$ for $u\in\bF_p^\times$. In particular if $p\ge 3$ then $\varphi([a,b])$ has order $2p$. Then Theorem \ref{thm_combinatorial_vdovin_intro}(a) implies
\begin{thm}[See Theorem \ref{thm_main_divisibility}]\label{thm_orbit_congruence_intro} For every $p\ge 3$, every $\Out^+(\Pi)$-orbit on $\Epi(\Pi,\SL_2(\bF_p))_{-2}/\GL_2(\bF_p)$ has cardinality divisible by $p$. In particular, every $\Gamma$-orbit on $\bX^*(p)$ has cardinality divisible by $p$.
\end{thm}

In particular, every $\Gamma$-orbit must have size at least $p$, so combined with Theorem \ref{thm_bgs_intro}(b), this establishes Conjecture \ref{conj_bgs_intro} for all but a finite (and in fact explictly bounded) set of primes:

\begin{thm}[{Main application; see Theorems \ref{thm_bgs_conj}, \ref{thm_main_applications}}]\label{thm_bgs_conj_intro} The following are true.
\begin{itemize}
\item[(a)] The ``exceptional'' set $\bE_\bgs$ is finite and explicitly bounded.
\item[(b)] For all primes $p\notin \bE_\bgs$, $\bX$ satisfies strong approximation at $p$.
\end{itemize}
\end{thm}

Here, by explicitly bounded, we mean that it is possible to write down an explicit upper bound on the largest prime in $\bE_\bgs$. This follows from the fact that the methods of \cite{BGS16arxiv} are effective, and an explicit upper bound was obtained by Elena Fuchs\footnote{Private communication. A preliminary upper bound is approximately $10^{5547000}$.}. In particular, we have effectively reduced the conjecture to a finite computation. The conjecture has been verified for all primes $p < 3000$ by de-Courcy-Ireland and Lee in \cite{IL20}, so it is reasonable to expect that the final computation will indeed verify Conjecture \ref{conj_bgs_intro}. If so, then it would follow from Proposition \ref{prop_empty_conics} below that there are no congruence constraints on Markoff numbers mod $p$ other than the ones first noted by Frobenius in 1913 \cite{Frob13}, namely that if $p\equiv 3\mod 4$ and $p\ne 3$ then a Markoff number cannot be $\equiv 0,\frac{\pm 2}{3}\mod p$. Using the work of Meiri-Puder \cite{MP18}, one can also deduce strong approximation mod $n$ for most squarefree integers $n$ (Theorem \ref{thm_composite_moduli}).


Our methods also give congruences for the generalized Markoff equations $x^2 + y^2 + z^2 = t + 2$.

\begin{thm}[Congruences for generalized Markoff equations - see Theorems \ref{thm_congruence_on_orbits} and \ref{thm_general_congruence}]\label{thm_general_congruence_intro} Let $q\ge 3$ be a prime power. Let $t\in\bF_q - \{2,-2\}$. Then we may write $t = \omega+\omega^{-1}$ for some $\omega\in\bF_{q^2}^\times - \{-1,1\}$, and the set $\{\omega,\omega^{-1}\}$ is uniquely determined by $t$. Let $P = (A,B,C)$ be an $\bF_q$-point of the hypersurface defined by
\begin{equation}\label{eq_generalized_Markoff}
x^2 + y^2 + z^2 - xyz = t + 2
\end{equation}
Let $\Gamma$ act on the hypersurface \eqref{eq_generalized_Markoff} via the same formulas as for its action on $\bX$. Suppose at least two of $\{A,B,C\}$ are nonzero. Let $\ell$ be an odd prime. Let $r := \ord_\ell(|\omega|)$ and write $\ord_\ell(|\SL_2(\bF_q)|) = \ord_\ell(q^3-q) = r+s$. Then the $\Gamma$-orbit of $P$ has cardinality $\equiv 0\mod\ell^{\max\{r-s,0\}}$.
\end{thm}

See Theorem \ref{thm_congruence_on_orbits} for a stronger and more precise result. Here, the case $t = -2$ gives the Markoff equation, which was addressed in Theorem \ref{thm_orbit_congruence_intro}; the case $t = 2$ is called the \emph{Cayley cubic} and is addressed in \cite[\S5]{IL20}. We remark that our methods fail to give anything new for the Cayley cubic. The reason is that the $\bF_q$-points of the Cayley cubic correspond to representations with image contained in a Borel subgroup of $\SL_2(\bF_q)$, and many of our results break down when the image is close to abelian (see \S\ref{ss_remarks_on_md} for the case of representations with dihedral image). As announced in \cite{BGS16}, the methods of \cite{BGS16arxiv} also carry over to yield results analogous to Theorem \ref{thm_bgs_intro} for these more general equations \eqref{eq_generalized_Markoff} (with an appropriately defined $\bX^*(p)$). Thus, combined with Theorem \ref{thm_general_congruence_intro}, it may be possible to establish conjectures analogous to Conjecture \ref{conj_bgs_intro} for these more general equations as well.

\subsubsection{Connectedness of Hurwitz stacks}

Recall that $\cAdm^0(G)$ denotes the stack of smooth $G$-covers of elliptic curves, only ramified above the origin. If $G\le\GL_2(\bF_q)$ and $t\in\bF_q$, we write $\cAdm^0(G)_t$ to denote the open and closed substack classifying $G$-covers whose Higman invariant has trace $t$. When $G = \SL_2(\bF_p)$ and $t = -2$, there are two possible Higman invariants with trace $-2$; they are represented by $\spmatrix{-1}{1}{0}{-1},\spmatrix{-1}{a}{0}{-1}$ where $a\in\bF_p^\times$ is a non-square, and they are swapped by the conjugation action of $\GL_2(\bF_p)$. Let us write $\fc_1,\fc_2$ for these two conjugacy classes. 


From the combinatorial characterization of the components of $\cAdm(G)$ described in \S\ref{ss_admissible_covers_intro}, we find that for $p\ge 5$, $p\notin\bE_\bgs$, the substack $\cAdm(\SL_2(\bF_p))_{-2}$ is a disjoint union of two components, corresponding to the two non-central conjugacy classes $\fc_1,\fc_2$ of trace $-2$. In other words, for $i = 1,2$, the stack $\cAdm(\SL_2(\bF_p))_{\fc_i}$ is \emph{connected}. In light of Conjecture \ref{conj_MW_intro}, we expect this to hold for \emph{any} conjugacy class of $\SL_2(\bF_p)$ (this can also be interpreted in terms of the $\Gamma$-action on the generalized Markoff equations).

\begin{remark} This connectedness result can be phrased in another way. The set $\Epi(\Pi,\SL_2(\bF_p))_{-2}/\GL_2(\bF_p)$ is in bijection with the fiber over $\cM(1)$ of the stack $\cM(\SL_2(\bF_p))_{-2}^\abs$ of elliptic curves with ``absolute $\SL_2(\bF_p)$-structures''; the transitivity of the $\Out^+(\Pi)$-action in Theorem \ref{thm_combinatorial_bgs_intro} is equivalent to the connectedness of this stack.
\end{remark}

Using the coordinatization of the monodromy action given by the character variety, we are able to compute the ramification behavior of the coarse scheme of $\cAdm(\SL_2(\bF_p))_{fc_1}$ over $\ol{M(1)}$. Applying Riemann-Hurwitz, we are able to obtain an explicit formula for the genus (Theorem \ref{thm_genus_formula}). For example, we will show:

\begin{thm}[Theorem \ref{thm_genus_formula}, \ref{thm_finiteness_of_SL2p_structures}] Let $p$ be an odd prime not in the finite set $\bE_\bgs$. The stacks $\cAdm(\SL_2(\bF_p))_{\fc_i}$ for $i = 1,2$ are isomorphic and connected. Let $\ol{M_p}$ be the coarse scheme of $\cAdm(\SL_2(\bF_p))_{\fc_i}$ for either $i = 1,2$. For a positive integer $n$, let $\Phi(n) := \sum_{d\mid n}\frac{\phi(d)}{d}$, where $\phi$ is the Euler function. Then we have
$$\genus(\ol{M_p}) = \frac{1}{12}p^2 - \frac{p-1}{4}\Phi(p-1) - \frac{p+1}{4}\Phi(p+1) + \epsilon(p)$$
where $\epsilon(p)\sim p$ (for $p$ large); it is an explicit function depending on the residue class of $p\mod 8$ which we define in Theorem \ref{thm_genus_formula}. Moreover, for $p\ge 13$, $\genus(\ol{M_p})\ge 2$. For $p = 5,7,11$, $\ol{M_p}$ has genus $0,0,1$ respectively. In particular, by Falting's theorem, for any $p\ge 13$, $p\notin\bE_\bgs$, and any number field $K$, only finitely many elliptic curves admit a $\SL_2(\bF_p)$-cover with ramification index $2p$.
\end{thm}

One interpretation of this result is that there exist only finitely many noncongruence modular curves of a given genus classifying elliptic curves with $\SL_2(\bF_p)$-covers with ramification index $2p$. This can be viewed as an analog of Rademacher's conjecture, proved by Dennin \cite{Den75}, that there are only finitely many congruence modular curves of a given genus. The same statement is false for noncongruence modular curves, so to obtain finiteness one needs to restrict the types of noncongruence modular curves considered. This theorem gives finiteness for the curves $\ol{M_p}$, which are shown to be noncongruence for a density 1 set of primes $p$ (Corollary \ref{cor_noncongruence}).

\subsection{Summary of the paper}\label{ss_summary_of_paper_intro}

\subsubsection{Admissible covers and $G$-structures}
In \S\ref{section_admissible_G_covers}, we give an overview of the moduli stacks we will work with. In \S\ref{ss_admissible_G_covers}, we define admissible $G$-covers and the stacks $\cAdm(G)$, following \cite{ACV03}. In \S\ref{ss_ramification_divisor} and \S\ref{ss_higman_invariant}, we define the reduced ramification divisor and the Higman invariant of an admissible $G$-cover and prove some basic results. Crucially, the reduced ramification divisor of an admissible $G$-cover $\pi : C\rightarrow E$ over $S$ is \emph{finite \'{e}tale} over $S$, $G$ acts transitively on its connected components, and the Galois group of each component can be controlled group theoretically. In \S\ref{ss_comparison_with_smGc}, we show that an admissible $G$-cover can be equivalently characterized as a curve equipped with a $G$-action satisfying certain properties. While the properness of the stacks $\cAdm(G)$ are crucial to the proof of our main congruence, it is somewhat inconvenient that the forgetful map $\cAdm(G)\rightarrow\ol{\cM(1)}$ is not finite, even when restricted to $\cAdm^0(G)$. In \S\ref{ss_comparison_with_G_structures}, we compare the stacks $\cAdm^0(G)$ to the moduli stacks $\cM(G)$ of elliptic curves with $G$-structures \cite{Chen18, Ols12,JP95,DM69}, whose forgetful map to $\cM(1)$ is finite \'{e}tale, and hence can be studied using Galois theory. The main result is that there is a map $\cAdm^0(G)\rightarrow\cM(G)$ which is an \'{e}tale gerbe with group $Z(G)$; in particular it induces an isomorphism on coarse schemes, so the connected components of $\cAdm(G)$ also have a group theoretic interpretation via the Galois correspondence applied to $\cM(G)\rightarrow\cM(1)$. This also allows us to define a compactification $\ol{\cM(G)}$ of $\cM(G)$.

\subsubsection{The main congruence and sketch of the argument}\label{sss_sketch_of_proof}

The purpose of section \S\ref{section_main_result} is to prove the base form of the main congruence (Theorem \ref{thm_congruence}) from which we will deduce Theorem \ref{thm_vdovin_intro}. In \S\ref{ss_dualizing_sheaf}-\ref{ss_degree_formalism}, we review some standard results and describe their extensions to the setting of stacks. In \S\ref{ss_congruence}, we prove the main congruence. This congruence essentially comes from noting that for any component $X\subset\Adm(G)$, modulo some technical considerations, the pullback of a certain line bundle on $\ol{M(1)}$ to $X$ is an $e$th tensor power, where $e$ is the ramification index of covers classified by $X$, and hence the forgetful map $X\rightarrow\ol{M(1)}$ must have degree which is divisible by $e$. We thank Johan de Jong for showing us this idea. Here we briefly sketch the argument. We work over an algebraically closed field $k$ of characteristic 0.


We begin by presenting the prototype of the argument. Suppose $f : X\rightarrow X(1)$ is a finite map of (connected) smooth proper curves, $E(1)\rightarrow X(1)$ is a family of 1-generalized elliptic curves (stable pointed curves of genus 1) with zero section $\sigma_O'$. Let $E\rightarrow X$ be its pullback via $f$ with section $\sigma_O$. Suppose we are given a diagram
\[\begin{tikzcd}
	C\ar[r,"\pi"]\ar[rd] & E\ar[d]\ar[r,"\tilde{f}"] & E(1)\ar[d] \\
	 & X\ar[r,"f"]\ar[ul,bend left = 30,"\sigma"]\ar[u,bend right = 30, "\sigma_O"'] & X(1)\ar[u,bend right = 30, "\sigma_O'"']
\end{tikzcd}\]
where $\pi\circ\sigma = \sigma_O$ and $\pi$ is an admissible $G$-cover, only branched above $\sigma_O$, where it has ramification index $e$. This implies that $C/X$ together with the ramification divisor of $\pi$ is a stable marked curve. Here, $\sigma$ is to be viewed as a ramified section of $\pi$, and we require that it lands in the smooth locus. It follows from the \'{e}tale local description of admissible $G$-covers (see \S\ref{ss_restriction}) that

\begin{equation}\label{eq_key_isom}
\sigma_O^*\Omega_{E/X} = \sigma^*\pi^*\Omega_{E/X}\cong (\sigma^*\Omega_{C/X})^{\otimes e}.
\end{equation}

Since cotangent sheaves commute with base change, we also have
$$\sigma_O^*\Omega_{E/X} = \sigma_O^*\tilde{f}^*\Omega_{E(1)/X(1)} = f^*\sigma_O'^*\Omega_{E(1)/X(1)}.$$
Let $\lambda := \sigma_O'^*\Omega_{E(1)/X(1)}$, then since $f$ is finite flat, we have
$$\deg f\cdot\deg\lambda = \deg f^*\lambda = e\cdot\deg(\sigma^*\Omega_{C/X})$$
and since degrees are integers, we would get
\begin{equation}\label{eq_prototype_congruence}
\deg f\equiv 0\mod \frac{e}{\gcd(e,\deg\lambda)}	.
\end{equation}
This is the essence of the main congruence. We wish to run this prototype argument when $E(1)\rightarrow X(1)$ is the universal family $\cE(1)\rightarrow\ol{\cM(1)}$, $X$ is a component $\cX\subset\cAdm(G)$, $f$ is the forgetful map $\ff : \cX\rightarrow\ol{\cM(1)}$, and $\pi : C\rightarrow E$ is the universal family $\pi : \cC\rightarrow\cE$ over $\cX$. In this case, we find that $\lambda$ is the Hodge bundle, which has degree $\frac{1}{24}$ (Proposition \ref{prop_hodge_bundle}). However, in this setting we are immediately presented with two difficulties:

\begin{itemize}
\item The universal admissible $G$-cover	 $\pi : \cC\rightarrow\cE$ may not admit a ramified section $\sigma$.

The reduced ramification divisor $\cR_\pi$ of the universal admissible $G$-cover $\pi : \cC\rightarrow\cE$ is always finite \'{e}tale over $\cX$, but it may not admit a section. If $\cR\subset\cR_\pi$ is a component, then $\cR\rightarrow\cX$ is the minimal extension over which $\pi : \cC\rightarrow\cE$ admits a ramified section. Thus to apply the prototype argument, we must make a further base change to $\cR$. This has the potential to weaken the resulting congruence. Let $d_\cX$ denote the degree of the map on coarse schemes $R\rightarrow X$ induced by $\cR\rightarrow\cX$. Over $\cR$, the universal cover $\pi : \cC_\cR\rightarrow\cE_\cR$ admits a ramified section $\sigma$.

\item Since $\cR$ is a stack, the degree of $\sigma^*\Omega_{\cC_\cR/\cR}$ may not be an integer.

Since $\cX$ is Deligne-Mumford, at least we have $\deg\sigma^*\Omega_{\cC_\cR/\cR}\in\bQ$. For a geometric point $x : \Spec k\rightarrow\cX$, $\sigma^*\Omega_{\cC_\cR/\cR}$ defines a rank 1 representation of $\Aut_\cX(x)$, which we call the \emph{local character} at $x$. Using a theorem of Olsson \cite{Ols12}, the denominator of the rational number $\deg\sigma^*\Omega_{\cC_\cR/\cR}$ can be bounded in terms of the orders of the local characters. Since $\cX$ is Noetherian, let $m_\cX$ denote the minimum positive integer such that $(\sigma^*\Omega_{\cC_\cR/\cR})^{\otimes m_\cX}$ has trivial local characters.

\end{itemize}

The integers $d_\cX,m_\cX$ defined above quantify the obstructions to achieving a congruence of the form \eqref{eq_prototype_congruence}. Since the forgetful map $\ff : \cX\rightarrow\ol{\cM(1)}$ is generally not finite, we phrase the congruence in terms of the induced finite map on coarse schemes $f : X\rightarrow \ol{M(1)}$. The purest form of our main result is:

\begin{thm}[Main congruence; see Theorem \ref{thm_congruence}]\label{thm_main_congruence_intro} Let $\cX\subset\cAdm(G)$ be a connected component classifying $G$-covers of elliptic curves with ramification index $e$ above the origin. Let $d_\cX,m_\cX$ be as above. Then the forgetful map $f : X\rightarrow\ol{M(1)}$ satisfies
$$\deg(X\stackrel{f}{\lra}\ol{M(1)}) \equiv 0\mod \frac{12e}{\gcd(12e,m_\cX d_\cX)}$$
\end{thm}
In this form, the congruence is difficult to use; one must first check that $m_\cX,d_\cX$ do not share too many divisors with $e$, so we are motivated to give more easily accessible bounds for $m_\cX,d_\cX$. If $c\in G$ represents the Higman invariant of covers classified by $\cX$, then by Proposition \ref{prop_stacky_RRD}, $d_\cX$ must divide $|C_G(c)/\langle c\rangle|$, though we do not know how sharp this bound is. The integer $m_\cX$ is more accessible. It is related to the order of automorphism groups of geometric points of $\cR$. For points corresponding to smooth covers, these automorphism groups are relatively easy to control (see Proposition \ref{prop_smooth_pointed_vertical_automorphisms}). For points lying over the boundary of $\cAdm(G)$, the possible non-irreducibility of degenerate covers complicates the situation. The purpose of \S\ref{section_cusps} is to give a combinatorial characterization of such degenerate covers, and to express their automorphism groups in terms of combinatorial data. This will yield an integer $m_\cX'$, defined purely combinatorially, such that $m_\cX$ divides $m_\cX'$, and such that $m_\cX'$ differs from $m_\cX$ by at most a factor of 12.

\subsubsection{Combinatorial characterization of the cusps}
In \S\ref{section_cusps} we give a combinatorial characterization of the boundary of $\cAdm(G)$ using Galois theory. Borrowing terminology from the classical theory of the moduli of elliptic curves, we call points lying on the boundary \emph{cusps}, and we call the objects they correspond to \emph{cuspidal objects}. We work over an algebraically closed field $k$ of characteristic 0. The approach is as follows: let $E$ be a ``nodal elliptic curve'' (a degenerate stable 1-pointed curve of genus 1). If $\pi : C\rightarrow E$ denotes an admissible $G$-cover, then taking normalizations, we obtain a finite $G$-cover $\pi' : C'\rightarrow E'$, only branched over 3 points, where $C'$ is smooth but possibly disconnected. The normalization maps $C'\rightarrow C$ and $E'\rightarrow E$ moreover equip $\pi'$ with the data of a $G$-equivariant bijection $\alpha$ of the fibers above the preimages in $E'$ of the node of $E$. We will show that giving $\pi$ is equivalent to giving $(\pi',\alpha)$, and describe the subcategory of pairs $(\pi',\alpha)$ which correspond to admissible $G$-covers of $E$. In \S\ref{ss_the_map_Inv} we will attach to any such object a label ``the $\delta$-invariant'', which is an equivalence class of a generating pair of $G$. In \S\ref{ss_cuspidal_automorphisms}, we will calculate the automorphism group of a cuspidal object from its $\delta$-invariant. This allows us to define in purely combinatorial terms the integer $m_\cX'$ described above. Together with the bound on $d_\cX$ mentioned above, we give a purely combinatorial corollary of the main congruence (Theorem \ref{thm_combinatorial_congruence}). In \S\ref{ss_congruences_for_NAFSG} we deduce the congruences of Theorems \ref{thm_vdovin_intro} and \ref{thm_combinatorial_vdovin_intro} from this combinatorial statement.

\subsubsection{Applications to Markoff triples and the geometry of $\cM(\SL_2(\bF_q))$}
In \S\ref{section_applications} we prove the applications of our congruences described above. In \S\ref{ss_character_variety} we describe the theory of the character variety for $\SL_2$-representations of a free group of rank 2, and we make precise the connection between the stacks $\cM(\SL_2(\bF_q))$ and the Markoff equation. In \S\ref{ss_automorphism_groups}, we use the connection with the character variety to show that the automorphism groups of $\cAdm(\SL_2(\bF_q))$ are as small as possible: they are reduced to the center $Z(\SL_2(\bF_q))$. In \S\ref{ss_congruences_for_SL2} we use this calculation as input to the main congruence and obtain congruences for the degrees of components of $\cAdm(\SL_2(\bF_q))$ (equivalently the sizes of $\Gamma$-orbits of $\bF_q$-points on the associated generalized Markoff surface). In \S\ref{ss_bgs_conjecture} we bring everything together and to prove Theorem \ref{thm_bgs_conj_intro}. In \S\ref{ss_genus_formulas}, we give formulas for the genus of the components of $\Adm(\SL_2(\bF_p))_{-2}$ for $p\notin\bE_\bgs$, and show that they are noncongruence for a density 1 set of primes.

\subsection{Further directions and related work}\label{ss_further_directions}

Recall that for a component $\cX\subset\cAdm(G)$ classifying covers with ramification index $e$, our main congruence (Theorem \ref{thm_main_congruence_intro}) gives
$$\deg(X\rightarrow\ol{M(1)})\equiv 0\mod\frac{12e}{\gcd(12e,m_\cX d_\cX)}$$
It is an interesting question to ask how sharp this congruence is, how real the obstructions $m_\cX,d_\cX$ are, and when one can expect a good congruence (for example $\equiv 0\mod e$). Here is what we know: When $G = D_{2k}$ is a dihedral group, the congruence $\equiv 0\mod e$ can fail arbitrarily badly, and $m_\cX$ is totally responsible (see \S\ref{ss_remarks_on_md}).


However, it seems that when $G$ is nonabelian simple, there is still hope that a congruence $\equiv 0\mod e$ can be obtained. This congruence would be sharp at least in the sense that for $G = \PSL_2(\bF_7)$, there are two components of $\Adm(\PSL_2(\bF_7))$ classifying $\PSL_2(\bF_7)$-covers with ramification index 7; in each case Theorem \ref{thm_vdovin_intro} gives a congruence $\equiv 0\mod 7$, and one can compute that indeed each component has degree precisely 7 over $\ol{M(1)}$.


We have checked computationally that if $G$ is any nonabelian finite simple group of order $\le 29120$, then for any component $M\subset M(G)$ classifying covers with ramification index $e$, the congruence $\equiv 0\mod e$ holds\footnote{This includes $\PSL_2(\bF_q)$ for $q = 5,\ldots,37$, the alternating groups $A_5,\ldots,A_8$, the Matthieu group $M_{11}$, $\PSL_3(\bF_3)$, $\PSL_3(\bF_4)$, $\PSU_3(\bF_3)$, $O_5(\bF_3)$, and the Suzuki group $\Sz(8)$, for a total of 352 nonisomorphic connected components.}. It is natural to ask the (possibly naive) question:

\begin{question}\label{question_congruence_intro} If $G$ is a nonabelian simple group, must every component $X\subset\Adm(G)$ classifying covers with ramification index $e$ satisfy
\begin{equation}\label{eq_sober_congruence_intro}
\deg(X\rightarrow\ol{M(1)})\equiv 0\mod e?
\end{equation}
\end{question}

Combinatorially speaking, the question asks: If $\Pi$ is a free group generated by $a,b$ and $\varphi : \Pi\twoheadrightarrow G$ is a surjection such that $\varphi([a,b])$ has order $e$, then must the $\Out^+(\Pi)$-orbit of $\varphi$ in $\Epi^\ext(\Pi,G)$ have cardinality divisible by $e$? Theorem \ref{thm_vdovin_intro} gives partial results towards this.


If one is to believe that Question \ref{question_congruence_intro} has a positive answer, then it is natural to ask if the congruence can be obtained by proving that the obstructions $d,m$ to the congruence \eqref{eq_sober_congruence_intro} are sufficiently small. For any nonabelian finite simple group $G$ of order $\le 29120$, if $\pi : \cC\rightarrow\cE\rightarrow\cAdm(G)$ is the universal family with reduced ramification divisor $\cR_\pi$, then we have checked computationally (using the combinatorial characterization of the automorphism groups given in Theorem \ref{thm_cuspidal_automorphisms} below) that the vertical automorphism groups of geometric points of $\cR_\pi$ all vanish, so in these cases $m\mid 12$ and hence provides no obstruction to the congruence $\equiv 0\mod e$. Unfortunately we do not have a combinatorial characterization of $d$, so we do not know how to check if the obstruction $d$ also vanishes. In fact, we are not aware of an example of a nonabelian finite simple group $G$ such that the reduced ramification divisor of the universal family over $\cAdm(G)$ is not totally split.




Finally, we list some related work that we have not yet mentioned.

\begin{itemize}
\item A related problem is to understand the mapping class group orbits on the \emph{integral} points of character varieties. In the case of $\SL_2$-representations of the fundamental group of one-holed torus with trace invariant $-2$, this question is resolved by a classical result of Markoff \cite{Mar79} (see Theorem \ref{thm_markoff_transitivity}). In particular, in this case we obtain 5 orbits, represented by $(0,0,0),(3,3,3),(3,-3,-3),(-3,3,-3),(-3,-3,3)$. This should be viewed as a type of ``finite generation'' result on the set of integral points up to the action of the mapping class group. An analogous finite generation result for integral points for more general character varieties is proven in \cite{Whang20} using techniques from differential geometry.
\item For a component $M\subset M(G)_\bC$, let $g$ denote the genus of the covers it classifies. Then forgetting the base curve gives a natural map $M\rightarrow M_g$, whose image is a \emph{Teichmuller curve} as first studied by Veech \cite{Vee89} in the context of dynamics of billiard tables (also see \cite{Chen17,HS06,Loc03,MT02,Zor06}). Specifically, these are Teichmuller curves generated by a square-tiled surface (called ``origami curves'' in \cite{HS09,Sch04}). If $M$ is a component of $M(G)_\bC/\Out(G)$, then it follows from \cite{Sch04} that the group $\Gamma_M\le\SL_2(\bZ)$ is the \emph{Veech group} of the corresponding square tiled surface. In this language, our congruence can be interpreted as a congruence on the index of the Veech group inside $\SL_2(\bZ)$. For genus $g = 2$, the Teichmuller curves in $M_2$ have been studied extensively in \cite{Mcm03, Mcm05, Dur19}. In \cite{Mcm05} McMullen faced a similar issue of connectedness of a certain moduli space, which he was able to solve by reducing it to a problem in combinatorial number theory. It can be shown that his moduli space is a subspace of $M(S_d)\sqcup M(A_d)/\Out(A_d)$, where $A_d$ (resp. $S_d$) denotes the alternating (resp. symmetric group) on $d$ letters. Specifically, \cite[Corollary 1.5]{Mcm05} can be viewed as saying that for $d\ge 4$, the subscheme of $M(S_d)\sqcup M(A_d)/\Out(A_d)$ consisting of covers with Higman invariant the class of a 3-cycle has 1 or 2 components, according to whether $d$ is even or odd.

\item In \cite{BBCL20}, for every prime $\ell$, we realized infinitely many alternating and symmetric groups as quotients of the tame fundamental group of $\pi_1(\bP^1_{\bF_\ell} - \{0,1,\infty\})$. This work is in the spirit of a tame version of ``Abhyankar's conjecture'' \cite{HOPS18}. The precise result stated in \cite[Theorem 3.5.1]{BBCL20} only holds for those primes $p$ for which $\Gamma$ acts transitively on $\bX^*(p)$. Thus, our main theorem (\ref{thm_bgs_conj_intro}) can be viewed as a strengthening of \cite[Theorem 3.5.1]{BBCL20}.
\end{itemize}

\subsection{Acknowledgements}

We thank Peter Sarnak for bringing the problem to the author's attention, for his encouragement, and many helpful discussions. We thank Johan de Jong for sharing the idea behind the argument sketched in \S\ref{sss_sketch_of_proof}, as well as for many helpful comments as the manuscript was being prepared. We would also like to thank Pierre Deligne for helpful letters, as well as Dan Abramovich, Dave Aulicino, Jeremy Booher, Pat Hooper, Rafael von Kanel, and John Voight for helpful conversations. Finally we would like to thank Columbia University for their support and hospitality while this work was being prepared.

\subsection{Notations and conventions}\label{ss_conventions}


In our usage of stacks, we will follow the definitions of the stacks project \cite[026O]{stacks}. We will typically use script letters $\cX,\cY,\cZ,\ldots$ to denote stacks, and use Roman letters $X,Y,Z,\ldots$ to denote schemes. Typically $X$ will be the coarse scheme of $\cX$.


The symbol $\pi_1$ applied to a geometric object will by default refer to the \'{e}tale fundamental group. If $X$ is a topological space and $x\in X$ then $\pi_1^\tp(X,x)$ denotes its topological fundamental group. We take the convention that if $\gamma_1,\gamma_2$ are two loops in $X$ based at $x$, then the product $\gamma_1\gamma_2$ in $\pi_1^\tp(X,x)$ is represented by the loop that first follows $\gamma_2$ and then follows $\gamma_1$. This is consistent with our use of Galois theory, where we adopt the convention that Galois actions are right-actions, and monodromy actions are left actions. In particular the \'{e}tale fundamental group is the automorphism group of a fiber functor, and hence acts on fibers from the left.


For elements $a,b$ of a group $G$, $a\sim b$ means that $a$ is conjugate to $b$, $[a,b] := aba^{-1}b^{-1}$ denotes the commutator, $a^b := b^{-1}ab$, and $\la{b}a := bab^{-1}$. The order of $a\in G$ is denoted $|a|$.


Here is a summary of our notation.

\begin{itemize}
\item $G$ will always be a finite group.
\item Given groups $A,B$, $\Epi^\ext(A,B) := \Epi(A,B)/\Inn(B)$ is the set of conjugacy classes of surjections $A\rightarrow B$ (called ``exterior epimorphisms'' in \cite{DM69}).
\item $\bS$ is the universal base scheme. Often we will take $\bS = \Spec\bZ[1/|G|]$.
\item $\Qbar$ is the algebraic closure of $\bQ$ inside $\bC$.
\item $\cM(1)$ is the moduli stack of elliptic curves. $\ol{\cM(1)}$ is the compactification of $\cM(1)$ by stable curves.
\item $\cM(G)$ is the moduli stack of elliptic curves with $G$-structures. This is finite \'{e}tale over $\cM(1)$, but does not always carry a universal family of covers.
\item $\cAdm(G)$ is the moduli stack of admissible $G$-covers of 1-generalized elliptic curves (i.e., stable pointed curves of genus 1). This carries a universal family of covers, but is typically not finite over $\ol{\cM(1)}$.
\item $\cAdm^0(G)\subset\cAdm(G)$ is the open substack classifying smooth covers. It is an \'{e}tale gerbe over $\cM(G)$ with group $Z(G)$. In particular $\cAdm^0(G)$ and $\cM(G)$ have the same coarse schemes.
\item $\ol{\cM(G)}$ denotes the \emph{rigidification} $\cAdm(G)\fs Z(G)$ (Definition \ref{def_MGbar}).
\item $I$ will generally denote the matrix $I = \spmatrix{1}{0}{0}{1}$. Thus the center of $\SL_2(\bF_q)$ is $\{\pm I\}$.
\end{itemize}

\section{Admissible $G$-covers of elliptic curves}\label{section_admissible_G_covers}
In this section we recall the theory of admissible $G$-covers and their moduli stacks $\cAdm(G)$, following \cite{AV02,ACV03,Ols07}, and will relate them to the moduli stacks $\cM(G)$ of elliptic curves with $G$-structures as studied in \cite{Chen18} (also see \cite[\S5]{DM69} and \cite{JP95}). Here $G$ will always be a finite group. We will work universally over the base scheme $\bS := \Spec\bZ[1/|G|]$.

\subsection{Admissible $G$-covers and their moduli}\label{ss_admissible_G_covers}

\begin{defn} A morphism of schemes $f : C\rightarrow S$ is a \emph{nodal curve} if $f$ is flat, proper, of finite presentation, and whose geometric fibers are of pure dimension 1 and whose only singularities are ordinary double points. The morphism $f : C\rightarrow S$ is a \emph{prestable curve} if $f$ is a nodal curve with connected geometric fibers.
\end{defn}

\begin{defn} A \emph{prestable $n$-pointed curve of genus $g$} is a prestable curve $f : C\rightarrow S$ equipped with $n$ mutually disjoint sections $\sigma_1,\ldots,\sigma_n : S\rightarrow C$ lying in the smooth locus $C_\sm\subset C$, such that
\begin{itemize}
\item The geometric fibers of $f$ are connected curves of arithmetic genus $g$,
\end{itemize}
The prestable $n$-pointed curve $(C,\{\sigma_i\})$ is moreover \emph{stable} if for any geometric point $\ol{s} : \Spec k\rightarrow S$, the geometric fiber $C_{\ol{s}}$ satisfies any of the following equivalent conditions \cite[\S5, Lemma 1.2.1]{Manin99} :
\begin{itemize}
\item $C_{\ol{s}}$ has only finitely many $k$-automorphisms that fix the sections $\sigma_1,\ldots,\sigma_n$.
\item For any irreducible component $Z\subset C_{\ol{s}}$ with normalization $Z'$, if $Z'$ has genus 0, then it must contain at least three special\footnote{For an irreducible component $Z\subset C_{\ol{s}}$ with normalization $Z'$, we say that a point $z\in Z'$ is special if either it maps to a marking or a node in $C_{\ol{s}}$.} points, and if $Z'$ has genus 1, then it must contain at least one special point.
\item $\omega_{C_{\ol{s}}}(\sum_i\sigma_i)$ is ample, where $\omega_{C_{\ol{s}}}$ is the dualizing sheaf.
\end{itemize}
The sections $\sigma_1,\ldots,\sigma_n$ are called \emph{markings}. A point in $C$ is called a node if it is the image of a node of a geometric fiber of $C/S$. The \emph{generic locus} of a (pre)stable $n$-pointed curve $f : C\rightarrow S$ is the open complement of the special points, denoted $C_\gen\subset C$.
\end{defn}

\begin{defn} A \emph{1-generalized elliptic curve} is a stable 1-pointed curve of genus 1. We will denote the section by $\sigma_O$, and the divisor it defines by $O$. The moduli stack of 1-generalized elliptic curves is denoted $\ol{\cM(1)}$. The open substack classifying (smooth) elliptic curves is denoted $\cM(1)$.\footnote{Here, the ``1'' in ``$\cM(1)$'' denotes ``no level structures'' (or trivial level structures). In general $\cM(G)$ will denote the moduli stack of elliptic curves with $G$-structures (c.f. \S\ref{ss_comparison_with_G_structures} below). In the literature $\cM(1)$ (resp. $\ol{\cM(1)}$) is often denoted $\cM_{1,1}$ (resp. $\ol{\cM_{1,1}}$). However since for the most part we do not consider the moduli of higher genus curves or curves with more than 1 marked point, to keep our notation clean and to be consistent with \cite{Chen18}, we will use $\cM(1)$ (resp. $\ol{\cM(1)}$).} Let $M(1)\cong\bA^1_\bS$ and $\ol{M(1)}\cong\bP^1_\bS$ denote their coarse schemes (see Definition \ref{def_coarse_space} below), where the isomorphisms are given by the $j$-invariant.
\end{defn}



For a scheme $X$ and a geometric point $\ol{p} : \Spec k\rightarrow X$, let $\cO_{X,\ol{p}}$ be the strict henselization of the local ring (or just the strict local ring) of $X$ at the image $p$ of $\ol{p}$. Its residue field is the separable closure of the residue field $\kappa(p)$ inside $k$. For the purposes of the following definition, we will use the notation
$$X_{(\ol{p})} := \Spec\cO_{X,\ol{p}}$$
and we say that $X_{(\ol{p})}$ is the strict localization of $X$ at $p$. 

\begin{defn}[c.f. {\cite[\S4]{ACV03}}]\label{def_admissible} An \emph{admissible $G$-cover} of a prestable $n$-pointed curve $(D\stackrel{f}{\rightarrow} S,\{\sigma_i\})$ is a finite map $\pi : C\rightarrow D$ equipped with a \emph{right} action of $G$ on $C$ leaving $\pi$ invariant, where
\begin{enumerate}[label=(\arabic*)]
\item $C\rightarrow S$ is a prestable curve.
\item\label{part_admissible_nodes_to_nodes} Every node of $C$ maps to a node of $D$.
\item\label{part_admissible_torsor} The restriction of $\pi$ to the preimage of $D_\gen$ is a $G$-torsor.
\item\label{part_admissible_local_marking} Let $\ol{p} : \Spec k\rightarrow C$ be a geometric point whose image in $D$ lands in a marking. Let $\ol{s} := f(\pi\circ\ol{p})$ be its image in $S$, with strict local ring $\cO_{S,\ol{s}}$. For some integer $e\ge 1$, let
$$\pi' : C' := \Spec\cO_{S,\ol{s}}[\xi]\longrightarrow D' := \Spec\cO_{S,\ol{s}}[x]$$
be given by $x\mapsto\xi^e$. Let $\ol{p}' : \Spec k\rightarrow C'$ be a geometric point with image the point $\xi = 0$. Then for some choice of $e$ as above, there is a commutative diagram
\[\begin{tikzcd}
C_{(\ol{p})}\ar[r,"\pi"]\ar[d,"\cong"] & D_{(\pi\circ\ol{p})}\ar[r,"f"]\ar[d,"\cong"] & S_{(\ol{s})}\ar[d,equals] \\
C'_{(\ol{p}')}\ar[r,"\pi'"] & D'_{(\pi'\circ\ol{p}')}\ar[r,"f'"] & S_{(\ol{s})}
\end{tikzcd}\]
where the vertical maps are isomorphisms, and $f'$ is induced by the structure map $D'\rightarrow S_\slc{s}$.
\item\label{part_admissible_local_nodes} Let $\ol{p} : \Spec k\rightarrow C$ be a geometric point whose image in $D$ is a node, and let $\ol{s} := f(\pi(\ol{p}))$ be its image in $S$. For some integer $r\ge 1$ and $a$ in the maximal ideal $\fm_{S,\ol{s}}\subset\cO_{S,\ol{s}}$, let
$$\pi' : C' := \Spec \cO_{S,\ol{s}}[\xi,\eta]/(\xi\eta-a)\longrightarrow D' := \Spec \cO_{S,\ol{s}}[x,y]/(xy-a^r)$$
be given by $(x,y)\mapsto (\xi^r,\eta^r)$. Let $p' : \Spec k\rightarrow C'$ be a geometric point with image $(\xi,\eta) = (0,0)$. Then there is a commutative diagram
\[\begin{tikzcd}
C_{(\ol{p})}\ar[r,"\pi"]\ar[d,"\cong"] & D_{(\pi\circ\ol{p})}\ar[r,"f"]\ar[d,"\cong"] & S_{(\ol{s})}\ar[d,equals] \\
C'_{(\ol{p}')}\ar[r,"\pi'"] & D'_{(\pi'\circ\ol{p}')}\ar[r,"f'"] & S_{(\ol{s})}
\end{tikzcd}\]
where the vertical maps are isomorphisms, and $f'$ is induced by the structure map $D'\rightarrow S_\slc{s}$.

\item\label{part_admissible_balanced} If $\ol{p}$ is a geometric point landing in a node of $C$ with image $\ol{s}\in S$, then in the notation of (5) applied to the fiber $C_\ol{s}$, the stabilizer $G_{\ol{p}} := \Stab_G(\ol{p})$ is cyclic and the action of a generator $g\in G_{\ol{p}}$ on $(C_{\ol{s}})_\slc{p}$ is given \'{e}tale locally by sending $\xi\mapsto \zeta\xi, \eta\mapsto \zeta^{-1}\eta$ for some primitive $e$-th root of unity $\zeta\in k(\ol{s})$. Here we say that the node $\ol{p}$ is \emph{balanced}.
\end{enumerate}
An admissible $G$-cover $\pi : C\rightarrow D$ is \emph{smooth} if $C/S$ is smooth.
\end{defn}

\begin{remark}\label{remark_covers}
Here we record some observations about admissible $G$-covers.

\begin{enumerate}[label=(\alph*)]
\item\label{part_nodes} We note that part \ref{part_admissible_nodes_to_nodes} of the definition is implied by \ref{part_admissible_torsor} and \ref{part_admissible_local_marking}. It follows from \ref{part_admissible_local_nodes} that admissible $G$-covers map smooth points to smooth points. Thus, for an admissible $G$-cover $\pi : C\rightarrow D$, a point $x\in C$ is a node if and only if $\pi(x)\in D$ is a node. In particular, $C/S$ is smooth if and only if $D/S$ is smooth.

\item\label{part_ACV_admissible} Our definition of admissble covers differs from the definition used in Abramovich-Corti-Vistoli \cite[Definition 4.3.1]{ACV03} in that we require the covering curve $C$ to be prestable, and hence has \emph{connected} geometric fibers, whereas \cite{ACV03} only requires that $C$ be a nodal curve, so its geometric fibers can be disconnected. An \emph{ACV-admissible $G$-cover} is a map $\pi : C\rightarrow D$ where $C/S$ is a nodal curve and $\pi$ satisfies conditions (2)-(6) of the definition. Thus an admissible $G$-cover is equivalently an ACV-admissible $G$-cover with connected geometric fibers.\footnote{We choose this convention because on the one hand it is in accord with the definition of admissible covers as introduced in Harris-Mumford \cite[\S4]{HM82}, and on the other hand we will not need to consider disconnected covers in this paper.}
\item\label{part_balanced1} A finite map $\pi : C\rightarrow E$ is an \emph{admissible cover} if it satisfies (1),(2),(3),(4),(5). If it admits a $G$-action leaving $\pi$ invariant which moreover satisfies (6), then we say that the $G$-action is ``balanced''. Thus, an admissible $G$-cover is an admissible cover admitting a balanced $G$-action. In general, for a cover satisfying (2), the action of a generator of $G_p$ on the local ring of a node $p$ can send $\xi\mapsto\zeta\xi$, $\eta\mapsto\zeta'\eta$ where $\zeta,\zeta'$ are arbitrary primitive $r$th roots of unity. However if $\zeta'\ne\zeta^{-1}$, this would force $a = 0$ in the notation of (5), and so the node cannot appear in a generically smooth family. Since we will be interested in compactifications of the moduli stack of \emph{smooth admissible $G$-covers}, it suffices to restrict our attention to balanced actions.

\item\label{part_balanced2} Let $\ol{p}$ be a geometric point landing in a node of $C$ with image $\ol{s}\in S$. By Proposition \ref{prop_normalized_coordinates}, the balanced condition at $\ol{p}$ can be equivalently phrased as follows: The normalization map $\nu : C_{\ol{s}}'\rightarrow C_\ol{s}$ induces a decomposition of the cotangent space $T_{C_{\ol{s}},\ol{p}}^*$ into a sum of two 1-dimensional subspaces (the branches of the node). The $G$-action is \emph{balanced} at $\ol{p}$ if the (left) action of $G_{\ol{p}} := \Stab_G(\ol{p})$ on this cotangent space preserves this decomposition and acts faithfully via mutually inverse characters on each summand.

\end{enumerate}
\end{remark}


\begin{defn} Given admissible (resp. ACV-admissible) $G$-covers of genus $g$ stable $n$-pointed curves $(C\rightarrow D\rightarrow S,\{\sigma_i\}_{1\le i\le n})$ and $(C'\rightarrow D'\rightarrow S',\{\sigma_i'\}_{1\le i\le n})$, a morphism $(C'\rightarrow D'\rightarrow S',\{\sigma_i\})\rightarrow(C\rightarrow D\rightarrow S,\{\sigma_i'\})$ is a diagram
\[\begin{tikzcd}
C'\ar[r,"f"]\ar[d] & C\ar[d] \\
D'\ar[r,"\ol{f}"]\ar[d] & D\ar[d] \\
S'\ar[r] & S
\end{tikzcd}\]
where all squares are cartesian, $\ol{f}$ sends each $\sigma_i$ to $\sigma_i'$, and $f$ is $G$-equivariant. Since $C\rightarrow D, C\rightarrow S$ are epimorphisms \cite[023Q]{stacks}, any such diagram is determined by morphism $f : C'\rightarrow C$. Moreover, we will see (Proposition \ref{prop_ac2smGc}) that an admissible $G$-cover $C\rightarrow D$ induces an isomorphism $C/G\cong D$, so any $G$-equivariant map $f : C'\rightarrow C$ determines a diagram as above, whence a morphism of admissible (resp. ACV-admissible) $G$-covers. The category of ACV-admissible $G$-covers of stable $n$-pointed curves of genus $g$ is fibered in groupoids over $\Sch/\bS$. In \cite[\S4.3]{ACV03}, this category is denoted $\cAdm_{g,n}(G)$. The full subcategory category of admissible $G$-covers forms an open and closed substack $\cAdm_{g,n}^\conn(G)\subset\cAdm_{g,n}(G)$. In our case, since we will only consider connected covers of 1-generalized elliptic curves, we will abbreviate
$$\cAdm(G) := \cAdm_{1,1}^\conn(G)$$
Let $\cAdm^0(G)\subset\cAdm(G)$ denote the open substack consisting of smooth covers. If $\phi : \cAdm(G)\rightarrow\ol{\cM(1)}$ denotes the functor sending an admissible cover $C\rightarrow E\rightarrow S$ to $C\rightarrow E$, then we have $\cAdm^0(G) = \phi^{-1}(\cM(1))$. As usual let $\Adm(G),\Adm^0(G)$ denote the corresponding coarse spaces (see Theorem \ref{thm_admG} below).
\end{defn}


\begin{remark} Beware that despite the nomenclature, given admissible $G$-covers $\pi : C\rightarrow E$, $\pi' : C'\rightarrow E$ of the same 1-generalized elliptic curve $E$, a morphism of admissible $G$-covers $\pi\rightarrow \pi'$ \emph{need not} induce the identity on $E$. In other words, morphisms of admissible $G$-covers are not necessarily ``morphisms of covers''. In some sense, it is thus better to think of an admissible $G$-cover as the curve $C$ equipped with a $G$-action and a marking divisor. This perspective is taken in \S\ref{ss_comparison_with_smGc} below.
\end{remark}

\begin{defn}\label{def_coarse_space} Recall that a coarse space of an algebraic stack $\cX$ is a map $c : \cX\rightarrow X$ with $X$ an algebraic space which satisfies:
\begin{itemize}
\item[(a)] Any map $\cX\rightarrow T$ with $T$ an algebraic space factors uniquely through $c : \cX\rightarrow X$.
\item[(b)] For any algebraically closed field $k$, $c$ induces a bijection of sets $\cX(k)/\cong\;\rightiso X(k)$.
\end{itemize}
\end{defn}

The coarse space, if it exists, is uniquely determined by (a). If $\cX$ is locally of finite presentation over $\bS$ and has finite inertia (e.g., if it is a separated and locally finitely presented Deligne-Mumford stack), then the coarse space always exists and enjoys the following additional properties, which we will use repeatedly without mention.

\begin{thm}[Keel-Mori theorem]\label{thm_keel_mori} Let $\cX$ be an algebraic stack locally of finite presentation (over $\bS$) with finite inertia. Then there exists a coarse space $c : \cX\rightarrow X$ such that moreover we have
\begin{itemize}
	\item[(a)] $c : \cX\rightarrow X$ is proper, quasi-finite, and a universal homeomorphism,
	\item[(b)] formation of $X$ commutes with flat base change,
	\item[(c)] $c_*\cO_\cX = \cO_X$,
	\item[(d)] if $\cX$ is separated over $\bS$, then $X$ is also separated over $\bS$,
	\item[(e)] if $\bS$ is locally Noetherian then $X$ is locally of finite presentation over $\bS$, and
	\item[(f)] if $\bS$ is locally Noetherian and $\cX$ is proper over $\bS$, then $X$ is also proper over $\bS$.
\end{itemize}
\end{thm}
\begin{proof} See \cite[04XE]{stacks} for the definition of the topological space of a stack. By \cite[0DUT]{stacks}, $c$ is separated, quasi-compact, and a universal homeomorphism, and commutes with flat base change. Since $\cX$ is locally of finite type, the same is true of $c$, so it is of finite type. This implies that $c$ is proper and quasi-finite \cite[0G2M]{stacks}. This establishes (a) and (b). Part (c) follows from (b) and the universal property of coarse spaces (see \cite[Theorem 2.2.1]{AV02}). Part (d) is \cite[0DUY]{stacks}. Part (e) is \cite[0DUX]{stacks}. Part (f) follows from (a),(d),and (e).
\end{proof}

In dimension 1, smoothness of $\cX$ also often implies smoothness of $X$.

\begin{lemma}\label{lemma_coarse_scheme_is_smooth} Let $\cX$ be a smooth proper Deligne-Mumford stack over a regular Noetherian scheme $S$ whose fibers have pure dimension 1. Suppose its coarse space $X$ is a scheme. Then $X$ is smooth and proper over $S$.
\end{lemma}
\begin{proof} By Theorem \ref{thm_keel_mori}(f), $X$ is proper over $S$, so it remains to establish smoothness. By the local structure of Deligne-Mumford stacks \cite[Theorem 11.3.1]{Ols16}, $\cX$ admits an \'{e}tale covering by schemes $\{U_i\rightarrow\cX\}$ such that $\cX\times_X U_i\cong [V_i/G_i]$ for some finite $U_i$-scheme $V_i$ equipped with an action of a finite group $G_i$. From the proof we may even take $V_i$ to be affine, and hence $U_i = V_i/G_i$ is also affine. Since $V_i\rightarrow\cX$ is \'{e}tale, each $V_i$ is a smooth affine curve over $S$. By \cite[Theorem on p508]{KM85}, the quotients $U_i = V_i/G_i$ are also smooth affine curves over $S$, so $X$ is smooth over $S$ \cite[036U]{stacks}.
\end{proof}


\begin{thm}[\cite{AV02, ACV03}]\label{thm_admG}\quad
\begin{itemize}
\item[(a)] The category $\cAdm(G)$ is a smooth proper Deligne-Mumford stack of pure dimension 1 (over $\bS$). In particular, it has finite diagonal.
\item[(b)] The natural map $\phi : \cAdm(G)\rightarrow\ol{\cM(1)}$ sending an admissible cover $C\rightarrow E\rightarrow S$ to $E\rightarrow S$ is flat, proper, and quasi-finite\footnote{Quasi-finiteness for morphisms of algebraic stacks is defined in \cite[0G2L]{stacks} (also see \cite[Definition 1.8]{Vis89})}. In particular, $\cAdm^0(G)\subset\cAdm(G)$ is open and dense.
\item[(c)] $\cAdm(G)$ admits a coarse scheme $\Adm(G)$ satisfying the properties of Theorem \ref{thm_keel_mori}, and the map $\Adm(G)\rightarrow\ol{M(1)}$ induced by $\phi$ is finite. In particular, $\Adm(G)$ is a scheme.
\end{itemize}
\end{thm}
\begin{proof} The statements are preserved by base change, so we may assume that $\bS$ is Noetherian. In \cite{ACV03}, these facts are proven for the stack $\cB_{1,1}^\bal(G)$ of \emph{twisted $G$-covers} of 1-generalized elliptic curves, which is equivalent to $\cAdm_{1,1}(G)$ \cite[Theorem 4.3.2]{ACV03}, and so the theorem follows from the fact that $\cAdm(G)$ is an open and closed substack of $\cAdm_{1,1}(G)$.


Specifically, that $\cAdm(G)$ is a proper Deligne-Mumford stack follows from \cite[Theorem 2.1.7]{ACV03}(1) (also see \cite[Theorem 1.4.1]{AV02}), which also implies the finiteness of the diagonal\footnote{A proper morphism is separated, so $\cAdm(G)$ has proper diagonal. Since $\cAdm(G)$ is Deligne-Mumford, its diagonal is also unramified, hence locally quasi-finite, hence finite by Zariski's main theorem \cite[0A4X]{stacks}.}. The smoothness and 1-dimensionality of $\cAdm(G)$ follows from \cite[Theorem 3.0.2]{ACV03}. This proves (a).


The properties of the map $\phi$ follows from \cite[Corollary 3.0.5]{ACV03}. To see that $\cAdm^0(G)\subset\cAdm(G)$ is open dense, note that a flat locally of finite presentation morphism of algebraic stacks induces an open map of topological spaces \cite[06R7]{stacks}, so the map $\phi$ is open. If $\cX\subset\cAdm(G)$ is a closed substack containing $\cAdm^0(G)$, then its complement $\cU\subset\cAdm(G)$ is open, which maps onto an open substack of $\ol{\cM(1)}$ containing the cusp, which must intersect $\cM(1)$ nontrivially since $\cM(1)\subset\ol{\cM(1)}$ is open dense. But this contradicts the fact that $\cAdm^0(G) = \phi^{-1}(\cM(1))$\footnote{This denseness also follows from the deformation theory (see Proposition \ref{prop_deformations}).} This proves (b).


For (c), note that finite diagonal implies finite inertia, so the existence and properties of the coarse space follows from Theorem \ref{thm_keel_mori}. The rest of part (c) follows from \cite[Theorem 1.4.1]{AV02} (also see \cite[\S2.2]{ACV03}), where in their notation they show that the map
$$B_{1,1}(G) = \textbf{K}_{1,1}(BG,0)\lra\textbf{K}_{1,1}(\bS,0) = \ol{M(1)}$$
is finite, where $B_{1,1}(G)$ is the coarse space of $\cB_{1,1}(G)$. Since $\cB_{1,1}^\bal(G)$ is an open and closed substack of $\cB_{1,1}(G)$, $\Adm(G)$ is also open and closed inside $B_{1,1}(G)$, so $\Adm(G)\rightarrow\ol{M(1)}$ is finite as desired. Since $\ol{M(1)}$ is a scheme, the finiteness also implies that $\Adm(G)$ is a scheme \cite[03ZQ]{stacks} (one could also use \cite[03XX]{stacks}).
\end{proof}

\subsection{The reduced ramification divisor of an admissible $G$-cover}\label{ss_ramification_divisor}
We begin with a lemma.

\begin{lemma}\label{lemma_etale_unique} Let $C\rightarrow S$ be a smooth morphism. Let $R,R'\subset C$ be two closed subschemes, each finite \'{e}tale over $S$. If $R,R'$ have the same support, then $R = R'$.	
\end{lemma}
\begin{proof} If $S$ is reduced, then $R$ and $R'$ are both reduced closed subschemes with the same support, so they are equal. In the general case, let $S_\red$ denote the reduction of $S$, so $S_\red$ is reduced and $S_\red\rightarrow S$ is a universal homeomorphism. Then $R\times_S S_\red = R'\times_S S_\red$, but by topological invariance of the \'{e}tale site \cite[04DZ]{stacks}, this means $R = R'$.	
\end{proof}

Let $\pi : C\rightarrow D$ be an admissible $G$-cover of a stable $n$-pointed curve $D$ over $S$. Here we will define its \emph{reduced ramification divisor} $\cR_\pi$ of $\pi$, which is a closed subscheme of $C$, supported on the smooth points of $C$ with nontrivial inertia groups\footnote{The inertia group of $x\in C$ is the subgroup of $\Stab_G(x)$ which acts trivially on the residue field $\kappa(x)$ \cite[Expos\'{e} V, \S2]{SGA1}}, and is finite \'{e}tale over $S$. By Lemma \ref{lemma_etale_unique}, it is uniquely determined by this property. Moreover it inherits an action of $G$, and its connected components can be controlled by group-theoretic properties of $G$. For the definition we follow \cite[discussion right before \S4.1.3]{BR11}.


For a nontrivial cyclic subgroup $H\le G$, let $C^H\subset C$ denote the closed subscheme of fixed points. Namely, for $h\in H$, let $C^h\hookrightarrow C$ denote the equalizer of the maps $\id,h : C\rightrightarrows C$, which is a closed subscheme of $C$ since $C$ is separated\footnote{This is also the fiber product $C^h = C\times_{(\id,h),C\times_S C,\Delta} C$.}. Let $C^H := \bigcap_{h\in H}C^h$ be the scheme theoretic intersection \cite[0C4H]{stacks}. Note that if $h\in H$ is a generator, then we have $C^H = C^h$. If $U := \Spec A\subset C$ is a $G$-invariant open affine, then $C^H\cap U = \Spec A_H$, where $A_H$ is the ring of coinvariants $A/\langle \{ha-a\}_{a\in A, h\in H}\rangle$.

\begin{prop} Let $\pi : C\rightarrow D$ be an admissible $G$-cover of a stable $n$-pointed curve $D$ over $S$. Let $C_\sm\subset C$ denote the smooth locus of $C/S$. For a nontrivial cyclic subgroup $H\le G$, let $C_\sm^H := C^H\cap C_\sm = (C_\sm)^H$. Then either $C_\sm^H$ is empty, or $C_\sm^H\hookrightarrow C$ is an effective Cartier divisor finite \'{e}tale over $S$. Moreover $C_\sm^H$ commutes with arbitrary base change.
\end{prop}
\begin{proof} Let $h\in H$ be a generator. For any map $T\rightarrow S$, the universal property of equalizers implies that $(C^H)_T$ is the equalizer of $\id,h : C_T\rightrightarrows C_T$. Since taking the smooth locus commutes with arbitrary base change \cite[0C3H]{stacks}, so does $C_\sm^H$.


Now suppose $C_\sm^H$ is nonempty. For any geometric point $\ol{z}\in C^H_\sm$, its image in $C$ must land in a marking, so \'{e}tale locally in $C$, $C^H_\sm\hookrightarrow C$ looks like $\Spec A[t]/\langle(\zeta-1)t\rangle\rightarrow\Spec A[t]$, where $A$ is the strict local ring at the image of $\ol{z}$ in $S$, $\zeta$ is a primitive $|H|$-th root of unity, and $h$ acts on $A[t]$ linearly in $A$ sending $t\mapsto \zeta t$. This shows that $C^H_\sm\hookrightarrow C$ is a closed immersion, and since $|G|$ is invertible on the base, $\zeta-1\in A^\times$, so $C^H_\sm$ is also \'{e}tale over $S$. Since $C\rightarrow S$ is proper this implies that $C^H_\sm$ is moreover finite \'{e}tale, so it is an effective Cartier divisor as desired.
\end{proof}

Let $H\le G$ be a nontrivial cyclic subgroup, and let $K\supset H$ be a cyclic subgroup containing $H$. Then $C_\sm^K\subset C_\sm^H$ is a closed immersion of finite \'{e}tale $S$-schemes, so $C_\sm^K$ is an open and closed subscheme of $C_\sm^H$. Let
$$\Delta(H) := C_\sm^H - \bigcup_{K\supsetneq H} C_\sm^K$$
where the union runs over all cyclic subgroups of $G$ properly containing $H$. Thus, the support of $\Delta(H)$ consists precisely of the points $x\in C_\sm$ such that
\begin{itemize}
	\item[(a)] $Hx = x$,
	\item[(b)] $H$ acts trivially on the residue field $\kappa(x)$, and
	\item[(c)] no strictly larger subgroup $K\supsetneq H$ satisfies (a) and (b).
\end{itemize}
It follows from this description that $\Delta(H)\cap\Delta(K) = \emptyset$ if $H,K\le G$ are distinct nontrivial cyclic subgroups.
\begin{defn} Let $\pi : C\rightarrow D$ be an admissible $G$-cover of a stable $n$-pointed curve $D$ over $S$. The \emph{reduced ramification divisor} is the divisor
$$\cR_\pi := \bigsqcup_{H\le G}\Delta(H)$$
where $H$ runs over all nontrivial cyclic subgroups of $G$.
\end{defn}

\begin{prop}\label{prop_RRD} Let $\pi : C\rightarrow D$ be an admissible $G$-cover of a stable $n$-pointed curve $(D,\sigma_1,\ldots,\sigma_n)$ over $S$. Suppose $S$ is connected and let $e_i$ denote the ramification index of any point $x\in\pi^{-1}(\sigma_i)$. Then the reduced ramification divisor $\cR_\pi\subset C$ is an effective Cartier divisor finite \'{e}tale over $S$, supported on the set of non-\'{e}tale points of $C_\sm\rightarrow D_\sm$. If $e_i \ge 2$, then $\sigma_i^*\cR_\pi$ is finite \'{e}tale over $S$ of degree $|G|/e_i$.
\end{prop}
\begin{proof} The reduced ramification divisor $\cR_\pi$ is a disjoint union of finite \'{e}tale $S$-schemes, so it is also finite \'{e}tale over $S$, hence an effective Cartier divisor. The description of the degree and support follows from the description of $\Delta(H)$ and the local picture above a marking.
\end{proof}

\begin{remark}\label{remark_RRD} Beware that the ``reduced ramification divisor'' is not generally reduced! In fact it follows from the \'{e}tale local picture that it is reduced if and only if $S$ is reduced. We call it the reduced ramification divisor to avoid confusion with the ramification divisor that appears in the proof of the Riemann-Hurwitz formula, which can be defined as
$$\fR_\pi := \Div(\pi^*\omega_{D/S}\rightarrow\omega_{C/S})$$
where $\omega$ denotes the relative dualizing sheaf, and $\Div$ is taken in the sense of Knudsen-Mumford \cite[\S2]{KM76}. The restriction of $\fR_\pi$ to the preimage of a marking is a multiple of $\cR_\pi$. In the notation of Proposition \ref{prop_RRD}, we have $\sigma_i^*\fR_\pi = (e_i-1)\sigma_i^*\cR_\pi$ \cite[\S4.1.2]{BR11}.
\end{remark}


Let $\pi : C\rightarrow D$ be an admissible $G$-cover of a stable $n$-pointed curve $(D,\sigma_1,\ldots,\sigma_n)$ over $S$. If $H\le G$ is a cyclic subgroup, then $\Delta(H)^g = \Delta(g^{-1}Hg)$ for any $g\in G$, so the action of $G$ on $C$ restricts to an action on $\cR_\pi$ which is transitive on fibers over $D$. Since $\cR_\pi/S$ is finite \'{e}tale, we may study the structure of $\cR_\pi$ via Galois theory: 


\begin{prop}\label{prop_components_of_ramification_divisor} Let $\pi : C\rightarrow D$ be an admissible $G$-cover of a prestable pointed curve $(D,\sigma)$ over (a $\bZ[1/|G|]$-scheme) $S$ with ramification indices $e$ above $\sigma$. Let $\ol{s}$ be a geometric point of $S$, and let $\ol{x}$ be a geometric point of $C$ lying over $\sigma(\ol{s})$. Then
\begin{itemize}
\item[(a)] the connected components of $\cR_\pi$ are all isomorphic, and the connected component $R\subset\cR_\pi$ containing $\ol{x}$ is finite \'{e}tale Galois over $S$ with Galois group a subgroup of $N_G(G_{\ol{x}})/G_{\ol{x}}$, where $N_G(G_{\ol{x}})$ denotes the normalizer of $G_{\ol{x}}$ inside $G$.
\item[(b)] If moreover $S$ is regular and integral and $\Gamma(S,\cO_S)$ contains a primitive $e$th root of unity, where $e$ denotes the common ramification indices of $\pi$ above $\sigma$, then $\Gal(R/S)$ is isomorphic to a subgroup of $C_G(G_{\ol{x}})/G_{\ol{x}}$, where $C_G$ denotes the centralizer.
\end{itemize}
\end{prop}

\begin{remark} By working universally (see Proposition \ref{prop_stacky_RRD}), the conditions that $S$ be regular and integral in part (b) of the proposition can be removed.	
\end{remark}

\begin{proof} Let $\Pi := \pi_1(S,\ol{s})$, then by Galois theory we have commuting actions of $\Pi$ and $G$ on the geometric fiber $F := (\cR_\pi)_\ol{s}$. Thus $G$ acts on the set of $\Pi$-orbits of $F$, and for any $z\in F$, the decomposition group $\bD := \Stab_G(\Pi\cdot z)$ acts transitively on the orbit $\Pi\cdot z$. It follows that the inertia group $G_z$ acts trivially on $\Pi\cdot z$, so $G_z$ is normal in $\Stab_G(\Pi\cdot z)$, and the connected component $R\subset\cR_\pi$ corresponding to $\Pi\cdot z$ is Galois over $S$ with Galois group $\bD/G_z\le N_G(G_z)/G_z$. This proves (a).


Now suppose in addition that $S$ is regular and integral and contains a primitive $e$th root of unity. Let $\eta\in S$ be the generic point. Since $R/S$ is \'{e}tale, $R$ is irreducible. Let $\epsilon\in R$ be the unique generic point lying over $\eta$, and let $\ol{\epsilon}$ be a geometric point mapping to $\epsilon$. Then we find that $\bD = G_\epsilon := \Stab_G(\epsilon)$, and taking $\ol{s} = \ol{\eta}$ in the above discussion, it remains to show that $G_{\ol{\epsilon}}$ is contained in the center $Z(\bD)$ of $\bD$. Let $A := \what{\cO_{D_\eta,\sigma(\eta)}}$ be the complete local ring of the generic fiber $D_\eta$ at the branch point $\sigma(\eta)$, and let $K := \Frac(A)$. Let $B := \what{\cO_{C_\eta,\epsilon}}$, and let $L := \Frac(B)$, then $L/K$ is Galois with Galois group $\bD = G_\epsilon$ and inertia group $G_{\ol{\epsilon}}$. The vector space $\fm_B/\fm_B^2$ is a 1-dimensional vector space over $B/\fm_B$, and since we're in the tame case, the local representation $\chi_{\ol{\epsilon}} : G_\ol{\epsilon}\rightarrow\GL(\fm_B/\fm_B^2)$ is faithful \cite[\S IV.2]{SerreLF}. If $g\in \bD$ and $\gamma\in G_\ol{\epsilon}$, then since by assumption $\Gamma(S,\cO_S)$ contains all $e$th roots of unity, for any $v\in\fm_B/\fm_B^2$, we have $g^{-1}\gamma gv = g^{-1}(\zeta\cdot gv) = \zeta v$ for some $e$th root of unity $\zeta$. Thus $g^{-1}\gamma g$ also acts by multiplication by $\zeta$, but since $\chi_\ol{\epsilon}$ is faithful, this implies that $g^{-1}\gamma g = \gamma$ for all $g\in G$, so $G_\ol{\epsilon}\le Z(\bD)$, which proves (b).
\end{proof}

\subsection{The Higman invariant}\label{ss_higman_invariant}

Let $G$ be a finite group. The stacks $\cAdm(G)$ are typically not geometrically connected. Different connected components can often be distinguished by a natural combinatorial invariant called the Higman invariant, which analytically over $\bC$ is the conjugacy class in $G$ given topologically by the monodromy of a small positively oriented loop winding once around the branch point. When $G$ is a matrix group, the trace of this class is called the \emph{trace invariant} (see \S\ref{ss_trace_invariant}). Here we discuss the Higman invariant over algebraically closed fields.


Let $k$ be an algebraically closed field, and let $(D,O)$ be a 1-pointed prestable curve over $k$. Let $G$ be a finite group of order invertible in $k$, and let $\pi : C\rightarrow D$ be an admissible $G$-cover. If $k = \bC$, then for any base point $y\in D_\gen(\bC)$ and any point $x\in\pi^{-1}(y)$, for $\gamma\in\pi_1^\tp(D_\gen(\bC),y)$, let $\gamma\cdot x$ denote the endpoint of the unique lift of $\gamma$ to $C$ which starts at $x$. The monodromy representation at $x$ is the unique homomorphism
\begin{equation}\label{eq_monodromy_representation}
\varphi_x : \pi_1^\tp(D_\gen(\bC),y)\lra G\qquad\text{satisfying}\qquad\gamma\cdot x = x\cdot \varphi_x(\gamma)\quad\text{for all $\gamma\in\pi_1^\tp(D_\gen(\bC),y)$.}
\end{equation}

Varying the choice of $x\in\pi^{-1}(y)$ amounts to post-composing $\varphi_x$ with an inner automorphism of $G$. Thus, if $\gamma_O\in\pi_1^\tp(D_\gen(\bC),y)$ is a small loop winding once counter-clockwise around $O$, then $\{\varphi_x(\gamma_O)\;|\; x\in\pi^{-1}(y)\}$ is a conjugacy class of $G$, which we call the \emph{(topological) Higman invariant} of $\pi$ (at $O\in D$).


For a general algebraically closed field $k$, for any $x\in\pi^{-1}(O)$, let $e$ denote its ramification index. Since $\pi$ is admissible, $x$ is a smooth point of $C$. Let $T_x^*$ denote its cotangent space, then the stabilizer $G_x := \Stab_G(x)$ is cyclic of order $e$ and the right action of $G$ on $C$ defines a faithful \emph{local representation}
$$\chi_x : G_x\rightarrow\GL(T_x^*) = k^\times$$
so it gives an isomorphism $\chi_x : G_x\rightiso\mu_e(k)$. If $x'\in\pi^{-1}(O)$ is another point, then $x' = xg$ for some $g\in G$, so the conjugation $i_{g^{-1}} : h\mapsto g^{-1}hg$ induces an isomorphism $i_{g^{-1}} : G_x\rightiso G_{xg}$. Computing on $G$-invariant open affines, we find that the following diagram is commutative.
\[\begin{tikzcd}
	G_x\ar[r,"\chi_x"]\ar[d,"i_{g^{-1}}"] & \GL(T_x^*)\ar[d,equals] \\
	G_{xg}\ar[r,"\chi_{xg}"] & \GL(T_{xg}^*)
\end{tikzcd}\]
Thus, if $\zeta_e\in k^\times$ is a primitive $e$th root of unity, then the conjugacy class of $\chi_x^{-1}(\zeta_e)$ in $G$ is independent of the choice of $x\in\pi^{-1}(O)$. 

\begin{defn} In the above situation, if $\zeta_n\in k^\times$ is a primitive $n$th root of unity with $e\mid n$, then the \emph{(algebraic) Higman invariant of $\pi$ relative to $\zeta_n$} is the conjugacy class of $\chi_x^{-1}(\zeta_n^{n/e})\in G$ for some $x\in\pi^{-1}(O)$, denoted $\Hig_{\zeta_n}(\pi)$. 
\end{defn}


Now suppose our universal base scheme $\bS$ is such that there exists a primitive $n$th root of unity $\zeta_n\in\Gamma(\bS,\cO_\bS)$. Let $\xi : \Spec\Omega\rightarrow\cAdm(G)$ be a geometric point, corresponding to an admissible $G$-cover $\pi_\xi : C\rightarrow E$ over $\Omega$, where $(E,O)$ is a 1-generalized elliptic curve. Thus, to any geometric point $\xi$, we may associate an integer $e(\xi)$ (the ramification index of the corresponding admissible $G$-cover). If $e(\xi)\mid n$, then relative to our choice of $\zeta_n$, we may also associate the conjugacy class
$$\Hig_{\zeta_n}(\xi) := \Hig_{\zeta_n}(\pi_\xi)$$
Note that $e(\xi)$ is also just the order of (any representative of) the Higman invariant. Let $|\cAdm(G)|$ denote the underlying topological space of the stack $\cAdm(G)$ \cite[04XE]{stacks}. Since $e(\xi)$ and $\Hig_{\zeta_n}(\xi)$ are invariant under field extensions $\Omega\subset\Omega'$, they define functions on $|\cAdm(G)|$.

\begin{prop}\label{prop_Higman_invariant_is_locally_constant} Let $\Cl(G)$ denote the set of conjugacy classes of $G$. The functions
\begin{eqnarray*}
e : |\cAdm(G)| & \longrightarrow & \bN \\
\Hig_{\zeta_n} : |\cAdm(G)| & \longrightarrow & \Cl(G)	
\end{eqnarray*}
are locally constant.
\end{prop}
\begin{proof} The statement for $e$ is evident from the definition of admissible $G$-covers. The proof for $\Hig_{\zeta_n}$ is identical to the proof of \cite[Proposition 2.5.1]{BBCL20}. In the language of stable marked $G$-curves, this also follows from \cite[Proposition 3.2.5]{BR11}, using the equivalence of Theorem \ref{thm_aGc_equals_smGc}.
\end{proof}

\begin{remark}\label{remark_topological_higman} If $k = \bC$, then we leave it to the reader to verify that the (topological) Higman invariant of the analytification of $\pi$ agrees with the (algebraic) Higman invariant of $\pi$ relative to $\zeta_n = \exp(\frac{2\pi i}{n})$. This is equivalent to saying that monodromy around a branch point induces multiplication by $\exp(\frac{2\pi i}{n})$ on cotangent (equivalently, tangent) spaces. It will be useful to keep in mind that if $E$ is an elliptic curve over $\bC$, $x_0\in E^\circ(\bC)$, and $a,b\in\pi_1^\tp(E^\circ(\bC),x_0)$ is a basis for the fundamental group with positive intersection number (a ``positively oriented basis''), then the conjugacy class of the commutator $[b,a]\in\pi_1^\tp(E^\circ(\bC),x_0)$ is represented by a positively oriented loop in $E^\circ(\bC)$ winding once around the puncture. Thus, if $\pi : C\rightarrow E$ is an admissible $G$-cover, $x_0\in E^\circ(\bC)$ and $x\in\pi^{-1}(x_0)$ with associated monodromy representation $\varphi_x : \pi_1^\tp(E^\circ(\bC),x_0)\rightarrow G$, then $\varphi_x$ is surjective and the Higman invariant of $\pi$ is the conjugacy class of $\varphi_x([b,a]) = [\varphi_x(b),\varphi_x(a)]$. In particular, the Higman invariant of an admissible $G$-cover can always be expressed as a commutator of a generating pair of $G$. This implies for example that abelian $G$-covers of elliptic curves unramified away from the origin are in fact unramified everywhere, or equivalently, if $G$ is abelian then objects of $\cAdm(G)$ have trivial Higman invariant.
\end{remark}

\begin{defn}\label{def_higman_over_Qbar} In light of the remark, when $\bS = \Spec\Qbar$, for a geometric point $\xi : \Spec\Omega\rightarrow\ol{\cM(G)}$, let
$$\Hig(\xi) := \Hig_{\exp(2\pi i/|G|)}(\xi).$$
\end{defn}

From the local constancy of the Higman invariant, it follows that we have decompositions
$$\ol{\cM(G)}_{\Qbar} = \bigsqcup_{\fc\in\Cl(G)}\ol{\cM(G)}_\fc\quad\text{and}\quad \cAdm(G)_\Qbar = \bigsqcup_{\fc\in\Cl(G)}\cAdm(G)_\fc$$
where $\cM(G)_\fc\subset\cM(G)_{\Qbar}$ (resp. $\cAdm(G)_\fc\subset\cAdm(G)_\Qbar$) is the open and closed substack consisting of objects with Higman invariant $c$. Let $\ol{M(G)}_\fc,\Adm(G)_\fc$ denote their coarse schemes.

\subsection{Comparison with stable marked $G$-curves}\label{ss_comparison_with_smGc}

An admissible $G$-cover is a map $\pi : C\rightarrow D$ satisfying certain properties. An alternative approach is to forget $D$, and only remember the curve $C$ together with its $G$-action and a suitable marking divisor $R\subset C$. We will see that $\pi$ induces an isomorphism $C/G\cong D$, so nothing in lost in this approach. Moreover this perspective will be convenient later when we describe the deformation theory for admissible $G$-covers in Proposition \ref{prop_deformations} below. In this section we make precise the relationship between these two viewpoints. The results here are not new, and can be viewed as an exposition of \cite[Appendix B]{ACV03}. However our terminology here follows Bertin-Romagny \cite[Definition 4.3.4]{BR11}. We begin with a well-known lemma.

\begin{lemma}\label{lemma_tame_quotients} Let $C\rightarrow S$ be a flat proper finitely presented morphism whose geometric fibers are reduced of equidimension 1, and $G$ a finite group acting $S$-linearly on $C$. Then
\begin{enumerate}[label=(\alph*)]
	\item\label{quotient_part_criteria} $C$ is a union of $G$-invariant affine opens,
	\item\label{quotient_part_properties} the categorical quotient $C/G$ exists, and the projection $\pi : C\rightarrow C/G$ is finite and induces an isomorphism $\cO_{C/G}\rightiso\pi_*\cO_C^G$, and
	\item\label{quotient_part_base_change} the quotient $C/G$ commutes with arbitrary base change.
\end{enumerate}
\end{lemma}
\begin{proof} By \cite[Expos\'{e} V, Corollaire 1.5, Proposition 1.8]{SGA1}, Part \ref{quotient_part_criteria} implies \ref{quotient_part_properties}, in which case \ref{quotient_part_base_change} follows from \cite[Proposition A7.1.3(4)]{KM85} (using our standing tameness assumption).


It remains to prove \ref{quotient_part_criteria}. Since the $G$-action preserves the fibers, it suffices to work Zariski-locally on the base. Thus we will view $S$ as a ``small affine open'', and will shrink it as necessary. By Noetherian approximation \cite[01ZM,081C]{stacks} (also see Remark \ref{remark_noetherian_approximation}), there is an affine map $S\rightarrow S_0$ with $S_0$ of finite type over $\bZ$ and a flat proper finitely presented map $C_0\rightarrow S_0$ with equidimension 1 fibers such that $C = C_0\times_{S_0}S$. Replacing $S_0$ by an open subscheme containing the image of $S$, we may assume that $C_0/S_0$ also has geometrically reduced fibers \cite[0C0E]{stacks}. Thus, we are moreover reduced to the case where $S$ is of finite type over $\bZ$.


A scheme $X$ satisfies property (AF) if every finite set of points is contained in an affine open \cite[Appendix B]{Rydh13}. If $C$ has property (AF), then for any point $x\in C$, let $W\subset C$ be an open affine containing the orbit $Gx$. Then $\cap_{g\in G} gW$ is a $G$-invariant open affine neighborhood of $x$, so we would obtain \ref{quotient_part_criteria}. Thus for \ref{quotient_part_criteria} and \ref{quotient_part_properties}, it suffices to show that $C$ has property (AF). Since $C/S$ has geometrically reduced fibers, the $S$-smooth locus $C_\sm$ is open dense inside every fiber of $C/S$ \cite[056V]{stacks}. Thus there is an \'{e}tale morphism $p : U\rightarrow S$ such that $C_U$ admits a section lying in $(C_U)_\sm$. Shrinking $U$, we may moreover assume that $p$ is quasi-finite. By \cite[03I1]{stacks}, there is an open affine $V\subset S$ such that $p^{-1}(V)\rightarrow V$ is finite \'{e}tale. Replacing $S$ with $V$, we have found a finite surjective map $S'\rightarrow S$ such that $C_{S'}$ admits a section. Repeating this finitely many times and possibly further shrinking $S$, we may assume that there is a finite surjective map $S'\rightarrow S$ such that $C_{S'}$ admits pairwise disjoint sections $\sigma_1,\ldots,\sigma_n$ which meet every irreducible component of every fiber. Since $C/S$ has relative dimension 1, this implies that $\cO_{C_{S'}}(\sum_i\sigma_i)$ is ample \cite[0B5Y,0D2S]{stacks}, so $C_{S'}$ satisfies (AF) \cite[II, Corollaire 4.5.4]{EGA}. Since $S,S'$ are of finite type over $\bZ$, \cite[Corollary 48]{Kol12}) implies that $C$ also satisfies (AF) as desired.
\end{proof}

\begin{remark} Note that the proofs of parts \ref{quotient_part_criteria} and \ref{quotient_part_properties} are valid with no tameness assumptions on $|G|,S$.	
\end{remark}

\begin{defn}\label{def_marked_G_curves} A marking on a prestable curve $C\rightarrow S$ is an effective Cartier divisor $R\subset C$ \'{e}tale over $S$. In particular, it must lie in the smooth locus of $C$ and must be finite \'{e}tale over $S$. A \emph{marked prestable curve} is a pair $(C/S,R)$ where $C/S$ is a prestable curve and $R\subset C$ is a marking. A stable marked curve is a marked prestable curve whose geometric fibers have finite automorphism groups preserving the marking. A stable marked $G$-curve is a stable marked curve equipped with a faithful \emph{right}-action of $G$ satisfying:
\begin{enumerate}[label=(\alph*)]
	\item the $G$-action preserves the divisor $R$,
	\item $C\rightarrow C/G$ is \'{e}tale on $C_\sm - R$ (equivalently, $G$ acts with trivial inertia on $C_\sm - R$), and
	\item the action at every geometric node is \emph{balanced} in the sense of Remark \ref{remark_covers}\ref{part_balanced2}.
\end{enumerate}
A morphism of stable marked $G$-curves is a morphism of the underlying prestable curves which both preserves the marking and is $G$-equivariant.
\end{defn}

\begin{prop}\label{prop_stability_of_quotient} Given a stable marked $G$-curve $(C/S,R)$, the quotient $(C/G,R/G)$ is also a stable marked curve and the quotient map $\pi : C\rightarrow C/G$ is finite flat.
\end{prop}
\begin{proof} 
By Noetherian approximation (Remark \ref{remark_noetherian_approximation}), we are reduced to the case where $S$ is of finite type over $\bZ$. First, the quotient exists by Lemma \ref{lemma_tame_quotients}. By \cite[Expos\'{e} V, Corollaire 1.5]{SGA1}, $\pi$ is finite and $C/G$ is separated of finite presentation over $S$. The finiteness implies $C/G\rightarrow S$ is universally closed, hence proper. Since $C/S$ is flat, by the fiberwise criteria of flatness \cite[039B]{stacks}, to check that $\pi$ and $C/G\rightarrow S$ are flat, it suffices to check that $\pi$ is flat on fibers over $S$, so it suffices to take $S = \Spec k$ where $k$ is an algebraically closed field, but flatness here follows immediately from the \'{e}tale-local picture (see Proposition \ref{prop_normalized_coordinates}).


Finally we claim that the quotient $(C/G,R/G)$ is stable. By Lemma \ref{lemma_tame_quotients}\ref{quotient_part_base_change}, we may assume $S = \Spec k$ with $k$ an algebraically closed field. It is easy to check that $C/G$ is prestable, so it remains to check that $(C/G,R/G)$ is stable. Suppose $Z\subset C/G$ be a non-stable component, with normalization $Z'$. Then $Z'$ has genus at most 1. If $Z'$ has genus 1, then we must have $Z = Z'$ and it must have no nodes or markings, but this implies that if $W$ is any irreducible component of the preimage of $Z$ in $C$, then $W$ contains no nodes or markings, so $W\rightarrow Z$ is \'{e}tale, so by Riemann-Hurwitz, $W$ also has genus 1, so $W$ is also unstable. Now suppose $Z'\cong \bP^1$ has genus 0, and let $W\subset C$ be an irreducible component mapping to $Z$, with normalization $W'$. Then $W'\rightarrow Z'$ is \'{e}tale away from the complement of two points, so $W'\rightarrow Z'$ is a totally ramified cyclic cover, so $W$ must also be unstable.
\end{proof}

\begin{defn}\label{def_marking} Let $\pi : C\rightarrow D$ be an admissible $G$-cover of an $n$-pointed stable curve $(D,\{\sigma_i\}_{1\le i\le n})$. Let $\cR_\pi$ be the reduced ramification divisor, and let $J\subset\{1,\ldots,n\}$ be the subset of indices $j$ such that $\cR_\pi$ does not meet the divisor $Z_j := C\times_{D,\sigma_j} S$. For each $i$, let
$$R_i := \left\{\begin{array}{rl}
	Z_i & \text{if } i\in J \\
	\cR_\pi\times_{D,\sigma_i} S & \text{if } i\notin J
\end{array}\right.$$
We call $\sqcup_{i=1}^n R_i$ the \emph{marking associated to the admissible $G$-cover $\pi$}.
\end{defn}

\begin{prop}\label{prop_ac2smGc} Let $\pi : C\rightarrow D$ be an admissible $G$-cover of an $n$-pointed stable curve $(D,\{\sigma_i\}_{1\le i\le n})$, and let $R := \sqcup_{i=1}^n R_i$ be the associated marking as in Definition \ref{def_marking}. Then $(C/S,R)$ is a stable marked $G$-curve. The categorical quotient $C/G$ exists, commutes with arbitrary base change, and $\pi$ induces an isomorphism $C/G\rightiso D$.
\end{prop}
\begin{proof} Since $\pi$ is admissible, the $G$-action is balanced on the nodes. Since $\pi$ is \'{e}tale above $D_\gen$, and the support of $\cR_\pi$ is precisely the set of points in $C_\sm$ at which $\pi$ is not \'{e}tale, the image of $\cR_\pi$ in $D$ is contained in the marking divisor and $\sqcup_{j\in J}Z_j$ is \'{e}tale over $S$. Thus $R\subset C$ is an effective Cartier divisor, finite \'{e}tale over $S$. To show that $(C/S,R)$ is stable marked, we may assume $S = \Spec k$ for $k$ an algebraically closed field. Let $Z\subset D$ be an irreducible component with normalization $Z'$, and let $W\subset C$ be an irreducible component mapping to $Z$, with normalization $W'$. We say that a point of $W'$ (resp. $Z'$) is special if it maps to a node or marking of $W$ (resp. $Z$). By Riemann-Hurwitz, the stability of $W$ is clear if $Q'$ has genus $\ge 2$, and since a point of $W$ is special if and only if it maps to a special point of $Z$ (Remark \ref{remark_covers}\ref{part_nodes}), stability is also clear if $Z'$ contains at least three special points. The only remaining case is when $Z'$ has genus 1, containing at least one special point, but again in this case we find $W'$ has genus at least 1 with at least one node or marking, so $W'$ is also stable.


By Lemma \ref{lemma_tame_quotients}, the categorical quotient $C/G$ exists, commutes with arbitrary base change, and is defined affine locally by taking $G$-invariants. Then $\pi$ factors uniquely through a finite map $\alpha : C/G\rightarrow D$ which by Definition \ref{def_admissible}\ref{part_admissible_torsor} must be an isomorphism over $D_\gen$. It follows from the local picture at the nodes and markings that $\alpha$ is an isomorphism there as well.
\end{proof}

\begin{defn}\label{def_stack_of_G_curves} Let $\ol{\cH}_{g,n,G}$ denote the category whose objects are stable marked $G$-curves $(C/S,R)$ equipped with a decomposition $R = \bigsqcup_{i=1}^n R_i$ into open and closed subschemes such that
\begin{itemize}
\item[(1)] $G$ preserves each $R_i$,
\item[(2)] the map $R_i/G\rightarrow S$ is an isomorphism\footnote{Equivalently, $G$ acts transitively on the geometric fibers of $R_i\rightarrow S$.} for each $i$, and
\item[(3)] $C/G$ is a prestable curve of genus $g$,
\end{itemize}
and whose morphisms are fiber squares, preserving the decomposition $R = \bigsqcup_{i=1}^n R_i$. Then $\ol{\cH}_{g,n,G}$ is a category fibered in groupoids over $\Sch/\bS$. Let $\cH_G\subset\ol{\cH}_G$ denote the subcategory consisting of pairs $(C/S,R)$ where $C/S$ is smooth. Let $\ol{\cH}_G := \ol{\cH}_{1,1,G}$, and similarly let $\cH_G := \cH_{1,1,G}$.
\end{defn}

\begin{prop}\label{prop_smGc2ac} Let $(C/S,R = \bigsqcup_{i=1}^n R_i)$ be an object of $\ol{\cH}_{g,n,G}$. Let $\sigma_i$ denote the section of $C/G$ determined by $R_i/G$. Then $(C/G,\{\sigma_i\}_{1\le i\le n})$ is a stable $n$-pointed curve and the quotient map $\pi : C\rightarrow C/G$ is an admissible $G$-cover.
\end{prop}
\begin{proof} By Proposition \ref{prop_stability_of_quotient}, $(C/G,\{\sigma_i\})$ is a stable $n$-pointed curve. It remains to check that $\pi$ is an admissible $G$-cover. The most difficult thing to check is that $\pi$ has the correct local picture at the nodes and markings, and that the $G$-action is balanced at the nodes. This is done in Proposition \ref{prop_normalized_coordinates} in the appendix. This local picture then implies that $\pi$ maps nodes to nodes. Since $G$ acts without inertia on $C_\sm - R$, $\pi$ is a $G$-torsor above $(C/G)_\gen$, so $\pi$ is an admissible $G$-cover as desired.
\end{proof}

\begin{thm}\label{thm_aGc_equals_smGc} For $g,n\in\bZ_{\ge 0}$, the map $\Phi$ sending an object $(C/S,R = \sqcup_{i=1}^n R_i)$ in $\ol{\cH}_{g,n,G}$ to the admissible $G$-cover $C\rightarrow C/G$ together with the $n$ sections of $C/G$ determined by $R_i/G$ gives an equivalence of categories
$$\Phi : \ol{\cH}_{g,n,G}\rightiso\cAdm^\conn_{g,n}(G).$$
A quasi-inverse is given by sending an admissible $G$-cover $(C\stackrel{\pi}{\rightarrow} D, \{\sigma_i : S\rightarrow D\}_{1\le i\le n})$ to $(C/S,R = \sqcup_{i=1}^n R_i)$, where here $R_i$ denotes the $i$th component of the marking associated to the admissible $G$-cover $\pi$ (Definition \ref{def_marking}). In particular, $\ol{\cH}_G$ is a smooth proper Deligne-Mumford stack of pure dimension 1 over $\bS$.
\end{thm}
\begin{proof} $\Phi$ is fully faithful since any admissible $G$-cover $C\rightarrow D$ is a quotient map (Proposition \ref{prop_ac2smGc}), so maps in $\cAdm_{g,n}^\conn(G)$ are precisely given by $G$-equivariant maps of the covering curve preserving the decomposition of the marking $R = \sqcup_{i=1}^n R_i$. It is essentially surjective since $\Psi\circ\Phi$ is equal to the identity functor on $\ol{\cH}_{g,n,G}$.
\end{proof}

\subsection{Relation to the moduli stack of elliptic curves with $G$-structures}\label{ss_comparison_with_G_structures}

In this section we compare $\cAdm(G)$ (equivalently $\ol{\cH}_G$) to the moduli stack $\cM(G)$ of elliptic curves with $G$-structures \cite{Chen18, Ols12, JP95, DM69}. The main result is that the open substack $\cAdm^0(G)\subset\cAdm(G)$ corresponding to smooth covers is an \'{e}tale gerbe over $\cM(G)$, and $\cM(G)$ is can be obtained from $\cAdm^0(G)$ by a rigidification process removing $Z(G)$ from all the automorphism groups of objects in $\cAdm^0(G)$. The same process applied to $\cAdm(G)$ results in a smooth compactification of $\cM(G)$, which we denote $\ol{\cM(G)}$. In this section we explain these relationships, following \cite{ACV03}, \cite{Chen18}, and \cite{BR11}.


We recall some definitions. Let $G$ be a finite group. A $G$-torsor over a scheme $X$ is a finite \'{e}tale morphism $p : Y\rightarrow X$ together with an $X$-linear right action of $G$ on $Y$ such that $p$ acts freely and transitively on geometric fibers. A morphism of $G$-torsors over $X$ is a $G$-equivariant morphism over $X$. Let $\cT_G^\pre$ denote the presheaf
\begin{equation}\label{eq_teichmuller_sheaf}
\begin{array}{rcl}
	\cT_G^\pre : \cM(1) & \longrightarrow & \Sets \\
	E/S & \mapsto & \{\text{$G$-torsors $X\rightarrow E^\circ$ with geometrically connected fibers over $S$}\}/\cong
\end{array}
\end{equation}
Let $\cT_G$ be the sheafification of $\cT_G^\pre$ with respect to the topology on $\cM(1)$ inherited from $(\Sch/\bS)_\et$. Then $\cM(G)$ is the stack associated $\cT_G$. For an elliptic curve $E/S$, a $G$-structure\footnote{Originally called a \emph{Teichmuller structure of level $G$} in \cite[\S5]{DM69}.} on $E/S$ is by definition an element of the set $\cT_G(E/S)$.


If $\alpha,\alpha'$ are isomorphism classes of $G$-torsors on $E/S$, $E'/S'$, then a morphism $(E'/S',\alpha')\rightarrow (E/S,\alpha)$ thus consists of a cartesian diagram
\[\begin{tikzcd}
	E'\ar[r,"f"]\ar[d] & E\ar[d] \\
	S'\ar[r] & S
\end{tikzcd}\]
such that $f^*\alpha = \alpha'$. Thus, morphisms in $\cM(G)$ are determined by morphisms of elliptic curves, so the map $\cM(G)\rightarrow\cM(1)$ is representable. Using the relative fundamental group, $\cT_G$ is checked to be finite locally constant on $\cM(1)$, and hence $\cM(G)\rightarrow\cM(1)$ is finite \'{e}tale. In particular, it is a smooth Deligne-Mumford stack.\footnote{As defined, $\cM(G)$ is an algebraic stack for the \'{e}tale topology. By \cite[076U]{stacks}, it is also an algebraic stack for the fppf topology (i.e., an algebraic stack in the sense of the stacks project).}

\begin{remark} We make some remarks.
\begin{enumerate}[label=(\alph*)]
\item Note that while the definition of $G$-structures above differs from the one used in \cite{Chen18}, the resulting objects are isomorphic. Namely, taking monodromy representations defines a map from $\cT_G^\pre$ to the presheaf of \cite[Definition 2.2.3]{Chen18} which is locally an isomorphism. Thus, their sheafifications are isomorphic.
\item If $S = \Spec k$ with $k$ a separably closed field, and $E/k$ is an elliptic curve, then the stalk of $\cT_G^\pre$ (resp. $\cT_G$) at $E/k$ is precisely $\cT_G^\pre(E/k)$ (resp. $\cT_G(E/k)$) \cite[06VW]{stacks}. Thus, since sheafification preserves stalks \cite[00Y8]{stacks}, $\cT_G^\pre(E/k) = \cT_G(E/k)$ is precisely the set of isomorphism classes of geometrically connected $G$-torsors over $E^\circ/k$.
\item If $E/S$ is an elliptic curve, then an object of $\cT_G(E/S)$ is given by an \'{e}tale covering $\{S_i\rightarrow S\}$, and $G$-torsors $X_i$ on each $E^\circ_i := E^\circ\times_S S_i$ with geometrically connected $S_i$-fibers whose common ``overlaps'' are isomorphic. However, there is no requirement that one can choose the isomorphisms to satisfy a cocycle condition, and hence the $G$-torsors $X_i$ need not glue to give a $G$-torsor on $X\rightarrow E^\circ$. If $G$ has trivial center, then such $G$-torsors have trivial automorphism groups, and hence any cocycle condition is automatic, so in this case we have $\cT_G^\pre = \cT_G$ \cite[Proposition 2.2.6(3)]{Chen18}.

\item Given an elliptic curve $E/S$, a $G$-torsor on $E^\circ$ with geometrically connected $S$-fibers defines a $G$-structure on $E/S$. This gives a map from the set of isomorphism classes of geometrically connected $G$-torsors on $E^\circ/S$ to the set of $G$-structures $\cT_G(E/S)$. As we saw above, this map is a bijection if either $S = \Spec k$ with $k$ a separably closed field, or if $G$ has trivial center, but in general it need not be injective or surjective\footnote{If $G$ is abelian, then the map is surjective but rarely injective: If $\pi : X\rightarrow E^\circ$ is a $G$-torsor geometrically connected over $S$ and $E^\circ$ admits a section $\sigma$, then for \emph{any} $G$-torsor $\xi$ over $S$, one can ``twist'' $\pi$ in a way that the restriction of the resulting $G$-torsor $\pi_\xi$ to $\sigma$ is isomorphic to $\xi$, but such that the torsors $\{\pi_\xi\}_{\xi}$ all determine the same $G$-structure. If $G$ has nontrivial center, then failure of descent implies that this map is typically not surjective.}. While this may make $G$-structures seem like a somewhat unnatural gadget, the upshot is that the forgetful map $\cM(G)\rightarrow\cM(1)$ is \emph{finite \'{e}tale}, and so it can be studied using Galois theory. On the other hand, while the objects of the related stack $\cAdm(G)$ are in some sense simpler to understand, the forgetful map $\cAdm(G)\rightarrow\ol{\cM(1)}$ is typically not representable, hence typically not finite, even above $\cM(1)$ (see \S\ref{sss_relation_between_cAdmG_MG}). Nonetheless, because $\cAdm(G)$ is both proper and carries a universal family of covers, it will be the central object of study. The main purpose of $\cM(G)$ is that it gives an approximation to $\cAdm(G)$ which allows us to use Galois theory to translate theorems about $\cAdm(G)$ into combinatorics.
\end{enumerate}
\end{remark}

\subsubsection{Review of $G$-structures}

Here we recall some of the salient features of the stacks $\cM(G)$.


\begin{thm}\label{thm_basic_properties} Let $G$ be a finite group. Let
$$\ff : \cM(G)\rightarrow\cM(1)$$
be the forgetful map. We work universally over $\bS = \Spec\bZ[1/|G|]$.
\begin{enumerate}[label=(\arabic*)]
    \item \label{part_etale} (\'{e}taleness) The category $\cM(G)$ is a Noetherian smooth separated Deligne-Mumford stack and the forgetful functor $\ff : \cM(G)\rightarrow\cM(1)$ is finite \'{e}tale.
    \item \label{part_coarse} (Coarse moduli and ramification) $\cM(G)$ admits a coarse moduli scheme $M(G)$ which is a normal affine scheme finite over $M(1)\cong\Spec\bZ[1/|G|][j]$, and smooth of relative dimension $1$ over $\bZ[1/|G|]$. Moreover, $M(G)$ is \'{e}tale over the complement of the sections $j = 0$ and $j = 1728$ in $M(1)$. If either $6\mid |G|$ or $S$ is a regular Noetherian $\bZ[1/|G|]$-scheme, then $M(G)\times_{\bZ[1/|G|]} S$ is the coarse moduli scheme of $\cM(G)\times_{\bZ[1/|G|]} S$, and is normal.
    \item \label{part_combinatorial} (Combinatorial description of $G$-structures) Let $\bL$ be the set of prime divisors of $|G|$. For any profinite group $\pi$, let $\pi^\bL$ denote the maximal pro-$\bL$-quotient of $\pi$. Let $E$ be an elliptic curve over $S$. Let $\ol{x}\in E^\circ$ be a geometric point, and let $\ol{s}\in S$ be the image of $\ol{x}$. The sequence $E^\circ_\ol{s}\hookrightarrow E\rightarrow S$ induces an outer representation
    $$\rho_{E,\ol{x}} : \pi_1(S,\ol{s})\rightarrow\Out(\pi_1^\bL(E^\circ_\ol{s},\ol{x}))$$
    from which we obtain a natural right action of $\pi_1(S,\ol{s})$ on the set
    $$\Epi^\ext(\pi_1^\bL(E^\circ_\ol{s},\ol{x}),G) := \Epi(\pi_1^\bL(E^\circ_\ol{s},\ol{x}),G)/\Inn(G)$$
    of surjective morphisms $\pi_1^\bL(E^\circ_\ol{s},\ol{x})\rightarrow G$ up to conjugation in $G$. By Galois theory this action corresponds to a finite \'{e}tale morphism $F\rightarrow S$, which fits into a cartesian diagram
    \[\begin{tikzcd}
    F\ar[d]\ar[r] & \cM(G)\ar[d,"\ff"]\\
    S\ar[r,"E/S"] & \cM(1)
    \end{tikzcd}\]	
    In particular, we obtain a bijection
    $$\cT_G(E/S)\rightiso \{\varphi\in\Epi^\ext(\pi_1(E^\circ_\ol{s},\ol{x}),G) \;\big|\; \varphi\circ\rho_{E,\ol{x}}(\sigma) = \varphi\quad\text{for all $\sigma\in\pi_1(S,\ol{s})$}\}.$$
    where recall that $\cT_G(E/S)$ is the set of $G$-structures on $E/S$.
    \item \label{part_fibers} (Fibers) Let $E$ be an elliptic curve over an algebraically closed field $k$ of characteristic not dividing $|G|$, and let $x_0\in E^\circ(k)$. Let $x_E : \Spec k\rightarrow\cM(1)$ be the geometric point given by $E$. The fiber $\mf{f}^{-1}(x_E)$ is in bijection with the set of connected $G$-torsors on $E^\circ$. Taking monodromy representations (see \S\ref{ss_higman_invariant}) gives a canonical bijection
	\begin{equation} \label{eq_fiber_bij_over_k}
    \ff^{-1}(x_E) \rightiso \Epi^{\ext}(\pi_1^{\et}(E^\circ,x_0),G).
    \end{equation}    
    If $E$ is an elliptic curve over $\bC$, then for $x_0\in E^\circ(\bC)$, write $\Pi := \pi_1^\tp(E^\circ(\bC),x_0)$. Taking monodromy representations gives a canonical bijection
    \begin{equation}\label{eq_fiber_bij_over_C}
	\ff^{-1}(x_E)\rightiso \Epi^{\ext}(\Pi,G).    	
    \end{equation}

    In particular, if $G$ is not generated by two elements, then $\cM(G)$ is empty. Let $a,b\in\Pi$ be generators. Let $\gamma_0,\gamma_{1728},\gamma_\infty,\gamma_{-I}\in\Aut(\Pi)$ be the automorphisms given by:
    $$\begin{array}{rcl}
\gamma_0 : (a,b) & \mapsto & (ab^{-1},a) \\
\gamma_{1728} : (a,b) & \mapsto & (b^{-1},a)
\end{array}\qquad
\begin{array}{rcl}
\gamma_\infty : (a,b) & \mapsto & (a,ab) \\
\gamma_{-I} : (a,b) & \mapsto & (a^{-1},b^{-1})
\end{array}$$    
    Let $f : \ol{M(G)}\rightarrow \ol{M(1)}_\bC$ be the map induced by $\ff$, where $\ol{M(G)}$ denotes a smooth compactification of $M(G)$ over $\bC$.\footnote{In fact we will see in Proposition \ref{prop_compactification}(e) below that $M(G)$ even admits a smooth modular compactification over $\bZ[1/|G|]$.} If $j(E)\ne 0,1728$, then viewing $\ol{M(1)}$ as the projective line with coordinate $j$, there is a bijection
    $$f^{-1}(j(E))\rightiso \Epi^\ext(\Pi,G)/\langle\gamma_{-I}\rangle$$
    such that via \eqref{eq_fiber_bij_over_C}, the map $\ff^{-1}(x_E)\rightarrow f^{-1}(j(E))$ induced by $\cM(G)\rightarrow M(G)$ is identified with the the canonical projection $\Epi^\ext(\Pi,G)\rightarrow \Epi^\ext(\Pi,G)/\langle\gamma_{-I}\rangle$. Letting $j(E)$ approach $j = 0,1728,\infty$ respectively, we also obtain bijections
    $$f^{-1}(0)\cong\Epi^\ext(\Pi,G)/\langle\gamma_0\rangle\qquad f^{-1}(1728)\cong\Epi^\ext(\Pi,G)/\langle\gamma_{1728}\rangle\qquad f^{-1}(\infty)\cong\Epi^\ext(\Pi,G)/\langle\gamma_{-I},\gamma_\infty\rangle$$

    \item \label{part_monodromy} (Monodromy) Let $E$ be an elliptic curve over $\bC$, $x_0\in E^\circ(\bC)$, and $\Pi := \pi_1^{top}(E^\circ(\bC),x_0)$. Let $x_E : \Spec\bC\rightarrow\cM(1)$ be the geometric point given by $E$. Then $\Pi$ is a free group of rank $2$, and the canonical map $\Pi\rightarrow H_1(E,\bZ)$ induces an isomorphism $\Pi/[\Pi,\Pi]\cong H_1(E,\bZ)$. Let $\Gamma_E$ denote the orientation-preserving mapping class group of $E^\circ(\bC)$, and let $\Out^+(\Pi)$ be the preimage of $\SL(H_1(E,\bZ))$ under the canonical map
    $$\alpha : \Out(\Pi)\rightarrow\GL(H_1(E,\bZ)).$$
    The outer action of $\Gamma_E$ on $\Pi$ is faithful and identifies $\Gamma_E$ with $\Out^+(\Pi)$. As $\alpha$ is an isomorphism, it induces canonical isomorphisms $\Gamma_E \rightiso \Out^+(\Pi) \rightiso \SL(H_1(E,\bZ))$. Note that the image of $\gamma_{-I}\in\Aut(\Pi)$ (see \ref{part_fibers}) in $\SL(H_1(E,\bZ))$ is central. The analytic theory identifies $\Gamma_E$ with the topological fundamental group of the analytic moduli stack of elliptic curves $\cM(1)^\an$, from which we obtain canonical isomorphisms $\Out^+(\Pi)^\wedge\rightiso\Gamma_E^\wedge\rightiso\pi_1(\cM(1)_{\Qbar},E)$ (where ${}^\wedge$ denotes profinite completion). In particular, we have a canonical map $\Out^+(\Pi)\hookrightarrow\pi_1(\cM(1)_{\Qbar},E)$ with dense image. Relative to this map, the bijection
    $$\ff^{-1}(x_E)\rightiso\Epi^\ext(\Pi,G)$$
    of \eqref{eq_fiber_bij_over_k} is $\Out^+(\Pi)$-equivariant.
    To summarize, we have canonical isomorphisms
    $$\pi_1^{\tp}(\cM(1)^\an,E) \cong\Gamma_E\cong \Out^+(\Pi)\cong \SL(H_1(E,\bZ))$$
    and
    $$\pi_1^{\et}(\cM(1)_{\Qbar},E) \cong \pi_1^{\tp}(\cM(1)^\an,E)^\wedge.$$
    \item \label{part_functoriality} (Functoriality) Let $\cC$ denote the category whose objects are finite groups generated by two elements, and whose morphisms are surjective homomorphisms. If $f : G_1\twoheadrightarrow G_2$ is a morphism in $\cC$, then we obtain a map $\cT_f^\pre : \cT_{G_1}^\pre\rightarrow\cT_{G_2}^\pre$ defined by sending the $G_1$-torsor $X^\circ\rightarrow E^\circ$ to the $G_2$-torsor $X^\circ/\ker(f)\rightarrow E^\circ$ where the $G_2$-action is given by the canonical isomorphism $G_1/\ker(f)\cong G_2$ induced by $f$. This induces a map $\cT_f : \cT_{G_1}\rightarrow\cT_{G_2}$, whence a map
$$\cM(f) : \cM(G_1)\rightarrow\cM(G_2)$$
The maps $\cM(f)$ make the rule sending $G\in\cC$ to the map $\cM(G)\rightarrow\cM(1)$ into an epimorphism-preserving functor from $\cC$ to the category\footnote{Here we mean the (1-)category associated to the (2,1)-category. Ie, the morphisms in this category are precisely the 2-isomorphism classes of 1-morphisms c.f. \cite[\S4]{Noo04}. We note that this category is equivalent to the category of finite locally constant sheaves on $\cM(1)$ (with respect to the \'{e}tale topology).} of stacks finite \'{e}tale over $\cM(1)$. Let $E,\Pi,\Gamma_E$ be as in (4), then in terms of the Galois correspondence for covers of $\cM(1)$, given a surjection $f : G_1\rightarrow G_2$, the induced map $\cM(G_1)\rightarrow\cM(G_2)$ (of $\bZ[1/|G|]$-stacks) is given by the $\Gamma_E$-equivariant map of fibers
    $$f_* : \Epi^{\ext}(\Pi,G_1)\rightarrow\Epi^{\ext}(\Pi,G_2)$$ 
    obtained by post-composing every surjection with $f$. 
    
    \item \label{part_cofinality} (Cofinality --- Asada's theorem) For any stack $\cM$ finite \'{e}tale over $\cM(1)_{\Qbar}$, there is a finite group $G$ such that $\cM$ is dominated by some connected component of $\cM(G)_{\Qbar}$. In particular, for any smooth projective curve $X$ over $\Qbar$, there is a component $\cM_X\subset\cM(G)_\Qbar$ and a finite \'{e}tale morphism $\cM_X\rightarrow X$. In particular, $\cM_X$ is a scheme. Here we can even arrange that $\cM_X\rightarrow X$ be Galois, and for the Galois action to be defined over the field of definition of $\cM_X$ as an element of $\pi_0(\cM(G)_\Qbar)$.
\end{enumerate}
\end{thm}

\begin{proof} Part \ref{part_etale} is \cite[Proposition 3.1.4]{Chen18} . Everything in \ref{part_coarse} except for normality and \'{e}taleness is  \cite[Proposition 3.3.4]{Chen18}.  The normality of $M(G)$ follows from the fact that $M(G)$ is the quotient of a smooth representable moduli problem by a finite group  \cite[\S3.3.3]{Chen18}.  Let $U\subset M(1)$ be the complement of $j = 0,1728$; to see $M(G)$ is \'{e}tale over the complement of $U$, consider a finite \'{e}tale surjection $\cM\rightarrow\cM(G)$ with $\cM$ representable (we may for example take $\cM$ to be the product of $\cM(G)$ with the moduli stack of elliptic curves with full level $p^2$ structure for some $p\mid |G|$).
By \cite[Corollary 8.4.5]{KM85}, the map $\cM_U\rightarrow U$ is \'{e}tale, which implies the \'{e}taleness of $\cM(G)_U\rightarrow U$ \cite[02KM]{stacks}. 


Part \ref{part_combinatorial} is \cite[Proposition 2.2.6(1,2)]{Chen18}. The bijections \eqref{eq_fiber_bij_over_k},\eqref{eq_fiber_bij_over_C} of \ref{part_fibers} follows from part \ref{part_combinatorial}, setting $S = \Spec k$. For the rest of \ref{part_fibers}, see \cite[Proposition 2.1.2, Corollary 2.1.3]{BBCL20}.


For \ref{part_monodromy}, a theorem of Nielsen  gives that  $\alpha$ is an isomorphism \cite[Theorem 3.1]{OZ81}, and the isomorphism $\Gamma_E^\wedge\cong\pi_1(\cM(1)_{\Qbar},E)$ follows from the Riemann existence theorem for stacks \cite[Theorem 20.4]{noo05}. The rest of \ref{part_monodromy} is simply an unfolding of definitions. Part \ref{part_functoriality} is \cite[Proposition 3.2.8]{Chen18}. All but the final sentence of \ref{part_cofinality} is Asada's theorem together with Belyi's theorem (see \cite[Theorem 3.4.2]{Chen18}, \cite{BER11}, \cite[\S7]{asa01}). The final sentence is \cite[Theorem 2.2.3]{BBCL20}.
\end{proof}

\subsubsection{The relation between $\cAdm(G)$ and $\cM(G)$}\label{sss_relation_between_cAdmG_MG}

The key relation between $\cAdm(G)$ and $\cM(G)$ is that the map $\cAdm^0(G)\rightarrow\cM(1)$ is an \'{e}tale gerbe. First we show that it is \'{e}tale. This is a problem in deformation theory (see Proposition \ref{prop_deformation_theoretic_criterion_of_etaleness}). In fact, following  \cite[\S5]{BR11}, we can even describe the ramification indices at points lying over the ``cusp'' of $\ol{\cM(1)}$ represented by a nodal cubic. For this it will be useful to work with the equivalent stack $\ol{\cH}_G$ of stable marked $G$-curves (Definition \ref{def_stack_of_G_curves}).

\begin{prop}[{\cite[Theorem 5.1.5]{BR11}}]\label{prop_deformations} Let $k$ be an algebraically closed field (of characteristic coprime to $|G|$) and $\ol{x} : \Spec k\rightarrow\ol{\cH}_G$ be a geometric point with image $\ol{y}\in\ol{\cM(1)}$. The point $\ol{x}$ corresponds to a stable marked $G$-curve $(C/k,R)$, and $\ol{y}$ is given by the 1-generalized elliptic curve $E := C/G$ with origin $O = R/G$. The natural map $\ol{\cH}_G\rightarrow\ol{\cM(1)}$ given by taking quotients by $G$ induces a morphism from the deformation functor of $\ol{x}$ to that of $\ol{y}$. Let $\Lambda$ be the Cohen ring\footnote{$\Lambda = k$ if $\ch(k) = 0$, and otherwise it is the unique \cite[Theorem 29.2]{Mat89} complete discrete valuation ring with residue field $k$ and maximal ideal $p\Lambda$.} with residue field $k$. The deformation functors for $\ol{x},\ol{y}$ in $\ol{\cH}_{G,\Lambda}$ and $\ol{\cM(1)}_{\Lambda}$ are prorepresentable, and the induced map on universal deformation rings is given (with respect to suitable coordinates) by
\begin{eqnarray*}
\Lambda\ps{T} & \longrightarrow & \Lambda\ps{t}	\\
T & \mapsto & t^e
\end{eqnarray*}
where $e = 1$ if $C$ is smooth, and otherwise $e$ is the order of the stabilizer $G_p$ of any node $p\in C$. In particular, the map $\ol{\cH}_G\rightarrow\ol{\cM(1)}$ is flat and $\cH_G\rightarrow \cM(1)$ is \'{e}tale. Moreover, the substack $\cH_G\subset\ol{\cH}_G$ is open and dense, and the same is true of $\cAdm^0(G)\subset\cAdm(G)$.
\end{prop}

\begin{proof} This statement is a special case of \cite[Theorem 5.1.5]{BR11}. Here we sketch the argument in our situation. Let $D_{C,G}$ (resp. $D_E$) denote the deformation functor of $C$ as a stable marked $G$-curve (resp. of $E$ as a 1-generalized elliptic curve). Since $\ol{\cH}_{G},\ol{\cM(1)}$ are Deligne-Mumford, all deformation functors are prorepresentable, and the universal deformation rings are the completions of the \'{e}tale local rings of $\ol{x}\in\ol{\cH}_{G,\Lambda}$ and $\ol{y}\in\ol{\cM(1)}_\Lambda$ (see Proposition \ref{prop_elr2udr}). Because $\ol{\cH}_{G,\Lambda}$ and $\ol{\cM(1)}_\Lambda$ are smooth and 1-dimensional over $\Spec\Lambda$, the universal deformation rings are power series rings in one variable over $\Lambda$ \cite[0DYL]{stacks}. Let
$$\pi : C\rightarrow E := C/G$$
be the quotient map. Every deformation of $C$ yields by taking quotients a deformation of $E$, so $\pi$ induces a map $D_{C,G}\rightarrow D_E$. It remains to describe the induced map of universal deformation rings. The deformation theory of $C$ is described by equivariant cohomology (see \cite[\S3]{BM00} or \cite[\S3]{BM06}). We briefly recall some definitions. An $(\cO_C,G)$-module is a coherent sheaf which locally on an open affine $\Spec A$ is given by an $A$-module $M$ equipped with a $G$-action satisfying $g(am) = g(a)g(m)$ for any $g\in G,a\in A,m\in M$. Given an $(\cO_C,G)$-module $\cF$ on $C$,  let $\pi_*^G(\cF)$ be the module on $E$ given by $U\mapsto \Gamma(U,\pi_*\cF)^G$, and let $\Gamma^G(C,\cF) := \Gamma(C,\cF)^G$. Let $H_G^i(C,\cF) := R^i\Gamma^G\cF$. Let $\cT_C := \cHom_C(\Omega^1_{C/k},\cO_C)$ be the tangent sheaf. Since $\pi$ is finite and $|G|$ is invertible in $k$, $\pi_*^G : \Mod_{\cO_C,G}\rightarrow \Coh(E)$ is exact, and hence we obtain a canonical isomorphism
$$H^1_G(C,\cT_C(-\cR_\pi))\cong H^1(E,\pi_*^G\cT_C(-\cR_\pi)).$$


By a local calculation \cite[Proposition 4.1.11]{BR11}\footnote{See \cite[\S4.1.2]{BR11} for the definition of points of type I,II,III. In their notation $D = C/G$, and by $\theta_D(-\Delta)$ they mean $\theta_D(-B)$}, we have $\pi_*^G\cT_C(-\cR_\pi) = \cT_E(-O)$ from which we obtain a canonical isomorphism
\begin{equation}\label{eq_isom_deformation_spaces}
H^1_G(C,\cT_C(-\cR_\pi))\cong H^1(E,\cT_E(-O)).
\end{equation}

If $C$ (equivalently $E$) is smooth, $H^1(E,\cT_E(-O))$ is the tangent space of $D_E$ (see \cite[Theorem 5.3]{HartDT} for the unmarked case, also see \cite[\S XI.3]{ACGII}), and $H^1_G(C,\cT_C(-\cR_\pi))$ is the tangent space of $D_{C,G}$ \cite[Proposition 3.2.1]{BM00}. By \eqref{eq_isom_deformation_spaces} these tangent spaces are isomorphic (and 1-dimensional), though this does not tell us that the map on deformation spaces induced by $C\mapsto C/G$ induces an isomorphism. To check this, one can use the fact that the deformations of $E$ and $C$ can be described explicitly by Cech cohomology (equivariant in the case of $C$ (see \cite[\S3.1]{BM00} and \cite[\S5.5]{Gro57}), but the added complication is minimal due to the invertibility of $|G|$ on $k$). There is a natural map of Cech complexes induced by the map $C\rightarrow E$, which (a) induces the isomorphism \eqref{eq_isom_deformation_spaces}, and (b) is easily seen to agree with the map on tangent spaces $D_{C,G}(k[\epsilon])\rightarrow D_E(k[\epsilon])$ induced by $C\mapsto E$. This shows that $\cH_{G,\Lambda}\rightarrow\cM(1)_\Lambda$ induces an isomorphism of deformation rings and hence is \'{e}tale by Proposition \ref{prop_deformation_theoretic_criterion_of_etaleness}. Since $\cH_G\rightarrow\cM(1)$ is flat (Theorem \ref{thm_admG}(b)), this implies that $\cH_G\rightarrow\cM(1)$ is \'{e}tale. 



Now suppose $C$ is nodal. Let $\cC_\Lambda$ be the category of Artin local $\Lambda$-algebras. If $x\in C$ is a node with stabilizer $G_x$, then the local deformation functor $D_{C,G_x,x}$ of $C$ at $x$ \cite[\S3]{BM06} is given by sending an object $A\in\cC_\Lambda$ to the set of deformations of $\what{\cO_{C,x}}\cong k\ps{u,v}/(uv)$ over $A$ as an $A$-algebra with $G_x$-action. The natural ``global-to-local'' morphism $D_{C,G}\rightarrow D_{C,G_x,x}$ is \emph{smooth} (see \cite[Theorem 4.3]{BM06} for the unmarked case). The map on tangent spaces $D_{C,G}(k[\epsilon])\rightarrow D_{C,G_x,x}(k[\epsilon])$ can be identified with the map
$$\varphi : \Ext^1_{\cO_C,G}(\Omega^1_{C/k},\cO_C(-\cR_\pi))\lra \Ext^1_{\what{\cO}_{C,x},G_x}(\what{\Omega}_{\what{\cO}_{C,x}/k},\what{\cO}_{C,x})$$
coming from the local-to-global (equivariant) Ext spectral sequence, which is surjective with kernel $H^1_G(C,\cT_C(-\cR_\pi)) = H^1(E,\cT_E(-O)) = 0$ (see \cite[Lemme 4.1]{BM06} for the unmarked case), so the map $D_{C,G}\rightarrow D_{C,G_x,x}$ induces an isomorphism on tangent spaces.


Since $G$ is invertible in $k$, $\Ext^i_{\what{\cO}_{C,x},G_x}(-,-) = \Ext^i_{\what{\cO}_{C,x}}(-,-)^G$ for all $i\ge 0$. The fact that the $G$-action is balanced at $x$ implies that the $G$-action on $\Ext^1_{\what{\cO}_{C,x}}(\what{\Omega}_{\what{\cO}_{C,x}/k},\what{\cO}_{C,x})$ is trivial (see \cite[\S5.2]{BM06} and \cite[Theorem 5.1.1]{BR11}). On the other hand, this latter group is also the tangent space of the usual (non-equivariant) local deformation functor of a node. In fact, if $D_{C,x}$ denotes the usual deformation functor of the node $x$ without $G_x$-action, then it can be shown using the results of \cite{ST18}, that  that the forgetful map $D_{C,G_x,x}\rightarrow D_{C,x}$ is an \emph{isomorphism}, and that a miniversal family for $D_{C,G_x,x}$ is given by $\Lambda\ps{U,V,T}/(UV-T)$ with $G_x$ action given by $gU = \chi(g)U, gV = \chi(g)^{-1}V$ for some primitive character $\chi : G_x\rightarrow k^\times$. In particular $\Lambda\ps{U,V,T}/(UV-T)$ (without $G_x$-action) defines a miniversal family for $D_{C,x}$.



Let $R_{C,G,x} = \Lambda\ps{t}$ be a miniversal ring of $D_{C,G,x}$, and let $\ul{R}_{C,G,x}$ denote the corresponding functor on $\cC_\Lambda$. By versality, we obtain a morphism $D_{C,G}\rightarrow \ul{R}_{C,G,x}$ which must be an isomorphism since it induces an isomorphism on tangent spaces and both functors are prorepresented by regular complete local rings of the same dimension. Let $y\in E$ be the node lying under $x$. Let $D_{E,y}$ be the local deformation functor of the node $y\in E$, with miniversal family given by $R_{E,y} := \Lambda\ps{T}\rightarrow \Lambda\ps{U,V,T}/(UV-T)$. Then similarly we have $D_{E,y}\cong\ul{R}_{E,y}$. Thus, the map $D_{C,G,x}\rightarrow D_{E,y}$ induced by $D_{C,G}\rightarrow D_E$ can be computed by examining the induced map on miniversal families at the nodes. From the local picture of a balanced node, choosing appropriate coordinates, this map on miniversal families is given by
\begin{eqnarray*}
\Lambda\ps{U,V,T}/(UV-T) & \longrightarrow & \Lambda\ps{u,v,t}/(uv-t) \\	
(U,V,T) & \mapsto & (u^e,v^e,t^e)
\end{eqnarray*}
where $e = |G_x|$, and hence the map on miniversal rings is given by
\begin{eqnarray*}
R_{E,y} = \Lambda\ps{T} & \longrightarrow & \Lambda\ps{t} = R_{C,G,x}\\
T & \mapsto & t^e
\end{eqnarray*}
which via the isomorphisms $D_{C,G}\cong \ul{R}_{C,G,x}$ and $D_{E,y}\cong\ul{R}_{E,y}$ also computes the map on universal deformation rings induced by $D_{C,G}\rightarrow D_E$. Finally, the form of the miniversal families given above implies that any point of $\ol{\cH_G}$ corresponding to a nodal curve is the specialization of a point corresponding to a smooth curve. Thus, $\cH_G$ (resp. $\cAdm^0(G)$) are open and dense inside $\ol{\cH}_G$ (resp. $\cAdm(G)$).
\end{proof}



\begin{cor}\label{cor_empty_if_not_2_gen} If $G$ cannot be generated by two elements, then $\cAdm(G)$ is empty.
\end{cor}
\begin{proof} By Proposition \ref{prop_deformations}, $\cAdm^0(G)\subset\cAdm(G)$ is dense, so it suffices to show that if $G$ is not 2-generated, then $\cAdm^0(G)$ is empty. If $E$ is an elliptic curve over an algebraically closed field $k$ (of characteristic prime to $|G|$, since we're working universally over $\bS = \Spec\bZ[1/|G|]$), then an admissible $G$-cover $\pi : C\rightarrow E$ corresponds by Galois theory to a surjection $\pi_1(E^\circ_k)\rightarrow G$. Since $G$ is prime to $p$, this surjection factors through the maximal prime-to-$p$ quotient of $\pi_1(E^\circ_k)$, which is 2-generated \cite[Expos\'{e} X, Corollaire 3.10]{SGA1}, so $G$ must be 2-generated, as desired.
\end{proof}

While $\cAdm^0(G)$ is \'{e}tale over $\cM(1)$, unlike $\cM(G)$, the map $\cAdm^0(G)\rightarrow\cM(1)$ is generally not representable, hence it is generally not finite. The obstruction to representability is the vertical automorphism groups:

\begin{defn}\label{def_vertical_automorphisms} For a map of algebraic stacks $f : \cX\rightarrow\cY$ and a $T$-valued point $t : T\rightarrow\cX$, there is a homomorphism
$$f_* : \Aut_{\cX(T)}(t)\rightarrow\Aut_{\cY(T)}(f(t))$$
The \emph{vertical automorphism group} of $t$ (relative to $f$) is by definition the kernel of $f_*$. Thus if $\cX,\cY$ are Deligne-Mumford then $f$ is representable if and only if $f_*$ is injective on geometric points \cite[Lemma 4.4.3]{ACV03} (equivalently, the vertical automorphism groups of geometric points are trivial).


For a $T$-valued point $t : T\rightarrow\cX$ where $\cX$ is a stack equipped with a map to $\ol{\cM(1)}$ (this will essentially always be the case in this paper), its \emph{vertical automorphism group} is by default defined to be its vertical automorphism group relative to the map to $\ol{\cM(1)}$, and we will denote it by $\Aut^v(t)$. 
\end{defn}

For a geometric point $\Spec\Omega\rightarrow\cAdm^0(G)$ corresponding to an admissible $G$-cover $\pi : C\rightarrow E$, its vertical automorphism group is the group of the $G$-equivariant automorphisms $\sigma$ of $C$ which induce the identity on $E$ - i.e., which satisfiy $\pi\circ\sigma = \pi$. Any $G$-equivariant automorphism $\sigma$ of $C$ inducing the identity on $E$ restricts to an automorphism of the $G$-torsor $\pi : \pi^{-1}(E_\gen)\rightarrow E_\gen$. Since $C$ is smooth and connected, it is irreducible, so $\pi^{-1}(E_\gen)$ is also irreducible, so every automorphism of $\pi^{-1}(E_\gen)\rightarrow E_\gen$ is given by the action of some $g\in G$. Since the automorphism is $G$-equivariant, we must have $g\in Z(G)$. We have proved:

\begin{prop}\label{prop_smooth_vertical_automorphisms} The vertical automorphism groups of geometric points of $\cAdm^0(G)$ are isomorphic to $Z(G)$.	
\end{prop}

Thus, if one takes $\cAdm^0(G)$ and considers all morphisms as defined ``modulo $Z(G)$'', then the corresponding fibered category should be representable over $\cM(1)$. This is achieved by the process of \emph{rigidification} (see \cite[\S5]{ACV03}, \cite[\S5]{Rom05}). The statement is the following:


\begin{thm}[{\cite[Theorem 5.1.5]{ACV03}}]\label{thm_rigidification} Let $H$ be a flat finitely presented separated group scheme over $\bS$, and let $\cX$ be an algebraic stack over $\bS$. Assume that for each object $\xi\in\cX(S)$, there is an embedding
$$i_\xi : H(S)\hookrightarrow\Aut_S(\xi),$$
which is compatible with pullback, in the following sense: Let $\phi : \xi\rightarrow\eta$ be a morphism in $\cX$ lying over a morphism of schemes $f : S\rightarrow T$, and let $g\in H(T)$; we require that the following diagram commutes:
\[\begin{tikzcd}
\xi\ar[r,"\phi"]\ar[d,"i_\xi(f^*g)"'] & \eta\ar[d,"i_\eta(g)"] \\
\xi\ar[r,"\phi"] & \eta
\end{tikzcd}\]
Since $\cX$ is fibered in groupoids, this implies that $i_\xi(f^*g) = \phi^*i_\eta(g)$. Then, the rigidification of $\cX$ by $H$ is a stack $\cX\fs H$, equipped with a smooth surjective finitely presented morphism $\cX\rightarrow \cX\fs H$ which satisfies

\begin{itemize}
\item[(a)] For any object $\xi\in\cX(S)$ with image $\eta\in (\cX\fs H)(S)$, we have that $H(S)$ lies in the kernel of $\Aut_S(\xi)\rightarrow\Aut_S(\eta)$.
\item[(b)] The morphism $\cX\rightarrow\cX\fs H$ is a gerbe and is universal for morphisms of stacks $\cX\rightarrow\cY$ satisfying (a) above.
\item[(c)] If $S$ is the spectrum of an algebraically closed field, then in (1), we have $\Aut_S(\eta) = \Aut_S(\xi)/H(S)$.
\item[(d)] If $c : \cX\rightarrow X$ is a coarse moduli space for $\cX$, then by (b) it factorizes through a unique morphism $c' : \cX\fs H\rightarrow X$, which is also a coarse moduli space for $\cX\fs H$. In particular, $\cX\rightarrow\cX\fs H$ induces a homeomorphism on topological spaces.
\item[(e)] If $\cX$ is Deligne-Mumford, then so is $\cX\fs H$, and the map $\cX\rightarrow\cX\fs H$ is \'{e}tale.
\item[(f)] If $\cX$ is smooth (resp. proper), then $\cX\fs H$ is smooth (resp. proper).
\end{itemize}
\end{thm}

\begin{remark} In \cite{ACV03} this rigidification would be denoted $\cX^H$, but since it seems closer to taking a quotient than taking fixed points, we will use the notation ``$\cX\fs\; H$'' as introduced in \cite{Rom05}.
\end{remark}

\begin{proof} Everything but (f) and the gerbiness in (b) is \cite[Theorem 5.1.5]{ACV03}. That $\cX\rightarrow\cX\fs H$ is a gerbe follows from the explicit description of $\cX\fs H$ given in \cite[\S5.1.7]{ACV03}, the key fact being that sheafification/stackification is locally surjective on sections/objects. If $\cX$ is smooth, then it admits a smooth covering by a smooth $\bS$-scheme $U\rightarrow \cX$, but then by $U\rightarrow\cX\rightarrow\cX\fs H$ is a smooth covering as well, so $\cX\fs H$ is also smooth. If $\cX$ is proper, then using the fact that $\cX\rightarrow\cX\fs H$ is a gerbe, it is straightforward to check that $\cX\fs H\rightarrow\bS$ satisfies the valuative criteria of properness \cite[0CLZ]{stacks}.
\end{proof}

Applying the theorem to $H = Z(G)$, by Theorem \ref{thm_rigidification}(b) we obtain a canonical factorization
\begin{equation}\label{eq_rigidification_factorization}
	\cAdm(G)\rightarrow\cAdm(G)\fs Z(G)\rightarrow\ol{\cM(1)}
\end{equation}

\begin{defn}\label{def_MGbar} Let $\ol{\cM(G)} := \cAdm(G)\fs Z(G)$. As usual we will write $\ol{M(G)}$ for its coarse moduli space.	
\end{defn}

Next we record some of the basic properties of $\ol{\cM(G)}$ and its relation to $\cAdm(G)$.
\begin{prop}\label{prop_compactification} We work over $\bS = \Spec\bZ[1/|G|]$. Let $\ol{\cM(G)}^0 := \cAdm^0(G)\fs Z(G)$ be the open substack classifying smooth covers.
\begin{itemize}
	\item[(a)] The map $\cAdm(G)\rightarrow\ol{\cM(G)}$ induces an isomorphism on coarse spaces $\Adm(G)\rightiso \ol{M(G)}$. In particular, it induces a homeomorphism on topological spaces. Both stacks are empty if $G$ cannot be generated by two elements.
	\item[(b)] Let $\varphi$ be the map
	$$\varphi : \cAdm^0(G)\rightarrow\cM(G)$$
	sending an admissible cover $\pi : C\rightarrow E$ to the $G$-structure on $E$ determined by the $G$-torsor $C_\gen\rightarrow E_\gen = E^\circ := E - O$. Then $\varphi$ is an \'{e}tale gerbe and factors through an isomorphism
	$$\ol{\cM(G)}^0 := \cAdm^0(G)\fs Z(G)\rightiso\cM(G)$$
	In particular, we find that $\cM(G)$ is naturally isomorphic to an open dense substack of $\ol{\cM(G)}$.
	\item[(c)] The stack $\ol{\cM(G)}$ is smooth and proper of relative dimension 1 over $\bS$.
	\item[(d)] The map $\ol{\cM(G)}\rightarrow\ol{\cM(1)}$ is flat, proper, and quasi-finite.
	\item[(e)] The scheme $\ol{M(G)}$, and hence $\Adm(G)$, is smooth and proper of relative dimension 1 over $\bS$.
	\item[(f)] Let $f : G_1\rightarrow G_2$ be a surjection of finite 2-generated groups, inducing a canonical isomorphism $\psi_f : G_2\rightiso G_1/\Ker(f)$. Consider the functor
	$$f_* : \cAdm(G_1)\rightarrow\cAdm(G_2)$$
	sending an admissible $G_1$-cover $\pi : C\rightarrow E$ to the cover $C/\Ker(f)\rightarrow E$ equipped with an action of $G_2$ via $\psi_f$. Its behavior on morphisms is defined using the universal property of quotients. Then ${f_*}$ induces a proper surjective and quasi-finite map
	$$\ol{\cM(f)} : \ol{\cM(G_1)}\rightarrow\ol{\cM(G_2)}$$
	of stacks over $\ol{\cM(1)}$ whose restriction to $\cM(G_1)$ agrees with the map described in Theorem \ref{thm_basic_properties}(6).
\end{itemize}
\end{prop}

\begin{proof} With begin with (a). If $G$ cannot be generated by two elements, then by Corollary \ref{cor_empty_if_not_2_gen}, $\cAdm(G)$ is empty, so $\ol{\cM(G)}$ is empty. The rest of (a) follows from Theorem \ref{thm_rigidification}(d). Part (c) follows from Theorem \ref{thm_rigidification}(f).




For (b), we first show that the map $p : \ol{\cM(G)}^0\rightarrow\cM(1)$ is finite \'{e}tale. By Theorem \ref{thm_rigidification}(c) and Proposition \ref{prop_smooth_vertical_automorphisms}, $p$ is representable. Next, since $\ol{\cM(G)}$ is proper, the map $\ol{\cM(G)}\rightarrow\ol{\cM(1)}$ is also proper, so $\ol{\cM(G)}^0\rightarrow\cM(1)$ is is also proper. By Proposition \ref{prop_deformations}, $\cAdm^0(G)\rightarrow\cM(1)$ is \'{e}tale. By Theorem \ref{thm_rigidification}(e), $\cAdm^0(G)\rightarrow\cAdm^0(G)\fs Z(G) = \ol{\cM(G)}^0$ is \'{e}tale surjective, so since $p : \ol{\cM(G)}^0\rightarrow\cM(1)$ is representable, we find that $p$ is also \'{e}tale \cite[0CIL]{stacks}. Thus $p$ is representable, proper, and \'{e}tale, so it is finite \'{e}tale \cite[02LS]{stacks}. Thus to show that the map $\ol{\cM(G)}^0\rightarrow\cM(G)$, it would suffice to show that it induces a bijection on geometric fibers over $\cM(1)$. This follows from the observation that if $E$ is an elliptic curve over an algebraically closed field $k$, then every admissible $G$-cover of $E$ restricts to give a $G$-torsor over $E^\circ$, and conversely every $G$-torsor over $E^\circ$ (necessarily tamely ramified over $O\in E$ since we're working universally over $\bS = \Spec\bZ[1/|G|]$) extends by normalization to an admissible $G$-cover of $E$, and these processes are mutually inverse. Thus $\ol{\cM(G)}^0\cong\cM(G)$ as desired. Finally, to see that $\ol{\cM(G)}^0$, and hence $\cM(G)$ is open dense inside $\ol{\cM(G)}$, it suffices to check that $\cAdm^0(G)\subset\cAdm(G)$ is open dense, but this follows from Proposition \ref{prop_deformations}. This proves (b).


For (d), flatness follows from the fact that $\cAdm(G)\rightarrow\ol{\cM(1)}$ is flat (see Theorem \ref{thm_admG}(b)), and that $\cAdm(G)\rightarrow\ol{\cM(G)}$ is flat and surjective. Properness follows from the fact that it is a map of proper stacks. quasi-finiteness follows from the same property for the map $\cAdm(G)\rightarrow\ol{\cM(1)}$ (see Theorem \ref{thm_admG}(b)).


Part (e) follows from Lemma \ref{lemma_coarse_scheme_is_smooth}.


For (f), first we check that $C/\Ker(f)\rightarrow E$ satisfies the conditions (1)-(6) of an admissible $G_1/\Ker(f)\cong G_2$-cover (Definition \ref{def_admissible}). Part (1) follows from Proposition \ref{prop_stability_of_quotient}. Part (6) is immediate from the same property of $C$. Part (3) is Galois theory. Parts (2),(4),(5) are local questions, and follow from the explicit \'{e}tale local description of admissible $G$-covers.


Since formation of the quotient $C/\Ker(f)$ commutes with arbitrary base change (Lemma \ref{lemma_tame_quotients}), the description of $f_*$ is a map of stacks $\cAdm(G_1)\rightarrow\cAdm(G_2)$. By the universal property of rigidification (c.f. Theorem \ref{thm_rigidification}(b)), the composition
$$\cAdm(G_1)\longrightarrow\cAdm(G_2)\longrightarrow\cAdm(G_2)\fs Z(G_2) = \ol{\cM(G_2)}$$
factors uniquely via
$$\cAdm(G_1)\longrightarrow\cAdm(G_1)\fs Z(G_1) = \ol{\cM(G_1)}\longrightarrow\ol{\cM(G_2)}$$
and $\ol{\cM(f)}$ will be defined to be the second arrow in the above factorization. As a map between proper stacks, it is proper. As a map between Deligne-Mumford stacks quasi-finite over $\ol{\cM(1)}$, it is quasi-finite. By construction its restriction to the smooth locus agrees with the map described in Theorem \ref{thm_basic_properties}(6). Since the induced maps over $\cM(1)$ are surjective, since $\cM(G)\subset\ol{\cM(G)}$ is open dense, the surjectivity of $\ol{\cM(f)} : \ol{\cM(G_1)}\rightarrow\ol{\cM(G_2)}$ is a consequence of its properness.

\end{proof}

\begin{defn} In light of Proposition \ref{prop_compactification}, we see that $\ol{\cM(G)}$ is a smooth modular compactification of the moduli stack of elliptic curves with $G$-structures. If $k$ is a field, then a $k$-point of $\ol{\cM(G)}$ (resp. $\ol{M(G)}, \cAdm(G), \Adm(G)$) not lying in $\cM(G)$ (resp. $M(G), \cAdm^0(G), \Adm^0(G)$) is called a \emph{cusp} or a \emph{cuspidal object}.
\end{defn}

\begin{remark} The terminology of ``cusp'' comes from the fact that the schemes $M(G)_\bC$ are disjoint unions of quotients of the upper half plane by finite index subgroups of $\SL_2(\bZ)$ acting via mobius transformations \cite[\S3.3.5]{Chen18}. Each component of $M(G)_\bC$ should be viewed as a (possibly noncongruence) modular curve, whose cusps in the sense of hyperbolic geometry correspond exactly to the points in $\ol{M(G)}_\bC - M(G)_\bC$.
\end{remark}


\subsubsection{A Galois theoretic mantra}\label{sss_mantra}
We will often need to use the Galois correspondence to translate between the category of stacks finite \'{e}tale over $\cM(1)_\Qbar$ and finite sets with $\pi_1(\cM(1)_\Qbar)$-action. Some setup is required to make this precise, so that Theorem \ref{thm_basic_properties} can be applied. To avoid describing the same setup repeatedly, we will gather it into the following ``situation'' (Situation \ref{situation_galois_theory} below) which we will refer to when needed. We begin with a definition of ``equivalence classes of $G$-structures''.


For any 2-generated group $G$, let $\cT_G : \cM(1)\rightarrow\Sets$ be the sheaf of $G$-structures defined in \S\ref{ss_comparison_with_G_structures}. By functoriality of $\cT_G$ in $G$ (c.f. Theorem \ref{thm_basic_properties}(6)), we obtain an action of $\Aut(G)$ on $\cT_G$. By looking at geometric fibers, we find that $\Inn(G)$ acts trivially, and that the induced action of $\Out(G)$ on $\cT_G$ is \emph{free}. In particular we get an action of $\Out(G)$ on the stack $\cM(G)$ over $\cM(1)$. Thus for any subgroup $A\subset\Out(G)$, the quotient $\cT_G/A$ is also a finite locally constant sheaf on $\cM(G)$.

\begin{defn} Let $A\subset\Out(G)$ be a subgroup. The sections of $\cT_G/A$ over an elliptic curve $E/S$ will be called $(G|A)$-structures on $E/S$. The moduli stack of elliptic curves with $(G|A)$-structures is the quotient $\cM(G)/A$, defined as the stack over $\cM(1)$ associated to the quotient sheaf $\cT_G/A$. When $A = \Out(G)$, we will call them \emph{absolute $G$-structures}\footnote{We distinguish this special case because absolute structures have a natural moduli interpretation. Specifically, let $k = \ol{k}$ have characteristic coprime to $|G|$, and $E/k$ an elliptic curve. An absolute $G$-structure on $E$ is represented by a $G$-torsor over $E^\circ$. Two absolute $G$-structures over $k = \ol{k}$ are the same if their corresponding torsors are isomorphic \emph{as covers of $E^\circ$}. Namely, we don't require the isomorphism to be $G$-equivariant. Thus an absolute $G$-structure on $E$ amounts to giving a finite Galois cover of $E$, branched only above the origin, whose Galois group is isomorphic to $G$ (c.f. \cite[\S2.4]{BBCL20})}, and the corresponding stack $\cM(G)/\Out(G)$ will be denoted $\cM(G)^\abs$.
\end{defn}

\begin{sit}\label{situation_galois_theory} Let $E$ be an elliptic curve over $\Qbar$. Let $t\in E^\circ(\bC)$ be a point. Let $\Pi := \pi_1^\tp(E^\circ(\bC),t)$. Let $a,b$ be a basis for $\Pi$ with intersection number $+1$ (we will call this a ``positively oriented basis''). Let $x_E : \Spec\Qbar\rightarrow\cM(1)$ be the geometric point corresponding to $E$. Recall that $\Aut^+(\Pi)$ is defined to be the subgroup of $\Aut(\Pi)$ which induce determinant 1 automorphisms of $\Pi/[\Pi,\Pi]\cong\bZ^2$, and $\Out^+(\Pi) := \Aut^+(\Pi)/\Inn(\Pi)$. By Theorem \ref{thm_basic_properties}(5), there is a canonical map $i : \Out^+(\Pi)\hookrightarrow\pi_1(\cM(1)_\Qbar,x_E)$ which is injective with dense image. It induces an isomorphism $\what{\Out^+(\Pi)}\cong\pi_1(\cM(1)_\Qbar,x_E)$ which we will use to identify the two groups. In particular, we obtain an action of $\Out^+(\Pi)$ (and hence $\Aut^+(\Pi)$) on the geometric fiber over $x_E$ of any stack finite \'{e}tale over $\cM(1)$. Let $G$ be a finite 2-generated group, and let
$$\ff : \cM(G)\rightarrow\cM(1)$$
be the forgetful map. The fiber $\mf{f}^{-1}(x_E)$ is the set of isomorphism classes of geometrically connected $G$-torsors over $E^\circ$. Taking monodromy representations (see \S\ref{ss_higman_invariant}) gives a bijection
$$\alpha_{G,x_E} : \ff^{-1}(x_E)\rightiso\Epi^\ext(\Pi,G)$$
which is $\Out^+(\Pi)$-equivariant, where $\Out^+(\Pi)$ acts on $\ff^{-1}(x_E)$ via $i$ (c.f. Theorem \ref{thm_basic_properties}). Via this bijection, the Higman invariant can also be expressed at the level of $\Epi^\ext(\Pi,G)$ (c.f. Remark \ref{remark_topological_higman}). To be precise, there is a commutative diagram
\[\begin{tikzcd}
	\ff^{-1}(x_E)\ar[r,"\alpha_{G,x_E}"]\ar[rd,"\Hig"'] & \Epi^\ext(\Pi,G)\ar[d,"{\varphi\mapsto\varphi([b,a])}"] \\
	 & \Cl(G)
\end{tikzcd}\]

Let $\cC$ be the category whose objects are 2-generated finite groups, and where morphisms are surjections. Let $\FinSets_{\Out^+(\Pi)}$ be the category of finite sets with $\Out^+(\Pi)$-action, and let $\FEt_{\cM(1)_\Qbar}$ be the category of finite \'{e}tale maps to $\cM(1)$. Let $\cM : \cC\rightarrow\FEt_{\cM(1)_\Qbar}$ be the epimorphism-preserving functor of Theorem \ref{thm_basic_properties}(6). Let $F_{x_E} : \FEt_{\cM(1)_\Qbar}\rightarrow\FinSets_{\Out^+(\Pi)}$ be the fiber functor at $x_E$ (this is an equivalence by Galois theory). Then the diagram
\[\begin{tikzcd}
	\cC\ar[r,"\cM"]\ar[rd,"{\Epi^\ext(\Pi,-)}"'] & \FEt_{\cM(1)_\Qbar}\ar[d,"F_{x_E}"] \\
	 & \FinSets_{\Out^+(\Pi)}
\end{tikzcd}\]
2-commutes, in the sense that the two paths are isomorphic as functors, with the isomorphism defined using the isomorphisms $\alpha_{G,x_E}$. In particular, if $A\subset\Out(G)$ is a subgroup, if $g : \cM(G)\rightarrow\cM(G)/A$ is the quotient map, and if $\ff_A : \cM(G)/A\rightarrow\cM(1)$ is the forgetful map, then the commutative diagram
\[\begin{tikzcd}
\cM(G)\ar[r,"g"]\ar[rd,"\ff"'] & \cM(G)/A\ar[d,"\ff_A"] \\
 & \cM(1)
\end{tikzcd}\quad\text{induces the diagram}\quad
\begin{tikzcd}
	\ff^{-1}(x_E)\ar[d,"g"]\ar[r,"\cong"] & \Epi^\ext(\Pi,G)\ar[d,"h"] \\
	\ff_A^{-1}(x_E)\ar[r,"\cong"] & \Epi^\ext(\Pi,G)/A
\end{tikzcd}
\]
where $h$ is the quotient map, and the bottom bijection is also $\Out^+(\Pi)$-equivariant.
\end{sit}

\section{Degrees of components of $\Adm(G)$ over $\ol{M(1)}$}\label{section_main_result}
In this section, after recalling some definitions and the formalism of the relative dualizing sheaf, ramification divisor, and degrees of line bundles on 1-dimensional stacks, in \S\ref{ss_congruence} we establish the basic form of our main congruence on the degrees of components of $\Adm(G)$ over $\ol{M(1)}$.

\subsection{The dualizing sheaf and sheaf of differentials for universal families over moduli stacks}\label{ss_dualizing_sheaf}

In this section we will recall what it means to give a sheaf on a stack, and we will describe the relationship between the sheaf of relative differentials and the relative dualizing sheaf in the setting of a flat proper representable map of algebraic stacks with Cohen-Macaulay fibers. There is essentially no difference to the case of schemes. The main content of this section is to carefully define the objects we will need. We work universally over a base scheme $\bS$.

\subsubsection{The sheaf of relative differentials}\label{sss_differentials}

Let $\bS$ be a scheme. Let $\cC$ be an algebraic stack over $\bS$. We can view $\cC$ as a site by giving $\cC$ the inherited topology from $(\Sch/\bS)_\fppf$. Recall that this means that given an object $\xi : T\rightarrow\cC$ with $T$ a scheme, a covering of $\xi$ is given by a family $\{\xi_i\}$ of objects of $\cC$ such that each $\xi_i$ is given by $\xi_i : T_i\stackrel{t_i}{\rightarrow} T\rightarrow\cC$ such that $\{t_i : T_i\rightarrow T\}$ is a covering family in $(\Sch/\bS)_\fppf$. A sheaf on $\cC$ is just a sheaf on the corresponding site \cite[06TF]{stacks}. The structure sheaf on $\cC$ is the sheaf of rings $\cO_\cC$ given by $(U\rightarrow\cC)\mapsto\Gamma(U,\cO_U)$. A sheaf $\cF$ of $\cO_\cC$-modules is quasicoherent if its restrictions to each $T\rightarrow\cC$ is a quasicoherent sheaf on $T$.


Let $f : \cC\rightarrow\cX$ be a representable map of algebraic stacks. Given $t : T\rightarrow\cC$ with $T$ a scheme, consider the diagram
\begin{equation}\label{eq_stacky_sheaf_diagram}
\begin{tikzcd}
\cC_T\ar[d]\ar[r,"\tilde{t}"] & \cC\ar[d,"f"] \\
T\ar[u,bend left = 30,"\tau"]\ar[ru,"t"]\ar[r] & \cX
\end{tikzcd}
\end{equation}
where the outer square is cartesian, and $\tau$ is the section induced by $t$. This diagram is commutative if one ignores $\tau$. Let $\Omega_{\cC/\cX}$ be the presheaf on $\cC$ given by
$$\Omega_{\cC/\cX}(t) := \Gamma(T,\tau^*\Omega_{\cC_T/T})$$
We note that given a diagram of the form \eqref{eq_stacky_sheaf_diagram}, it ``factors through'' a diagram (commutative if one ignores sections)
\[\begin{tikzcd}
\cC_T\ar[r,"\tilde{\tau}"]\ar[d] & \cC_{\cC_T}\ar[r]\ar[d] & \cC\ar[d,"f"] \\
T\ar[r,"\tau"] & \cC_T\ar[u,bend left = 30,"\tilde{\tau}"]\ar[r,"f\circ\tilde{t}"]\ar[ru,"\tilde{t}"] & \cX
\end{tikzcd}\]
where the right square is cartesian, and $\tilde{\tau}$ is the section induced by $\tilde{t}$ (thus the left square is also cartesian). Since sheaves of relative differentials commute with base change, we have
\begin{equation}\label{eq_as_expected}
\Omega_{\cC/\cX}(\tilde{t} : \cC_T\rightarrow\cC) := \Gamma(\cC_T,\tilde{\tau}^*\Omega_{\cC_{\cC_T}/\cC_T}) \cong \Gamma(\cC_T,\Omega_{\cC_T/T})	
\end{equation}

The restriction maps of $\Omega_{\cC/\cX}$ are defined as follows. Given a map $s : S\rightarrow\cC$ and a map $p : S\rightarrow T$ with $t\circ p = s$, we obtain a commutative diagram
\[\begin{tikzcd}
\cC_S\ar[d]\ar[r,"\tilde{p}"] & \cC_T\ar[d]\ar[r,"\tilde{t}"] & \cC\ar[d,"f"] \\
S\ar[u,bend left=30,"\sigma"]\ar[r,"p"] & T\ar[u,bend left = 30,"\tau"]\ar[r]\ar[ru,"t"] & \cX
\end{tikzcd}\]
with both squares cartesian and where the section $\sigma$ is induced by $s = t\circ p$. Note that $\cC_S,\cC_T$ are schemes since $f$ is representable. From the diagram we obtain natural restriction maps on global sections
\begin{equation}\label{eq_restriction_maps}
\rho_{p,t,s} : \Omega_{\cC/\cX}(t) := \Gamma(T,\tau^*\Omega_{\cC_T/T})\longrightarrow\Gamma(S,p^*\tau^*\Omega_{\cC_T/T})\rightiso\Gamma(S,\sigma^*\tilde{p}^*\Omega_{\cC_T/T})\rightiso\Gamma(S, \sigma^*\Omega_{\cC_S/S}) =: \Omega_{\cC/\cX}(s)
\end{equation}
The first map is the usual pullback map on sections. The second map comes from the unique isomorphism $p^*\tau^*\cong \sigma^*\tilde{p}^*$ of functors $\Mod(\cO_{\cC_T})\rightarrow\Mod(\cO_S)$, the uniqueness coming from the fact that $p^*\tau^*$ and $\sigma^*\tilde{p}^*$ are both left adjoints of $\tau_*\circ p_* = \tilde{p}_*\circ \sigma_* = (\tau\,\circ\, p)_*$. The final map is induced by the unique isomorphism $\tilde{p}^*\Omega_{\cC_T/T}\cong \Omega_{\cC_S/S}$ coming from the universal property of the sheaf of differentials.


Since the restriction of $\Omega_{\cC/\cX}$ to any scheme $t : T\rightarrow\cC$ is just the quasicoherent sheaf on $(\Sch/T)_\fppf$ associated to the usual quasicoherent sheaf $\tau^*\Omega_{\cC_T/T}$ on $T$. Thus $\Omega_{\cC/\cX}$ is a quasicoherent sheaf, and by \eqref{eq_as_expected} it agrees with the usual sheaf of relative differentials when $\cC,\cX$ is are schemes.


\begin{defn} Let $f : \cC\rightarrow\cX$ be a representable morphism of algebraic stacks. Let $\Omega_{\cC/\cX}$ denote the sheaf of relative differentials, as defined above.
\end{defn}

Alternatively, the sheaf $\Omega_{\cC/\cX}$ can be defined on a presentation for $\cC/\cX$. Namely, let $U$ be a scheme and let $U\rightarrow\cX$ now be a smooth surjective morphism. Then $\cC_U := \cC\times_\cX U\rightarrow\cC$ is also smooth and surjective. Let $R := \cC_U\times_\cC \cC_U$ and $S := U\times_\cX U$. This determines a commutative diagram

\[\begin{tikzcd}
R\ar[r,shift left,"\pr_0"]\ar[r,shift right,"\pr_1"']\ar[d] & \cC_U\ar[d]\ar[r] & \cC\ar[d] \\
S\ar[r,shift left,"\pr_0"]\ar[r,shift right,"\pr_1"'] & U\ar[r] & \cX
\end{tikzcd}\]
which induces isomorphisms $[\cC_U/R]\cong\cC$ and $[U/S]\cong\cX$ \cite[04T4]{stacks}. To give a quasicoherent sheaf on $\cC$ is the same as giving a quasicoherent sheaf $\cF$ on $\cC_U$ together with an isomorphism
$$\alpha : \pr_0^*\cF\rightiso \pr_1^*\cF$$
over $R$ satisfying a certain cocycle condition \cite[0441,06WT]{stacks}. Letting $\cF = \Omega_{\cC_U/U}$ and $\alpha$ the canonical isomorphism $\pr_0^*\Omega_{\cC_U/U}\cong\Omega_{R/S}\cong\pr_1^*\Omega_{\cC_U/U}$, one checks that $\alpha$ satisfies the cocycle condition and hence this determines a sheaf $\Omega_{\cC/\cX}$ on $\cC$ which agrees with our earlier definition.

\subsubsection{The relative dualizing sheaf}\label{sss_dualizing}

Let $f : C\rightarrow X$ be a flat proper finitely presented morphism of schemes. Recall that if $X$ is quasi-compact quasi-separated, then the functor $Rf_* : D(\cO_C)\rightarrow D(\cO_X)$ has a right adjoint $f^! : D(\cO_X)\rightarrow D(\cO_C)$ \cite[0B6S]{stacks}. In this case the relative dualizing complex is $\omega^\blt_{C/X} := f^!(\cO_X)$, and is equipped with a ``trace map'' $\tr_f : Rf_*\omega^\blt_{C/X}\rightarrow\cO_X$ coming from adjunction. For general $X$, the uniqueness of the adjoint implies that the relative dualizing complexes locally defined on affine opens of $X$ glue to yield a pair $(\omega_{C/X}^\blt,\tr_f)$, where $\omega_{C/X}^\blt\in D(\cO_X)$ and $\tr_f : Rf_*\omega^\blt_{C/X}\rightarrow\cO_X$ is such that the pair $(\omega_{C/X}^\blt,\tr_f)$ restricts to the relative dualizing complex over affine opens of $X$ \cite[0E61]{stacks}. Moreover, such a pair is unique up to unique isomorphism and commutes with arbitrary base change \cite[0E5Z,0E60]{stacks}.


If $f$ moreover has Cohen-Macaulay and geometrically connected fibers of constant relative dimension $d$, then the dualizing complex $\omega_{C/X}^\blt$ has a unique nonzero cohomology sheaf which is in degree $-d$ \cite[0BV8]{stacks}\footnote{Technically, in order to use this, one should first reduce to the case where $X$ is affine, and then use Noetherian approximation to note that $f$ is the base change of a morphism $f_0 : C_0\rightarrow X_0$ where $X_0$ is of finite type over $\bZ$, and $f_0$ is also flat proper finitely presented with Cohen-Macaulay fibers \cite[01ZA,081C,045U]{stacks}. In this situation \cite[0BV8]{stacks} applies, and by pullback we deduce the result for $f$.}. Let $\omega_{C/X} := H^{-d}(\omega_{C/X}^\blt)$ denote this unique nonzero cohomology sheaf, which is called the (relative) dualizing sheaf for $f$.


\begin{defn} Let $f : \cC\rightarrow\cX$ be a representable flat proper finitely presented morphism of algebraic stacks with geometrically connected and Cohen-Macaulay fibers. The relative dualizing sheaf $\omega_{\cC/\cX}$ is defined as follows. Given any diagram of the form \eqref{eq_stacky_sheaf_diagram}, define
$$\omega_{\cC/\cX}(t) := \Gamma(T,\tau^*\omega_{\cC_T/T})$$
As in the case for $\Omega_{\cC/\cX}$, we obtain natural restriction maps $\rho_{p,t,s}$ as in \eqref{eq_restriction_maps}. Since relative dualizing sheaves commute with arbitrary base change, the same discussion as above implies that the restriction maps are compatible and define a quasicoherent sheaf on $\cC$, which is called the \emph{relative dualizing sheaf} of $\cC/\cX$.
\end{defn}

\subsubsection{The canonical map $\Omega_{\cC/\cX}\rightarrow\omega_{\cC/\cX}$}

We begin with a well-known lemma.

\begin{lemma}\label{lemma_very_ample} Let $X$ be a scheme and let $(f : C\rightarrow X, R)$ be a stable marked curve. Then the invertible sheaf
$$\omega_{C/X}(R)^{\otimes 3} := (\omega_{C/X}\otimes_{\cO_C}\cO_C(R))^{\otimes 3}$$
is very ample relative to $C\rightarrow X$, and $f_*\omega_{C/X}(R)^{\otimes 3}$ is locally free.
\end{lemma}

\begin{proof} That $f_*\omega_{C/X}(R)^{\otimes 3}$ is locally free is \cite[Corollary 1.11]{Knud83II}, so it remains to prove that $\omega_{C/X}(R)^{\otimes 3}$ is very ample. By Noetherian approximation (see Remark \ref{remark_noetherian_approximation}), we may assume $X$ is Noetherian. Let $\cL := \omega_{C/X}(R)^{\otimes 3}$. First we note that being very ample is fpqc local on the target - this follows from \cite[01VR(4)]{stacks} and the fact that formation of $f_*\cL$ commutes with flat base change \cite[02KH]{stacks}. Thus \cite[Corollary 1.9]{Knud83II} implies that the restrictions of $\cL$ to fibers is very ample, so by \cite[0D2S]{stacks}, $\cL$ is $f$-relatively ample, so it is locally projective. Next, since $f$ has geometrically connected fibers, we have $f_*\cO_C \cong \cO_X$ \cite[0E0D]{stacks}. Since $H^1(C_x,\cL|_{C_x}) = 0$ for all $x\in X$ \cite[Theorem 1.8]{Knud83II}, we find that $f_*\cL$ is also an invertible sheaf \cite[Theorem 12.11]{HartAG}, so for any $x\in X$, there is an open neighborhood $x\in U\subset X$ and a map $j : f^{-1}(U)\rightarrow \bP^n_U$ such that $j^*\cO(1) \cong \cL|_{f^{-1}(U)}$ and $j|_{C_x}$ is a closed immersion. Since we can find an open $V\subset U$ containing $x$ such that $j|_{C_U}$ is a closed immersion \cite[Proposition 12.93]{GW20}, this implies that $\cL$ is very ample for $C_U\rightarrow U$. Since very ampleness is local on the target this shows that $\cL$ is very ample.
\end{proof}

The purpose of this section is to state the following result:

\begin{prop}\label{prop_canonical_map} Let $\bS$ be a Noetherian scheme. Suppose we have a commutative diagram of Noetherian algebraic stacks (over $\bS$)
\[\begin{tikzcd}
\cR\ar[rd,"h"']\ar[r,hookrightarrow] & \cC\ar[d,"f"] \\
 & \cX
\end{tikzcd}\]
where $f$ is a prestable curve and $h$ is a finite \'{e}tale map such that the base change of the diagram via any map $T\rightarrow\cX$ with $T$ a scheme defines a stable marked curve $(\cC_T/T,\cR_T)$ (Definition \ref{def_marked_G_curves}), where $\cR_T := \cR\times_\cX T$ is an effective Cartier divisor on $\cC_T$. Then there is a canonical map
$$\varphi : \Omega_{\cC/\cX}\longrightarrow\omega_{\cC/\cX}$$
which is an isomorphism at every geometric point $\Spec k\rightarrow\cC$ where $f$ is smooth.
\end{prop}

The map $\varphi$ comes from the theory of determinants \cite{KM76}. If $\cC,\cX$ are schemes, then $f : \cC\rightarrow\cX$ factors through an immersion into a smooth proper $\cX$-scheme $\cP$, in which case this is just \cite[0E9Z]{stacks} and \cite[\S1]{Knud83II}. We briefly explain how to construct the map $\varphi$ when $\cC,\cX$ are stacks.


Let $\cL := f_*(\omega_{\cC/\cX}(\cR)^{\otimes 3})$. One should take $\cP$ to be the projective bundle $\bP(f^*\cL)$ associated to $f^*\cL$, defined using the functorial characterization of projective space \cite[01NS]{stacks}. Thus $\cP := \bP(f^*\cL)$ is the stack over $\cX$ associated to a certain sheaf on $\cX$, whose restrictions to schemes agrees with the usual projective bundle. In particular this means that the map $\cP\rightarrow\cX$ is representable, so $\cP$ is an algebraic stack, which is moreover smooth since Lemma \ref{lemma_very_ample} implies $\cL$ is locally free, so $\cP$ is locally (over $\cX$) a projective space. The functorial characterization of $\cP = \bP(f^*\cL)$ also yields a factorization of $f$:

\begin{equation}\label{eq_projective_bundle_diagram}
\begin{tikzcd}
\cC\ar[d,"f"']\ar[r,"i"] & \cP\ar[ld,"g"] \\
\cX
\end{tikzcd}
\end{equation}

where $i$ is a closed immersion (since $f,g$ are representable, $i$ is also representable, so we may check this locally on $\cX$, using the very ampleness of $\omega_{\cC/\cX}(R)^{\otimes 3}$). Let $u : U\rightarrow\cX$ be an \'{e}tale covering by a scheme, which we may assume is Noetherian since $\cX$ is Noetherian. This defines a presentation of $\cX$ as $[U/R]$ where $R := U\times_\cX U$, and by pullback one obtains presentations of $\cC$ and $\cP$ as quotient stacks of $\cC_U$ and $\cP_U$, which fit into a diagram extending \eqref{eq_projective_bundle_diagram}. The standard theory (see \cite[0E9Z]{stacks} and \cite[\S1]{Knud83II}) yields a canonical map
$$\varphi_U : \Omega_{\cC_U/U}\rightarrow\omega_{\cC_U/U}.$$
which is an isomorphism over the smooth locus of $f$. We leave it to the reader to check that this map is compatible with the comparison isomorphisms associated to the presentation of $\cC$. This implies that $\varphi_U$ descends to a map $\varphi : \Omega_{\cC/\cX}\rightarrow\omega_{\cC/\cX}$ as desired \cite[06WT]{stacks}.

\subsection{The universal family over $\cAdm(G)$ and its reduced ramification divisor}\label{ss_ramification_divisor_of_universal_family}
Let $G$ be a finite group. We work universally over a $\bZ[1/|G|]$-scheme $\bS$. The universal admissible $G$-cover
\begin{equation}\label{eq_universal_family}
\cC(G)\stackrel{}{\lra}\cE(G)\lra\cAdm(G)	
\end{equation}

is defined as follows. The objects of $\cE(G)$ over a scheme $T$ are pairs $(\pi : C\rightarrow E,\sigma)$, where $\pi$ is an admissible $G$-cover of a 1-generalized elliptic curve $E$ over $T$ and $\sigma : T\rightarrow E$ is a section. Morphisms are morphisms in the category $\cAdm(G)$ respecting the sections $\sigma$. The map $\cE(G)\rightarrow\cAdm(G)$ is given by forgetting $\sigma$. Similarly, the objects of $\cC(G)$ over a scheme $T$ are pairs $(\pi : C\rightarrow E, \sigma)$ where $\pi$ is again an admissible $G$-cover of a 1-generalized elliptic curve $E$ over $T$, and $\sigma : T\rightarrow C$ is a section, with morphisms similarly defined. The map $\cC(G)\rightarrow\cE(G)$ sends $(\pi : C\rightarrow E,\sigma)$ to $(\pi : C\rightarrow E,\pi\circ\sigma)$. In particular, we find that for a geometric point $\ol{z}$ of $\cC(G)$ mapping to $\ol{x}$ in $\cAdm(G)$, the automorphism group of $\ol{z}$ is precisely the group of automorphisms of the admissible $G$-cover $\pi_{\ol{x}} : \cC(G)_{\ol{x}}\rightarrow\cE(G)_{\ol{x}}$ which fix $\ol{z}$.



It follows from the above discussion that the maps \eqref{eq_universal_family} are representable and for any scheme $T$ and map $T\rightarrow\cAdm(G)$ given by an admissible $G$-cover $\pi : C\rightarrow E$ over $T$, the pullback of \eqref{eq_universal_family} to $T$ is canonically isomorphic to $C\stackrel{\pi}{\rightarrow} E\ra T$.


Let $\cR_{\cC(G)/\cE(G)}$ be the strictly full subcategory of $\cC(G)$ whose objects over a scheme $T$ are pairs $(\pi : C\rightarrow E,\sigma)$ where $\sigma : T\rightarrow C$ factors (uniquely) through the closed immersion $\cR_\pi\hookrightarrow C$. If $T\rightarrow\cAdm(G)$ is a map given by an admissible $G$-cover $\pi : C\rightarrow E$, then $\cR_{\cC(G)/\cE(G)}\times_{\cAdm(G)}T = \cR_\pi$. It follows that the inclusion map $\cR_{\cC(G)/\cE(G)}\subset\cC(G)$ is a closed immersion.


Let $\cX\subset\cAdm(G)$ be a connected component, and let
$$\cC\stackrel{\pi}{\longrightarrow}\cE\lra\cX$$
denote the restriction of the universal family to $\cX$. Let $\cR_\pi = \cR_{\cC/\cE}$ denote the restriction of $\cR_{\cC(G)/\cE(G)}$ to $\cC$.

\begin{defn}\label{def_reduced_ramification_divisor} The \emph{reduced ramification divisor} of $\cC/\cE$ is the closed substack $\cR_{\cC/\cE}\subset\cC$ defined above.
\end{defn}

\begin{prop}\label{prop_stacky_RRD} The components of the reduced ramification divisor can be controlled as follows.
\begin{itemize}
	\item[(a)] Working over $\bS = \Spec\bZ[1/|G|]$, let $\cX\subset\cAdm(G) = \cAdm(G)_\bS$ be a connected component classifying covers with ramification index $e$. Let $\cC\stackrel{\pi}{\rightarrow}\cE\rightarrow\cX$ be the universal family. Then the reduced ramification divisor $\cR_\pi = \cR_{\cC/\cE}$ is finite \'{e}tale over $\cX$ of degree $|G|/e$. Let $\ol{z}$ be a geometric point of $\cR_\pi$ with stabilizer $G_{\ol{z}}\le G$. The connected components of $\cR_\pi$ are all isomorphic, each Galois over $\cX$ with Galois group isomorphic to a subgroup of $N_G(G_{\ol{z}})/G_{\ol{z}}$.
	\item[(b)] Working over $\bS = \Spec \bZ[1/|G|,\zeta_e]$, let $\cX\subset\cAdm(G) = \cAdm(G)_\bS$ be a connected component classifying covers with ramification index $e$. Let $\cC\stackrel{\pi}{\rightarrow}\cE\rightarrow\cX$ be the universal family. Then the conclusions of (a) hold and the Galois groups of components of $\cR_\pi$ are moreover isomorphic to a subgroup of $C_G(G_{\ol{z}})/G_{\ol{z}}$.
\end{itemize}
\end{prop}
\begin{proof} For part (a), from the discussion above, we may argue exactly as in Proposition \ref{prop_components_of_ramification_divisor}(a). In part (b), the conclusions of (a) hold by base change, so it remains to justify the claim about the Galois groups. If $e = 1$ then the statement is trivial, so we may assume $e\ge 2$, and hence $\cC\rightarrow\cX$ has fibers of genus $g\ge 2$. Let $\cR\subset\cR_\pi$ be a connected component. Suppose there exists a map $f : U\rightarrow\cX$ with $U$ a regular integral scheme such that $f^*\cR$ is connected. Then we may apply Proposition \ref{prop_components_of_ramification_divisor}(b) to the pullback $f^*\cC\rightarrow f^*\cE$, which would give us the desired result. To construct the map $f$, let $K := \bQ(\zeta_e)$, let $\cM_g$ (resp. $\cM_{g,n}$) denote the moduli stack of smooth curves of genus $g$ (resp. with $n$ distinct marked points) over $K$ (see \cite{Knud83II}), and let $\cY_K\subset\cX_K$ be the open substack consisting of smooth objects. There is a natural map
$$h : \cY_K\rightarrow\cM_g$$
sending an admissible cover $C\rightarrow E$ to the genus $g$ curve $C$. The Riemann-Hurwitz formula together with Hurwitz's automorphism theorem implies that for fixed $g$, there is a large enough $n$ such that curves of genus $g$ (in characteristic 0) do not have any automorphisms with $n$ fixed points, so the stack $\cM_{g,n}$ is a scheme. Let us fix such an $n$. Forgetting marked points yields natural maps
$$\cM_{g,n}\lra\cM_{g,n-1}\lra\cdots\lra\cM_{g,1}\lra\cM_{g,0} = \cM_g$$
where the source of each map is the universal family over the target. Thus the composition $\cM_{g,n}\rightarrow\cM_g$ is representable, smooth, and proper with geometrically connected fibers. Let $U := h^*\cM_{g,n}$, then since automorphisms of objects in $\cAdm(G)$ are determined by how it behaves on the covering curve $C$, the map $h$ is representable and hence $U$ is a scheme. Since $\cY_K$ is regular and connected and $U\stackrel{\pr}{\rightarrow}\cY_K$ is proper and smooth (hence open and closed) with connected fibers, $U$ is also regular and connected, so it is a regular integral scheme. Let $f$ be the composition
$$f : U\stackrel{\pr}{\lra}\cY_K\stackrel{i}{\lra}\cX$$
By Proposition \ref{prop_compactification}(e), the coarse scheme $X$ of $\cX$ is smooth over $\bS$, hence normal, hence irreducible since it is connected, so the same is true of the coarse scheme of $\cR$. This implies that $\cX$ and $\cR$ are irreducible, so the restriction $i^*\cR$ is also irreducible. Next, the map $f^*\cR\rightarrow i^*\cR$ is smooth proper with geometrically connected fibers since the same is true of $\pr$. Since $i^*\cR$ is connected, this implies that $f^*\cR$ is also connected, as desired.
\end{proof}

\subsection{Restriction of the relative dualizing sheaf to a ramified section}\label{ss_restriction}

Let $G$ be a finite group, we work universally over a $\bZ[1/|G|]$-scheme $\bS$.

\begin{lemma}\label{lemma_alg} Let $A$ be a ring, and $I\subset A$ an ideal which is $A$-flat, and $e\ge 1$ an integer.
\begin{itemize}
\item[(a)] The multiplication map $\prod_{i=1}^e I\rightarrow I^e$ induces a canonical $A$-module isomorphism
$$\psi : I^{\otimes e}\cong I^e$$
\item[(b)] For $x_1,\ldots,x_e\in I$ with images $\ol{x_i}\in I/I^2$ and $a\in A/I$, the map
\begin{eqnarray*}
\xi_e : I^{\otimes e}\otimes_A A/I & \longrightarrow & (I/I^2)^{\otimes e} \\
(x_1\otimes\cdots\otimes x_e)\otimes a & \mapsto & a(\ol{x_1}\otimes\cdots\otimes \ol{x_n})
\end{eqnarray*}
is an isomorphism.

\end{itemize}
\end{lemma}
\begin{proof} Part (a) is \cite[\S1 Theorem 2.4]{Liu02}. For (b), note that (a) applied to the unit ideal in $A/I$ gives a canonical isomorphism $\phi : (A/I)^{\otimes e}\rightiso A/I$. Next, tensoring the exact sequence $0\rightarrow I\rightarrow A\rightarrow A/I\rightarrow 0$ by the $A$-module $I$ and using (a) shows that $\xi_1 : I\otimes_A A/I\rightarrow I/I^2$ is an isomorphism. That $\xi_e$ is an isomorphism follows from noting that $\xi_e$ can also be described as the composition of isomorphisms
$$I^{\otimes e}\otimes_A A/I\stackrel{1\otimes\phi^{-1}}{\longrightarrow} I^{\otimes e}\otimes_A (A/I)^{\otimes e}\longrightarrow (I\otimes_A A/I)^{\otimes e}\stackrel{\xi_1^{\otimes e}}{\longrightarrow} (I/I^2)^{\otimes e}$$
\end{proof}

\begin{prop}\label{prop_restriction} Let $\cX\subset\cAdm(G)$ be a connected component with universal family $\cC\stackrel{\pi}{\ra}\cE\rightarrow\cX$. Let $\sigma_\cO : \cX\rightarrow\cE$ denote the zero section. Suppose $\cC\rightarrow\cE$ has ramification index $e$ above $\sigma_\cO$. Suppose further that we have a section $\sigma : \cX\rightarrow\cC$ making the following diagram commute:
\[\begin{tikzcd}
\cC\ar[r,"\pi"] & \cE\\
 & \cX\ar[lu,"\sigma"]\ar[u,"\sigma_\cO"']
\end{tikzcd}\]
Then, there is a canonical isomorphism $\sigma_O^*\omega_{\cE/\cX}\cong(\sigma^*\omega_{\cC/\cX})^{\otimes e}$.
\end{prop}
\begin{proof} Let $U$ be a scheme and $U\rightarrow\cX$ a surjective \'{e}tale map corresponding to an admissible $G$-cover $C\stackrel{\pi}{\rightarrow} E\rightarrow U$. Let $\sigma,\sigma_O$ denote the sections of $C,E$ pulled back from $\sigma,\sigma_\cO$, with respective sheaves of ideals $\cJ\subset\cO_C,\cI\subset\cO_E$. Since $\pi\circ\sigma = \sigma_O$, $\pi$ induces a map $\pi^*\cI\rightarrow\cJ$. The \'{e}tale local picture of an admissible cover above a marking implies that this map factors through an isomorphism $\pi^*\cI\rightiso\cJ^e\subset\cJ$. On the other hand, since $\sigma,\sigma_O$ land in the smooth loci of $C$ and $E$, the conormal exact sequence induces isomorphisms
$$\cJ/\cJ^2\cong\sigma^*\Omega_{C/U}\qquad \cI/\cI^2\cong\sigma_O^*\Omega_{E/U}$$
Thus, using Lemma \ref{lemma_alg}, we obtain canonical isomorphisms
$$\sigma_O^*\Omega_{E/U}\cong \cI/\cI^2\cong\cI\otimes_{\cO_E}\cO_E/\cI\cong \sigma_O^*\cI\cong \sigma^*\pi^*\cI\cong\sigma^*\cJ^e\cong(\cJ/\cJ^2)^{\otimes e}\cong(\sigma^*\Omega_{C/U})^{\otimes e}$$
Since the canonical map $\varphi_U : \Omega_{C/U}\rightarrow\omega_{C/U}$ relative to the reduced ramification divisor $\cR_\pi\subset\cC$ is an isomorphism on the smooth locus (and similarly for $E/U$), these isomorphisms give an isomorphism $\sigma_O^*\omega_{E/U}\cong(\sigma^*\omega_{C/U})^{\otimes e}$. One checks that these isomorphisms are compatible with the comparison isomorphisms associated to the presentation of $\cX$ induced by $U$, and hence they descend to an isomorphism $\sigma_\cO^*\omega_{\cE/\cX}\cong(\sigma^*\omega_{\cC/\cX})^{\otimes e}$ as desired. 
\end{proof}

\subsection{Degree formalism for line bundles on 1-dimensional stacks}\label{ss_degree_formalism}

The main congruence described in the introduction (Theorem \ref{thm_vdovin_intro}) originates from a congruence on the degree of a certain line bundle on $\cAdm(G)$. Here we recall the notion of the degree of a line bundle on a proper 1-dimensional algebraic stack and prove some basic properties.


Given a finite flat morphism of algebraic stacks $f : \cX\rightarrow\cY$, given a map $y : \Spec k\rightarrow\cY$ with $k$ a field, its degree $\deg_y(f)$ at $y$ is the $k$-rank of the fiber $\cX\times_{\cY,y}\Spec k$. This integer is locally constant on $\cY$. If $\cY$ is connected we denote it simply by $\deg(f)\in\bZ$. For a line bundle $\cL$ on a proper scheme $X$ of pure dimension 1 over a field $k$, its degree is \cite[0AYR]{stacks}
\begin{equation}\label{eq_degree_ratios}
\deg(\cL) = \deg_k(\cL) := \chi_k(X,\cL) - \chi_k(X,\cO_X)
\end{equation}

\begin{lemma}[{\cite[Appendix B.2]{BDP17}}] Let $\cX$ be a connected proper algebraic stack of pure dimension 1 over a field $k$, and for $i = 1,2$, let $q_i : U_i\rightarrow\cX$ denote finite flat surjective morphisms with $U_i$ a connected scheme. Let $\cL$ be an invertible sheaf on $\cX$, then
$$\frac{\deg(q_1^*\cL)}{\deg q_1} = \frac{\deg(q_2^*\cL)}{\deg q_2}$$
\end{lemma}
\begin{proof} Since $q_i$ is finite flat, $U_1\times_\cX U_2$ is a scheme. Let $U\subset U_1\times_\cX U_2$ be a connected component. Then $U_1,U_2,U$ are all proper connected $k$-schemes of pure dimension 1. We have a diagram with all maps finite flat surjective
\[\begin{tikzcd}
U\ar[rd,"p"]\ar[r,"p_2"]\ar[d,"p_1"'] & U_2\ar[d,"q_2"] \\
U_1\ar[r,"q_1"'] & \cX
\end{tikzcd}\]
Noting that $\frac{\deg(p_i^*q_i^*\cL)}{\deg(q_i^*\cL)} = \deg p_i$ \cite[0AYW,0AYZ]{stacks}, we find that the ratios in \eqref{eq_degree_ratios} are both equal to $\frac{\deg(p^*\cL)}{\deg p}$.
\end{proof}

\begin{defn}[{\cite[Appendix B.2]{BDP17}}] Let $\cX$ be a connected proper algebraic stack of pure dimension 1 over a field $k$ which admits a finite flat (equivalently, finite locally-free \cite[02KB]{stacks}) surjective map $p : U\rightarrow\cX$ with $U$ a scheme. For a line bundle $\cL$ on $\cX$, define
$$\deg(\cL) := \frac{\deg(p^*\cL)}{\deg p}\in\bQ$$
\end{defn}
It follows from the lemma that $\deg(\cL)$ is independent of the choice of the finite flat scheme cover $p : U\rightarrow\cX$. Recall that a Deligne-Mumford stack is \emph{tame} if its automorphism groups at every geometric point $\Spec\Omega\rightarrow\cX$ has order coprime to the characteristic of $\Omega$. We say it is generically tame if it has an open dense substack which is tame. The following lemma implies that it makes sense to speak of degrees of line bundles on $\cAdm(G),\ol{\cM(G)}$.

\begin{prop}[{\cite{KV04, EHKV01}}]\label{prop_finite_flat_scheme_cover} We work over a field $k$.
\begin{itemize}
\item[(a)] Let $\cX$ be a smooth separated generically tame 1-dimensional Deligne-Mumford stack admitting a coarse scheme (i.e., whose coarse space given by Theorem \ref{thm_keel_mori} is a scheme). Then there exists a finite flat cover $U\rightarrow\cX$ with $U$ a smooth $k$-scheme.
\item[(b)] Let $G$ be a finite group. If $\ch(k)\nmid 2|G|$, then for $\cM = \cAdm(G)_k$ or $\cM = \ol{\cM(G)}_k$, $\cM$ admits a finite flat surjective map $U\rightarrow\cM$ with $U$ a smooth $k$-scheme.	
\end{itemize}
\end{prop}
\begin{proof} By Theorem \ref{thm_keel_mori}, the coarse scheme is proper over $k$. Since proper schemes of dimension 1 are projective, (a) is \cite[Theorems 1, 2, 3]{KV04}. By Proposition \ref{prop_compactification}(e), $\ol{\cM(G)}$ is smooth with smooth projective coarse scheme. Since $\cAdm(G)\rightarrow\ol{\cM(G)}$ is an \'{e}tale gerbe, the same is true of $\cAdm(G)$, so (b) follows from (a).
\end{proof}


Next we obtain a criterion to detect when $\deg(\cL)$ is an integer. Let $\cL$ be an invertible sheaf on a Deligne-Mumford stack $\cL$. For any geometric point $x : \Spec\Omega\rightarrow\cX$, $x^*\cL$ is a rank 1 representation of $\Aut_\cX(x)$, which we call the \emph{local character} of $\cL$ at $x$. We will need the following result.

\begin{prop}[{\cite[Proposition 6.1]{Ols12}}]\label{prop_trivial_local_characters} Let $\cX$ be a locally finitely presented tame separated Deligne-Mumford stack with a coarse scheme $c : \cX\rightarrow X$. Then pullback $c^*$ induces an isomorphism between the category of invertible sheaves on $X$ and the full subcategory of invertible sheaves on $\cX$ whose local characters at all geometric points are trivial. A quasi-inverse is given by $c_*$.
\end{prop}


\begin{defn} Let $\cX$ be an irreducible Deligne-Mumford stack admitting a coarse scheme $c : \cX\rightarrow X$. Its \emph{generic automorphism group} is the automorphism group of a geometric generic point $\Spec\Omega\rightarrow\cX$.
\end{defn}

\begin{prop}\label{prop_degree_integrality} Let $\cX$ be an connected tame smooth proper 1-dimensional Deligne-Mumford stack over a field $k$ admitting a coarse scheme. Suppose its generic automorphism group has order $n$. Then for any invertible sheaf $\cL$ on $\cX$ with trivial local characters, we have
$$\deg(\cL) = \frac{1}{n}\deg(c_*\cL)\in\frac{1}{n}\bZ$$
\end{prop}
\begin{proof} By Proposition \ref{prop_finite_flat_scheme_cover}, we may find a finite flat map $p : Y\rightarrow\cX$ with $Y$ a smooth proper curve, so we may speak of degrees of line bundles on $\cX$. By Lemma \ref{lemma_coarse_scheme_is_smooth}, $X$ is a smooth proper curve, so the map $c\circ p : Y\rightarrow\cX\rightarrow X$ is finite flat. By Proposition \ref{prop_trivial_local_characters}, $\cL \cong c^*c_*\cL$, so it would suffice to show that $\deg(c\circ p) = \frac{\deg p}{n}$. By the local structure of Deligne-Mumford stacks \cite[Theorem 11.3.1]{Ols16}, $\cX\rightarrow X$ is \'{e}tale-locally given by $[U/\Gamma]\rightarrow U/\Gamma$ for some finite group $\Gamma$ acting on a scheme $U$ \'{e}tale over $\cX$. Let $\Gamma_1\le\Gamma$ be the kernel of the $\Gamma$-action on $U$, then by shrinking $U$ we may assume that $U$ is irreducible and $\Gamma/\Gamma_1$ acts with trivial inertia on $U$. In this case it follows that $\Gamma_1$ isomorphic to the generic automorphism group of $\cX$, so $|\Gamma_1| = n$. We have a diagram

\[\begin{tikzcd}
Y\times_\cX U\ar[r,"\gamma"]\ar[d,"\alpha"] & Y\times_\cX[U/\Gamma]\ar[r]\ar[d,"\delta"] & Y\ar[d,"p"] \\
U\ar[r]\ar[rd,"\beta"'] & {[U/\Gamma]}\ar[r]\ar[d,"\epsilon"] & \cX\ar[d,"c"] \\
& U/\Gamma\ar[r] & X
\end{tikzcd}\]
where all squares are cartesian. Thus $\gamma$ is finite \'{e}tale of degree $|\Gamma|$ and $\deg\alpha = \deg p$. Since $\Gamma/\Gamma_1$ acts with trivial inertia on $U$, $\beta$ is finite \'{e}tale of degree $|\Gamma/\Gamma_1| = \frac{1}{n}|\Gamma|$. Since $\gamma\circ\delta\circ\epsilon = \beta\circ\alpha$ are finite flat, we have
$$\deg(c\circ p) = \deg(\epsilon\circ\delta) = \frac{\deg \alpha\cdot\deg\beta}{\deg\gamma} = \frac{\deg p\cdot\frac{1}{n}|\Gamma|}{|\Gamma|} = \deg p\cdot\frac{1}{n}$$
as desired.	
\end{proof}

\begin{prop}\label{prop_degree_of_pullback} We work over a field $k$.
\begin{itemize}
\item[(a)] Let $f : \cY\rightarrow\cX$ be a finite flat map of connected proper 1-dimensional algebraic stacks of degree $d$, and $\cL$ is a line bundle on $\cX$, then
$$\deg(f^*\cL) = d\cdot\deg(\cL)$$
\item[(b)] Let $\cX$ be a connected tame smooth proper 1-dimensional Deligne-Mumford stack with generic automorphism group of order $n$ and admitting a coarse scheme $c : \cX\rightarrow X$. Then for a line bundle $\cL$ on $X$,
$$\deg(c^*\cL) = \frac{1}{n}\deg(\cL)$$
\item[(c)] Let $\cY,\cX$ be connected tame smooth proper 1-dimensional Deligne-Mumford stacks with generic automorphism groups of order $n_\cY,n_\cX$ respectively and admitting coarse schemes $Y,X$ respectively. If $f : \cY\rightarrow\cX$ induces a finite flat map $\ol{f} : Y\rightarrow X$ on coarse schemes, then for any line bundle $\cL$ on $\cX$, we have
$$\deg(f^*\cL) = \frac{n_\cX}{n_\cY}\deg(\ol{f})\cdot\deg(\cL)$$
\end{itemize}
\end{prop}
\begin{proof} For (a), see \cite[0AYW,0AYZ,02RH]{stacks}. Part (b) follows from Propositions \ref{prop_trivial_local_characters} and \ref{prop_degree_integrality}. For part (c), it follows from the \'{e}tale local picture of Deligne-Mumford stacks \cite[Lemma 2.2.3]{AV02} that for some integer $m\ge 1$, $\cL^{\otimes m}$ has trivial local characters. Since $\deg(\cL^{\otimes m}) = m\cdot\deg(\cL)$, we are reduced to the case where $\cL$ has trivial local characters, but in this case the result follows from (b).
\end{proof}

\begin{prop}\label{prop_hodge_bundle} Here we work over $\Qbar$. Let $\cE(1)\rightarrow\ol{\cM(1)}$ be the universal family of elliptic curves, with zero section $\sigma_O$. Then the \emph{Hodge bundle} $\lambda := \sigma_O^*\omega_{\cE(1)/\ol{\cM(1)}}$ is an invertible sheaf on $\ol{\cM(1)}$ of degree $\frac{1}{24}$.
\end{prop}
\begin{proof} Recall that if $\cM$ is an algebraic stack and $\cF$ is an $\cO_\cM$-module, its global sections is the set $\Gamma(\cM,\cF) := \Hom_{\cO_\cM}(\cO_\cM,\cF)$. Let $c : \ol{\cM(1)}\rightarrow\ol{M(1)}$ be the coarse scheme. Thus, we have 
$$\Gamma(\ol{\cM(1)},\lambda^{\otimes 12}) = \Hom_{\cO_{\ol{\cM(1)}}}(c^*\cO_{\ol{M(1)}},\lambda^{\otimes 12}) = \Hom_{\cO_{\ol{\cM(1)}}}(\cO_{\ol{M(1)}},c_*\lambda^{\otimes 12}) = \Gamma(\ol{M(1)},c_*\lambda^{\otimes 12})$$
Base changing to $\bC$, by standard GAGA arguments $\Gamma(\ol{\cM(1)},\lambda^{\otimes 12}) = \Gamma(\ol{M(1)},c_*\lambda^{\otimes 12})$ is isomorphic to the $\bC$-vector space of modular forms of level 1 and weight 12, which has dimension 2, generated by the Eisenstein series $E_{12}$ and the discriminant $\Delta$ \cite[Theorem 3.5.2]{DS06}. Since $\ol{M(1)}\cong\bP^1$, it follows that $c_*\lambda^{\otimes 12}$ has degree 1, so by Proposition \ref{prop_degree_of_pullback}(b), $\deg(\lambda^{\otimes 12}) = \frac{1}{2}$, so $\deg(\lambda) = \frac{1}{24}$.
\end{proof}

\subsection{The main congruence}\label{ss_congruence}
In this section we work universally over $\bS = \Spec k$ where $k$ is an algebraically closed field of characteristic 0. Note that by Corollary \ref{cor_empty_if_not_2_gen}, $\cAdm(G)$ is empty if $G$ is not 2-generated, so our results are only nontrivial for finite 2-generated groups $G$. It follows from Remark \ref{remark_topological_higman} that our results are also trivial if $G$ is abelian, so here one should think of $G$ as a finite nonabelian 2-generated group.

\begin{thm}\label{thm_congruence} Let $\cX\subset\cAdm(G)$ be a connected component, with universal family $\cC\stackrel{\pi}{\rightarrow}\cE\rightarrow\cX$ and reduced ramification divisor $\cR_\pi$. Let $\cX\rightarrow X$ be the coarse scheme of $\cX$, and let $\ol{\cM(1)}\rightarrow \ol{M(1)}\cong\bP^1_j$ be the coarse moduli scheme. Let $\cE(1)\rightarrow\ol{\cM(1)}$ be the universal family. Let $\cR\subset\cR_\pi$ be a connected component with coarse scheme $R$, and let $\epsilon : \cR\rightarrow\cX$ be the induced map. By definition, $\cC' := \cC\times_\cC\cR\rightarrow\cR$ admits a section $\sigma$ lying over the zero section of $\cE' := \cE\times_\cX\cR$. To $\cX$ we associate the three integers
\begin{itemize}
\item Let $e = e_\cX$ be the ramification index of any point of $C$ above the zero section of $\cE$.
\item Let $d = d_\cX$ be the degree of the induced map on coarse schemes $\ol{\epsilon} : R\rightarrow X$.
\item Let $m = m_\cX$ be the minimum positive integer such that $(\sigma^*\Omega_{\cC'/\cR})^{\otimes m}$ has trivial local characters.
\end{itemize}
Let $\ol{f} : X\rightarrow\ol{M(1)}$ be the map on coarse schemes induced by $f : \cX\rightarrow\ol{\cM(1)}$. Then we have
$$\deg(\ol{f}) \equiv 0\mod \frac{12e}{\gcd(12e,md)}$$
\end{thm}

\begin{proof} We have a commutative diagram with all squares cartesian (if you ignore the sections):
\[\begin{tikzcd}
\cC'\ar[d,"\pi'"]\ar[r] & \cC\ar[d,"\pi"] \\
\cE'\ar[r]\ar[d] & \cE\ar[d]\ar[r] & \cE(1)\ar[d] \\
\cR\ar[d]\ar[r,"\epsilon"]\ar[uu,bend left = 30, "\sigma"]\ar[u,bend right = 30,"\sigma_O'"'] & \cX\ar[r,"f"]\ar[d] & \ol{\cM(1)}\ar[d]\ar[u,bend right = 30,"\sigma_O"'] \\
R\ar[r,"\ol{\epsilon}"] & X\ar[r,"\ol{f}"] & \ol{M(1)}
\end{tikzcd}\]
where $\sigma_O,\sigma_O'$ denotes the zero sections. The sheaf $\lambda := \sigma_O^*\omega_{\cE(1)/\ol{\cM(1)}}$ is the Hodge bundle, which has degree $\frac{1}{24}$ by Proposition \ref{prop_hodge_bundle}. Since dualizing sheaves commute with arbitrary base change, by Proposition \ref{prop_restriction} we have
$$\epsilon^*f^*\lambda\cong\sigma_O'^*\omega_{\cE'/\cR}\cong(\sigma^*\omega_{\cC'/\cR})^{\otimes e}\cong(\sigma^*\Omega_{\cC'/\cR})^{\otimes e}$$
where the final isomorphism follows from the fact that $\sigma$ lies in the smooth locus by definition of admissible covers. Suppose $\cR$ has a generic automorphism group of order $n$. By Propositions \ref{prop_degree_integrality} and \ref{prop_degree_of_pullback}(c), we find that $\deg(\epsilon^*f^*\lambda)\in \frac{e}{mn}\bZ$. Since $\deg(\lambda) = \frac{1}{24}$, by (b) we have
$$\deg(\epsilon^*f^*\lambda) = \frac{2}{n}\deg(\ol{\epsilon})\cdot\deg(\ol{f})\cdot\frac{1}{24} = \frac{d}{12n}\cdot\deg(\ol{f})\in\frac{e}{mn}\bZ$$
and hence
$$\deg(\ol{f}) \in \frac{12e}{dm}\bZ\qquad\text{or equivalently (since $\deg(\ol{f})\in\bZ$)}\qquad\deg(\ol{f})\equiv 0\mod\frac{12e}{\gcd(12e,md)}.$$
\end{proof}

Theorem \ref{thm_congruence} takes as input a group $G$ and a component $\cX\subset\cAdm(G)$, and outputs a congruence which at best gives $\equiv 0\mod 12e$, but is possibly watered down by the integers $d_\cX,m_\cX$ associated to $\cX$. To obtain a good congruence, one wishes to show that $e_\cX$ does not share large divisors with $d_\cX$ and $m_\cX$. Suppose $\cX$ classifies covers with Higman invariant equal to the conjugacy class of $c\in G$. As we saw in \S\ref{ss_higman_invariant}, $e_\cX$ is just the order of $c$. Here we will describe some ways to control $d_\cX$ and $m_\cX$.



By Proposition \ref{prop_stacky_RRD}, $d_\cX$ must divide the order of $C_G(\langle c\rangle)/\langle c\rangle$. A trivial consequence of this is that if $\ell^r$ is a prime power dividing $|c|$ with $\ell^{r+1}\nmid |G|$, then $\ell\nmid d_\cX$. The integer $ m_\cX$ is more difficult to control. If $m_\cX'$ denotes the minimum positive integer required to kill all the vertical automorphism groups of geometric points of $\cR$ (Definition \ref{def_vertical_automorphisms}), then $12m_\cX$ kills all the automorphism groups of geometric points, so we must have $m_\cX\mid 12m_\cX'$. In fact, it follows from the local structure of admissible covers that the local characters of $\sigma^*\Omega_{\cC'/\cR}$ restrict to \emph{faithful representations} of the vertical automorphism groups, so up to a factor of 12, $m_\cX$ is equal to $m_\cX'$.


Let $\ol{r} : \Spec\Omega\rightarrow\cR$ be a geometric point with image $\ol{x}\in\cX$. Then $\ol{x}$ is given by a 1-generalized elliptic curve $E$ over $\Omega$ and an admissible $G$-cover $\pi : C\rightarrow E$, and $\ol{r}$ is given by $\pi$ together with a point $P\in C(\Omega)$ lying over $O\in E$. The vertical automorphism group of $\ol{x}$ is the group of $G$-equivariant automorphisms $\sigma\in\Aut(C)$ such that $\pi\circ\sigma = \pi$, and the vertical automorphism group of $\ol{r}$ is the subgroup consisting of vertical automorphisms of $\ol{x}$ satisfying $\sigma(P) = P$.


When $E$ is smooth, there is a simple description of the vertical automorphism group of $\ol{r}$:

\begin{prop}\label{prop_smooth_pointed_vertical_automorphisms} Let $\cY\subset\cAdm^0(G)$ be a connected component classifying covers with Higman invariant $\fc$. Let $c\in\fc$ be a representative. Let $\cR$ be a component of the ramification divisor of the universal family over $\cY$. Let $\ol{r}$ be a geometric point of $\cR$ with image $\ol{y}$ in $\cY$. Suppose $\ol{r}$ corresponds to the admissible cover $\pi : C\rightarrow E$ together with the point $P\in\pi^{-1}(O)$. Then the vertical automorphism groups of $\ol{y},\ol{r}$ are
$$\Aut^v(\ol{y}) = Z(G)\qquad\text{and}\qquad\Aut^v(\ol{r}) = Z(G)\cap \langle c\rangle$$
\end{prop}
\begin{proof} The fact that $\Aut^v(\ol{y}) = Z(G)$ is Proposition \ref{prop_smooth_vertical_automorphisms}. The subgroup of $G$-equivariant automorphisms which fix $P$ is then $Z(G)\cap G_P$, where $G_P := \Stab_G(P)$ is the stabilizer. By the definition of the Higman invariant, $G_P$ is conjugate to $\langle c\rangle$, so we have $\Aut^v(\ol{r}) = Z(G)\cap G_P = Z(G)\cap\langle c\rangle$.
\end{proof}

If $E$ is not smooth, then $C$ can fail to be irreducible, and the situation can be potentially be bad enough to make the congruence trivial (see \S\ref{ss_further_directions}). The main purpose of the next section is to give a precise group-theoretic characterization of the vertical automorphism groups of cuspidal objects of $\cAdm(G)$. This will allow us to control $m_\cX$ at least when $G = \SL_2(\bF_q)$ or a nonabelian finite simple group. In particular we will show that in these cases, one can often achieve nontrivial congruences (see \S\ref{ss_cuspidal_automorphisms}).

\section{Galois correspondence for cuspidal objects of $\cAdm(G)$}\label{section_cusps}

In this section we give a combinatorial characterization of the cuspidal objects of $\cAdm(G)$. The main purpose of this section is to characterize cuspidal admissible $G$-covers combinatorially in terms of group-theoretic information in $G$ and to recognize their automorphism groups. This is done in \S\ref{ss_cuspidal_automorphisms}, though it will need terminology from the preceding subsections. We will also formulate a combinatorial version of Theorem \ref{thm_congruence} (Theorem \ref{thm_combinatorial_congruence}), using which we will show that we can often obtain nontrivial congruences when $G$ is a nonabelian group (Corollary \ref{cor_vdovin_2}). The statements of these two results can be understood without consulting the previous subsections. 


We give an overview of our approach.  We begin in \S\ref{ss_precuspidal_preliminaries} by defining the notion of a ``precuspidal $G$-cover'' of a non-smooth 1-generalized elliptic curve $E$. The category of such objects is denoted $\cC_E^{pc}$. Given a precuspidal $G$-cover $\pi : C\rightarrow E$, taking normalizations we obtain a $G$-cover $\pi' : C'\rightarrow\bP^1$, ramified only above three points. We may assume that $0,\infty\in\bP^1$ are the preimages of the node in $E$. The normalization map also provides the data of a $G$-equivariant bijection $\alpha = \alpha_\pi : \pi'^{-1}(0)\rightiso \pi'^{-1}(\infty)$. We will let $\cC_{\bP^1}^\succ$ denote the category of such pairs $(\pi',\alpha)$. The usual Galois correspondence identifies $G$-covers of $\bP^1$ only branched over $\{0,1,\infty\}$ with finite sets equipped with commuting actions of $\Pi := \pi_1(\bP^1 - \{0,1,\infty\})$ and $G$. Accordingly, objects of $\cC_{\bP^1}^\succ$ can be identified with finite sets equipped with commuting actions of $\Pi$ and $G$ as well as a ``combinatorial $G$-equivariant bijection''. The category of such objects is denoted $\Sets^{(\Pi,G)_\delta,\succ}$, where $\delta$ denotes a path from a tangential base point at $0\in\bP^1$ to a tangential base point at $\infty\in\bP^1$. We will describe equivalences of categories:
$$\cC_E^{pc}\stackrel{\text{``}\Xi\text{''}}{\lra}\cC_{\bP^1}^\succ\stackrel{\text{``}F_{\delta}^\succ\text{''}}{\lra}\Sets^{(\Pi,G)_\delta,\succ}$$
Roughly speaking, the first is given by taking normalizations, and the second is given by taking geometric fibers. We say that an object of $\cC_E^{pc}$ is \emph{cuspidal} if it corresponds to a cuspidal object of $\cAdm(G)$. This amounts to the two additional conditions that the cover is \emph{connected} and (the $G$-action is) \emph{balanced}. The full subcategory of cuspidal objects of $\cC_E^{pc}$ is denoted $\cC_E^c$. To obtain a combinatorial characterization of $\cC_E^c$, we must describe what it means for an object of $\Sets^{(\Pi,G)_\delta,\succ}$ to come from a cuspidal (i.e., connected and balanced) object of $\cC_E^{pc}$. This is done in \S\ref{ss_combinatorial_balance_connectedness}. Using this, in \S\ref{ss_the_map_Inv} we give a combinatorial parametrization of all cuspidal objects of $\cAdm(G)$ in terms their ``$\delta$-invariant'' (where $\delta$ is the path mentioned above), which is an equivalence class of a generating pair of $G$. In \S\ref{ss_cuspidal_automorphisms}, we calculate the automorphism groups of cuspidal admissible $G$-covers in terms of their $\delta$-invariants, and in \S\ref{ss_congruences_for_NAFSG} we describe some applications to the cardinalities of Nielsen equivalence classes of generating pairs of finite groups.

\subsection{(Pre)cuspidal $G$-curves, (pre)cuspidal $G$-covers}\label{ss_precuspidal_preliminaries}

Throughout \S\ref{section_cusps}, we will work universally over $\bS = \Spec k$, where $k$ denotes an algebraically closed field of characteristic 0. Moreover, we will fix a compatible system of primitive $n$th roots of unity $\{\zeta_n\}_{n\ge 1}\subset k$, compatible in the sense that for all $d\mid n$, $\zeta_n^d = \zeta_{n/d}$. Essentially all of our methods are algebraic, so the same development should also make sense in all tame characteristics.


\begin{defn} A \emph{precuspidal $G$-curve} is a non-smooth prestable curve $C$ equipped with a faithful right action of $G$ and a $G$-invariant divisor $R\subset C$ finite \'{e}tale over $k$, such that the quotient map $C\rightarrow C/G$ sends nodes to nodes and is \'{e}tale on $C_\sm - R$, and $(C/G,R/G)$ is a nodal elliptic curve (i.e., a non-smooth 1-generalized elliptic curve). A \emph{cuspidal $G$-curve} is a precuspidal $G$-curve $C$ such that
\begin{itemize}
\item $C$ is \emph{connected}, and
\item the $G$-action is \emph{balanced} at the nodes in the sense of Remark \ref{remark_covers}\ref{part_balanced2}. In this case we say $C$ is balanced.
\end{itemize}
A morphism of (pre)cuspidal $G$-curves is a $G$-equivariant map preserving divisors.
\end{defn}

Note that up to isomorphism, there is only one nodal elliptic curve.

\begin{defn} A $G$-cover of a finite type equidimension 1 scheme $Y$ is a finite flat map $p : X\rightarrow Y$ equipped with a faithful right action of $G$ on $X$ such that $p$ induces an isomorphism $X/G\rightiso Y$.	
\end{defn}

\begin{defn} Let $E = (E,O)$ be nodal elliptic curve. A \emph{precuspidal $G$-cover} of $E$ is a $G$-cover $\pi : C\rightarrow E$ such that $(C,(C\times_E O)_\red)$ is a precuspidal $G$-curve. Equivalently, it is a $G$-cover $\pi : C\rightarrow E$ satisfying
\begin{itemize}
\item $\pi$ sends nodes to nodes,
\item $\pi$ is \'{e}tale on $C_\sm - \pi^{-1}(O)$
\end{itemize}
A precuspidal $G$-cover is \emph{connected} (resp. \emph{balanced}) if $C$ is connected (resp. the $G$-action is balanced at the nodes). A \emph{cuspidal $G$-cover} is a balanced connected precuspidal $G$-cover. A morphism of (pre)cuspidal $G$-covers is a $G$-equivariant map over $E$. Let $\cC_E^{pc}$ (resp. $\cC_E^c$) denote the category of precuspidal (resp. cuspidal) $G$-covers of $E$.
\end{defn}

In particular, a precuspidal $G$-cover is an admissible $G$-cover if and only if it is cuspidal, or equivalently, connected and balanced, or equivalently, is a cuspidal object of $\cAdm(G)$. Here is a precise statement.

\begin{prop}\label{prop_cuspidal_equivalence} Let $\cAdm^c(G) := \cAdm(G) - \cAdm^0(G)$ be the closed substack consisting of cuspidal (i.e. non-smooth) objects. Associating a cuspidal $G$-curve $(C,R)$ to the cuspidal $G$-cover $C\rightarrow C/G$ of the nodal elliptic curve $(C/G,R/G)$ gives an equivalence of categories
$$\{\text{cuspidal $G$-curves}\}	 \rightiso \cAdm^c(G)(k).$$
Let $E$ be a nodal elliptic curve. Let $\cAdm(G)_E$ denote the fiber category of $\cAdm(G)$ over $E\in\ol{\cM(1)}$. Thus, the objects of $\cAdm(G)_E$ are admissible $G$-covers of $E$, and morphisms are morphisms of admissible $G$-covers which induce the identity on $E$. Then we have an equality of categories
$$\cC_E^c = \cAdm(G)_E.$$
In particular, the automorphism groups of objects of $\cC_E^c$ (resp. $\cAdm(G)_E$) are the vertical automorphism groups of the corresponding objects of $\cAdm(G)$.
\end{prop}
\begin{proof} The first equivalence follows from Theorem \ref{thm_aGc_equals_smGc}, which in turn implies $\cC_E^c = \cAdm(G)_E$.
\end{proof}

\subsection{Tangential base points}\label{ss_tbp}

Here we follow Deligne \cite[\S15]{Del89}. Let $X$ be a smooth curve (a smooth finite type 1-dimensional scheme over $k = \ol{k}$), $x\in X$ a closed point, and $X^\circ := X - \{x\}$. Let $K_x := \Frac\cO_{X,x}$, and $v_x : K_x\rightarrow\bZ$ the discrete valuation. Then $K_x$ is filtered by $v_x$:
$$F^iK_x := \{f\in K_x \;|\; v_x(f)\ge i\}.$$
As a scheme, define the tangent space at $x$ by $\bT_x := \Spec\gr(\cO_{X,x})$, where $\gr(\cO_{X,x})$ is the associated graded ring with respect to the filtration $F^i$ defined above\footnote{Strictly speaking, this is really the tangent \emph{cone}. In particular, we are viewing $\bT_x$ as a scheme instead of as a vector space.}. Similarly the punctured tangent space $\bT_x^\circ := \bT_x - \{0\}$ is $\Spec\gr(K_x)$. These constructions are functorial in $(X,x)$. If $\pi^\circ : Y^\circ\rightarrow X^\circ$ is a finite \'{e}tale map which extends to a finite map $\pi : Y\rightarrow X$ of smooth curves, then we may associate to $\pi^\circ$ the maps
\begin{eqnarray*}
\pi_{(x)} : Y_{(x)} := \bigsqcup_{y\in \pi^{-1}(x)}\Spec\gr(\cO_{Y,y}) & \longrightarrow & \Spec\gr(\cO_{X,x}) =: \bT_x	\\
\pi_{(x)}^\circ : Y_{(x)}^\circ := \bigsqcup_{y\in \pi^{-1}(x)}\Spec\gr(K_y) & \longrightarrow & \Spec\gr(K_x) =: \bT_x^\circ
\end{eqnarray*}
defined as follows. The rings $\cO_{Y,y}$ and $K_y := \Frac\cO_{Y,y}$ are filtered according to the valuation $v_y : K_y\rightarrow\frac{1}{e}\bZ$ extending $v_x$, where $e$ is the ramification index of $\pi$ at $y$. Specifically, for $i\in\frac{1}{e}\bZ$,
$$F^iK_y = \{f\in K_y\;|\; v_y(f)\ge i\}.$$
If we also view $K_x$ as a $\frac{1}{e}\bZ$-filtered ring with $F^iK_x := F^iK_y\cap K_x$ for $i\in\frac{1}{e}\bZ$, then $\cO_{X,x}\rightarrow\cO_{Y,y}$ and $K_x\rightarrow K_y$ are filtered ring maps, and we define the maps $\pi_{(x)},\pi_{(x)}^\circ$ to be the maps which are induced by these filtered ring maps. Let $\varpi_x\in\cO_{X,x}$ be a uniformizer, and let $\varpi_y\in\cO_{Y,y}$ satisfy $\varpi_y^e \equiv \varpi_x\mod\fm_y^{e+1}$, so that $\varpi_y$ is a uniformizer of $\cO_{Y,y}$. Then the maps
$$\begin{array}{rclcrcl}
k[t] & \longrightarrow & \gr(\cO_{X,x})	& \qquad & k[s] & \rightiso & \gr(\cO_{Y,y}) \\
t & \mapsto & \varpi_x\mod \fm_x = F^1\cO_{X,x} & \qquad & s & \mapsto & \varpi_y\mod\fm_y = F^{1/e}\cO_{Y,y}
\end{array}$$
are isomorphisms \cite[00NO]{stacks} which moreover induce isomorphisms $k[t,t^{-1}]\cong\gr(K_x)$ and $k[s,s^{-1}]\cong\gr(K_y)$. Since $\varpi_y^e\equiv \varpi_x\mod\fm_y^{e+1}$, the restriction of $\pi_{(x)}^\circ$ to $\Spec\gr(K_y)$ fits into a commutative diagram
\[\begin{tikzcd}
	k[s,s^{-1}]\ar[r,"\cong"] & \gr(K_y) \\
	k[t,t^{-1}]\ar[u,"t\mapsto s^e"]\ar[r,"\cong"] & \gr(K_x)\ar[u,"\pi_{(x)}^\circ"']
\end{tikzcd}\]
In particular, we find that $\pi_{(x)}^\circ$ is finite \'{e}tale. For a scheme $S$, let $\FEt_{S}$ denote the category of finite \'{e}tale maps to $S$. Then the association $\pi^\circ\mapsto \pi^\circ_{(x)}$ defines an exact functor\footnote{Ie, it preserves finite limits and finite colimits. This can be checked using \cite[Expos\'{e} V, Proposition 6.1]{SGA1}.} $R_x : \FEt_{X^\circ}\rightarrow\FEt_{\bT^\circ_x}$, and hence if $t\in\bT^\circ_x$ is a closed point with associated fiber functor $F_{\bT^\circ_x,t}$, then the composition
$$F_{X^\circ,t} : \FEt_{X^\circ}\stackrel{R_x}{\longrightarrow}\FEt_{\bT^\circ_x}\stackrel{F_{\bT^\circ_x,t}}{\longrightarrow}\Sets$$
is also a fundamental functor for $\FEt_{X^\circ}$ \cite[Expos\'{e} V, Proposition 6.1]{SGA1}. That is to say, if $\Pi := \Aut(F_{X^\circ,t})$, then $F_{X^\circ,t}$ defines an equivalence of categories between $\FEt_{X^\circ}$ and the category of finite sets with $\Pi$-action.

\begin{defn} Let $X$ be a smooth curve, $x\in X$ a closed point, and $X^\circ := X - \{x\}$. A \emph{tangential base point} of $X$ at $x$ is a closed point $t\in\bT^\circ_x$. Given a finite \'{e}tale cover $\pi : Y^\circ\rightarrow X^\circ$, we will abuse notation and write:
$$Y^\circ_t = \pi^{-1}(t) := F_{X^\circ,t}(\pi)$$
and call $F_{X^\circ,t}$ the \emph{fiber of $\pi : Y^\circ\rightarrow X^\circ$ above $t$}. Accordingly we will write $\pi_1(X^\circ,t) := \Aut(F_{X^\circ,t})$.
\end{defn}

The group $\pi_1(\bT^\circ_x,t) := \Aut(F_{\bT^\circ_x,t})\cong\Zhat$ acts on $F_{X^\circ,t}$ in the obvious way, and this defines a canonical morphism of fundamental groups
$$\pi_1(\bT^\circ_x,t)\rightarrow \pi_1(X^\circ,t)$$
which is in fact injective \cite[V, Proposition 6.8]{SGA1}. There is a unique isomorphism $\bT_x\rightiso\bA^1$ sending $0\mapsto 0$ and $t\mapsto 1$ which induces an isomorphism $\bT_x^\circ\rightiso\bG_m := \bA^1 - \{1\}$ (which we will use to identify the two schemes). The Galois theory of $\bG_m$ is well-understood: for every integer $n\ge 1$, there is a unique connected degree $n$ finite \'{e}tale cover of $\bG_m$ given by $\pi_n : \bG_m\rightarrow\bG_m$ sending $z\mapsto z^n$. Let $g_{t,n}\in\Gal(\pi_n)$ be given by $z\mapsto \zeta_n z$.

\begin{defn}\label{def_canonical_generator_of_inertia} Let $X$ be a smooth curve, $x\in X$ a closed point, $X^\circ := X - \{x\}$, and $t\in\bT^\circ_x$ a tangential base point. The \emph{canonical generator of inertia} is the element $\gamma_t\in\pi_1(\bT^\circ_x,t)$ which acts on $F_{\bT^\circ_x,t}(\pi_n)$ via $g_{t,n} : z\mapsto \zeta_n z$ for every $n$. We will often view $\gamma_t$ as an element of $\pi_1(X^\circ,t)$ via the canonical map $\pi_1(\bT^\circ_x,t)\hookrightarrow\pi_1(X^\circ,t)$. Note that this definition depends on our choice of compatible system $\{\zeta_n\}_{n\ge 1}$.
\end{defn}

\begin{prop}\label{prop_tbp} Let $X$ be a smooth curve, $x\in X$ a closed point, and $X^\circ := X - \{x\}$.
\begin{itemize}
\item[(a)] Let $T_x^*X$ be the Zariski cotangent space at $x$ (a vector space), and let $T_0^*\bT_x$ be the Zariski cotangent space at $0\in\bT_x$. There is a canonical isomorphism
$$T_x^*X\rightiso T_0^*\bT_x$$
which is functorial in $(X,x)$. If $t$ is a tangential base point at $x$, then the canonical generator of inertia $\gamma_t$ acts on $T^*_xX$ by multiplication by $\zeta_n$.

\item[(b)] Let $\pi : Y\rightarrow X$ be a finite map of smooth curves, \'{e}tale over $X^\circ$. Every connected component of the scheme $Y_{(x)}$ has a unique point lying over $0\in\bT_x$. The fiber $\pi_{(x)}^{-1}(0)$ is canonically in bijection with $Y_x := \pi^{-1}(x)$. The map $Y_t := \pi_{(x)}^{-1}(t)\rightarrow \pi_{(x)}^{-1}(0)$ sending $y\in Y_t$ to the unique point of $Y_{(x)}$ lying over $0$ induces a canonical bijection (the specialization to a ramified fiber)
$$\langle\gamma_t\rangle\bs Y_t\rightiso \pi_{(x)}^{-1}(0)\rightiso Y_x$$
which is functorial in $\pi$, in the sense that it defines an isomorphism of functors $\FEt_{X^\circ}\rightarrow\Sets$. Given $y\in Y_t$, we will often let $[x]$ denote its image in $\langle \gamma_t\rangle\bs Y_t$ or its image in $Y_x$.
\end{itemize}
Moreover, the isomorphism in (a) is functorial in $(X,x)$, and the isomorphism in (b) \end{prop}
\begin{proof} For (a), to be precise let $\fm_0\subset\Spec\gr(\cO_{X,x})$ be the maximal ideal at $0$. Then for $r\ge 0$, $\fm_0^r = \oplus_{n\ge r}\gr_n(\cO_{X,x})$, so the Zariski cotangent space $\fm_0/\fm_0^2$ is canonically isomorphic to $\gr_1(\cO_{X,x}) = \fm_x/\fm_x^2 = T_x^*X$, and this isomorphism is clearly functorial in the pair $(X,x)$. It follows from the definition of the canonical generator of inertia that $\gamma_t$ acts by multiplication by $\zeta_n$ on tangent spaces, and hence it also acts by multiplication by $\zeta_n$ on cotangent spaces.


For (b), it suffices to observe that from the discussion above, the scheme $Y_{(x)} := \sqcup_{y\in \pi^{-1}(x)}\Spec\gr(\cO_{Y,y})$ is a disjoint union of affine lines indexed by $y\in\pi^{-1}(x)$, each mapping to $\bT_x$ (also an affine line) by the map $z\mapsto z^{e_y}$, where $e_y$ is the ramification index of $\pi$ at $y$.
\end{proof}

\begin{remark}\label{remark_analytic_theory} If $k = \bC$, there is a complementary analytic theory \cite[\S15.3-15.12]{Del89}, where the functor $R_x : \FEt_{X^\circ}\rightarrow \FEt_{\bT_x^\circ}$ is given by pulling back covers along the germ of a local isomorphism between a punctured neighborhood of $0\in\bT_x^\circ$ and a punctured neighborhood of $x\in X$. As an informal picture, one should imagine gluing a copy of $\bC^\times$ to $X^\circ$ along a (infinitesimally) small punctured disk at $0$ and $x$. In this picture a tangential base point is simply a point of $\bC^\times$. The resulting space is homeomorphic to $X^\circ$, and hence has the same fundamental group. Moreover, in this analytic theory, there is a good notion of ``path'' in the topological sense given by continuous maps from $[0,1]$, and so one can speak of true paths between points (normal, or tangential).


For general algebraically closed $k$ of characteristic 0, one can often deduce results from the analytic theory over $\bC$ as follows. First note that depending on the cardinality of $k$, either $k$ embeds in $\bC$ or it contains $\bC$. On the other hand, if $L/K$ is an extension of algebraically closed fields and $X$ is a $K$-scheme, then the map $X_L\rightarrow X$ induces an isomorphism on fundamental groups for any choice of base points \cite[Expos\'{e} XIII, Proposition 4.6]{SGA1}, so it induces an equivalence $\FEt_{X}\cong\FEt_{X_L}$. Finally, classical GAGA results \cite[Expos\'{e} XII]{SGA1}, provide an equivalence between the category of finite \'{e}tale covers of finite type $\bC$-schemes and finite \'{e}tale covers of their analytifications.
\end{remark}

\begin{remark} There is a simpler alternative version of tangential base points given as follows. Let $X$ be a smooth curve over $k$, $x\in X$ a closed point. Let $\varpi\in\cO_{X,x}$ be a uniformizer, then the map $\cO_{X,x}\rightarrow k\ps{t}$ sending $\varpi\mapsto t$ defines a map
$$t_\varpi : \Spec k\ls{z^{1/\infty}}\stackrel{i_z}{\longrightarrow}\Spec k\ls{z}\rightarrow\Spec k\ps{z}\rightiso\Spec\what{\cO_{X,x}}\longrightarrow X$$
such that $t_\varpi$ becomes a geometric point of $X^\circ$. This gives a true geometric point of $X^\circ$, and shares many of the same features as the tangential base points given above (canonical generators of inertia, specialization map to a ramified fiber). However in this definition, it is more difficult to describe the corresponding analytic theory over $\bC$, which we will need in two places.	
\end{remark}

\subsection{The fundamental group of $\bP^1 - \{0,1,\infty\}$}\label{ss_pi_1}
Here we set up some of the notation and convention that we will maintain for the rest of \S\ref{section_cusps}.


Let $\bP^* := \bP^1 - \{0,1,\infty\}$. Recall that the fundamental groupoid of $\bP^*$ is the category $\Pi_{\bP^*}$ whose objects are fundamental functors $\FEt_{\bP^*}\rightarrow\Sets$ (see \cite[\S V, Definition 5.1]{SGA1}) and whose morphisms are isomorphisms of functors. If $x\in\bP^*$ is a geometric point or tangential base point, let $F_x$ denote the corresponding fundamental functor, which we also call the \emph{fiber functor at $x$}. The fundamental group of $\bP^*$ with base point $x$ is the group $\pi_1(\bP^*,x) := \Aut(F_x)$. If $x,y$ are geometric points or tangential base points, then a path $\delta : x\leadsto y$ is an isomorphism of fiber functors $\delta : F_x\rightiso F_y$. If $\delta : x\leadsto y$ and $\delta' : y\leadsto z$ are paths, then we will write the composition $\delta'\circ\delta$ as $\delta'\delta$ or $\delta'\cdot\delta : x\leadsto z$. If $\gamma\in\pi_1(\bP^*,y) := \Aut(F_y)$ and $\delta : x\leadsto y$ is a path, then we write ``conjugation'' as:
\begin{equation}\label{eq_gamma_delta}
\gamma^\delta := \delta^{-1}\gamma\delta \in\pi_1(\bP^*,x) := \Aut(F_x)	
\end{equation}


An automorphism $f\in\Aut(\bP^*)$ determines by pullback an automorphism $f^* : \FEt_{\bP^*}\rightiso\FEt_{\bP^*}$ which induces an automorphism of the fundamental groupoid $f_* : \Pi_{\bP^*}\rightiso\Pi_{\bP^*}$ described by $f_*F_x := F_x\circ f^*$. There is a canonical isomorphism $F_x\circ f^*\cong F_{f(x)}$, using which we obtain the usual induced map $f_* : \pi_1(\bP^*,x)\rightarrow\pi_1(\bP^*,f(x))$.


Let $\iota : z\mapsto \frac{1}{z}$ denote the unique automorphism of $\bP^1$ fixing 1 and switching $0,\infty$. In what follows, we will fix a choice of tangential base point $t_0$ at $0\in\bP^1$, and define $t_\infty$ as the tangential base point at $\infty$ given by $\iota(t_0)$. If $\gamma_0,\gamma_\infty$ denote the canonical generators of inertia at $t_0,t_\infty$, then $\iota$ defines a homomorphism $\iota_* : \pi_1(\bP^*,t_0)\rightarrow\pi_1(\bP^*,t_\infty)$ sending $\gamma_0\mapsto \gamma_\infty$. If $\delta : t_0\leadsto t_\infty$ is a path, then $\iota_*\delta$ is a path $t_\infty\leadsto t_0$.

\begin{defn}\label{def_good_path} Having fixed a tangential base point $t_0$ at $0$, let $t_\infty := \iota(t_0)$. A path $\delta : t_0\leadsto t_\infty$ is \emph{good} if 
\begin{itemize}
	\item $\gamma_0$ and $\gamma_\infty^\delta := \delta^{-1}\gamma_\infty\delta$ topologically generate $\pi_1(\bP^*,t_0)$, and
	\item for some path $\epsilon : t_0\leadsto t_1$, we have $\gamma_\infty^\delta\gamma_1^\epsilon\gamma_0 = 1$.
\end{itemize}
The path $\delta$ is \emph{symmetric} if
$$\iota_*\delta = \gamma_0^r\cdot\delta^{-1}\cdot\gamma_\infty^s\qquad\text{for some $r,s\in\bZ$}$$
\end{defn}


It follows from the analytic theory (c.f. Remark \ref{remark_analytic_theory}) that symmetric good paths exist. Indeed, we can let $\delta$ be the path $t_0\leadsto t_\infty$ in $\bP^*$ which, for some small $\epsilon > 0$, traces the interval $(0,1-\epsilon)$, makes a small counterclockwise turn around $1\in\bP^1$, and continues along the interval $(1+\epsilon,\infty)$.

\subsection{The category $\cC_{\bP^1}^\succ$ and the normalization map $\Xi : \cC_E^{pc}\longrightarrow\cC_{\bP^1}^\succ$}\label{ss_normalization_map}
As in \S\ref{ss_pi_1}, in what follows let $\bP^* := \bP^1 - \{0,1,\infty\}$. We fix a choice of tangential base point $t_0$ at $0$. Let $t_\infty := \iota(t_0)$ denote the corresponding tangential base point at $\infty$ (see \S\ref{ss_pi_1}). Then associated to $t_0,t_\infty$ we have fiber functors
$$F_{t_0},F_{t_\infty} : \FEt_{X^\circ}\longrightarrow\Sets$$
and canonical generators of inertia $\gamma_0\in\pi_1(\bP^*,t_0),\gamma_\infty\in\pi_1(\bP^*,t_\infty)$.

\begin{defn} Let $\cC_{\bP^1}^\succ$ be the category of pairs $(q : D\rightarrow\bP^1,\alpha)$ where $q$ is a (possibly disconnected) smooth $G$-cover, \'{e}tale over $\bP^*$, and $\alpha$ is a $G$-equivariant bijection of fibers $\alpha : D_0\rightarrow D_\infty$, where $D_0 := q^{-1}(0)$ and $D_\infty := q^{-1}(\infty)$. A morphism between $(q : D\rightarrow\bP^1,\alpha)\rightarrow (q' : D'\rightarrow\bP^1,\alpha')$ in $\cC_{\bP^1}^\succ$ is a $G$-equivariant morphism $f : D\rightarrow D'$ over $\bP^1$ respecting the identifications of the fibers. In a formula, we require that $f$ satisfies
$$f(\alpha(x)) = \alpha'(f(x))\qquad\text{for all $x\in D_0$}$$
\end{defn}

Informally, $\cC_{\bP^1}^\succ$ is the category of three point covers with ``gluing data''. The symbol ``$\succ$'' represents that the objects come with gluing data.

\begin{defn} Given a nodal elliptic curve $E$, a \emph{standard normalization} of $E$ is a finite birational map $\nu : \bP^1\rightarrow E$ satisfying
\begin{itemize}
	\item $\nu(1) = O\in E$.
	\item $\nu|_{\nu^{-1}(E^\sm)} : \nu^{-1}(E^\sm)\rightarrow E^\sm$ is an isomorphism.
	\item Let $z\in E$ be the node, then $\nu^{-1}(z) = \{0,\infty\}\subset\bP^1$.
\end{itemize}
The map $\nu : \bP^1\rightarrow E$ identifies $\bP^1$ with the normalization of $E$, and hence up to precomposing with $\iota : z\mapsto 1/z$, the map $\nu$ is uniquely determined by these properties.
\end{defn}

Let $\nu : \bP^1\rightarrow E$ be a standard normalization, and let $p : C\rightarrow E$ be a precuspidal $G$-cover. Let $\nu_C : C'\rightarrow C$ denote the normalization of $C$, then the $G$-action on $C$ extends uniquely to $C'$. Moreover, by the universal property of normalization \cite[035Q]{stacks} there is a unique map $p' : C'\rightarrow\bP^1$ fitting into the commutative diagram
\[\begin{tikzcd}
C'\ar[r,"\nu_C"]\ar[d,"p'"] & C\ar[d,"p"] \\
\bP^1\ar[r,"\nu"] & E
\end{tikzcd}\]
such that $p'$ is a $G$-cover. Note that $C'$ may be disconnected even if $C$ is connected. Moreover, the normalization $\nu_C$ defines a $G$-equivariant bijection (a ``gluing datum'')
$$\alpha_p : C'_0\rightiso C'_\infty$$
defined by sending $x\in C'_0$ to the unique point $y\in C'_\infty$ such that $\nu_C(x) = \nu_C(y)$. Let $C'^\circ := p'^{-1}(\bP^*)$, then the restriction $p'|_{C'^\circ} : C'^\circ\rightarrow\bP^*$ is a finite \'{e}tale $G$-cover, and so the pair $(p',\alpha_p)$ defines an object of $\cC_{\bP^1}^\succ$. Moreover, by the universal property of normalization, a morphism of objects in $\cC_E^{pc}$ induces a morphism of their images in $\cC_{\bP^1}^\succ$. Thus, taking normalizations and remembering the gluing data defines a functor
\begin{eqnarray*}
\Xi_\nu : \cC_E^{pc} & \longrightarrow & \cC_{\bP^1}^\succ \\
(p : C\rightarrow E) & \mapsto & (p',\alpha_p).	
\end{eqnarray*}
We note that if $\nu'$ is another standard normalization, then $\nu' = \nu\circ\iota$, and $\iota : z\mapsto 1/z$ induces an isomorphism of functors $\Xi_\nu\cong\Xi_{\nu'}$. We will eventually show that $\Xi_\nu$ is an equivalence of categories. Let $(q : D\rightarrow\bP^1,\alpha)$ be an object of $\cC_{\bP^1}^\succ$. We summarize some of the relevant structures associated to $(q,\alpha)$.

\begin{itemize}
\item[(a)] A path $\delta : t_0\leadsto t_\infty$ defines a bijection of fibers
$$\delta(q) : D_{t_0}\rightiso D_{t_\infty}$$
which is $G$-equivariant since $\delta$ is an isomorphism of functors. Sometimes we will abuse notation and simply write $\delta = \delta(q)$. Moreover, it determines an isomorphism of fundamental groups:
\begin{eqnarray*}
\la{\delta}(\cdot) : \pi_1(\bP^*,t_0) & \rightiso & \pi_1(\bP^*,t_\infty) \\
\gamma & \mapsto & \la{\delta}\gamma := \delta\gamma\delta^{-1}	
\end{eqnarray*}
If we let $\pi_1(\bP^*,t_0)$ act on $D_{t_\infty}$ via this isomorphism, then $\delta(q)$ is $\pi_1(\bP^*,t_0)$-equivariant.
\item[(b)] For any point $x_0\in D_{t_0}$, the freeness and transitivity of the $G$-action on fibers of $q|_{q^{-1}(\bP^*)}$ yield monodromy representations
\begin{eqnarray*}
\varphi_{x_0} : \pi_1(\bP^*,t_0) & \lra & G \\
\varphi_{\delta x_0} : \pi_1(\bP^*,t_\infty) & \lra & G
\end{eqnarray*}
defined in the usual way by the relations:
$$\begin{array}{rcll}
\gamma\cdot x_0 & = & x_0\cdot \varphi_{x_0}(\gamma) & \text{for any }\gamma\in\pi_1(\bP^*,t_0) \\
\gamma\cdot \delta x_0 & = & \delta x_0\cdot \varphi_{\delta x_0}(\gamma) & \text{for any }\gamma\in\pi_1(\bP^*,t_\infty)
\end{array}$$
\item[(c)] Let $[\cdot]$ denote the map
\begin{eqnarray*}
{[\cdot]} : D_{t_0} & \longrightarrow & \langle\gamma_0\rangle\bs D_{t_0}	 \\
x & \mapsto & {[x]} := \langle\gamma_0\rangle x
\end{eqnarray*}
and similarly for points of $D_{t_\infty}$.
\item[(d)] Let $\gamma_\infty^\delta := \delta^{-1}\gamma_\infty\delta$. As in Proposition \ref{prop_tbp}, there are canonical bijections functorial in $D$
$$\begin{array}{rcl}
\xi_0 : \langle \gamma_0\rangle\bs D_{t_0} & \rightiso & D_0 \\[2pt]
\xi_\infty : \langle\gamma_\infty\rangle\bs D_{t_\infty} & \rightiso & D_\infty \\[2pt]
\xi_\infty\circ\delta : \langle\gamma_\infty^\delta\rangle\bs D_{t_0} & \rightiso & D_\infty
\end{array}$$
Here, functoriality means the following. Let $F_0$ (resp. $F_0'$) be the functors $\cC_{\bP^1}^\succ\rightarrow\Sets$ sending $(q : D\rightarrow \bP^1,\alpha)$ to $D_0$ (resp. $\langle\gamma_0\rangle\bs D_{t_0}$). Then $\xi_0$ defines a natural isomorphism of functors $\xi_0 : F_0'\rightiso F_0$. The map $\xi_0\circ [\cdot] : D_{t_0}\rightarrow D_0$ should be thought of as the ``specialization map to the ramified fiber'' induced by the ``specialization'' $t_0\leadsto 0$ in $\bP^1$. Thus for $x\in D_{t_0}$, we will often abuse notation and view $[x]$ as an element of $D_0$, and similarly for points of $D_{t_\infty}$. Via $\xi_0,\xi_\infty$, we will often view $\alpha$ as a bijection of coset spaces
$$\alpha : \langle\gamma_0\rangle\bs D_{t_0}\rightiso\langle\gamma_\infty\rangle\bs D_{t_\infty}$$
\item[(e)] Given a point $x\in D_0$, let $G_x := \Stab_G(x)$. The right-action of $G$ on $D$ induces a left action on the cotangent space $T^*_xD$, whence a \emph{local representation} 
\begin{equation}\label{eq_local_representation}
\chi_x : G_x\lra \GL(T_x^*D)
\end{equation}
\end{itemize}




Next we record some basic computations that we will use freely in what follows. Recall that for a path $\delta : x\leadsto y$ and a loop $\gamma$ at $y$, we write $\gamma^\delta := \delta^{-1}\gamma\delta$ to be the loop at $x$ given by $x\stackrel{\delta}{\leadsto}y\stackrel{\gamma}{\leadsto}y\stackrel{\delta^{-1}}{\leadsto} x$ (see \eqref{eq_gamma_delta}).

\begin{prop}\label{prop_obvious} Let $(q : D\rightarrow\bP^1,\alpha)\in\cC_{\bP^1}^\succ$, let $x_0\in D_{t_0}$, and let $\delta : t_0\leadsto t_\infty$ be a path. Then we have
\begin{itemize}
\item[(a)] $G_{[x_0]} = \langle\varphi_{x_0}(\gamma_0)\rangle$ and $\chi_{[x_0]}(\varphi_{x_0}(\gamma_0)) = \zeta_n$. 
\item[(b)] $G_{[\delta x_0]} = \langle\varphi_{\delta x_0}(\gamma_\infty)\rangle$ and $\chi_{[\delta x_0]}(\varphi_{\delta x_0}(\gamma_\infty)) = \zeta_n$. 
\item[(c)] $\varphi_{\delta x_0}(\gamma) = \varphi_{x_0}(\gamma^\delta)$ for all $\gamma\in\pi_1(\bP^*,t_\infty)$.
\item[(d)] $\varphi_{\delta^{-1}x_\infty}(\gamma) = \varphi_{x_\infty}(\gamma^{\delta^{-1}})$ for all $\gamma\in\pi_1(\bP^*,t_\infty)$, $x_\infty\in D_{t_\infty}$.
\item[(e)] $\varphi_{x g}(\gamma) = g^{-1}\varphi_{x}(\gamma)g$ for all $g\in G, \gamma\in\pi_1(\bP^*,q(x))$, and any unramified or tangential base point $x\in D$.
\end{itemize}
\end{prop}
\begin{proof} Parts (a) and (b) follows from the local picture. For (c), note that $\delta x_0\varphi_{\delta x_0}(\gamma) = \gamma \delta x_0 = \delta\delta^{-1}\gamma\delta x_0$ which forces $x_0\varphi_{\delta x_0}(\gamma) = \delta^{-1}\gamma\delta x_0$. This precisely says that $\varphi_{x_0}(\delta^{-1}\gamma\delta) = \varphi_{\delta x_0}(\gamma)$ as desired. The proof of (d) and (e) are similar.
\end{proof}

In the remainder of this section, we will show that $\Xi_\nu : \cC_E^{pc}\rightarrow \cC_{\bP^1}^\succ$ is an equivalence.

\begin{lemma}\label{lemma_pushout} Given an object $(\pi : D\rightarrow\bP^1,\alpha)$ of $\cC_{\bP^1}^\succ$. Suppose $\pi^{-1}(0)$ has cardinality $n$. Let
$$\beta : \underbrace{\sqcup_{x\in\pi^{-1}(\{0,\infty\})}\Spec k}_{Z}\longrightarrow \underbrace{\sqcup_{i=1}^n\Spec k}_{W}.$$
Then the pushout $D_\alpha := D\cup_{Z,\beta} W$ exists in the category of schemes. Let $\nu : D\rightarrow D_\alpha$ be the canonical injection. Then we have
\begin{itemize}
\item[(a)] The map $\nu : D\rightarrow D_\alpha$ is finite and surjective.
\item[(b)] $\nu$ identifies the topological space of $D_\alpha$ with the topological quotient of $D$ by the equivalence relation given by $\alpha$.
\item[(c)] For an open $U\subset D_\alpha$, we have $\Gamma(U,\cO_{D_\alpha}) = \{f\in\Gamma(\nu^{-1}(U),\cO_D)\;|\; f(x) = f(\alpha(x))\text{ for all }x\in \pi^{-1}(0)\}$.
\item[(d)] $D_\alpha$ is a prestable curve, and every point in $\nu(\pi^{-1}(\{0,\infty\}))$ is a node.
\item[(e)] $\nu$ identifies $D$ with a normalization of $D_\alpha$.
\item[(f)] $\nu$ restricts to an isomorphism on $D - \pi^{-1}(\{0,\infty\})$.
\item[(g)] Let $(E_0,O)$ be the nodal elliptic curve obtained by gluing $\bP^1$ along $\{0,\infty\}$, where we set the origin $O$ to be the image of $1\in\bP^1$. Then the canonical map $D_\alpha\rightarrow E_0$ is precuspidal $G$-cover of $E_0$, and sending $(\pi : D\rightarrow\bP^1,\alpha)\mapsto (\pi_\alpha : D_\alpha\rightarrow E_0)$ defines a ``gluing'' functor
$$\Glue : \cC_{\bP^1}^\succ\lra\cC_{E_0}^{pc}$$
\end{itemize}
We say that $D_\alpha$ is the prestable curve obtained by gluing $D$ along $\alpha$. Note that if $\pi$ is the degree 1 cover $D = \bP^1\rightarrow\bP^1$, then $D_\alpha$ is a nodal cubic, and isomorphic 
\end{lemma}
\begin{proof} By \cite[0E25]{stacks}, the pushout $D\cup_{Z,\beta} W$ exists and satisfies (b) and (c) (which in turn determines the pushout uniquely), and moreover the pushout diagram is also cartesian, so $\nu$ is the pullback of $\beta$ and hence is finite surjective. Part (d) is \cite[Theorem 3.4]{Knud83II}. Since $D$ is smooth, part (e) follows from the universal properties of normalization \cite[035Q]{stacks}, and (f) follows from (e).


Finally, for (g), the $G$-equivariance of $\alpha$ implies that the $G$-action descends to $D_\alpha$. Write $Z = \Spec R$ and $W = \Spec S$. Give $W$ the unique $G$-action making $\beta : Z\rightarrow W$ $G$-equivariant. Note that $E_0 = \bP^1\cup_{Z/G}(W/G)$. The behavior of $\Glue$ on morphisms is defined by the universal property of pushouts, so it remains to show that $\pi_\alpha : D_\alpha\rightarrow E_0$ is a precuspidal $G$-cover. By (f), this is obvious on the smooth locus, so let $U = \Spec A\subset D$ be a $G$-invariant open affine containing $\pi^{-1}(\{0,\infty\})$ (see Lemma \ref{lemma_tame_quotients}). Let $U_\alpha := \nu(U)$, then $U_\alpha = \Spec A\times_R S$, and $U_\alpha/G = \Spec (A\times_R S)^G$. Since taking $G$-invariants is left exact, this is equal to $\Spec A^G\times_{R^G} S^G$, which is precisely the corresponding open affine neighborhood of the node in $E_0$, so $\pi_\alpha$ induces an isomorphism $D_\alpha/G\cong E_0$. This shows that $\pi_\alpha$ is finite. Since the maps $D\rightarrow D/G$, $Z\rightarrow Z/G$, $W\rightarrow W/G$ are all flat, so is $\pi_\alpha : D_\alpha\rightarrow E_0$ \cite[0ECL]{stacks}. Thus $\pi_\alpha$ is a $G$-cover. Part (f) implies that $\pi_\alpha$ sends nodes to nodes and is \'{e}tale at smooth points not mapping to $O\in E_0$, so $\pi_\alpha$ is a precuspidal $G$-cover of $E_0$, as desired.
\end{proof}

\begin{thm}\label{thm_equivalence} Let $E$ be a nodal elliptic curve, and $\nu : \bP^1\rightarrow E$ a standard normalization. The categories $\cC_E^{pc},\cC_{\bP^1}^\succ$ are groupoids, and the functor
$$\Xi_\nu : \cC_E^{pc}\longrightarrow \cC_{\bP^1}^\succ$$
sending $p : C\rightarrow E$ to $(p' : C'\rightarrow\bP^1,\alpha_p)$ is an equivalence of categories. We note that if $\nu'$ is another standard normalization (equivalently, $\nu' = \nu\circ\iota$), then $\iota$ defines an isomorphism $\Xi_\nu\cong\Xi_{\nu'}$. Let $E_0$ be the nodal elliptic curve obtained by gluing 0 to $\infty$ in $\bP^1$. A quasi-inverse to $\Xi_\nu$ is given by composing the gluing functor of Lemma \ref{lemma_pushout}(g) with the isomorphism $\cC_{E_0}^{pc}\rightiso \cC_E^{pc}$ induced by any isomorphism $E_0\rightiso E$.
\end{thm}
\begin{proof} By the universal properties of normalization \cite[035Q]{stacks}, this normalization procedure defines a functor. Now consider a pair of objects $(q_1 : D_1\rightarrow\bP^1,\alpha_1), (q_2 : D_2\rightarrow\bP^1,\alpha_2)$ in $\cC_{\bP^1}^\succ$. Suppose we have a morphism $s : (q_1,\alpha_1)\rightarrow(q_2,\alpha_2)$ in $\cC_{\bP^1}^\succ$. Since it respects $\alpha_1,\alpha_2$, by the universal property of pushouts it induces a unique map $s_\alpha : (D_1)_{\alpha_1}\rightarrow (D_2)_{\alpha_2}$ in $\Sch/\bP^1$ fitting into a commutative diagram
\begin{equation}\label{eq_norm1}
\begin{tikzcd}
D_1\ar[r,"s"]\ar[d,"\nu_1"] & D_2\ar[d,"\nu_2"] \\
(D_1)_{\alpha_1}\ar[r,"s_\alpha"] & (D_2)_{\alpha_2}
\end{tikzcd}
\end{equation}
of schemes over $\bP^1$. By Lemma \ref{lemma_pushout}(e), the canonical maps $\nu_1,\nu_2$ are also normalization maps. The universal property of pushouts also gives unique maps $p_i : (D_i)_{\alpha_i}\rightarrow E_0$ making the diagrams
\[\begin{tikzcd}
D_i\ar[r,"\nu_i"]\ar[d,"q_i"] & (D_i)_{\alpha_i}\ar[d,"p_i"] \\
\bP^1\ar[r] & E_0
\end{tikzcd}\]
commute (for $i = 1,2$), and forming a commutative prism with (\ref{eq_norm1}). By Lemma \ref{lemma_pushout}(e,g), each $p_i$ is a precuspidal $G$-cover and $q_i : D_i\rightarrow\bP^1$ is a normalization of $p_i$. Pulling back $p_i$ via some isomorphism $E\rightiso E_0$, we find that $q_i$ is the normalization of a precuspidal $G$-cover of $E$, so $\Xi_\nu$ is essentially surjective.


Since $\nu_1,\nu_2$ are birational morphisms between separated schemes, $s$ determines $s_\alpha$, so $\Xi_\nu$ is faithful. Since $s$ can be recovered from $s_\alpha$ as the induced map on normalizations $\Xi_\nu$ is fully faithful, so $\Xi_\nu$ is an equivalence.


Finally, to see that they are both groupoids, let $\FEt_{\bP^*}^G$ denote the category of finite \'{e}tale $G$-covers of $\bP^*$. Consider the ``forgetful functor'' $\cC_{\bP^1}^\succ\rightarrow\FEt_{\bP^*}^G$ sending $(q,\alpha)$ in $\cC_{\bP^1}^\succ$ to the restriction of $q$ to $\bP^*$. Since all schemes considered are separated, this functor is faithful. The category $\FEt_{\bP^*}^G$ is a groupoid (any $G$-equivariant map between \'{e}tale $G$-covers is an isomorphism).  Since the inverse of a map which preserves the identifications of fibers (in $\cC_{\bP^1}^\succ$) must also preserve identifications of fibers, this functor is also conservative (i.e. ``isomorphism reflecting''), and hence we find that $\cC_{\bP^1}^\succ$ is a groupoid, and hence so is $\cC_E^{pc}$.
\end{proof}

\begin{defn} We say that an object $(C'\rightarrow\bP^1,\alpha)$ of $\cC_{\bP^1}^\succ$ is \emph{connected} if the precuspidal $G$-curve $C'_\alpha$ is connected.
\end{defn}

\subsection{Balanced objects of $\cC_{\bP^1}^\succ$}\label{ss_balanced_objects_of_normalization}
Here we describe the property of being balanced in terms of the internal logic of $\cC_{\bP^1}^\succ$.


Let $(q : D\rightarrow\bP^1,\alpha)$ be an object of $\cC_{\bP^1}^\succ$. Let $x_0\in D_{t_0}$ be a point, and $\delta : t_0\leadsto t_\infty$ a path. Then $\alpha([x_0]) = [\delta x_0\cdot h]$ for some $h\in G$, where $h$ is unique up to the coset
$$H_{\delta,x_0} := G_{[\delta x_0]} h \in G_{[\delta x_0]}\bs G$$
Thus the coset $H_{\delta,x_0}\subset G$ is a well-defined function of the quadruple $(q,\alpha,\delta,x_0)$.


\begin{lemma}\label{lemma_hprime} Let $\delta'$ be another path from $t_0\leadsto t_\infty$. Then for any $h\in H_{\delta,x_0}$, we have
$$[\delta' x_0\cdot \varphi_{x_0}(\delta'^{-1}\delta)h] = [\delta x_0\cdot h]\qquad\text{or equivalently}\qquad \varphi_{x_0}(\delta'^{-1}\delta)h\in H_{\delta',x_0}$$
\end{lemma}
\begin{proof} The first equality follows from the definition of $\varphi_{x_0}$, and the equivalence follows from the definition of $H_{\delta,x_0}$ and $H_{\delta',x_0}$.
\end{proof}

The following proposition describes the objects of $\cC_{\bP^1}^\succ$ which correspond to balanced objects of $\cC_E^{pc}$.

\begin{prop}\label{prop_balanced_condition} For an object $(q : D\rightarrow\bP^1,\alpha)$ in $\cC_{\bP^1}^\succ$, the following are equivalent
\begin{itemize}
\item[(a)] For some (equivalently any) choices of $x_0\in D_{t_0}$ and $x_\infty\in D_{t_\infty}$ satisfying $[x_\infty] = \alpha([x_0])$, we have
$$\varphi_{x_0}(\gamma_0)^{-1} = \varphi_{x_\infty}(\gamma_\infty)$$
\item[(b)] For some (equivalently any) choices of $x_0\in D_{t_0}$, $\delta : t_0\leadsto t_\infty$, and $h\in H_{\delta,x_0}$, we have:
\begin{equation}\label{eq_condition}
\varphi_{x_0}(\gamma_0)^{-1} = h^{-1}\varphi_{x_0}(\gamma_\infty^\delta)h
\end{equation}
\end{itemize}
Moreover, if $p : C\rightarrow E$ is a precuspidal $G$-cover, then $p$ is balanced if and only if $\Xi(p) = (p',\alpha_p)\in\cC_{\bP^1}^\succ$ satisfies either of the equivalent conditions (a) or (b).
\end{prop}

\begin{defn} We say that an object $(q,\alpha)$ of $\cC_{\bP^1}^\succ$ is balanced if the equivalent conditions (a), (b) of Proposition \ref{prop_balanced_condition} are satisfied.
\end{defn}

\begin{proof}[Proof of Proposition \ref{prop_balanced_condition}] Fix a standard normalization $\nu : \bP^1\rightarrow E$. Since $\Xi_\nu : \cC_E^{pc}\rightarrow\cC_{\bP^1}^\succ$ is an equivalence, we may assume that $q : D\rightarrow\bP^1$ fits into a commutative diagram
\[\begin{tikzcd}
	D\ar[r,"\nu_C"]\ar[d,"q"] & C\ar[d,"p"] \\
	\bP^1\ar[r,"\nu"] & E
\end{tikzcd}\]
such that $\nu_C,\nu$ are normalization maps, $q$ is the map of normalizations induced by $p$, and $\alpha$ is the $G$-equivariant bijection $D_0\rightiso D_\infty$ induced by $\nu_C$. Equivalently, $(q,\alpha) \cong \Xi(p)$. We wish to show that $p$ is balanced if and only if $(q,\alpha)$ satisfies (a) if and only if it satisfies (b).


By Proposition \ref{prop_obvious}(a), $\varphi_{x_0}(\gamma_0)$ generates $G_{[x_0]}$ and is the unique element of $G_{[x_0]}$ inducing $\zeta_n$ on the cotangent space at $[x_0]$. Let $x_\infty\in D_{t_\infty}$ satisfying $[x_\infty] = \alpha([x_0])$. Again by Proposition \ref{prop_obvious}(a), $\varphi_{x_\infty}(\gamma_\infty)$ is the unique element of $G_{[x_\infty]} = G_{\alpha([x_0])}$ inducing $\zeta_n$ on the cotangent space at $[x_\infty]$. Since $G_{[x_0]} = G_{\alpha_p([x_0])} = G_{[x_\infty]}$ (due to $G$-equivariance of $\alpha$), we find that $p$ is balanced at $\nu_C([x_0]) = \nu_C([x_\infty])$ if and only if

\begin{equation}\label{eq_balanced_at_x_0}
\varphi_{x_0}(\gamma_0)^{-1} = \varphi_{x_\infty}(\gamma_\infty)	.
\end{equation}
In this case, using Proposition \ref{prop_obvious}(e), for any $g\in G$ we have
$$\varphi_{x_0g}(\gamma_0)^{-1} = g^{-1}\varphi_{x_0}(\gamma_0)^{-1}g = g^{-1}\varphi_{x_\infty}(\gamma_\infty)g = \varphi_{x_\infty g}(\gamma_\infty).$$
Since $\alpha([x_0g]) = \alpha([x_0])g = [x_\infty]g = [x_\infty g]$, we have shown that $p$ is balanced if and only if it is balanced at the image of $x_0$ if and only if \eqref{eq_balanced_at_x_0} is satisfied. Thus, the choices of $x_0,x_\infty$ are irrelevant, so in part (a), ``some'' is equivalent to ``any'', and $p$ being balanced is equivalent to (a).


For part (b), we note that having fixed $x_0$, $x_\infty$ satisfies $[x_\infty] = \alpha([x_0])$ if and only if $x_\infty = \delta x_0h$ for some $\delta : t_0\leadsto t_\infty$ and $h\in H_{\delta,x_0}$. Let $x_\infty$ be given in this way. Then using Proposition \ref{prop_obvious}, we find that
$$h^{-1}\varphi_{x_0}(\gamma_\infty^\delta)h = h^{-1}\varphi_{\delta x_0}(\gamma_\infty)h = \varphi_{\delta x_0h}(\gamma_\infty) = \varphi_{x_\infty}(\gamma_\infty).$$
This implies that (b) is equivalent to (a).
\end{proof}

\subsection{Galois correspondence for precuspidal $G$-covers}

Recall that $t_0$ denotes a tangential base point at $0\in\bP^1$, and $t_\infty := \iota(t_0)$ is the corresponding tangential base point at $\infty\in\bP^1$. In the remainder of \S\ref{section_cusps}, let $\Pi := \pi_1(\bP^*,t_0)$, and let $\delta$ be a \emph{good path} $t_0\leadsto t_\infty$ (c.f. Definition \ref{def_good_path}). Then $\Pi$ is a free profinite group of rank 2 topologically generated by $\gamma_0,\gamma_\infty^\delta := \delta^{-1}\gamma\delta$.

\begin{defn} For a good path $\delta : t_0\leadsto t_\infty$, let $\Sets^{(\Pi,G)_\delta,\succ}$ denote the category of pairs $(F,\alpha)$ where $F$ is a finite set equipped with a left $\Pi$-action which commutes with a free and transitive right $G$-action, and
$$\alpha : \langle\gamma_0\rangle\bs F\rightiso \langle\gamma_\infty^\delta\rangle\bs F$$
is a $G$-equivariant bijection (which we view as a ``combinatorial gluing datum''). We will call such a pair $(F,\alpha)$ a \emph{precuspidal $G$-datum} (relative to $\delta$). Morphisms are given by $(\Pi,G)$-equivariant maps respecting $\alpha$'s. Let $F_{\delta}^\succ$ be the functor
$$F_{\delta}^\succ : \cC_{\bP^1}^\succ\longrightarrow\Sets^{(\Pi,G)_\delta,\succ}$$
which sends $(q : D\rightarrow\bP^1,\alpha)$ to the pair $(D_{t_0},F_{\delta}^\succ(\alpha))$, where $F_{\delta}^\succ(\alpha)$ is the bijection
$$F_{\delta}^\succ(\alpha) \;:\; \langle\gamma_0\rangle\bs D_{t_0}\stackrel{\xi_0}{\longrightarrow} D_0\stackrel{\alpha}{\longrightarrow} D_\infty\stackrel{\xi_\infty^{-1}}{\longrightarrow}\langle\gamma_\infty\rangle\bs D_{t_\infty}\stackrel{\delta^{-1}}{\lra}\langle \gamma_\infty^\delta\rangle\bs D_{t_0}$$
where $\xi_0,\xi_\infty$ are bijections defined as in Proposition \ref{prop_tbp}(b).

\end{defn}

\begin{prop}\label{prop_Galois_for_pairs} The functor $F_{\delta}^\succ : \cC_{\bP^1}^\succ\longrightarrow\Sets^{(\Pi,G)_\delta,\succ}$ is an equivalence of categories.
\end{prop}
\begin{proof} Given objects $(D,\alpha),(D',\alpha')\in\cC_{\bP^1}^\succ$, the usual Galois correspondence implies giving a $G$-equivariant morphisms $f : D\rightarrow D'$ over $\bP^1$ is the same as giving a $(\Pi,G)$-equivariant map $D_{t_0}\rightarrow D'_{t_0}$. We must show that $f$ respects $\alpha$ if and only if the induced map $f_* : \langle\gamma_0\rangle\bs D_{t_0}\rightarrow\langle\gamma_0\rangle\bs D'_{t_0}$ respects $F_\delta^\succ(\alpha), F_\delta^\succ(\alpha')$. Consider the diagram
\[\begin{tikzcd}
	\langle\gamma_0\rangle\bs D_{t_0}\ar[d,"f_*"]\ar[r,"\xi_0"] & D_0\ar[d,"f_*"]\ar[r,"\alpha"] & D_\infty\ar[d,"f_*"]\ar[r,"\xi_\infty^{-1}"] & \langle\gamma_\infty\rangle\bs D_{t_\infty}\ar[d,"f_*"]\ar[r,"\delta^{-1}"] & \langle\gamma_\infty^\delta\rangle\bs D_{t_0}\ar[d,"f_*"] \\
	\langle\gamma_0\rangle\bs D_{t_0}'\ar[r,"\xi_0"] & D_0'\ar[r,"\alpha"] & D_\infty'\ar[r,"\xi_\infty^{-1}"] & \langle\gamma_\infty\rangle\bs D_{t_\infty}'\ar[r,"\delta^{-1}"] & \langle\gamma_\infty^\delta\rangle\bs D_{t_0}'
\end{tikzcd}\]
The composition of the top row is just $F_\delta^\succ(\alpha)$, and the composition of the bottom row is $F_\delta^\succ(\alpha')$. From left to right, the first, third, and fourth squares commute because $\xi_0,\xi_\infty^{-1},\delta^{-1}$ are all functorial in $D$. Thus the diagram commutes if and only if the second square commutes, which happens if and only if $f$ respects $\alpha$. Thus $F_\delta^\succ$ is fully faithful.



To show that $F_{\delta}^\succ$ is essentially surjective, we will define a quasi-inverse functor, denoted $H : \Sets^{(\Pi,G)_\delta,\succ}\rightarrow\cC_{\bP^1}^\succ$. Given a precuspidal $G$-datum $(F,\alpha)$, by the usual Galois correspondence we obtain a $G$-cover $\pi : X\rightarrow\bP^1$, \'{e}tale over $\bP^*$, equipped with a $(\Pi,G)$-equivariant bijection $\varphi : F\rightiso X_{t_0}$. Moreover, $\delta$ defines a $G$-equivariant bijection $\delta : X_{t_0}\rightiso X_{t_\infty}$, which induces a $G$-equivariant bijection $\langle\gamma_\infty^\delta\rangle\bs X_{t_0}\rightiso \langle\gamma_\infty\rangle\bs X_{t_\infty}$. Using $\varphi : F\rightiso X_{t_0}$ and $\delta\circ \varphi : F\rightiso X_{t_\infty}$, we obtain $G$-equivariant bijections
$$\begin{array}{rcccl}
\langle\gamma_0\rangle\bs F & \stackrel{\varphi}{\longrightarrow} & \langle\gamma_0\rangle\bs X_{t_0} & \stackrel{\xi_0}{\longrightarrow} & X_0 \\
\langle\gamma_\infty^\delta\rangle\bs F & \stackrel{\delta\circ\varphi}{\longrightarrow} & \langle\gamma_\infty\rangle\bs X_{t_\infty} & \stackrel{\xi_\infty}{\longrightarrow} & X_\infty
\end{array}$$
Connecting these bijections via $\alpha$, we obtain a $G$-equivariant bijection $\alpha_\pi : X_0\rightiso X_\infty$, and we define $H(F,\alpha) = (\pi : X\rightarrow\bP^1,\alpha_\pi)$. It's straightforward to check that $F_{\delta}^\succ\circ H\cong \id_{\Sets^{(\Pi,G)_\delta,\succ}}$, so $F_{\delta}^\succ$ is essentially surjective.
\end{proof}

Informally, an object $(F,\alpha)$ represents a ``three-point-cover with gluing data $X\rightarrow\bP^1$'' whose fibers at $t_0$ and $t_\infty$ are both represented by $F$, and where the bijection $\delta : X_{t_0}\rightiso X_{t_\infty}$ between the fibers becomes ``normalized'' to be the identity $\id_F : F\rightiso F$. The orbit spaces $\langle\gamma_0\rangle \bs F$ and $\langle\gamma_\infty^\delta\rangle \bs F$ correspond to the fibers $X_0,X_\infty$, and $\alpha$ corresponds to the gluing data $X_0\rightiso X_\infty$.

\subsection{Combinatorial balance, combinatorial connectedness}\label{ss_combinatorial_balance_connectedness}

For a good path $\delta : t_0\leadsto t_\infty$, we have defined equivalences of categories
$$\cC_E^{pc}\stackrel{\Xi}{\longrightarrow}\cC_{\bP^1}^\succ\stackrel{F_{\delta}^\succ}{\longrightarrow}\Sets^{(\Pi,G)_\delta,\succ}$$

Next we will define when a precuspidal $G$-datum corresponds to a balanced object of $\cC_{\bP^1}^\succ$. Given a precuspidal $G$-datum $(F,\alpha)$ (relative to a good path $\delta : t_0\leadsto t_\infty$), let $[\cdot]_0,[\cdot]_\infty$ be the ($G$-equivariant) projections
\begin{eqnarray*}
{[\cdot]}_0 : F & \longrightarrow & \langle\gamma_0\rangle\bs F \\
{[\cdot]}_\infty : F & \longrightarrow & \langle\gamma_\infty^\delta\rangle\bs F
\end{eqnarray*}
For $x\in F$, we have a homomorphism
$$\varphi_x : \Pi\longrightarrow G\qquad\text{defined by}\quad\gamma\cdot x = x\cdot \varphi_x(\gamma)\quad\text{for all $\gamma\in\Pi$}$$ 
Note that we have $G_{[x]_0} = \varphi_x(\langle\gamma_0\rangle)$ and $G_{[x]_\infty} = \varphi_x(\langle\gamma_\infty^\delta\rangle)$. Finally for $x\in F$ let
\begin{equation}\label{eq_H_delta_x}
H_{\delta,x} := \{h\in G \;|\; \alpha([x]_0) = [x]_\infty\cdot h\}\in G_{[x]_\infty}\bs G.
\end{equation}

\begin{defn} We say a precuspidal $G$-datum $(F,\alpha)\in \Sets^{(\Pi,G)_\delta,\succ}$ is \emph{balanced} if for some (equivalently any) choices of $x\in F$ and $h\in H_{\delta,x}$, we have
$$\varphi_x(\gamma_0)^{-1} = h^{-1}\varphi_x(\gamma_\infty^\delta)h$$
\end{defn}
It follows from Proposition \ref{prop_balanced_condition} that this definition makes sense and agrees with the notion of balanced objects in $\cC_{\bP^1}^\succ$ relative to the equivalence $F_{\delta}^\succ$.


Next we want to express the notion of connectedness for precuspidal $G$-covers in terms of precuspidal $G$-data. For this it is useful to introduce the graph of components of a prestable curve.


\begin{defn} Following \cite{Bass93}, a graph $\Gamma$ consists of the data $(\cV_\Gamma,\cE_\Gamma,\sigma,\tau)$ where
\begin{itemize}
\item $\cV_\Gamma$ is a set, called the set of \emph{vertices} of $\Gamma$,
\item $\cE_\Gamma$ is a set called the set of (directed) \emph{edges} (or arrows) of $\Gamma$,
\item $\sigma : \cE_\Gamma\rightarrow\cV_\Gamma$ is a function, called the ``incidence map'' or the ``source map''.
\item $\tau : \cE_\Gamma\rightarrow\cE_\Gamma$ is a fixed-point free involution. For $e\in\cE_\Gamma$ we also write $\ol{e} := \tau(e)$.
\end{itemize}
If $e\in\cE_\Gamma$, then we say that $e$ is an arrow from $\sigma(e)$ to $\sigma(\ol{e})$, and we write $\sigma(e)\stackrel{e}{\lra}\sigma(\ol{e})$.


A morphism of graphs is given by a pair of functions between the vertex and edge sets which commute with the $\sigma,\tau$ maps. Let $\Graphs$ denote the category of graphs.
\end{defn}

Let $\pi : C\rightarrow E$ be a precuspidal $G$-torsor. Let $(\pi' : C'\rightarrow\bP^1, \alpha_\pi : C_0'\rightiso C_\infty') = \Xi(\pi)$ be the associated object of $\cC_{\bP^1}^\succ$ obtained by normalization. Define a graph $\Gamma_\pi$ by
\begin{itemize}
\item $\cV_{\Gamma_\pi} = \pi_0(C')$ is the set of components of $C'$.
\item $\cE_{\Gamma_\pi} = C'_0\sqcup C'_\infty$.
\item For $e\in C'_0\sqcup C'_\infty$, $\sigma(e)$ is the component on which $e$ lies.
\item For $e\in C'_0$, $\ol{e} = \alpha_\pi(e)$, and if $e\in C'_\infty$ then $\ol{e} := \alpha_\pi^{-1}(e)$.
\end{itemize}
The action of $G$ on $C$ induces an action on $\Gamma_\pi$. It is immediate that
\begin{prop} $C$ is connected if and only if $\Gamma_\pi$ is connected.
\end{prop}

Let $\delta : t_0\leadsto t_\infty$ be a good path in $\bP^*$, and let $(F,\alpha)\in\Sets^{(\Pi,G)_\delta,\succ}$ be an object. Let $\Gamma_{F,\alpha}$ denote the graph:
\begin{itemize}
\item $\cV_{\Gamma_{F,\alpha}} = \Pi\bs F$.
\item $\cE_{\Gamma_{F,\alpha}} = \langle\gamma_0\rangle\bs F\sqcup \langle \gamma_\infty^\delta\rangle\bs F$.
\item For $e\in\cE_{\Gamma_{F,\alpha}}$, then $e = [x]_0$ or $e = [x]_\infty$ for some $x\in F$. In either case, let $\sigma(e) := \Pi\cdot x$.
\item For $e\in \langle\gamma_0\rangle\bs F$, define $\ol{e} := \alpha(e)$. For $e\in \langle \gamma_\infty^\delta\rangle\bs F$, define $\ol{e} := \alpha^{-1}(e)$.
\end{itemize}
$\Gamma_{F,\alpha}$ comes with a natural action of $G$ induced by its action on $F$.



\begin{prop}\label{prop_connectedness} The graph associated to a precuspidal $G$-cover $\pi : C\rightarrow E$ is isomorphic to the graph associated to the precuspidal $G$-datum $F_{\delta}^\succ(\Xi(\pi))$. Given a precuspidal $G$-datum $(F,\alpha)$ and $x\in F$, let $H_{\delta,x}$ be as in \eqref{eq_H_delta_x}. Let $M_x := \varphi_x(\Pi)$ be the ``monodromy group at $x$''. Then the graph $\Gamma_{F,\alpha}$ is connected if and only if either of the two equivalent conditions hold:
\begin{itemize}
\item[(a)] For some (equivalently any) $x\in F$, $G$ is generated by $M_x$ and $H_{\delta,x}$.
\item[(b)] For some (equivalently any) $x\in F$, $G$ is generated by $M_x$ and any element $h\in H_{\delta,x}$.
\end{itemize}
\end{prop}
\begin{proof} That the graphs of $\pi$ and $F_{\delta}^\succ(\Xi(\pi))$ are isomorphic follows from their definitions. Since $G_{[x]_\infty} = \varphi_x(\langle\gamma_\infty^\delta\rangle)\subset M_x$, we see that (a) and (b) are equivalent. It remains to show that they are equivalent to the connectedness of $\Gamma_{F,\alpha}$. Suppose $\Gamma_{F,\alpha}$ is connected. Then for any $g\in G$, there is a path in $\Gamma_{F,\alpha}$ from the vertex $\Pi x$ to $\Pi xg$. Setting $g_n = g, g_0 = 1$, this means there is a sequence $1 = g_1,g_2,\ldots,g_n = g\in G$ and edges $e_1,e_2,\ldots,e_n$ fitting into a path
\begin{equation}\label{eq_path}
\Pi x = \Pi xg_0\stackrel{e_1}{\lra} \Pi xg_1\stackrel{e_2}{\lra}\Pi xg_2\stackrel{e_3}{\lra}\cdots\stackrel{e_n}{\lra} \Pi xg_n = \Pi xg
\end{equation}
For $i\in[1,n]$ there are two possibilities for the edge $\Pi xg_{i-1}\stackrel{e_i}{\lra}\Pi x g_i$:
\begin{itemize}
\item $e_i = [\gamma x g_{i-1}]_0 = [x\varphi_x(\gamma)g_{i-1}]_0$ for some $\gamma\in\Pi$. In this case, since $\alpha([x]_0) = [x\cdot h]_\infty$ for any $h\in H_{\delta,x}$, by $G$-equivariance we have $\ol{e_i} = [xh\varphi_x(\gamma)g_{i-1}]_\infty$. The fact that $\sigma(\ol{e_i}) = \Pi x g_i$ says precisely that $[xh\varphi_x(\gamma)g_{i-1}]_\infty = \langle\gamma_\infty^\delta\rangle xh\varphi_x(\gamma)g_{i-1}\subset\Pi xg_i$. Thus, for some $\gamma'\in\Pi$, we have
$$xh\varphi_x(\gamma)g_{i-1} = \gamma'xg_i = x\varphi_x(\gamma')g_i\qquad\text{so}\qquad g_i\in M_xhM_xg_{i-1}$$
\item $e_i = [\gamma x g_{i-1}]_\infty = [x\varphi_x(\gamma)g_{i-1}]_\infty$ for some $\gamma\in\Pi$. In this case, a similar argument shows $\ol{e_i} = [xh^{-1}\varphi_x(\gamma)g_{i-1}]_0$. The fact that $\sigma(\ol{e_i}) = \Pi x g_i$ says precisely that $[xh^{-1}\varphi_x(\gamma)g_{i-1}]_0 = \langle\gamma_\infty^\delta\rangle xh^{-1}\varphi_x(\gamma)g_{i-1}\subset\Pi xg_i$. Thus, for some $\gamma'\in\Pi$, we have
$$xh^{-1}\varphi_x(\gamma)g_{i-1} = \gamma'xg_i = x\varphi_x(\gamma')g_i\qquad\text{so}\qquad g_i\in M_xh^{-1}M_xg_{i-1}$$
\end{itemize}
Thus, by induction we find that for $j = \{1,\ldots,n\}$, there exist $i_j\in\{\pm 1\}$ and $m_j,m_j'\in M_x$ such that
$$g_n = g = m_n'h^{i_n}m_nm_{n-1}'h^{i_{n-1}}m_{n-1}\cdots m_1'h^{i_1}m_1.$$
Since this holds for any $g\in G$, we find that $G$ is generated by $M_x$ and $h$.


Conversely, suppose $G$ is generated by $M_x$ and $h$ for some $h\in H_{\delta,x}$. Then for every $g\in G$ we may write it as $g = h^{i_n}m_nh^{i_{n-1}}m_{n-1}\cdots h^{i_1}m_1$ with $i_j\in\{\pm1\}$ and $m_j\in M_x$. Then for $k\in[0,n]$ define $g_k$ inductively by $g_0 = 1$ and $g_k = h^{i_k}m_kg_{k-1}$ for $k\ge 1$ (so $g_n = g$). Define edges $e_1\ldots,e_n$ by
$$e_k := \left\{\begin{array}{ll}
[xm_kg_{k-1}]_0 & \text{if } i_j = 1 \\
{[xm_kg_{k-1}]_\infty} & \text{if } i_j = -1
\end{array}\right.$$
Then $e_1,\ldots,e_n$ define a path
$$\Pi x =\Pi xg_0\stackrel{e_1}{\lra} \Pi xg_1\stackrel{e_2}{\lra}\Pi xg_2\stackrel{e_3}{\lra}\cdots\stackrel{e_n}{\lra} \Pi xg_n = \Pi xg.$$
Since this holds for every $g\in G$, this implies that $\Gamma_{F,\alpha}$ is connected. Since $x$ was arbitrary in the discussion above, it follows that ``some'' is equivalent to ``any'' in (a) and (b).
\end{proof}


\begin{defn} For a good path $\delta : t_0\leadsto t_\infty$, we say that a precuspidal $G$-datum $(F,\alpha)\in\Sets^{(\Pi,G)_\delta,\succ}$ is \emph{connected} if the equivalent conditions of Proposition \ref{prop_connectedness} hold. A precuspidal $G$-datum is \emph{cuspidal} if it is connected and balanced.
\end{defn}

\begin{prop}\label{prop_equivalences} For a good path $\delta : t_0\leadsto t_\infty$, the equivalences $\Xi : \cC_E^{pc}\rightiso\cC_{\bP^1}^\succ$ and $F_{\delta}^\succ : \cC_{\bP^1}^\succ\rightiso\Sets^{(\Pi,G)_\delta,\succ}$ restrict to equivalences between the full subcategories of connected objects, balanced objects, and connected+balanced objects.
\end{prop}
\begin{proof} Follows from the definitions together with Propositions \ref{prop_balanced_condition} and \ref{prop_connectedness}.
\end{proof}

\subsection{Parametrizing isomorphism classes of cuspidal admissible $G$-covers: the $\delta$-invariant}\label{ss_the_map_Inv}
Here we give an explicit characterization of the isomorphism classes of $\cC_E^c$ in terms of the internal logic of $G$. As usual let $\Pi := \pi_1(\bP^*,t_0)$, and let $\delta : t_0\leadsto t_\infty$ be a good path in $\bP^*$.

\begin{prop} A precuspidal $G$-datum $(F,\alpha)\in\Sets^{(\Pi,G)_\delta,\succ}$ is cuspidal (i.e. balanced and connected) if and only if both of the following conditions hold:
\begin{itemize}
\item[(a)] For some (equivalently any) $x\in F$ and $h\in H_{\delta,x}$ (see \eqref{eq_H_delta_x}), we have	
$$\varphi_x(\gamma_0)^{-1} = h^{-1}\varphi_x(\gamma_\infty^\delta)h\qquad\text{or equivalently}\qquad\varphi_x(\gamma_\infty^\delta) = h\varphi_x(\gamma_0)^{-1}h^{-1}$$
\item[(b)] For some (equivalently any) $x\in F$ and $h\in H_{\delta,x}$, $G$ is generated by $h$ and $\varphi_x(\gamma_0)$.
\end{itemize}
\end{prop}
\begin{proof} Condition (a) is the definition of balancedness. The group $\Pi$ is generated by $\gamma_0$ and $\gamma_\infty^\delta$, and hence in the presence of (a), $G$ is generated by $M_x := \varphi_x(\Pi)$ and $h$ if and only if it is generated by $\varphi_x(\gamma_0)$ and $h$. Thus in the presence of (a), by Proposition \ref{prop_connectedness}, (b) is equivalent to $(F,\alpha)$ being connected.	
\end{proof}


We will define a bijection between the set of isomorphism classes of cuspidal objects of $\Sets^{(\Pi,G)_\delta,\succ}$ and an explicit finite set built from $G$. Let $G$ act on the set $G\times G$ by conjugation:
$$g\cdot (u,h) = (gug^{-1},ghg^{-1})$$
If two pairs $(u,h)$ and $(u',h')$ are conjugate by this action, then we will write $(u,h)\sim (u',h')$. Let $\bZ$ act on $G\times G$ on the right by the rule
$$(u,h)\cdot k = (u,u^kh)\qquad k\in\bZ$$
These actions of $G$ and $\bZ$ on $G\times G$ commute and they preserve the subset of generating pairs. Let $\bI(G)$ denote the subset of the orbit space $G\bs(G\times G)/\bZ$ represented by generating pairs:
$$\bI(G) := G\bs\{(u,h) : \text{$u,h$ generate $G$}\}/\bZ$$

\begin{defn} For a generating pair $(u,h)\in G\times G$, let $\ps{u,h}$ denote its image in $\bI(G)$.
\end{defn}

\begin{defn} For a generating pair $(u,h)\in G\times G$, let $(F_{u,h},\alpha_{u,h})\in\Sets^{(\Pi,G)_\delta,\succ}$ be given as follows
\begin{itemize}
\item $F_{u,h}$ is the pointed set $G$ (with distinguished element $1_G$) viewed as a right $G$-torsor via right-multiplication. Give $F_{u,h}$ the structure of a left $\Pi$-set using the unique left $\Pi$-action which both commutes with the right $G$-action and also satisfies:
$$\gamma_\infty^\delta\cdot 1_G = 1_Gu\qquad\text{and}\qquad \gamma_0\cdot 1_G = 1_Gh^{-1}u^{-1}h$$
Explicitly, this left $\Pi$-action has the following description:
$$\gamma_\infty^\delta \cdot x = ux\qquad \gamma_0\cdot x = h^{-1}u^{-1}hx\qquad\forall x\in F_{u,h}$$
where the right hand sides of the equalities involve multiplication in $G$. Thus, the $\Pi$ action is given by left multiplication by the subgroup
\begin{equation}\label{eq_Muh}
M_{u,h} := \langle u,h^{-1}u^{-1}h\rangle\le G = \varphi_{1_G}(\Pi)	
\end{equation}
We call this subgroup the \emph{monodromy group} at $1_G$.
\item $\alpha_{u,h}$ is the $G$-equivariant bijection
\begin{eqnarray*}
\alpha_{u,h} : \langle\gamma_0\rangle\bs F_{u,h} & \lra & \langle\gamma_\infty^\delta\rangle\bs F_{u,h} \\
\langle\gamma_0\rangle \cdot x & \mapsto & \langle\gamma_\infty^\delta\rangle\cdot hx
\end{eqnarray*}
where $x\in F_{u,h}$ and $hx$ is multiplication in $G$. We note that this is the unique $G$-equivariant bijection satisfying $\alpha_{u,h}([1]_0) = [1]_\infty\cdot h = [h]_\infty$. Equivalently, in terms of the $\Pi$-action, we have
\begin{eqnarray*}
\alpha_{u,h} : \langle h^{-1}u^{-1}h\rangle\bs G & \lra & \langle u\rangle\bs G \\
\langle h^{-1}u^{-1}h\rangle\cdot x & \mapsto & \langle u\rangle\cdot hx
\end{eqnarray*}
Note that $H_{\delta,1_G} = \varphi_{1_G}(\langle \gamma_\infty^\delta\rangle)\cdot h = \langle u\rangle h$.
\end{itemize}
\end{defn}

It is easy to check that $(F_{u,h},\alpha_{u,h})$ is balanced and connected, hence cuspidal.

\begin{thm}[The $\delta$-invariant]\label{thm_delta_invariant} Let $\delta : t_0\leadsto t_\infty$ be a good path, and $(F,\alpha)\in\Sets^{(\Pi,G)_\delta,\succ}$ be a precuspidal $G$-datum. For any $x\in F$ and $h\in H_{\delta, x}$, define
\begin{eqnarray*}
\Inv_\delta : \Sets^{(\Pi,G)_\delta,\succ}  & \lra & G\bs (G\times G)/\bZ  \\
(F,\alpha) & \mapsto  & \ps{\varphi_x(\gamma_\infty^\delta),h}
\end{eqnarray*}
We will say that $\Inv_\delta(F,\alpha)$ is the \emph{invariant} of $(F,\alpha)$ with respect to the path $\delta$. Let $\cD_\delta$ be the full subcategory of $\Sets^{(\Pi,G)_\delta,\succ}$ consisting of cuspidal objects, and let $\pi_0(\cD_\delta)$ denote the set of isomorphism classes of $\cD_\delta$. Then $\Inv_\delta(F,\alpha)$ is independent of $x$ and $h$ and its restriction to $\cD_\delta$ yields a bijection
$$\Inv_\delta : \pi_0(\cD_\delta)\rightiso \bI(G)$$
with inverse induced by $(u,h)\mapsto (F_{u,h},\alpha_{u,h})$.
\end{thm}

\begin{defn}\label{def_delta_invariant} Let $E$ be a nodal elliptic curve. Given an object of $\cC_E^{pc}$ or $\cC_{\bP^1}^\succ$, its $\delta$-invariant is the $\delta$-invariant of its image in $\Sets^{(\Pi,G)_\delta,\succ}$ via the equivalences $\Xi$ and $F_{\delta}^\succ$. In particular, using the equality $\cC_E^c = \cAdm(G)_E$, taking $\delta$-invariants induces a bijection
$$\pi_0(\cAdm(G)_E)\rightiso \bI(G)$$
\end{defn}

\begin{proof}[Proof of Theorem \ref{thm_delta_invariant}] First we note that for $x' = xg\in F$, we have
$$\varphi_{xg}(\gamma) = g^{-1}\varphi_x(\gamma) g\quad\text{$\forall \gamma\in\Pi$}\quad\text{and}\quad H_{\delta,xg} = g^{-1}H_{\delta,x}g$$
Thus, changing $x$ results in a conjugate pair, so $\Inv_\delta$ is independent of $x$. On the other hand, the only other choices of $h$ are given by $\varphi_x(\gamma_\infty^\delta)^k\cdot h$, so by the definition of the $\bZ$-action on $(G\times G)$, $\Inv_\delta$ is also independent of $h$.


Next we note that for any choice of $x,h$, if $(F,\alpha)$ lies in $\cD_\delta$, then by connectedness of $(F,\alpha)$, $G$ is generated by $\varphi_x(\gamma_0),\varphi_x(\gamma_\infty^\delta)$, and $h$. On the other hand, since $(F,\alpha)$ is balanced, $\varphi_x(\gamma_0) = h^{-1}\varphi_x(\gamma_\infty^\delta)^{-1}h$, so $G$ is also generated by $\varphi_x(\gamma_\infty^\delta),h$. Thus $\Inv_\delta$ applied to an element of $\cD_\delta$ represents an element of $\bI(G)$.


Next, if $(F,\alpha)\cong (F',\alpha')$ via a $(\Pi,G)$-equivariant bijection $f : F\rightiso F'$ respecting $\alpha,\alpha'$, then one checks that
$$\varphi_{f(x)}(\gamma) = \varphi_x(\gamma)\quad\forall x\in F,\gamma\in\Pi\qquad\text{and}\qquad H_{\delta,x} = H_{\delta,f(x)}\quad\forall x\in F$$
so $\Inv_\delta$ is an isomorphism invariant.


We claim that the map $\{(u,h)\in G\times G \;|\; \text{$u,h$ generate $G$}\}\rightarrow\pi_0(\cD_\delta)$ defined by sending $(u,h)$ to the isomorphism class of $(F_{u,h},\alpha_{u,h})$ gives an inverse to $\Inv_\delta$. Indeed, we have $\Inv_\delta(F_{u,h},\alpha_{u,h}) = \ps{u,h}$ by construction. On the other hand, if $(F,\alpha)\in\Sets^{(\Pi,G)_\delta,\succ}$ has $\Inv_\delta(F,\alpha) = \ps{u,h}$, then for some $x\in F$, we have $\varphi_x(\gamma_\infty^\delta) = u$ and $h\in H_{\delta,x}$. Then one can check that the unique $G$-equivariant map $F_{u,h}\rightiso F$ sending $1_G\mapsto x$ defines an isomorphism $(F,\alpha)\rightiso (F_{u,h},\alpha_{u,h})$ as desired.
\end{proof}


\begin{prop}\label{prop_cuspidal_higman_invariant} Let $\delta : t_0\leadsto t_\infty$ be a good path. Let $\pi : C\rightarrow E$ be a cuspidal $G$-cover with $\delta$-invariant $\Inv_\delta(\pi) = \ps{u,h}$. Then its Higman invariant is the conjugacy class of $[u^{-1},h^{-1}] := u^{-1}h^{-1}uh\sim [u,h]$.
\end{prop}
\begin{proof} 

Let $e$ denote the common ramification indices of $\pi$ above $O$. Let $\Xi(\pi) = (\pi' : C'\rightarrow\bP^1,\alpha)$ be the normalization-with-gluing data. Since normalization induces an isomorphism on the smooth locus, every point in $\pi'^{-1}(1)$ is also ramified with index $e$, and the Higman invariant of $\pi$ is the Higman invariant of $\pi'$ at $1\in\bP^1$. We note that the conjugacy class of $[u,h]$ does not depend on the representative $(u,h)\in\ps{u,h}$. Thus from the definition of $\Inv_\delta$, we may assume that for some $x_0\in C'_{t_0}$, we have

$$u = \varphi_{x_0}(\gamma_\infty^\delta)\qquad\text{and}\qquad h \in H_{\delta,x_0}	$$

For any $x\in \pi'^{-1}(1)$, let $\chi_x : G_x\rightarrow\GL(T_x^*C')$ be the local representation as in \eqref{eq_local_representation}, then the Higman invariant of $\pi$ is precisely $\chi_x^{-1}(\zeta_e)$. Since $\delta$ is good, for some path $\epsilon : t_0\leadsto t_1$, we have $\gamma_1^\epsilon = (\gamma_\infty^\delta)^{-1}\gamma_0^{-1}$. 
By Proposition \ref{prop_tbp}, there is a canonical $G$-equivariant bijection
$$\xi_1 : \langle\gamma_1\rangle\bs C_{t_1}'\rightiso C_1'$$
Let $x_1 := \epsilon x_0\in C_{t_1}'$, and let $x := \xi_1([x_1])$. Then from the local picture (also see Proposition \ref{prop_obvious}), we have
$$\chi_x^{-1}(\zeta_e) = \varphi_{x_1}(\gamma_1) = \varphi_{x_0}(\gamma_1^\epsilon)$$
but by our choice of $\epsilon$, $\varphi_{x_0}(\gamma_1^\epsilon) = \varphi_{x_0}(\gamma_\infty^\delta)^{-1}\varphi_{x_0}(\gamma_0)^{-1} = u^{-1}\varphi_{x_0}(\gamma_0)^{-1}$, and since the $G$-action is balanced, we have $\varphi_{x_0}(\gamma_0)^{-1} = h^{-1}\varphi_{x_0}(\gamma_\infty^\delta)h = h^{-1}uh$. Thus we find that the Higman invariant of $\pi$ is the conjugacy class of
$$\chi_x^{-1}(\zeta_e) = [u^{-1},h^{-1}] \sim [u,h].$$
\end{proof}


\subsection{The $\delta$-invariant of the $[-1]$-pullback of a cuspidal $G$-cover}

The sole purpose of this section is to prove the following proposition. To simplify calculations, we will assume that $\delta$ is a good path which is moreover \emph{symmetric} (c.f. Definition \ref{def_good_path}).

\begin{prop}\label{prop_iota_pullback} Let $\iota\in\Aut(\bP^1)$ denote the unique automorphism fixing $1$ and switching $0,\infty$. Let $t_0$ be a tangential base point at $0\in\bP^1$, and let $t_\infty := \iota(t_0)$. Let $\delta : t_0\leadsto t_\infty$ be a \emph{symmetric} good path. Let $E$ be a nodal elliptic curve over $k$. Let $p : C\rightarrow E$ be a cuspidal $G$-cover. If $\Inv_\delta(p) = \ps{u,h}$, then $\Inv_\delta([-1]^*p) = \ps{u^{-1},h^{-1}}$. In particular, if $[-1]$ denotes the unique nontrivial automorphism of $E$, then the following are equivalent
\begin{itemize}
\item[(a)] $[-1]$ lifts to an automorphism of $C$.
\item[(b)] $[-1]^*p\cong p$ as cuspidal $G$-covers.
\item[(c)] $\ps{u,h} = \ps{u^{-1},h^{-1}}$,
\item[(d)] $(u^{-1},h^{-1})$ is conjugate to $(u,u^rh)$ for some $r > 0$.
\end{itemize}
\end{prop}

\begin{proof} The equivalence of (a),(b),(c),(d) follows from the statement that $\Inv_\delta([-1]^*p) = \ps{u^{-1},h^{-1}}$, so this is what we will prove.


Via $\Xi$, $p : C\rightarrow E$ corresponds to an object $(q : D\rightarrow\bP^1,\alpha)\in\cC_{\bP^1}^\succ$. Let $\ol{D} := \iota^*D$, so we have a cartesian diagram
\[\begin{tikzcd}
\ol{D}\ar[r,"\tilde{\iota}"]\ar[d,"\ol{q}"] & D\ar[d,"q"] \\
\bP^1\ar[r,"\iota"] & \bP^1
\end{tikzcd}\]
with $\tilde{\iota}$ a $G$-equivariant isomorphism. Via the standard normalization $\bP^1\rightarrow E$, $[-1]$ induces the automorphism $\iota$, so it suffices to show that if $\Inv_\delta(q,\alpha) = \ps{u,h}$, then $\Inv_\delta(\iota^*q,\iota^*\alpha) = \ps{u^{-1},h^{-1}}$.


The map $\tilde{\iota}$ induces bijections $\tilde{\iota} : \ol{D}_0\rightiso D_\infty$ and $\tilde{\iota} : \ol{D}_\infty\rightiso D_0$. Thus $\alpha : D_0\rightiso D_\infty$ induces a bijection
$$\ol{\alpha} := \iota^*\alpha : \ol{D}_0\stackrel{\tilde{\iota}}{\lra} D_\infty\stackrel{\alpha^{-1}}{\longrightarrow} D_0\stackrel{\tilde{\iota}^{-1}}{\lra}\ol{D}_\infty$$


Let $x_0\in D_{t_0}$ be a point, and as usual let
$$H_{\delta,x_0} := \{h\in G \;|\; \alpha([x_0]) = [\delta x_0 h]\} \in G_{[\delta x_0]}\bs G$$
Fix $h\in H_{\delta, x_0}$, then we have
$$\Inv_\delta(q,\alpha) = \ps{\varphi_{x_0}(\gamma_\infty^\delta),h}$$
Since $\delta$ is symmetric, write $\iota_*\delta = \gamma_0^r\delta^{-1}\gamma_\infty^s$ for some $r,s\in\bZ$. Let $x_\infty := \gamma_\infty^{-s}\delta x_0 h\in D_{t_\infty}$, so that $\alpha([x_0]) = [x_\infty]$. The canonical isomorphism $F_{t_0}\circ \iota^* \cong F_{\iota(t_0)} = F_{t_\infty}$ evaluated at $q : D\rightarrow\bP^1$ is realized by the bijection
$$\tilde{\iota} : \ol{D}_{t_0}\rightiso D_{t_\infty}$$
which is $G$-equivariant and moreover (by definition of $\iota_* : \pi_1(\bP^*,t_0)\rightarrow\pi_1(\bP^*,t_\infty)$ satisfies:
\begin{equation}\label{eq_iota_pullback_1}
\tilde{\iota}(\gamma z) = i_*\gamma\cdot\tilde{\iota}(z)\qquad \text{for all $z\in\ol{D}_{t_0}$ and $\gamma\in\pi_1(\bP^*,t_0)$}	
\end{equation}
In particular, this implies that $\tilde{\iota} : \ol{D}_{t_0}\rightarrow D_{t_\infty}$ commutes with $[\cdot]$, and the same is true for $\tilde{\iota} : \ol{D}_{t_\infty}\rightarrow D_{t_0}$.


Let $\ol{x_\infty} := \tilde{\iota}^{-1}(x_\infty)\in \ol{D}_{t_0}$. Then by \eqref{eq_iota_pullback_1} and $G$-equivariance of $\tilde{\iota}$, for any $\gamma\in\pi_1(\bP^*,t_0)$ we have
\begin{equation}\label{eq_iota_pullback_2}
x_\infty\varphi_{\ol{x_\infty}}(\gamma) = \tilde{\iota}(\ol{x_\infty}\cdot\varphi_{\ol{x_\infty}}(\gamma)) = \tilde{\iota}(\gamma\cdot\ol{x_\infty}) = \iota_*\gamma\cdot x_\infty = x_\infty\varphi_{x_\infty}(\iota_*\gamma)
\end{equation}
Since $\iota_*\delta = \gamma_0^r\delta^{-1}\gamma_\infty^s$, we have $\iota_*(\gamma_\infty^\delta) = \gamma_0^{\delta^{-1}\gamma_\infty^s}$. Thus since the $G$-action on $D_{t_\infty}$ is free, we get

$$\begin{array}{rcll}
\varphi_{\ol{x_\infty}}(\gamma_\infty^\delta) & = & \varphi_{x_\infty}(\gamma_0^{\delta^{-1}\gamma_\infty^s}) & \text{take $\gamma = \gamma_\infty^\delta$ in \eqref{eq_iota_pullback_2}}\\
 & = & \varphi_{\delta^{-1}\gamma_\infty^s x_\infty}(\gamma_0) & \text{by Proposition \ref{prop_obvious}(d)} \\
 & = & \varphi_{x_0h}(\gamma_0) & \text{since $x_\infty := \gamma_\infty^{-s}\delta x_0 h$} \\
 & = & h^{-1}\varphi_{x_0}(\gamma_0)h & \text{by Proposition \ref{prop_obvious}(e)} \\
 & = & h^{-2}\varphi_{x_0}(\gamma_\infty^\delta)^{-1}h^2 & \text{since $q$ is balanced (see Proposition \ref{prop_balanced_condition})}
\end{array}$$
This computes the first part of $\Inv(\iota^*q,\iota^*\alpha)$. To compute the second part, note that
\begin{equation}\label{eq_iota_pullback_3}
\ol{\alpha}([\ol{x_\infty}]) = \tilde{\iota}^{-1}(\alpha^{-1}([x_\infty])) = \tilde{\iota}^{-1}(\alpha^{-1}([\gamma_\infty^{-s}\delta x_0h])) = \tilde{\iota}^{-1}(\alpha^{-1}([\delta x_0h])) = \tilde{\iota}^{-1}([x_0])
\end{equation}

On the other hand, for $g\in G$, we have
\begin{equation}\label{eq_iota_pullback_4}
\tilde{\iota}(\delta\ol{x_\infty}g) = \iota_*\delta\cdot\tilde{\iota}(\ol{x_\infty})\cdot g = \gamma_0^r\delta^{-1}\gamma_\infty^s x_\infty g	
\end{equation}
Combining \eqref{eq_iota_pullback_3} with \eqref{eq_iota_pullback_4}, we get
\begin{equation}\label{eq_iota_pullback_5}
\ol{\alpha}([\ol{x_\infty}]) = [\delta\ol{x_\infty}g]\iff \tilde{\iota}^{-1}([x_0]) = [\tilde{\iota}^{-1}(\gamma_0^r\delta^{-1}\gamma_\infty^s x_\infty g)] \iff [x_0]  = [\gamma_0^r\delta^{-1}\gamma_\infty^s x_\infty g] = [x_0hg]
\end{equation}
Here we have used the fact that $\tilde{\iota}$ commutes with $[\cdot]$. On the other hand, note that
$$G_{[\delta\ol{x_\infty}]} = G_{[\ol{x_\infty}]} = G_{[x_\infty]} = G_{[\delta x_0h]} = G_{[x_0h]} = h^{-1}G_{[x_0]}h$$
Here we have used the fact that $\tilde{\iota}$ and (the fiber bijections induced by) $\delta$ are both $G$-equivariant. Thus \eqref{eq_iota_pullback_5} holds if and only if $hg\in G_{[x_0]}$, equivalently $g\in h^{-1}G_{[x_0]} = G_{[\delta\ol{x_\infty}]}h^{-1}$. This shows that
$$H_{\delta,\ol{x_\infty}} = G_{[\delta \ol{x_\infty}]}h^{-1}$$
Thus $\Inv_\delta(\iota^* q,\iota^*\alpha) = \ps{h^{-2}\varphi_{x_0}(\gamma_\infty^\delta)^{-1}h^2,h^{-1}} = \ps{\varphi_{x_0}(\gamma_\infty^\delta)^{-1},h^{-1}}$, which proves the Proposition.
\end{proof}

\subsection{Automorphism groups of cuspidal objects of $\cAdm(G)$}\label{ss_cuspidal_automorphisms}
By Theorem \ref{thm_delta_invariant}, every object $(F,\alpha)\in\Sets^{(\Pi,G)_\delta,\succ}$ is isomorphic to an object of the form $(F_{u,h},\alpha_{u,h})$ where $(u,h)$ represents the class $\Inv_\delta(F,\alpha)\in\bI(G)$. Here we calculate the automorphism group of the pair $(F_{u,h},\alpha_{u,h})$.


First we consider automorphisms of the $(\Pi,G)$-set $F_{u,h}$. Let $M_{u,h} := \langle u,h^{-1}u^{-1}h\rangle$ as in \eqref{eq_Muh}. Let $S_G$ denote the symmetric group on the underlying set of $G$, then an automorphism of the $(\Pi,G)$-set $F_{u,h}$ is given by a permutation $\sigma\in S_G$ such that:
\begin{itemize}
\item ($G$-equivariance) $\sigma(xg) = \sigma(x)g$ for all $x\in F_{u,h}$, $g\in G$.
\item ($\Pi$-equivariance) $\sigma(mx) = m\sigma(x)$ for all $x\in F_{u,h}$, $m\in M_{u,h}$.
\end{itemize}
Equivariance in $G$ implies that $\sigma(g) = \sigma(1\cdot g) = \sigma(1)\cdot g$ so $\sigma$ is determined by $\sigma(1)\in G$. On the other hand $\Pi$-equivariance implies that $m\sigma(1) = \sigma(m\cdot 1) = \sigma(m) = \sigma(1\cdot m) = \sigma(1)\cdot m$, so $\sigma(1)\in C_G(M_{u,h})$. Conversely, for any $a\in C_G(M_{u,h})$, the permutation
$$\sigma_a : G\rightarrow G\qquad\text{sending}\qquad g \mapsto ag$$
defines a permutation of $G = F_{u,h}$ which is both right $G$-equivariant and left $M_{u,h}$-equivariant, so it defines an automorphism of $F_{u,h}$. Next we seek to identify the elements $a\in C_G(M_{u,h})$ such that $\sigma_a$ respects the gluing $\alpha_{u,h}$. For $a\in C_G(M_{u,h})$, $\sigma_a$ respects the gluing if and only if we have an equality of cosets
$$\langle u\rangle\cdot h\sigma_a(x) = \langle u\rangle\cdot\sigma_a(hx)\qquad \text{for all $x\in F_{u,h}$}$$
This is equivalent to saying
\begin{equation}\label{eq_dehn_twist}
\langle u\rangle aha^{-1} = \langle u\rangle h\qquad\text{or equivalently}\qquad aha^{-1} = u^{k_a}h \quad\text{for some $k_a\in\bZ$}
\end{equation}

Thus, automorphisms of $(F_{u,h},\alpha_{u,h})$ are precisely the permutations $\sigma_a$ such that $a\in G$ centralizes $M_{u,h} = \langle u,h^{-1}u^{-1}h\rangle$ and satisfies \eqref{eq_dehn_twist}. Let
\begin{equation}\label{eq_AGuh}
A_{G,u,h} := \{a\in C_G(M_{u,h}) : aha^{-1} = u^{k_a}h\text{ for some $k_a\in\bZ$}\}\le C_G(M_{u,h})	
\end{equation}

\begin{prop}\label{prop_AGuh} Let $A_{G,u,h}$ be as in \eqref{eq_AGuh}. The map
\begin{eqnarray*}
z : A_{G,u,h} & \lra & \langle u\rangle \\
a & \mapsto & [a,h] := aha^{-1}h^{-1}
\end{eqnarray*}
is a group homomorphism which fits into an exact sequence
$$1\lra Z(G)\lra A_{G,u,h}\stackrel{z}{\lra}\langle u^{k_{u,h}}\rangle\lra 1$$
where $k_{u,h}$ is the smallest positive integer such that $(u,h)$ is conjugate to $(u,u^{k_{u,h}}h)$. In particular $A_{G,u,h}$ is an extension of a cyclic group of order $|u|/k_{u,h}$ by $Z(G)$. If $M_{u,h} = G$, then $A_{G,u,h} = Z(G)$ and $k_{u,h} = |u|$.
\end{prop}
\begin{remark} Let $(u,h)$ correspond to the cuspidal $G$-cover $\pi : C\rightarrow E$ (Theorem \ref{thm_delta_invariant}). It follows from the equivalence $\cC_E^{pc}\cong\Sets^{(\Pi,G)_\delta,\succ}$ that $A_{G,u,h}$ is precisely the vertical automorphism group $\Aut^v(\pi)$ (viewing $\pi$ in $\cAdm(G)$). Certainly $k_{u,h}$ always divides $|u|$, and accordingly $Z(G)$ is always a subgroup of $A_{G,u,h} \cong \Aut^v(\pi)$. We have equality if and only if the generating pairs $(u,h),(u,uh),(u,u^2h),\ldots,(u,u^{|u|-1}h)$ are all non-conjugate. Thus, equality fails if and only if there is some ``unexpected'' relation between these generating pairs, corresponding to an unexpected vertical automorphism that does not lie in $Z(G)$. In \S\ref{ss_automorphism_groups}, we will show that for $G = \SL_2(\bF_q)$ ($q\ge 5$), there are no unexpected automorphisms, so $k_{u,h} = |u|$ and $A_{G,u,h} = Z(G)$ for any generating pair of $\SL_2(\bF_q)$.
\end{remark}

\begin{proof}[Proof of Proposition \ref{prop_AGuh}] The description of $A_{G,u,h}$ follows from the exactness of the sequence, so it suffices to establish exactness. If $a,b\in A_{G,u,h}$, then
$$z(ab) = abhb^{-1}a^{-1}h^{-1} = az(b)ha^{-1}h^{-1} = z(b)aha^{-1}h^{-1} = z(b)z(a) = z(a)z(b)$$
so $z$ is a homomorphism. The kernel of $z$ must commute with both $h$ and $M_{u,h}$, but $G$ is generated by $M_{u,h}$ and $h$, so the kernel is precisely $Z(G)$. Since $A_{G,u,h}\subset C_G(M_{u,h})$, if $(u,h)$ is conjugate to $(u,u^{k_{u,h}}h)$, then there must be a $g\in G$ which centralizes $u$ and satisfies $\la{g}h := ghg^{-1} = u^{k_{u,h}}h$. But then
$$\la{g}(h^{-1}u^{-1}h) = \la{g}h^{-1}\la{g}u^{-1}\la{g}h = h^{-1}u^{-k_{u,h}}u^{-1}u^{k_{u,h}}h = h^{-1}u^{-1}h$$
so $g$ also centralizes $h^{-1}u^{-1}h$, so we have $g\in C_G(M_{u,h})$, hence $g\in A_{G,u,h}$. This shows that $z$ is surjective onto $\langle u^{k_{u,h}}\rangle$, so the sequence is exact. If $M_{u,h} = G$, then $C_G(M_{u,h}) = Z(G)$ and hence $A_{G,u,h}$ is both contained in and contains $Z(G)$, so it is equal to $Z(G)$. The exactness of the sequence then forces $k_{u,h} = |u|$.
\end{proof}

Let $\ev_1$ denote the ``evaluation at 1'' map
\begin{eqnarray*}
\ev_1 : \Aut_{\Sets^{(\Pi,G)_\delta,\succ}}(F_{u,h},\alpha_{u,h}) & \lra & G \\
\sigma & \mapsto & \sigma(1)
\end{eqnarray*}

\begin{thm}\label{thm_cuspidal_automorphisms} As usual we work over an algebraically closed field $k$ of characteristic 0. Let $t_0$ be a tangential base point at $0\in\bP^1$. Let $\iota$ be the unique automorphism of $\bP^1$ fixing $1$ and swapping $0,\infty$. Let $t_\infty := \iota(t_0)$, and let $\delta : t_0\leadsto t_\infty$ be a symmetric good path. The map $\ev_1$ induces an isomorphism of \emph{groups} (note that $\sigma(\tau(1)) = \sigma(1\cdot\tau(1)) = \sigma(1)\tau(1)$)
$$\ev_1 : \Aut_{\Sets^{(\Pi,G)_\delta,\succ}}(F_{u,h},\alpha_{u,h})\rightiso A_{G,u,h}$$
Let $\pi : C\rightarrow E$ be a cuspidal object of $\cAdm(G)(k)$ with $\Inv_\delta(\pi) = \ps{u,h}$. Recall that its Higman invariant is $[u^{-1},h^{-1}]$ (Proposition \ref{prop_cuspidal_higman_invariant}). Let $\Aut^v(\pi)$ denote its vertical automorphism group. This is precisely the group of $G$-equivariant automorphisms of $C$ inducing the identity on $E$. For a generating pair $(u,h)$ of $G$, let $k_{u,h}$ be the minimal positive integer such that $(u,h)$ is conjugate to $(u,u^{k_{u,h}}h)$. Then
\begin{itemize}
\item[(a)] Let $A_{G,u,h}$ be as in \eqref{eq_AGuh}. There is an isomorphism
$$\Aut^v(\pi)\rightiso A_{G,u,h}$$
Thus, the vertical automorphism group of $\pi$ is isomorphic to a subgroup of $G$ which is an extension of a cyclic group of order $|u|/k_{u,h}$ by $Z(G)$.
\item[(a')] The vertical automorphism groups of geometric points of $\cAdm(G)$ are all reduced to $Z(G)$ (equivalently, the map $\ol{\cM(G)}\rightarrow\ol{\cM(1)}$ is representable) if and only if $k_{u,h} = |u|$ for all generating pairs $(u,h)$ of $G$.
\item[(b)] If $\ps{u,h} \ne \ps{u^{-1},h^{-1}}$, then $\Aut_{\cAdm(G)(k)}(\pi) = \Aut^v(\pi)$. If $\ps{u,h} = \ps{u^{-1},h^{-1}}$, then $\Aut_{\cAdm(G)(k)}(\pi)$ is an extension of $\Aut(E)\cong\bZ/2\bZ$ by $\Aut^v(\pi)$.
\item[(c)] Every irreducible component of $C$ is Galois over $E$ with Galois group isomorphic to $M_{u,h} = \langle u,h^{-1}u^{-1}h\rangle$. The number of irreducible components of $C$ is $[G:M_{u,h}]$. In particular $C$ is irreducible if and only if $G$ is generated by $u,h^{-1}u^{-1}h$, in which case $\Aut^v(\pi) = Z(G)$.
\item[(d)] As usual let $O\in E$ be the origin. Let $x\in \pi^{-1}(O)$. Then an automorphism $\sigma\in\Aut^v(\pi)$ fixes $x$ if and only if it acts trivially on $\pi^{-1}(O)$, and there is an isomorphism
\begin{equation}\label{eq_vertical_automorphisms_of_RRD_1}
\{\sigma\in\Aut^v(\pi)\;|\; \text{$\sigma$ acts trivially on $\pi^{-1}(O)$}\} = \Stab_{\Aut^v(\pi)}(x) \rightiso A_{G,u,h}\cap \langle [u^{-1},h^{-1}]\rangle
\end{equation}
In particular, let $\pi : \cC\rightarrow\cE$ be the universal family over $\cAdm(G)$ and let $\cR_\pi$ be its reduced ramification divisor, then for any cuspidal geometric point $z\in\cR_\pi$ whose image in $\cAdm(G)$ has $\delta$-invariant $\ps{u,h}$, its vertical automorphism group is
\begin{equation}\label{eq_vertical_automorphisms_of_RRD_2}
\Aut^v(z) \cong A_{G,u,h}\cap\langle[u^{-1},h^{-1}]\rangle
\end{equation}

\end{itemize}
  \end{thm}
\begin{proof} Parts (a),(a'), and (c) follows immediately from Proposition \ref{prop_AGuh} and the discussion above, noting that the map $\cAdm(G)\rightarrow\ol{\cM(G)}$ removes $Z(G)$ from all automorphism groups of geometric points.








For (b), note that since automorphisms of $\pi$ in $\cAdm(G)$ are $G$-equivariant, they descend to automorphisms of $E \cong C/G$, so we have an exact sequence
$$1\lra \Aut^v(\pi)\lra\Aut_{\cAdm(G)(k)}(\pi)\lra\Aut(E)$$
Since $E$ is a nodal elliptic curve, $\Aut(E)$ is cyclic of order 2 so let $[-1]$ be the generator. Then the map $\Aut_{\cAdm(G)(k)}(\pi)\rightarrow\Aut(E)$ is surjective if and only if there is an isomorphism $[-1]^*\pi \rightarrow \pi$ inducing the identity on $E$. By Proposition \ref{prop_iota_pullback}, this happens if and only if $\ps{u,h} = \ps{u^{-1},h^{-1}}$.


It remains to prove (d). Since automorphisms of $\pi$ are $G$-equivariant and $G$-acts transitively on all fibers, if an automorphism fixes a point, then it must fix the entire fiber. Next we prove the isomorphism \eqref{eq_vertical_automorphisms_of_RRD_1}. Let $\nu : \bP^1\rightarrow E$ be a standard normalization, and let $\Xi_\nu(\pi) = (\pi' : C'\rightarrow \bP^1,\alpha_\pi)$ be the normalized-cover-with-gluing-data. Let $x'\in C'$ be the unique point lying over $x$, so $x'\in C'_1 := \pi'^{-1}(1)$. Let $\Aut_{\bP^1}(\pi')$ be the automorphism group of $(\pi',\alpha_\pi)\in\cC_{\bP^1}^\succ$, then by the equivalence $\Xi_\nu : \cC_E^{pc}\rightiso\cC_{\bP^1}^\succ$, it suffices to define an isomorphism
$$\Stab_{\Aut_{\bP^1}(\pi')}(x')\rightiso A_{G,u,h}\cap\langle u^{-1}h^{-1}uh\rangle$$

Let $t_1$ be a tangential base point at $1\in\bP^1$, and $\gamma_1\in\pi_1(\bP^*,t_1)$ the canonical generator of inertia. Since $\delta$ is a good path, for some path $\epsilon : t_0\leadsto t_1$, we have $(\gamma_\infty^\delta)^{-1}\gamma_0^{-1} = \gamma_1^\epsilon$. Let $\xi_1 : \langle\gamma_1\rangle\bs C_{t_1}'\rightiso C'_1$ be the canonical isomorphism given by the local picture near 1. Using the path $\epsilon$, we obtain a $G$-equivariant bijection 
$$\xi_1\circ\epsilon : \langle\gamma_1^\epsilon\rangle\bs C_{t_0}'\rightiso\langle\gamma_1\rangle\bs C_{t_1}'\rightiso C'_1$$
By Proposition \ref{prop_tbp}(b), these bijections define natural isomorphisms of functors $\cC_{\bP^1}^\succ\rightarrow\Sets$, and hence we can compute the action of $\Aut_{\bP^1}(\pi')$ on $C_1'$ via its action on $\langle\gamma_1^\epsilon\rangle\bs C_{t_0}'$. Passing to $\Sets^{(\Pi,G)_\delta,\succ}$, let $F_{\delta}^\succ(\pi') = (C'_{t_0},\alpha)\in\Sets^{(\Pi,G)_\delta,\succ}$, then the same computations as above shows that for any choice of $x_0\in C'_{t_0}$, $\Aut(C'_{t_0},\alpha)\cong A_{G,u,h}$, where $u = \varphi_{x_0}(\gamma_\infty^\delta)$, and $h\in H_{\delta,x_0}$. Since $\pi$ is balanced, we also have $\varphi_{x_0}(\gamma_0) = h^{-1}u^{-1}h$. Here, for $a\in A_{G,u,h}$, let $\sigma_a\in\Aut(C'_{t_0},\alpha)$ be the corresponding automorphism. Then $\sigma_a$ acts on $C_{t_0}'$ by $\sigma_a(x_0g) = x_0ag$ for any $g\in G$. Since $\gamma_1^\epsilon = (\gamma_\infty^\delta)^{-1}\gamma_0^{-1}$, $\varphi_{x_0}(\gamma_1^\epsilon) = u^{-1}h^{-1}uh$, and for any $g\in G$ we have
\begin{eqnarray*}
\sigma_a(\gamma_1^\epsilon x_0g) & = & \sigma_a(x_0\varphi_{x_0}(\gamma_1^\epsilon)g) \\
 & = & \sigma_a(x_0u^{-1}h^{-1}uhg) \\
 & = & x_0(au^{-1}h^{-1}uhg)
\end{eqnarray*}
so $\sigma_a$ fixes the coset $\langle\gamma_1^\epsilon\rangle x_0g$ if and only if $a\in\langle u^{-1}h^{-1}uh\rangle$. This establishes the isomorphism \eqref{eq_vertical_automorphisms_of_RRD_1}. The isomorphism \eqref{eq_vertical_automorphisms_of_RRD_2} follows immediately from the definition of $\cR_\pi$ in \S\ref{ss_ramification_divisor_of_universal_family}.
\end{proof}

\begin{remark}\label{remark_abelian_level_structures} If $G$ is abelian, then it follows from the discussion above that the vertical automorphism group of any cuspidal $G$-cover $\pi : C\rightarrow E$ is equal to $G$. In particular, the map $\ol{\cM(G)} := \cAdm(G)\fs Z(G)\rightarrow\ol{\cM(1)}$ is representable. For $n\ge 3$, it is known that $\cM(n) := \cM(\bZ/n\bZ\times\bZ/n\bZ)$ is representable. On the other hand, it follows from Theorem \ref{thm_cuspidal_automorphisms}(b) that for $n\ge 3$, the cuspidal objects of $\ol{\cM(n)} := \ol{\cM(\bZ/n\bZ\times\bZ/n\bZ)}$ also have no automorphisms, so $\ol{\cM(n)}$ is also representable (even as a stack over $\bZ[1/n]$), and hence must agree with the Deligne-Rapoport moduli stack of generalized elliptic curves with full level $n$ structures. This recovers \cite[Theorem 2.7]{DR72}.
\end{remark}

Using this, we are able to give a purely combinatorial statement of Theorem \ref{thm_congruence}.

\begin{thm}[Combinatorial congruence]\label{thm_combinatorial_congruence} Let $G$ be a finite group, let $\fc\in\Cl(G)$ be a conjugacy class, and let $c\in\fc$ be a representative. Let $A_{G,u,h}$ be as in \eqref{eq_AGuh}.
\begin{itemize}
	\item[(a)] Let $d' = d'_\fc := |C_G(\langle c\rangle)/\langle c\rangle|$.
	\item[(b)] Let $m' = m'_\fc$ be the least positive integer which kills $A_{G,u,h}\cap \langle [u^{-1},h^{-1}]\rangle$ for any generating pair $(u,h)$ of $G$ with $[u,h]\in\fc$.
\end{itemize}
Then for any component $\cX\subset\cAdm(G)_\fc$ with coarse scheme $X$, the map to the $j$-line $X\rightarrow\ol{M(1)}$ satisfies
$$\deg(X\rightarrow\ol{M(1)})\equiv 0\mod \frac{|c|}{\gcd(|c|,m'd')}.$$
Combinatorially speaking, let $F_2 = \langle a,b\rangle$ be a free group of rank 2 and let $\gamma_{-I}\in\Aut^+(F_2)$ be the automorphism $(a,b)\mapsto (a^{-1},b^{-1})$. Note that the image of $\gamma_{-I}$ in $\Out^+(F_2)\cong\SL_2(\bZ)$ is central. Let $\Epi^\ext(F_2,G)_\fc\subset\Epi^\ext(F_2,G)$ be the subset represented by surjections $\varphi : F_2\twoheadrightarrow G$ satisfying $\varphi([a,b])\in\fc$, then every $\Out^+(F_2)$-orbit on $\Epi^\ext(F_2,G)_\fc/\langle\gamma_{-I}\rangle$ has cardinality divisible by $\frac{|c|}{\gcd(|c|,m'd')}$.\footnote{Viewing $F_2$ as the fundamental group of a punctured torus and $a,b$ as a positively oriented basis, then in the setup of Situation \ref{situation_galois_theory}, the Higman invariant is technically given by $\varphi([b,a])$ (as opposed to $\varphi([a,b])$). However the automorphism of $F_2$ given by $(a,b)\mapsto (b,a)$ induces an isomorphism on the corresponding $\Out^+(F_2)$-orbits, and so the sizes of the orbits are the same.}
\end{thm}

\begin{remark} Recall that if $G$ is not 2-generated, then $\cAdm(G)$ is empty (Corollary \ref{cor_empty_if_not_2_gen}), so the theorem is only nontrivial for finite 2-generated groups $G$.
\end{remark}

\begin{proof} Let $\cC\stackrel{\pi}{\rightarrow}\cE\rightarrow\cAdm(G)_\fc$ denote the universal family. Let $\cR_\pi\subset\cC$ denote the reduced ramification divisor. Suppose $u,h$ is a generating pair of $G$ with $[u,h]\in\fc$. By Proposition \ref{prop_smooth_pointed_vertical_automorphisms} the vertical automorphism groups of geometric points of $\cR_\pi$ which lie over $\cM(1)$ are equal to
$$Z(G)\cap \langle c\rangle = Z(G)\cap \langle [u^{-1},h^{-1}]\rangle\subset A_{G,u,h}\cap \langle [u^{-1},h^{-1}]\rangle$$
so they are all killed by $m'$. By Theorem \ref{thm_cuspidal_automorphisms}(d), the vertical automorphism group of any geometric point of $\cR_\pi$ is killed by $m'$, so the full automorphism groups are killed by $12m'$. By Proposition \ref{prop_stacky_RRD}, the components of $\cR_\pi$ have degree over $\cX$ dividing $d'$, so Theorem \ref{thm_congruence} gives the congruence on degrees. For the congruence on $\Out^+(F_2)$-orbits, first by Proposition \ref{prop_compactification}(b), $\cAdm^0(G)$ is a gerbe over $\cM(G)$, so $X$ is the smooth compactification of a unique component $M\subset M(G)_\fc$. Thus, $\deg(X\rightarrow \ol{M(1)}) = \deg(M\rightarrow M(1))$. On the other hand, the fiber of the finite \'{e}tale map $\cM(G)_\fc\rightarrow\cM(1)$ above a geometric point $x_E\in\cM(1)$ is in bijection with $\Epi^\ext(F_2,G)_\fc$, and for a component $\cM\subset\cM(G)_\fc$, the fiber of $\cM\rightarrow\cM(1)$ is in bijection with an $\Out^+(F_2)$-orbit on $\Epi^\ext(F_2,G)_\fc$. By Theorem \ref{thm_basic_properties}\ref{part_fibers} (also see \cite[Proposition 2.1.2(1)]{BBCL20}), any unramified geometric fiber of $M\rightarrow M(1)$ is in bijection with the quotient of a geometric fiber of $\cM\rightarrow\cM(1)$ by $\gamma_{-I}$. Thus we obtain the congruence on $\Out^+(F_2)$-orbits from the congruence on degrees.
\end{proof}

\begin{remark} In Theorem \ref{thm_combinatorial_congruence}, the full automorphism groups of geometric points of $\cR_\pi$ can also be expressed totally combinatorially, but the payoff is limited to at most an improvement of the resulting congruence by a factor of 12, so we do not treat it here.
\end{remark}

\subsection{Remarks on the obstructions $m,d$}\label{ss_remarks_on_md}

Let $\cX\subset\cAdm(G)$ be a component classifying covers with ramification index $e$. Our main congruence (Theorem \ref{thm_congruence}) gives
$$\deg(X\rightarrow \ol{M(1)})\equiv 0\mod\frac{12e}{\gcd(12e,m_\cX d_\cX)}$$
which is possibly diluted by the integers $d_\cX$ and $m_\cX$. In the best possible case, we can obtain $\deg(X\rightarrow \ol{M(1)})\equiv 0\mod 12e$. However in general this is too much to hope for since for example the congruence modular curve $\Adm(\bZ/2\bZ) = X_1(2)$ for elliptic curves with ``$\Gamma_1(2)$-structures'' corresponds to ramification index $e = 1$ but it has degree 3 over $\ol{M(1)}$; thus, the conditions on $m_\cX, d_\cX$ are necessary. A more reasonable question is to ask if we can always obtain the congruence
\begin{equation}\label{eq_sober_congruence}
\deg(X\rightarrow \ol{M(1)})\equiv 0\mod e	
\end{equation}

It turns out this also does not hold in general. As an explicit example, let $D_{2k} := \bZ/k\bZ\rtimes\mu_2$ denote the dihedral group of order $2k$, where $k$ is odd. In this case, its commutator subgroup is $\bZ/k\bZ$, its abelianization is $\mu_2$, and its center is trivial. By \cite[Theorem 4.2.2]{Chen18}, $\Adm(D_{2k})$ is isomorphic to $\frac{\phi(k)}{2}$ copies of the curve $\Adm(\bZ/2\bZ)$. In particular, every component of $\Adm(D_{2k})$ has degree 3 over $\ol{M(1)}$, whereas one can check that the commutator of any generating pair of $D_{2k}$ has order $k$ \cite[\S4.2]{Chen18}, so every component of $\Adm(D_{2k})$ classifies $D_{2k}$-covers with ramification index $e = k$. Letting $k\to\infty$ we see that in the general case, even the congruence \eqref{eq_sober_congruence} can fail arbitrarily badly. From the combinatorial congruence (Theorem \ref{thm_combinatorial_congruence}), one can check that for $G = D_{2k}$, $d' = 2$, so $d_\cX\mid 2$ and hence $d_\cX$ does not affect the congruence; Thus the culprit must be $m_\cX$. Indeed, let $(u,h)$ be the generating pair
$$(u,h) = ((1,1),(0,-1))\in D_{2k} = \bZ/k\bZ\rtimes\mu_2$$
In this case, we find that $(u,h)$ is conjugate to $(u,uh)$, so in this case we have $A_{D_{2k},u,h} = [D_{2k},D_{2k}]$ has order $k$, so the vertical automorphism group of the cusp in $\cAdm(D_{2k})$ with $\delta$-invariant $\ps{u,h}$ has order $k$, which implies that $k\mid m_\cX$. Thus $m_\cX$ is totally responsible for the failure of the congruence; in other words, $m_\cX$ is a real obstruction. For lack of examples, we do not know if $d_\cX$ is a real obstruction. 


\subsection{Congruences for nonabelian finite simple groups}\label{ss_congruences_for_NAFSG}

If $f : G_1\rightarrow G_2$ is a surjection of 2-generated finite groups, then by Theorem \ref{thm_basic_properties}\ref{part_functoriality}, $f$ induces a finite etale surjection of moduli spaces $\cM(G_1)\rightarrow \cM(G_2)$ (corresponding to the surjection of sets $f_* : \Epi^\ext(F_2,G_1)\ra\Epi^\ext(F_2,G_2)$). Thus, in terms of understanding the sizes of $\Out^+(F_2)$-orbits, a crucial special case is when $G$ is a finite simple group. If $G$ is cyclic, then $\Out^+(F_2)$ acts transitively on $\Epi^\ext(F_2,G)$ with size equal to the index of the congruence subgroup $\Gamma_1(n)\le\SL_2(\bZ)$ (c.f. \cite[\S3.9]{DS06}). Moreover, note that in this case our results say nothing - when $G$ is abelian, $|c| = 1$ and hence the congruence in Theorem \ref{thm_combinatorial_congruence} is trivial.


In this section, we will show that Theorem \ref{thm_combinatorial_congruence} often yields nontrivial congruences when $G$ is a nonabelian finite simple group. For this we wish to control the integers $d_\fc'$ and $m_\fc'$ of Theorem \ref{thm_combinatorial_congruence}. More generally we also obtain nontrivial congruences if we have some control on the proper normal subgroups of $G$. We begin with some group theoretic lemmas:


\begin{lemma}\label{lemma_intersection_of_conjugate_cyclic_subgroups} Let $G$ be a group and $U\le G$ a cyclic subgroup of finite index. Then for any $g\in G$, $g$ normalizes $U\cap gUg^{-1}$.	
\end{lemma}
\begin{proof} Let $u\in U$ and $x\in U\cap gUg^{-1}$ be generators. Then for some integer $i,j$, we have $x = u^i = (gug^{-1})^j = gu^jg^{-1}$. Since $x\in U$, $gxg^{-1}\in gUg^{-1}$, so it would suffice to show that $gxg^{-1}\in U$. We have $gxg^{-1} = gu^ig^{-1}$, so it suffices to show that $\langle gu^ig^{-1}\rangle = \langle gu^jg^{-1}\rangle$. Note that $U\cap gUg^{-1}$ has the same (finite) index in both $U$ and $gUg^{-1}$. We have
$$[gUg^{-1} : \langle gu^ig^{-1}\rangle] = [gUg^{-1} : \langle gxg^{-1}\rangle] = [U : \langle x\rangle] = [gUg^{-1} : \langle x\rangle] = [gUg^{-1} : \langle gu^jg^{-1}\rangle]$$
Since $gUg^{-1}$ is cyclic, this implies $\langle gu^ig^{-1}\rangle = \langle gu^jg^{-1}\rangle$ as desired.
\end{proof}

\begin{lemma}\label{lemma_abelian_subgroup} Let $G$ be a nonabelian finite group. In the notation of Theorem \ref{thm_combinatorial_congruence}, let $\ell$ is a prime and let $k\ge j\ge 0$ be integers such that
\begin{itemize}
	\item[(a)] $k := \ord_\ell(m_\fc')$ and\footnote{Note that if $G$ cannot be generated by two elements, then $m'_\fc$ is the least positive integer which satisfies a trivial condition, hence $m'_\fc = 1$. In particular this lemma implicitly assumes that $G$ can be generated by two elements.}
	\item[(b)] $G$ does not contain a proper normal subgroup of order divisible by $\ell^{j+1}$.
\end{itemize}
Then $G$ must contain a subgroup isomorphic to $\bZ/\ell^k\bZ\times\bZ/\ell^{k-j}\bZ$. In particular, we must have $\ell^{2k-j}\mid|G|$.
\end{lemma}
\begin{proof} By Theorem \ref{thm_combinatorial_congruence}, there must exist a generating pair $(u,h)$ of $G$ with $[u,h]\in\fc$ such that there exists an element $z\in A_{G,u,h}\cap\langle [u^{-1},h^{-1}]\rangle$ of order $\ell^k$. By Proposition \ref{prop_AGuh}, $|A_{G,u,h}|$ divides $|u|\cdot|Z(G)|$. Since $G$ is nonabelian, (b) implies that $\ell^{k-j}$ divides $|u|$. Moreover $z$ centralizes $M_{u,h} = \langle u,h^{-1}u^{-1}h\rangle$. Let $U := \langle u\rangle$ and $U^h := \langle h^{-1}uh\rangle$. By Lemma \ref{lemma_intersection_of_conjugate_cyclic_subgroups}, $h$ normalizes $U\cap U^h$. Since $U$ is cyclic, $u$ normalizes $U\cap U^h$, so $G = \langle u,h\rangle$ also normalizes $U\cap U^h$. Since $G$ is not cyclic, (b) implies that $\ell^{j+1}$ does not divide $|U\cap U^h|$. Thus $|\langle z\rangle\cap U|$ and $|\langle z\rangle\cap U^h|$ cannot both be divisible by $\ell^{j+1}$. Note that since $z$ centralizes $M_{u,h}$, both $\langle z,u\rangle$ and $\langle z,u^h\rangle$ are abelian subgroups of $G$. For simplicity assume $\ell^{j+1}$ does not divide $|\langle z\rangle\cap U|$. Choose $z'\in\langle z\rangle, u'\in\langle u\rangle$ such that $|z'| = \ell^k$ and $|u'| = \ell^{k-j}$, then $\ell^{j+1}$ does not divide $|\langle z'\rangle\cap\langle u'\rangle|$, so $\langle z'^j\rangle\cap\langle u'\rangle$ is trivial, so $\langle z'^j,u'\rangle\cong (\bZ/\ell^{k-j}\bZ)^2$. It follows that $\langle z',u'\rangle\cong\bZ/\ell^k\bZ\times\bZ/\ell^{k-j}\bZ$.
\end{proof}

The $\ell$-adic valuation $\ord_\ell(d'_\fc)$ can be easily bounded if we know the relative valuations of $\ord_\ell(|c|)$ and $\ord_\ell(|G|)$. One use of Lemma \ref{lemma_abelian_subgroup} is to do the same for $\ord_\ell(m'_\fc)$. A theorem of Vdovin gives another method to control $d_\fc',m_\fc'$ when $G$ is a nonabelian finite simple group other than $\PSL_2(\bF_q)$:

\begin{thm}[Vdovin {\cite[Theorem A]{Vdo99}}]\label{thm_vdovin} If $G$ is a nonabelian finite simple group which is not isomorphic to $\PSL_2(\bF_q)$ (for any prime power $q$), then for any abelian subgroup $A\le G$, we have $|A| < |G|^{1/3}$.	
\end{thm}

To summarize, when $G$ is nonabelian, $d'_\fc$ and $m'_\fc$ can be controlled as follows.

\begin{cor}\label{cor_vdovin} Let $G$ be a finite group. Let $c\in G$ and let $\fc$ be its conjugacy class. Let $\cX\subset\cAdm(G)_\fc$ be a connected component with coarse scheme $X$. For a prime $\ell$, let $r := \ord_\ell(|c|)$. Let $d_\fc',m_\fc'$ be as in Theorem \ref{thm_combinatorial_congruence}.
\begin{itemize}
\item[(a)] Suppose $G$ is nonabelian. Write $\ord_\ell(|G|) = r+s$, and let $j\ge 0$ be an integer such that $G$ does not contain any proper normal subgroups of order divisible by $\ell^{j+1}$. Then
\begin{itemize}
\item[$\bullet$] $\ord_\ell(d_\fc')\le s$.
\item[$\bullet$] $\ord_\ell(m_\fc')\le \floor{\frac{r+s+j}{2}}$.
\end{itemize}
\item[(b)] Suppose $G$ is nonabelian and simple.
\begin{itemize}
\item[$\bullet$] If $\ell^{r+1}\ge |G|^{1/3}$ and $G$ is not isomorphic to $\PSL_2(\bF_q)$ for any $q$, then $\ord_\ell(d_\fc') = 0$.
\item[$\bullet$] If $\ell^{k+1}\ge |G|^{1/3}$ ($k\in\bZ$) and $G$ is not isomorphic to $\PSL_2(\bF_q)$ for any $q$, then $\ord_\ell(m_\fc')\le \floor{\frac{k}{2}}$.
\end{itemize}
\end{itemize}
\end{cor}
\begin{proof} Part (a) follows from Lemma \ref{lemma_abelian_subgroup}. For (b), if $\ord_\ell(d'_\fc) > 0$ then there exists an $\ell$-power torsion element $z\in C_G(\langle c\rangle)$ which does not lie in $\langle c\rangle$. Then $\langle z,c\rangle$ is abelian of order $\ell^{r+1}$, which is forbidden by Vdovin's theorem. Similarly, (c) immediately follows from Lemma \ref{lemma_abelian_subgroup} and Vdovin's theorem.
\end{proof}

This allows us to guarantee a nontrivial congruence in a number of settings. For example, we have

\begin{cor}\label{cor_vdovin_2} Let $G$ be a finite group. Let $c\in G$ and let $\fc$ be its conjugacy class. Let $\cX\subset\cAdm(G)_\fc$ be a connected component with coarse scheme $X$. For a prime $\ell$, let $r := \ord_\ell(|c|)$. Then we have
\begin{itemize}
\item[(a)] Write $\ord_\ell(|G|) = r+s$, and let $j\ge 0$ be an integer such that $G$ does not contain any proper normal subgroup of order divisible by $\ell^{j+1}$. Then
$$\deg(X\rightarrow\ol{M(1)})\equiv 0\mod \ell^{\ceil{\frac{r-3s-j}{2}}}$$
Combinatorially speaking, using the notation of Theorem \ref{thm_combinatorial_congruence}, every $\Out^+(F_2)$-orbit on $\Epi^\ext(F_2,G)_\fc/\langle\gamma_{-I}\rangle$ has cardinality divisible by $\ell^{\ceil{\frac{r-3s-j}{2}}}$.
\item[(b)] Suppose $G$ is nonabelian and simple. If $\ell^{r+1}\ge |G|^{1/3}$ and $G$ is not isomorphic to $\PSL_2(\bF_q)$ for any $q$, then
$$\deg(X\rightarrow\ol{M(1)})\equiv 0\mod \ell^{\ceil{\frac{r}{2}}}$$
Combinatorially speaking, using the notation of Theorem \ref{thm_combinatorial_congruence}, every $\Out^+(F_2)$-orbit on $\Epi^\ext(F_2,G)_\fc/\langle\gamma_{-I}\rangle$ has cardinality divisible by $\ell^{\ceil{\frac{r}{2}}}$.
\end{itemize}
\end{cor}
\begin{proof} We note that if $G$ is abelian then $\cAdm(G)_\fc$ is empty for any nontrivial conjugacy class $\fc$, so (a) holds trivially in this case (see Remark \ref{remark_topological_higman}). If $G$ is nonabelian, (a) and (b) follow immediately from Corollary \ref{cor_vdovin} and Theorem \ref{thm_combinatorial_congruence}.
\end{proof}

\begin{remark}\label{remark_combinatorial_improvement} We make a few observations.
\begin{itemize}
\item If $\ell\ge 3$ and $\fc$ is a class of $\SL_2(\bF_\ell)$ of order divisible by $\ell$, then since $\ell^2\nmid |\SL_2(\bF_\ell)|$, we may apply Corollary \ref{cor_vdovin_2}(a) with $r = 1$, $s = j = 0$ to obtain a congruence mod $\ell$. We will later recover this fact from a more general analysis in \S\ref{ss_automorphism_groups} below, where we will even show that when $G = \SL_2(\bF_q)$ and $q\ge 3$, $A_{\SL_2(\bF_q),u,h} = Z(\SL_2(\bF_q))$ for any generating pair of $\SL_2(\bF_q)$. Moreover we will explicitly compute $d_\fc'$ and $m_\fc'$ for any Higman invariant $\fc$.
\item Lemma \ref{lemma_abelian_subgroup} can be slightly sharpened by weakening condition (b). The proof only requires that $\ell^{j+1}$ does not divide $|Z(G)|$ or the order of any proper normal \emph{cyclic} subgroup. Accordingly part (a) of the above corollaries can also be sharpened.
\item Part (b) of the above corollaries use Vdovin's theorem and hence can be slightly sharpened by noting that the proof of Lemma \ref{lemma_abelian_subgroup} often gives an abelian subgroup which is larger than $(\bZ/\ell^k\bZ)^2$.
\end{itemize}
\end{remark}




\section{Applications to Markoff triples and the geometry of $\cM(\SL_2(\bF_q))$}\label{section_applications}

In this section we will specialize our discussions above to the case of admissible $\SL_2(\bF_q)$-covers of elliptic curves. The key new features we use here is the well-developed theory of the character variety for $\SL_2$-representations of a free group of rank 2, which is explained in \S\ref{ss_character_variety}, and the work of Bourgain, Gamburd, and Sarnak as explained in \S\ref{ss_bgs_conjecture}. In \S\ref{ss_automorphism_groups}, we use the theory of this character variety and the results obtained in \S\ref{ss_cuspidal_automorphisms} to compute the vertical automorphism groups of $\cAdm(\SL_2(\bF_q))$: we will prove that these automorphism groups are reduced to the center of $\SL_2(\bF_q)$. In \S\ref{ss_congruences_for_SL2}, we put everything together to obtain congruences for many Markoff-type equations. The application to the Markoff equation $x^2+y^2+z^2-3xyz=0$ is explicitly described in \S\ref{ss_bgs_conjecture}. In \S\ref{ss_genus_formulas} we give a genus formula for the components of $M(\SL_2(\bF_p))_{-2}$.


We work universally over a base scheme $\bS$ over which $|G|$ is invertible.

\subsection{The trace invariant}\label{ss_trace_invariant}
If $\pi : C\rightarrow E$ is an admissible $G$-cover of a 1-generalized elliptic curve $E$ over an algebraically closed field $k$, then relative to a compatible system of roots of unity $\{\zeta_n\}_{n\ge 1}$, we have defined its Higman invariant in \S\ref{ss_higman_invariant}, which is a conjugacy class in $G$.

\begin{defn} Let $q = p^r$ be a prime power. Let $G \le \GL_2(\bF_q)$ be a subgroup. Then the trace invariant of a geometric point of $\cAdm(G)_\Qbar$ is the trace of its Higman invariant relative to $\{\exp(\frac{2\pi i}{n})\}_{n\ge 1}$.
\end{defn}

As in \S\ref{ss_higman_invariant}, we also have decompositions
\begin{equation}\label{eq_trace_decomposition}
\ol{\cM(\SL_2(\bF_q))}_\Qbar = \bigsqcup_{t\in\bF_q}\ol{\cM(\SL_2(\bF_q))}_t \quad\text{and}\quad \cAdm(\SL_2(\bF_q))_\Qbar = \bigsqcup_{t\in\bF_q}\cAdm(\SL_2(\bF_q))_t
\end{equation}
into open and closed substacks corresponding to objects with trace invariant $t$.

\begin{defn} We say that a conjugacy class $\fc\in\Cl(\SL_2(\bF_q))$ (resp. an element $t\in\bF_q$) is $q$-admissible if it is the conjugacy class (resp. trace) of the commutator of a generating pair of $\SL_2(\bF_q)$.
\end{defn}

The $q$-admissible classes (resp. traces) were classified by McCullough and Wanderley \cite{MW11}.

\begin{prop}\label{prop_Higman_vs_trace} We state the following results in pairs, beginning with a group-theoretic statement, followed by a geometric consequence.
\begin{itemize}
\item[(a)] If $q = 2, 4, 8$ or $q\ge 13$, then the $q$-admissible traces are $\bF_q - \{2\}$. If $q = 3,9,11$, then the $q$-admissible traces are $\bF_q - \{1,2\}$. If $q = 5$, the $q$-admissible traces are $\in\bF_q - \{0,2,4\}$. If $q = 7$, the $q$-admissible traces are $\in\bF_q - \{0,1,2\}$.

\item[(a')] For any $t\in\bF_q$, $\ol{\cM(\SL_2(\bF_q))}_t$ is nonempty if and only if $t$ is $q$-admissible.

\item[(b)] If $q$ is even, then for any $t\in\bF_q - \{\pm 2\} = \bF_q - \{0\}$, there exists a unique conjugacy class in $\SL_2(\bF_q)$ with trace $t$. If $t = \pm 2 = 0$, then there exist precisely two classes with trace $t$, represented by $\spmatrix{\pm1}{0}{0}{\pm1}$ and $\spmatrix{\pm1}{1}{0}{\pm1}$, neither of which are $q$-admissible. In particular, every $q$-admissible trace $t\in\bF_q$ is the trace of a \emph{unique} $q$-admissible class $\fc$, and for any $a\in\bF_q$, any noncentral element of $\SL_2(\bF_q)$ with trace $a$ has the same order.

\item[(b')] If $q$ is even, then for any $q$-admissible class $\fc$ with trace $t$, the map
$$\cAdm(\SL_2(\bF_q))_\fc\lra \cAdm(\SL_2(\bF_q))_t$$
is an isomorphism.

\item[(c)] If $q$ is odd, then for any $t\in\bF_q - \{\pm 2\}$, there exists a unique conjugacy class in $\SL_2(\bF_q)$ with trace $t$. If $t = \pm 2$, then there exist precisely three classes with trace $t$, represented by $\spmatrix{\pm1}{0}{0}{\pm1},\spmatrix{\pm1}{1}{0}{\pm1}$, and $\spmatrix{\pm1}{a}{0}{\pm1}$, where $a\in\bF_q^\times$ is not a square. In particular, every $q$-admissible $t$ other than $-2$ is the trace of a \emph{unique} $q$-admissible class, and for any $a\in\bF_q$, any noncentral element of $\SL_2(\bF_q)$ with trace $a$ has the same order. The $q$-admissible classes of trace $-2$ are represented by $\spmatrix{-1}{1}{0}{-1},\spmatrix{-1}{a}{0}{-1}$ where $a\in\bF_q^\times$ is a nonsquare. If $u\in\bF_{q^2}^\times - \bF_q^\times$ with $u^2\in\bF_q$ and $\gamma := \spmatrix{u}{0}{0}{u^{-1}}$, then conjugation by $\gamma$ switches the two $q$-admissible classes of trace $-2$.

\item[(c')] If $q$ is odd, then for any $q$-admissible class $\fc$ with trace $t$, the map
$$\cAdm(\SL_2(\bF_q))_\fc\lra \cAdm(\SL_2(\bF_q))_t$$
is an isomorphism except for $t = -2$. If $t = -2$, let $\fc_1,\fc_2$ denote the two $q$-admissible classes of trace $-2$. Then conjugation by $\gamma$ induces an isomorphism $\cAdm(\SL_2(\bF_q))_{\fc_1}\rightiso\cAdm(\SL_2(\bF_q))_{\fc_2}$, and
$$\cAdm(\SL_2(\bF_q))_{-2} = \bigsqcup_{\fc\in\{\fc_1,\fc_2\}}\cAdm(\SL_2(\bF_q))_\fc$$
\end{itemize}
\end{prop}
\begin{proof} Part (a) is precisely \cite[Theorem 2.1]{MW11}. Parts (b) and (c) follow from \cite[\S5]{MW13}. Parts (a') and (b') follow immediately from (a) and (b). For part (c'), assume $q$ odd, and let $\fc_1,\fc_2$ denote the two $q$-admissible classes of trace invariant -2. Let $i_\gamma\in\Aut(\SL_2(\bF_q))$ be given by $g\mapsto \gamma g\gamma^{-1}$, then $i_\gamma$ defines a map
$$\cAdm(\SL_2(\bF_q))\lra\cAdm(\SL_2(\bF_q))$$
sending an admissible $G$-cover $\pi : C\rightarrow E$ to the same map $\pi : C\rightarrow E$, but with $G$-action defined via the isomorphism $i_\gamma$. Since $i_\gamma$ has finite order, this is an equivalence. Since $u^2\in\bF_q^\times$ is a nonsquare, $i_\gamma$ switches the two $q$-admissible classes of trace invariant $-2$, and hence it restricts to an equivalence
$$\cAdm(\SL_2(\bF_q))_{\fc_1}\rightiso\cAdm(\SL_2(\bF_q))_{\fc_2}.$$
\end{proof}

Let $I := \spmatrix{1}{0}{0}{1}$. By Proposition \ref{prop_Higman_vs_trace}(b,c), for any $a\in\bF_q$, with the exception of the conjugacy classes of $\pm I$, the conjugacy classes of $\SL_2(\bF_q)$ with trace $a$ all have the same order. Thus, we may define:

\begin{defn}\label{def_na} Let $q = p^r$ be a prime power. For $a\in \bF_q$, let $n_q(a)$ denote the order of any matrix $A\in\SL_2(\ol{\bF_q}) - \{\pm I\}$ with trace $a$.
\end{defn}

\begin{prop}\label{prop_na} Let $q = p^r$ be a prime power. For $a\in\bF_q$, there is a set $\{\omega,\omega^{-1}\}\subset\bF_{q^2}$ which is uniquely determined by the property that $a = \omega + \omega^{-1}$. The integer $n_q(a)$ satisfies:

$$n_q(a) = \left\{\begin{array}{ll}
	2 & \text{if $q$ is even and $a = \pm 2 = 0$} \\
	p & \text{if $q$ is odd and $a = 2$} \\
	2p & \text{if $q$ is odd and $a = -2$} \\
	|\omega| & \text{if $a\ne \pm2$}
\end{array}\right.$$
\end{prop}
\begin{proof} If $\omega+\omega^{-1} = a$, then $\{\omega,\omega^{-1}\}$ are precisely the roots of the polynomial $x^2 - ax + 1$, so they are determined by $a$. The description of $n_q(a)$ follows from Proposition \ref{prop_Higman_vs_trace}, noting that if $a\ne \pm 2$, then $\omega\ne\omega^{-1}$ so any matrix with trace $a$ is diagonalizable over $\bF_{q^2}$.
\end{proof}

\subsection{$\SL_2(\bF_q)$-structures as $\bF_q$-points of a character variety}\label{ss_character_variety}


Let $G$ be a finite group. Let $E,\Pi,a,b,x_E,\alpha$ be as in Situation \ref{situation_galois_theory}. Then the fiber of $\cM(G)_\Qbar/\cM(1)_\Qbar$ over $x_E$ is canonically in bijection with
$$\Epi^\ext(\Pi,G)$$
Under this bijection, the orbits of the $\Out^+(\Pi)$-action correspond to the components of $\cM(G)_\Qbar$. In this section we will show that when $G = \SL_2(\bF_q)$, $\Epi^\ext(\Pi,\SL_2(\bF_q))$ is closely related to the $\bF_q$ points of a certain character variety. This correspondence is central in our geometric approach to the conjecture of Bourgain, Gamburd, and Sarnak \ref{conj_bgs_intro}, and moreover it allows us to calculate of the automorphism groups of geometric points of $\cM(\SL_2(\bF_q))$, which forms the key input to Theorem \ref{thm_congruence}.



Using the oriented basis $a,b\in\Pi$, there is a bijection
\begin{equation}\label{eq_generating_pairs}
\begin{array}{rcl}
\Epi(\Pi,G) & \rightiso & \{(A,B)\;|\;\text{$A,B$ generate $G$}\} \\
\varphi & \mapsto & (\varphi(a),\varphi(b))
\end{array}
\end{equation}

For the purposes of understanding the $\Out^+(\Pi)$-orbits, it suffices to consider the orbits of $\Aut^+(\Pi)$ on the set of generating pairs of $G$. An elementary Nielsen move applied to a generating pair $(A,B)$ of $G$ sends $(A,B)$ to any of:
\begin{equation}\label{eq_nielsen_move}
(A,AB), (B,A), \;\text{or}\; (A,B^{-1}).
\end{equation}


Two generating pairs of $G$ are said to be Nielsen equivalent if they are related by a sequence of elementary Nielsen moves. One can check that the elementary Nielsen moves, applied to the generators $(a,b)$ of $\Pi$, define a set of generators of $\Aut(\Pi)$ \cite[\S3]{MKS04}, and that two generating pairs of $G$ are Nielsen equivalent if and only if they lie in the same $\Aut(\Pi)$-orbit via the bijection in \eqref{eq_generating_pairs}. Questions about Nielsen equivalence were traditionally studied from the point of view of combinatorial group theory \cite[\S2]{Pak00}. However, if $R$ is a ring, $S$ an $R$-algebra, and $G$ is obtained as the $S$-points of an algebraic group $\cG/R$, then the set $\Hom(\Pi,G)$ can be viewed as the $S$-points of the functor $\Hom(\Pi,\cG) : \Sch/R\rightarrow\Sets$ sending $T\mapsto \Hom(\Pi,\cG(T))$. Since $\Pi$ is free on $a,b$, this functor is representable by the scheme $\cG\times\cG$, and hence $\Epi(\Pi,G)$ inherits an algebraic structure as a subset of the $S$-points of $\cG\times\cG$. Similarly $\Epi^\ext(\Pi,G)$ in many cases can be approximated by a subset of the $S$-points of
$$\Hom(\Pi,\cG)/\cG = (\cG\times\cG)/\cG$$
where the action of $\cG$ is by simultaneous conjugation, and the quotient is taken in a suitable sense (see Theorem \ref{thm_moduli_interpretation} and Remark \ref{remark_nakamoto} below). Moreover, since $\Aut(\Pi)$ is generated by automorphisms of the form \eqref{eq_nielsen_move}, $\Aut(\Pi)$ also acts on $\Hom(\Pi,\cG)$ and $\Hom(\Pi,\cG)/\cG$ (on the right!) as automorphisms of schemes. When $\cG = \SL_2$, these observations together with the work of Brumfiel-Hilden \cite{BH95} and Nakamoto \cite{Nak00} will allow us to connect the sets $\Epi^\ext(\Pi,\SL_2(\bF_q))$ to the Markoff equation (Theorem \ref{thm_moduli_interpretation}). In \S\ref{ss_automorphism_groups}, we will also use this relationship to compute the vertical automorphism groups of $\cM(\SL_2(\bF_q))$.


When $\cG = \SL_{2,R}$, it is more natural to consider the quotient by $\GL_{2,R}$ acting by conjugation.

\begin{thm}\label{thm_character_variety} Let $R$ be any ring. Let $A[\Pi]$ denote affine ring of $\Hom(\Pi,\SL_{2,R}) \cong \SL_{2,R}\times\SL_{2,R}$. Let $X_{\SL_{2,R}} := \Spec A[\Pi]^{\GL_{2,R}} = \Hom(\Pi,\SL_{2,R})\git\GL_{2,R}$. Then $X_{\SL_{2,R}}$ is a \emph{universal categorical quotient} in the sense of geometric invariant theory (here $\GL_{2,R}$ acts by conjugation). If $R = \bZ$ we will simply write $X_{\SL_2} := X_{\SL_{2,\bZ}}$. Let $\Tr$ denote the map
$$\Tr : \Hom(\Pi,\SL_{2,R})\longrightarrow \bA^3_R$$
defined on $A$-valued points\footnote{Since every scheme is covered by affine opens, to define a morphism of schemes, it suffices to define it on $T$-valued points for all affine schemes $T$.} for various $R$-algebras $A$ by sending $\varphi : \Pi\rightarrow\SL_2(A)$ to $\tr\varphi(a),\tr\varphi(b)$, and $\tr\varphi(ab)$ respectively. Then $\Tr$ induces an isomorphism (which we also denote by $\Tr$)
\begin{equation}\label{eq_trace_map}
\Tr : X_{\SL_{2,R}}\rightiso\bA^3_R	
\end{equation}
\end{thm}
\begin{proof} When $R = \bC$ this amounts to a classical result of Fricke and Vogt (see \cite{Gold03,Gold09}). For the general case, let $A,B,C$ denote the functions $\tr\varphi(a),\tr\varphi(b),\tr\varphi(ab)\in A[\Pi]$ respectively. By \cite[Proposition 3.5 and Proposition 9.1(ii)]{BH95}, we have $A[\Pi]^{\GL_2(R)} = R[A,B,C]$, but we also have the inclusions
$$R[A,B,C]\subset A[\Pi]^{\GL_{2,R}}\subset A[\Pi]^{\GL_2(R)} = R[A,B,C]$$
so all the inclusions must be equalities, as desired. The map $\Hom(\Pi,\SL_2)\rightarrow X_{\SL_2,R}$ is clearly a categorical quotient. That it's universal follows from the fact that the isomorphism $\Tr$ holds over any ring $R$.
\end{proof}

\begin{defn} With notation as in Theorem \ref{thm_character_variety}, $\Hom(\Pi,\SL_{2,R})$ the called the \emph{representation variety} for $\SL_{2,R}$-representations of $\Pi$, and $X_{\SL_{2,R}}$ is the \emph{character variety} for $\SL_{2,R}$-representations of $\Pi$. Since $X_{\SL_{2,R}}$ is a universal categorical quotient, for most applications it will suffice to work with $X_{\SL_2} = X_{\SL_{2,\bZ}}$.	
\end{defn}

\begin{defn} For a ring $A$ and $\varphi\in\Hom(\Pi,\SL_{2}(A))$, we will call $\Tr(\varphi)\in \bA^3(A) = A^3$ the ``trace coordinates of $\varphi$''. The \emph{trace invariant} of $\varphi$ is the element $\tr\varphi([b,a]) = \tr\varphi([a,b])\in A$.	
\end{defn}

There is a natural right action of $\Aut(\Pi)$ on the functor $\Hom(\Pi,\SL_{2})$ commuting with the conjugation action of $\GL_2$. By the Yoneda lemma, it follows that $\Aut(\Pi)$ acts on trace coordinates by \emph{polynomials}, and the action descends to a right action of $\Aut(\Pi)$ on the character variety $X_{\SL_{2}} \cong \bA^3$. Moreover, we will see that this action preserves the trace invariant (Proposition \ref{prop_trace_fibration} below). The key calculation is the following:

\begin{lemma} Let $R$ be any ring, and let $\Pi$ be a free group with generators $a,b$. Let $\varphi : \Pi\rightarrow\SL_2(R)$ be a homomorphism. Let $(A,B,C) := (\tr \varphi(a),\tr\varphi(b),\tr\varphi(ab))$. Then the trace invariant of $\varphi$ can be computed as follows:
$$\tr(\varphi([a,b])) = A^2 + B^2 + C^2 - ABC - 2\in R$$
\end{lemma}
\begin{proof} See \cite[Proposition A.1*.10(iii)]{BH95}, and also \cite[\S2.2]{Gold09}.
\end{proof}

\begin{prop}\label{prop_trace_fibration} Let $T : \bA^3\rightarrow\bA^1$ (over $\bZ$) be given by
$$T(x,y,z) = x^2 + y^2 + z^2 - xyz - 2$$
Let $\tau : X_{\SL_2}\rightarrow\bA^1$ be given by $\varphi \mapsto \tr\varphi([a,b])$. Then the following diagram is commutative
\[\begin{tikzcd}
X_{\SL_2}\ar[r,"\Tr"]\ar[rd,"\tau"'] & \bA^3\ar[d,"T"] \\
 & \bA^1
\end{tikzcd}\]
where $\Tr$ is the isomorphism \eqref{eq_trace_map}. The right action of $\Aut(\Pi)$ on $X_{\SL_2}$ and its induced action on $\bA^3$ preserve the fibers of $\tau$ and $T$. Viewing the induced action on $\bA^3$ as a left action, $\Tr$ defines an injective \emph{anti-homomorphism}

$$\Tr_* : \Aut(\Pi)\hookrightarrow\Aut(\bA^3)$$
which can be described as follows. The group $\Aut(\Pi)$ is generated by three elements $r,s,t$ (defined below), and their images in $\Aut(\bA^3)$ under the anti-homomorphism $\Tr_*$ are given as follows
$$\begin{array}{rcl}
r : (a,b) & \mapsto & (a^{-1},b)\\
s : (a,b) & \mapsto & (b,a) \\
t : (a,b) & \mapsto & (a^{-1},ab)
\end{array}\quad\stackrel{\Tr_*}{\lra}\quad
\begin{array}{rcl}
R_3 : (x,y,z) & \mapsto & (x,y,xy-z) \\
\tau_{12} : (x,y,z) & \mapsto & (y,x,z) \\
\tau_{23} : (x,y,z) & \mapsto & (x,z,y)
\end{array}$$
\end{prop}
\begin{proof} We begin with the final statement. That $r,s,t$ generate $\Aut(\Pi)$ follows from \cite[\S3]{MKS04} (also see \cite[\S2]{MW13}). That $(s,t)\mapsto (\tau_{12},\tau_{23})$ is easy to check. To see that $r\mapsto R_3$, one should use the Fricke identity
\begin{equation}\label{eq_fricke_identity}
\tr(AB) + \tr(A^{-1}B) = \tr(A)\tr(B)	
\end{equation}
valid for $A,B\in\SL_2(R)$ for any ring $R$. To verify this identity, one uses the Cayley-Hamilton theorem to deduce that $A+A^{-1} = \Tr(A)$. Multiplying both sides by $B$ and taking traces yields the identity \eqref{eq_fricke_identity}.


The commutativity of the diagram follows from the lemma. Using the explicit form of the $\Aut(\Pi)$-action, the fact that $\Aut(\Pi)$ preserves the fibers amounts to the observation that for ring $R$ and any homomorphism $\varphi : \Pi\rightarrow\SL_2(R)$, the following identities hold:
$$T(\tau_{12}(x,y,z)) = T(\tau_{23}(x,y,z)) = T(R_3(x,y,z)) = T(x,y,z)$$
where $\tau_{12},\tau_{23},R_3$ are as in Proposition \ref{prop_trace_fibration}. This is easily verified by hand.
\end{proof}

\begin{remark} The fact that $\Aut^+(\Pi)$ preserves the fibers of $\tau$ also follows from the local constancy of the trace invariant (see Proposition \ref{prop_Higman_invariant_is_locally_constant}). The fact that $\Aut(\Pi)$ preserves the fibers of $\tau$ can be explained as follows. If $\Pi = \pi_1^\tp(E^\circ(\bC),x_0)$ where $E$ is an elliptic curve over $\bC$, then the fact that $\Aut(\Pi)$ preserves the fibers of $\tau$ amounts to the fact that any automorphism of $\Pi$ is represented by a self homeomorphism of $E^\circ(\bC)$ (this is the Dehn-Nielsen-Baer theorem \cite[\S8]{FM11}). Any such self-homeomorphism must send a neighborhood of the puncture to another neighborhood of the puncture, and hence must either preserve the free homotopy class of a small loop winding once around the puncture, or reverse its orientation. Thus, in light of Remark \ref{remark_topological_higman}, if the points of $X_{\SL_{2,k}}$ are viewed as covers of $E^\circ(\bC)$, then such homeomorphisms must either preserve the Higman invariant of the cover, or send it to its inverse, and hence it must preserve the trace invariant, but $\tau$ is precisely the map that sends the monodromy representation of a cover to its trace invariant.
\end{remark}


\begin{defn} Let $T$ be a scheme, $n\ge 1$ an integer, and $G$ a group. A representation $\varphi : G\rightarrow\GL_n(T) = \GL_n(\Gamma(T,\cO_T))$ is \emph{absolutely irreducible} if the induced algebra homomorphism $\Gamma(T,\cO_T)[G]\rightarrow M_n(\Gamma(T,\cO_T))$ is surjective. A subgroup $G\subset\GL_n(T)$ is \emph{absolutely irreducible} if the inclusion $G\hookrightarrow\GL_n(T)$ is an absolutely irreducible representation.
\end{defn}

When $T = \Spec k$ with $k$ a field, then this notion of absolute irreducibility is the same as the nonexistence of nontrivial $G$-invariant subspaces of $\ol{k}^2$, where $\ol{k}$ denotes the algebraic closure \cite[\S XVII, Corollary 3.4]{Lang02}.

\begin{defn} Let $R$ be any ring. Let $\Hom(\Pi,\SL_{2,R})^\ai\subset \Hom(\Pi,\SL_{2,R})$ denote the subfunctor corresponding to the absolutely irreducible representations. By \cite[\S3]{Nak00}, this is represented by an open subscheme of $\Hom(\Pi,\SL_{2,R})$. Accordingly, let
$$X_{\SL_{2,R}}^\ai := \Hom(\Pi,\SL_{2,R})^\ai\git\GL_2$$
As usual if $R = \bZ$ then we will simply write $X_{\SL_2}^\ai := X_{\SL_{2,\bZ}}^\ai$.
\end{defn}

\begin{lemma}\label{lemma_absolutely_irreducible} Let $R$ be a ring. Let $\varphi : \Pi\rightarrow \SL_2(R)$ be a representation. Let $(A,B,C) := (\tr\varphi(a), \tr\varphi(b), \tr\varphi(ab))$. Then $\varphi$ is absolutely irreducible if and only if
$$A^2 + B^2 + C^2 - ABC - 4 \in R^\times$$
In particular, we have $X_{\SL_{2}}^\ai = X_{\SL_{2}} - \tau^{-1}(2)$.
\end{lemma}
\begin{proof} The first statement is \cite[Proposition 4.1]{BH95}. The second statement follows from the description of $\tau$ in Proposition \ref{prop_trace_fibration}.
\end{proof}

We have the following ``moduli interpretation'' for $X_{\SL_{2,R}}^\ai(\bF_q)$.

\begin{thm}\label{thm_moduli_interpretation} Let $\Pi$ be a free group on the generators $a,b$. Let $q = p^r$ be a prime power. Let $\GL_2(\bF_q)$ act on the set $\Hom(\Pi,\SL_2(\bF_q))$ by conjugation. The map
\begin{eqnarray*}
\Tr : \Hom (\Pi , \SL_2(\bF_q)) & \longrightarrow & \bF_q^3 \\
\varphi & \mapsto & (\tr\varphi(a),\tr\varphi(b),\tr\varphi(ab)).
\end{eqnarray*}
is \emph{surjective}. Moreover, the following maps induced by $\Tr$ are \emph{bijections}
\begin{equation}\label{eq_moduli_interpretation}
\begin{tikzcd}
\Hom(\Pi,\SL_2(\bF_q))^\ai/\GL_2(\bF_q)\;\ar[r,"\alpha"] & X_{\SL_2}^\ai(\bF_q)\ar[r,"\Tr"] & \bF_q^3 - T^{-1}(2)
\end{tikzcd}
\end{equation}
\end{thm}
The surjectivity of $\Tr$ is \cite[Theorem 1]{Mac69}. We will give two proofs of the bijectivity statement. The first is elementary and amounts to a classical result of Macbeath \cite{Mac69} which involves explicit calculations in $\SL_2(\ol{\bF_p})$. The second is significantly more general and uses results of Nakamoto \cite{Nak00} on character varieties, which in particular shows that the map $\Hom(\Pi,\SL_{2,R})^\ai\rightarrow X_{\SL_{2,R}}^\ai$ is a universal geometric quotient. This argument uses the triviality of the Brauer group of a finite field, and also easily generalizes to character varieties for absolutely irreducible representations of arbitrary groups in arbitrary degree.
\begin{proof}[Proof 1] The second map $\Tr : X^\ai_{\SL_{2}}(\bF_q)\longrightarrow\bF_q^3 - T^{-1}(2)$ is already a bijection by Proposition \ref{prop_trace_fibration}. Thus it suffices to show that the composition is bijective. This amounts to a classical result of Macbeath. For a homomorphism $\varphi : \Pi\rightarrow\SL_2(\bF_q)$, let $(A,B,C) = (\tr\varphi(a), \tr\varphi(b), \tr\varphi(ab))$ and let 
$$Q_\varphi(x,y,z) = Q_{A,B,C}(x,y,z) := x^2 + y^2 + z^2 + Ayz + Bxz + Cxy$$
We say that $\varphi$ is nonsingular if the projective conic defined by $Q_\varphi$ is smooth (equivalently geometrically integral). Similarly given any triple $(A,B,C)\in\bF_q^3$, we say that it is nonsingular if the form $Q_{A,B,C}$ defines a smooth projective conic. We claim that $\varphi$ is nonsingular if and only if it is absolutely irreducible. Indeed, the discriminant of $Q_\varphi = Q_{A,B,C}$ is
$$\disc(Q_\varphi) = \disc(Q_{A,B,C}) = -(A^2 + B^2 + C^2 - ABC - 4)$$
and the associated conic fails to be smooth if and only if $\disc(Q_\varphi) = 0$\footnote{The definition of the discriminant and its association with the smoothness of the associated conic is classical in characteristic $p \ne 2$, but in characteristic 2 some care is needed. A modern treatment describing the discriminant for an arbitrary projective hypersurface over arbitrary fields is given in Demazure \cite{Dem12}. The definition of discriminant is \cite[Definition 4]{Dem12}, and its association with smoothness is \cite[Proposition 12]{Dem12}. For a ternary quadratic form $q(x,y,z)$ over a field $k$, to calculate its discriminant one should first lift $q$ to a quadratic form $\tilde{q}$ in characteristic 0 (e.g., over a Cohen ring of $k$), compute its discriminant there, defined as one half of the determinant of the Gram matrix of the bilinear form $b(v,w) = \tilde{q}(v+w) - \tilde{q}(v) - \tilde{q}(w)$, and then take its image in $k$.}, or equivalently, when $T(A,B,C) = 2$, where $T$ is as in Proposition \ref{prop_trace_fibration}. By Lemma \ref{lemma_absolutely_irreducible}, this is equivalent to $\varphi$ not being absolutely irreducible, so the nonsingular representations are precisely the absolutely irreducible representations, and the nonsingular triples are precisely those in $\bF_q^3 - T^{-1}(2)$. 


In \cite[Theorem 3]{Mac69}, Macbeath shows that every nonsingular triple in $\bF_q^3$ is the trace of a nonsingular $\varphi$, and conversely any two nonsingular representations $\varphi,\varphi'$ are conjugate by some $P\in\SL_2(\ol{\bF_q})$. Since $\varphi,\varphi'$ are absolutely irreducible, the associated algebra homomorphisms $\what{\varphi},\what{\varphi'} : \bF_q[\Pi]\rightarrow M_2(\bF_q)$ are surjective, and hence such a $P$ must normalize $M_2(\bF_q)$, so it must normalize $\SL_2(\bF_q)$. Finally, it follows from the description of the normalizer $N_{\GL_2(\ol{\bF_q})}(\SL_2(\bF_q))$ in Proposition \ref{prop_normalizer} that the action factors through the action of $\GL_2(\bF_q)$, so Macbeath's result establishes the desired bijectivity.
\end{proof}

\begin{proof}[Proof 2] Since $\Tr$ is a bijection (\ref{thm_character_variety}, \ref{prop_trace_fibration}), it remains to prove the bijectivity of $\alpha$. It follows from \cite[Corollary 2.13, Corollary 6.8]{Nak00} that for any ring $R$ the map
$$\xi : \Hom(\Pi,\SL_{2,R})^\ai\rightarrow X_{\SL_{2,R}}^\ai$$
is a \emph{universal geometric quotient} by $\GL_{2,R}$ in the sense of geometric invariant theory \cite[Definition 0.6]{MFK94}. Thus for algebraically closed fields $\Omega$ over $R$, $\xi$ induces a bijection
$$\Hom(\Pi,\SL_2(\Omega))^\ai/\GL_2(\Omega)\rightiso X_{\SL_{2,R}}^\ai(\Omega)$$
Setting $\Omega = \ol{\bF_q}$, Proposition \ref{prop_normalizer} implies (arguing as in the first proof) that two absolutely irreducible representations $\varphi,\varphi' : \Pi\rightarrow\SL_2(\bF_q)$, are conjugate in $\GL_2(\ol{\bF_q})$ if and only if they are conjugate in $\GL_2(\bF_q)$. This implies the injectivity of $\alpha$. Since the center of $\GL_{2,R}$ acts trivially, $\xi$ is also a universal geometric quotient by $\PGL_{2,R}$. Since our representations are absolutely irreducible, the $\PGL_{2,R}$-action is \emph{free} \cite[Corollary 6.5]{Nak00}, so $\xi$ is moreover a principal $\PGL_{2,R}$-bundle. In particular, for any field $k$ and map $x : \Spec k\rightarrow  X_{\SL_{2,R}}^\ai$, the restriction $x^*\xi : x^*\Hom(\Pi,\SL_{2,R})^\ai\rightarrow \Spec k$ is a principal $\PGL_{2,k}$ bundle. When $k = \bF_q$, by Lang's theorem \cite[Theorem 5.12.19]{Poo17}, principal $\PGL_{2,\bF_q}$-bundles over $\Spec \bF_q$ are trivial, and hence $\alpha$ is surjective.\footnote{Alternatively, one could argue surjectivity as follows. Let $k^s$ denote a separable closure of $k$. The exact sequence
$$1\lra\bG_m(k^s)\lra\GL_2(k^s)\lra\PGL_2(k^s)\lra 1$$
induces a longer exact sequence of Galois cohomology sets \cite[I,\S5.7, Proposition 43]{SerreGC}
$$\cdots\lra H^1(k,\GL_2(k^s))\lra H^1(k,\PGL_2(k^s))\lra H^2(k,\bG_m(k^s))\lra\cdots$$
Since $\PGL_2$ is smooth, every principal $\PGL_2$-bundle over $k^s$ is trivial, and hence $H^1(k,\PGL_2(k^s))$ classifies principal $\PGL_2$-bundles over $k$ \cite[\S5.12.4]{Poo17}. On the other hand, the first term of the sequence vanishes by Hilbert's theorem 90 \cite[Proposition 1.3.15]{Poo17}, and the last is the Brauer group, which vanishes for finite fields (or any field of cohomological dimension $\le 1$). In particular if $k = \bF_q$ then $x^*\xi$ is a trivial $\PGL_{2,k}$-bundle, and hence $\alpha$ is surjective.}
\end{proof}

\begin{remark}\label{remark_nakamoto} The methods used in the second proof also hold for character varieties of absolutely irreducible representations of arbitrary groups in arbitrary dimension. Precisely, if $\Pi$ temporarily denotes an \emph{arbitrary} group, $n\ge 1$ an integer, and $R$ any ring, then by \cite[Corollary 6.8]{Nak00} if $X_{\Pi,n}^\ai$ denotes the $\GL_{n,R}$-quotient of the representation variety of absolutely irreducible representations $\Hom(\Pi,\GL_{n,R})^\ai$, then the quotient map $\Hom(\Pi,\GL_{n,R})^\ai\rightarrow X_{\Pi,n}^\ai$ is a universal geometric quotient and a principal $\PGL_{n,R}$-bundle. Since the Brauer group of a finite field is trivial, the argument used in the second proof also shows that we also obtain a bijection
\begin{eqnarray*}
\Hom(\Pi,\GL_n(\bF_q))/\GL_2(\bF_q) & \rightiso & X_{\GL_n}(\bF_q)
\end{eqnarray*}
where $X_{\GL_n} := \Hom(\Pi,\GL_n)\git\GL_n$. Thus, combined with Theorem \ref{thm_congruence}, restricting to the subsets with image of a particular type, this bijection can potentially be used to establish congruences on the $\Aut(\Pi)$-orbits on $\bF_q$ points of more general character varieties.
\end{remark}


By Theorem \ref{thm_moduli_interpretation}, elements of $X_{\SL_{2}}(\bF_q)$ do not quite correspond to $\SL_2(\bF_q)$-structures, but rather $\GL_2(\bF_q)$-equivalence classes of such structures. 

\begin{defn}\label{def_Dq} For a prime power $q$, let $i : \GL_2(\bF_q)\rightarrow\Aut(\SL_2(\bF_q))$ be the map which sends $A\in \GL_2(\bF_q)$ to the corresponding action by conjugation on $\SL_2(\bF_q)$. Let $D(q)$ be the image of $\GL_2(\bF_q)$ in $\Out(\SL_2(\bF_q))$. We call $D(q)$ the group of ``diagonal'' outer automorphisms of $\SL_2(\bF_q)$.


For a general subgroup $G\le\SL_2(\bF_q)$, let $D(q,G)$ denote the image of $i : N_{\GL_2(\bF_q)}(G)\rightarrow\Out(G)$ where $i$ is given by the conjugation action. Thus if $G = \SL_2(\bF_q)$, then $D(q,G) = D(q)$.
\end{defn}

\begin{prop}\label{prop_Dq} Let $q = p^r$ be a prime power. If $q$ is even, $D(q)$ is trivial. If $q$ is odd, $D(q)$ has order 2, and the nontrivial element is represented by any matrix $A\in\GL_2(\bF_q)$ with nonsquare determinant. When $q = p$ is a prime, $D(p) = \Out(\SL_2(\bF_p))$.
\end{prop}

\begin{proof} The fact that $D(p) = \Out(\SL_2(\bF_p))$ follows from \cite[Theorem 30]{Stein16}, noting that for $\SL_2(\bF_p)$, there are no graph or field automorphisms. The rest is clear.
\end{proof}


\begin{defn}\label{def_dihedral_type} For a prime power $q$, we say that a subgroup $G\le\SL_2(\bF_q)$ is \emph{of dihedral type} if its image in $\PSL_2(\bF_q) := \SL_2(\bF_q)/\{\pm I\}$ is either trivial or isomorphic to the dihedral group of order $2n$ for some $n\ge 1$. For a representation $\varphi : \Pi\rightarrow\SL_2(\bF_q)$, we say that it is of \emph{dihedral type} if its image is a subgroup of dihedral type.
\end{defn}

For the results in the next section we will eventually need to exclude the representations of dihedral type. Next we make some preliminary observations about such representations.


\begin{defn}\label{def_Xq} Let $R$ be a ring, and let $t\in R$ be an element corresponding to a map $t : \Spec R\rightarrow\bA^1$. Let
$$X_{\SL_{2},t} := \tau^{-1}(t) := X_{\SL_2}\times_{\tau,\bA^1,t}\Spec R$$
be the fiber corresponding to representations with $\tr\varphi([a,b]) = t$. Via $\Tr : X_{\SL_{2,R}}\rightiso\bA_R^3$, $X_{\SL_2,t} = \tau^{-1}(t)$ is identified with the affine surface $T^{-1}(t)\subset\bA^3_R$
$$T^{-1}(t) : x^2 + y^2 + z^2 - xyz - 2 = t$$
In particular, when $t = -2\in\bZ$, this is the Markoff surface $\bX\subset\bA^3_\bZ$. By Theorem \ref{thm_moduli_interpretation}, for $t\in\bF_q$, every $\bF_q$ point of $T^{-1}(t)$ is the image under $\Tr$ of a representation $\varphi : \Pi\rightarrow\SL_2(\bF_q)$.
\begin{itemize}
	\item Let $X_t(q)$ denote the set of $\bF_q$-points of the affine surface $T^{-1}(t) \subset\bA^3_{\bF_q}$;
	\item Let $X^*_t(q)\subset X_t(q)$ denote the subset corresponding to absolutely irreducible representations which are not of dihedral type (Definition \ref{def_dihedral_type});
	\item Let $X^\circ_t(q)\subset X_t(q)$ denote the subset corresponding to \emph{surjective} representations $\Pi\twoheadrightarrow\SL_2(\bF_q)$;
	\item Let $X(q) := \bigsqcup_{t\in\bF_q} X_t(q)$, let $X^*(q) := \bigsqcup_{t\in\bF_q} X_t^*(q)$, and let $X^\circ(q) := \bigsqcup_{t\in\bF_q} X_t^\circ(q)$.
\end{itemize}
\end{defn}

\begin{prop}\label{prop_dihedral_representations} The following are true.
\begin{itemize}
\item[(a)] For all prime powers $q$, $X^*_2(q)$ and $X^\circ_2(q)$ are empty.
\item[(b)] For $q\ge 3$, $X^\circ(q)\subset X^*(q)$.
\item[(c)] For $q = 2$, $X^*(2)$ is empty whereas $X^\circ(2)$ is not.
\item[(d)] For any prime power $q$, if $t\ne 2\in\bF_q$, then
$$X^*_t(q) = \{(x,y,z)\in X_t(q) \;|\;\text{At least two of $\{x,y,z\}$ are nonzero in $\bF_q$}\}$$
\end{itemize}

\end{prop}
\begin{proof} Part (a) follows immediately from Lemma \ref{lemma_absolutely_irreducible}. Part (b) follows from the fact that $\PSL_2(\bF_q)$ is not dihedral for $q\ge 3$. For $q = 2$, $X^*(q)$ is empty since every subgroup of $\SL_2(\bF_2) = \PSL_2(\bF_2)$ is either dihedral or cyclic. On the other hand, there exist surjections $\Pi\twoheadrightarrow\SL_2(\bF_2)$, so $X^\circ(2)$ is nonempty. This proves (c).


Finally we prove (d). Let $\varphi : \Pi\rightarrow\SL_2(\bF_q)$ be an absolutely irreducible representation with $A := \varphi(a), B := \varphi(b)$. Recall that $\varphi$ is absolutely irreducible if and only if $\tr\varphi([a,b]) = 2$ (Lemma \ref{lemma_absolutely_irreducible}). Thus, $X^*_2(q)$ is empty, and it remains to show that $\varphi$ is of dihedral type if and only if two of $\tr A,\tr B,\tr AB$ are equal to 0. Note that a group is dihedral if and only if it is generated by two elements $g,h$ with $g^2 = h^2 = 1$, and that a matrix $A\in\SL_2(\bF_q)$ maps to an element of order 2 in $\PSL_2(\bF_q)$ if and only if $\tr(A) = 0$. Thus if two of $\tr A,\tr B, \tr AB$ are equal to zero, then $\varphi$ must be of dihedral type. Conversely, if $\varphi$ is of dihedral type, let $\ol{A},\ol{B}\in\PSL_2(\bF_q)$ be the images of $A,B$, then $\ol{A},\ol{B}$ must satisfy $\ol{A}^2 = \ol{B}^2 = 1$, and up to switching $A,B$, there are three cases:
\begin{itemize}
\item $|\ol{A}| = |\ol{B}| = 1$. Then $\varphi(\Pi) = \langle A,B\rangle\le Z(\SL_2(\bF_q)$, so $\varphi$ is not absolutely irreducible, so this case cannot occur.
\item $|\ol{A}| = 1$ and $|\ol{B}| = 2$. Then $A = \pm I$ and $\tr(B) = \tr(AB) = 0$ In this case $\Tr(\varphi) = (\pm 2,0,0)$ (though this case is actually forbidden since $\langle A,B\rangle$ is cyclic so $\varphi$ cannot be absolutely irreducible in this case).
\item $|\ol{A}| = |\ol{B}| = 2$. Then $\tr(A) = \tr(B) = 0$. In this case $\Tr(\varphi) = (0,0,\tr(AB))$
\end{itemize}

Thus a representation of dihedral type must have at least two of $\tr\varphi(a),\tr\varphi(b),\tr\varphi(ab)$ equal to 0, as desired.	
\end{proof}

For us, an important special case is when $q = p$ and $t = -2$,

\begin{prop}\label{prop_surjective_representation} The following are true.
\begin{itemize}
\item[(a)] For all primes $p\ge 3$, $X^*_{-2}(p) = X^\circ_{-2}(p)$ consists of the $\bF_p$-points of the surface
$$\bX : x^2 + y^2 + z^2 - xyz = 0$$
other than $(0,0,0)$. In other words, for $p\ge 3$, we have $X^*_{-2}(p) = X^\circ_{-2}(p) = X_{-2}(p) - \{(0,0,0)\} = \bX^*(p)$.
\item[(b)] For $p = 2$, $X^\circ_{-2}(2) = X^*_{-2}(2)$ is empty but
$$X_{-2}(2) - \{(0,0,0)\} = \{(1,1,0),(1,0,1),(0,1,1),(1,1,1)\}$$
\end{itemize}
\end{prop}
\begin{proof} Note that in $X_{-2}(p)$, if two of the coordinates are zero, then the third must also be zero. The fact that $X^\circ_{-2}(p) = X_{-2}(p) - \{(0,0,0)\}$ for $p\ge 3$ follows from the analysis of \cite[\S11]{MW13}. This also implies $X^*_{-2}(p) = X_{-2}(p) - \{(0,0,0)\}$ for $p\ge 3$ by Proposition \ref{prop_dihedral_representations}(b). For $p = 2$, $-2 = 2$, so the statement in this case follows from Proposition \ref{prop_dihedral_representations}(a).
\end{proof}

\begin{remark} For general prime powers $q$ and traces $t\in\bF_q$, the discussion in \cite[\S11]{MW13} gives a complete (albeit more complicated) description of the subsets $X^\circ_t(q)\subset X(q)$.
\end{remark}

\begin{defn}\label{def_trace_invariant_of_subgroup} Given a subgroup $G\le\SL_2(\bF_q)$, we may speak of the trace of elements $g\in G$. For any $t\in\bF_q$, let $\cAdm(G)_t\subset\cAdm(G)_\Qbar$ denote the open and closed substack classifying $G$-covers whose Higman invariant has trace $t$.	 We have analogous notions of $\Adm(G)_t,\cM(G)_t,M(G)_t$.
\end{defn}

Proposition \ref{prop_trace_fibration} and Theorem \ref{thm_moduli_interpretation} allows us to relate the $\Aut(\Pi)$-action on $X^*(q)$ to the geometry of the moduli stacks $\cM(G)_\Qbar/D(q,G)$ which we summarize here. 

\begin{prop}\label{prop_monodromy_formulas} Let $p\ge 3$ be a prime. Let $E,\Pi,a,b,x_E$ be as in Situation \ref{situation_galois_theory}. Let $G\le\SL_2(\bF_q)$ be the image of an absolutely irreducible representation $\varphi : \Pi\rightarrow\SL_2(\bF_q)$, and let $D(q,G)$ be as in Definition \ref{def_Dq}. Let
$$\ff : \cM(G)_\Qbar/D(q,G)\lra \cM(1)$$
be the forgetful map. We have a commutative diagram
\begin{equation}\label{eq_correspondence}
\begin{tikzcd}
\ff^{-1}(x_E) = \text{$\{(G|D(q,G))$-structures on $E\}$}\ar[rrd,bend right = 18,"\text{trace invariant}"']\ar[r,"\cong"] & \Epi^\ext(\Pi,G)/D(q,G)\ar[rd,"{\tr\varphi([a,b])}"']\ar[r,hookrightarrow] & \;\; X^*(q)\subset\bF_q^3\ar[d,"x^2+y^2+z^2-xyz-2"] \\
& & \;\;\bF_q - \{2\}
\end{tikzcd}
\end{equation}
where the map ``trace invariant'' is the trace invariant of any $G$-structure in the $D(q,G)$-orbit. The automorphisms $\gamma_0,\gamma_{1728},\gamma_\infty,\gamma_{-I}\in\Aut^+(\Pi)$ and their induced actions on $X(q) := \sqcup_{t\in\bF_q}X_t(q) = \bF_q^3$ are given as follows:

$$\begin{array}{rcl}
\gamma_0 : (a,b) & \mapsto & (ab^{-1},a) \\
\gamma_{1728} : (a,b) & \mapsto & (b^{-1},a) \\
\gamma_\infty : (a,b) & \mapsto & (a,ab) \\
\gamma_{-I} : (a,b) & \mapsto & (a^{-1},b^{-1})
\end{array}\quad\stackrel{\Tr_*}{\longrightarrow}\quad
\begin{array}{rcl}
(x,y,z) & \mapsto & (xy-z,x,x^2y-xz-y) \\
(x,y,z) & \mapsto & (y,x,xy-z) \\
\rot_1 : (x,y,z) & \mapsto & (x,z,xz-y) \\
(x,y,z) & \mapsto & (x,y,z)
\end{array}$$

By Galois theory, to every $\Aut^+(\Pi)$-orbit on $X(q)$ is associated a connected component of $\cM(G)_\Qbar/D(q,G)$ for some $G$ as above. If $P\in X(q)$ corresponds to a geometric point $x_P\in\cM(G)_\Qbar/D(q,G)$, then the connected component $\cM(x_P)\subset\cM(G)_\Qbar/D(q,G)$ containing $P$ has degree over $\cM(1)_\Qbar$ equal to the size of the $\Aut^+(\Pi)$-orbit of $P$. Let $M(x_P)$ denote the coarse scheme, let $\ol{M(x_P)}$ denote its smooth compactification and let
$$f : \ol{M(x_P})\rightarrow\ol{M(1)}_\Qbar$$
denote the forgetful map. If we identify $\ol{M(1)}_\Qbar$ with the projective line with coordinate $j$, then $f$ is a branched covering of smooth proper curves over $\Qbar$, \'{e}tale over the complement of the points $j = 0,1728,\infty$. If $j(E)\ne 0,1728$, then the bijection in \eqref{eq_correspondence} induces a bijection
$$f^{-1}(x_E) \rightiso \left(\Aut^+(\Pi)\cdot P\right)/\langle\gamma_{-I}\rangle = \Aut^+(\Pi)\cdot P\quad\text{(since $\gamma_{-I}$ acts trivially on $X(q)$)}$$
For $j = 0,1728,\infty$, there are bijections
$$f^{-1}(0) = \left(\Aut^+(\Pi)\cdot P\right)/\langle\gamma_0\rangle\qquad f^{-1}(1728) = \left(\Aut^+(\Pi)\cdot P\right)/\langle\gamma_{1728}\rangle\qquad f^{-1}(\infty) = \left(\Aut^+(\Pi)\cdot P\right)/\langle\gamma_\infty\rangle$$
such that the ramification index of a given point in $f^{-1}(0)$ (resp. $f^{-1}(1728),f^{-1}(\infty)$) is equal to the size of the corresponding orbit under $\gamma_0$ (resp. $\gamma_{1728},\gamma_\infty$) in $\left(\Aut^+(\Pi)\cdot P\right)$.
\end{prop}

\begin{proof} The properties of the diagram \eqref{eq_correspondence} follows from Theorem \ref{thm_moduli_interpretation}, Proposition \ref{prop_trace_fibration}, and the discussion in Situation \ref{situation_galois_theory}. The formulas for the automorphisms of $X(q) := \bigsqcup_{t\in\bF_q}X_t(q) = \bF_q^3$ induced by $\gamma_0,\gamma_{1728},\gamma_\infty$ can be verified using Proposition \ref{prop_trace_fibration}, noting that
$$\gamma_0 = s\circ r\circ s\circ t\circ r\circ s\qquad \gamma_{1728} = s\circ r\qquad\gamma_\infty = t\circ r\qquad \gamma_{-I} = \gamma_{1728}^2$$
Here we must remember that $\Tr_* : \Aut(\Pi)\rightarrow\Aut(X(q))$ is an \emph{anti-homomorphism}. Since $\gamma_{-I} = \gamma_{1728}^2$ acts trivially on $X^\circ(p)$, the statement about ramification indices follows from \cite[Proposition 2.2.3]{BBCL20}. 
\end{proof}

\begin{remark} The automorphisms $\gamma_0,\gamma_{1728},\gamma_\infty,\gamma_{-I}$ are perhaps more familiar in terms of the corresponding matrices they determine in $\SL_2(\bZ)$. Namely, under the map $\Pi\rightarrow\bZ^2$ sending $a,b$ to the canonical basis of $\bZ^2$, we find that $\gamma_0,\gamma_{1728},\gamma_\infty,\gamma_{-I}$ correspond to the matrices
$$\spmatrix{1}{1}{-1}{0},\qquad\spmatrix{0}{1}{-1}{0},\qquad\spmatrix{1}{-1}{0}{1},\qquad\spmatrix{-1}{0}{0}{-1}$$
respectively. Viewing the matrices as fractional linear transformations of the upper half plane $\cH$, the $j$-function takes the values $0,1728,\infty$ respectively on their sets of fixed points. Finally $\gamma_{-I}$ corresponds to $\spmatrix{-1}{0}{0}{-1}$ and acts trivially on $\cH$.	
\end{remark}

\begin{remark} By Proposition \ref{prop_surjective_representation}(a), $X^*_{-2}(p) = X^\circ_{-2}(p)$ (for $p\ge 3$) is precisely the set of nonzero $\bF_p$-points of $x^2 + y^2 + z^2 - xyz = 0$. Thus it follows from the combinatorial analysis in the previous section that every $\Out^+(\Pi)$-orbit on this set has cardinality divisible by $p$ (see Corollary \ref{cor_vdovin_2}(a)). This proves the Theorem \ref{thm_orbit_congruence_intro} given in the introduction, and hence resolves the conjecture of Bourgain, Gamburd, and Sarnak for all but finitely many primes. In the following sections we will give a different proof of this fact using the explicit form of the $\Aut^+(\Pi)$-action on $X(q) = \sqcup_{t\in\bF_q}X_t(q)$; this will yield additional congruences not directly implied by Corollary \ref{cor_vdovin_2}.
\end{remark}

\subsection{Automorphism groups of $\cAdm(\SL_2(\bF_q))_\Qbar$}\label{ss_automorphism_groups}
In this section we use the formalism developed in \S\ref{ss_character_variety} to describe the vertical automorphism groups of geometric points of $\cAdm(\SL_2(\bF_q))_\Qbar$. The main result is that for any $q\ge 3$, the map $\ol{\cM(\SL_2(\bF_q)}_\Qbar\rightarrow\ol{\cM(1)}_\Qbar$ is representable. By Proposition \ref{prop_compactification}(b), the restriction to the preimage above $\cM(1)_\Qbar$ is just the finite \'{e}tale map $\cM(G)_\Qbar\rightarrow\cM(1)_\Qbar$ (Theorem \ref{thm_basic_properties}\ref{part_etale}), so there the representability is a consequence of finiteness. It remains to show representability at the cusps - ie, that the vertical automorphism groups remain trivial at the cusps. For this it will be important to consider the ``Dehn twist'' $\gamma_\infty\in\Aut(\Pi)$ given by
$$\gamma_\infty = t\circ r : (a,b)\mapsto (a,ab)$$
which induces the ``rotation'' (using the terminology of \cite{BGS16arxiv})
$$\rot_1 := \Tr_*(t\circ r) = \Tr_*(r)\circ\Tr_*(t) = R_3\circ\tau_{23} : (x,y,z)\mapsto (x,z,xz-y)$$

By Theorem \ref{thm_delta_invariant}, the cusps of $\cAdm(\SL_2(\bF_q))_\Qbar$ are classified by the set $\bI(\SL_2(\bF_q))$. Let $\Pi,a,b$ be as in Situation \ref{situation_galois_theory}. Then the map
\begin{eqnarray*}
\Epi^\ext(\Pi,\SL_2(\bF_q)) & \lra & \bI(\SL_2(\bF_q)) \\
\varphi & \mapsto & \ps{\varphi(a),\varphi(b)}
\end{eqnarray*}
induces a bijection $\Epi^\ext(\Pi,\SL_2(\bF_q))/\langle\gamma_\infty\rangle\rightiso\bI(\SL_2(\bF_q))$. It follows from Theorem \ref{thm_cuspidal_automorphisms}(a) that the following are equivalent:
\begin{itemize}
	\item For every $\varphi\in\Epi^\ext(\Pi,\SL_2(\bF_q))$, the orbit $\{\varphi\circ\gamma_\infty^i\;|\; i\in\bZ\}$ has size $|\varphi(a)|$.
	\item The vertical automorphism groups of geometric points of $\cAdm(\SL_2(\bF_q))$ are reduced to $\{\pm I\} = Z(\SL_2(\bF_q))$.
\end{itemize}
Using the moduli interpretation of the points $X_{\SL_2}(\bF_q)$ (Theorem \ref{thm_moduli_interpretation}), we will show something even stronger:

\begin{lemma}\label{lemma_free} For any prime power $q\ge 3$, let $\varphi : \Pi\rightarrow\SL_2(\bF_q)$ be an absolutely irreducible representation with $(\tr\varphi(b),\tr\varphi(ab))\ne (0,0)$ (in particular $\varphi$ is not of dihedral type, see Proposition \ref{prop_dihedral_representations}). Then the $\gamma_\infty$-orbit of $\varphi$ viewed as an element of $\Hom(\Pi,\SL_2(\bF_q))^\ai/D(q)$ has size $|\varphi(a)|$.
\end{lemma}

The proof of the lemma will be given below. Recall that for $t\in\bF_q$, $n_q(t)$ is the order of any noncentral element of $\SL_2(\bF_q)$ of trace $t$ (Definition \ref{def_na}). The lemma implies:

\begin{thm}\label{thm_representable} Let $q\ge 3$ be a prime power. Let $\varphi : \Pi\ra\SL_2(\bF_q)$ be an absolutely irreducible representation which is not of dihedral type (equivalently, $\varphi(\Pi)$ is not a subgroup of dihedral type and $\tr\varphi([a,b])\ne 2$). Then $G := \varphi(\Pi)$ contains $Z(\SL_2(\bF_q))$ and we have
\begin{itemize}
\item[(a)] The vertical automorphism groups of geometric points of $\cAdm(G)_\Qbar$ are all reduced to $Z(\SL_2(\bF_q)) = \{\pm I\}$
\item[(b)] The forgetful map $\ff : \ol{\cM(G)}_\Qbar\longrightarrow\ol{\cM(1)}_\Qbar$ is representable.
\end{itemize}
Let $t\ne 2\in\bF_q$ and let $\cX\subset\cAdm(G)_t$ be a component with universal family $\pi : \cC\rightarrow\cE$ and reduced ramification divisor $\cR_\pi$. The vertical automorphism groups of geometric points of $\cR_\pi$ are as follows
\begin{itemize}
\item[(c1)] If $q$ is even, then for any geometric point $x\in\cR_\pi$, $\Aut^v(x)$ is trivial.
\item[(c2)] If $q$ is odd and $t = -2$, then for any geometric point $x\in\cR_\pi$, $\Aut^v(x)$ has order 2.
\item[(c3)] If $q$ is odd, $t\ne -2$ and $n_q(t)$ is even, then for any geometric point $x\in \cR_\pi$, $\Aut^v(x)$ has order 2.
\item[(c4)] If $q$ is odd, $t \ne -2$ and $n_q(t)$ is odd, then for any geometric point $x\in\cR_\pi$, $\Aut^v(x)$ is trivial.
\end{itemize}
\end{thm}
\begin{proof}[Proof of Theorem \ref{thm_representable}] That $G = \varphi(\Pi)$ must contain the $Z(\SL_2(\bF_q))$ follows from Macbeath's classification of the 2-generated subgroups of $\SL_2(\bF_q)$ \cite[\S4]{Mac69} (also see Proposition \ref{prop_ai_images} in the appendix). Parts (a) and (b) follow immediately from Lemma \ref{lemma_free} and the preceding discussion. It remains to address (c1)-(c4). Suppose $\cX$ classifies covers with Higman invariant $\fc$. Let $c\in\fc$ be a representative. Using Theorem \ref{thm_cuspidal_automorphisms} and Lemma \ref{lemma_free}, we find $\Aut^v(x) = Z(\SL_2(\bF_q))\cap \langle c\rangle$. If $q$ is even then $Z(\SL_2(\bF_q)) = 1$ so we're done. Now assume $q$ is odd. If $t = -2$, then $c$ is conjugate to $\spmatrix{-1}{u}{0}{-1}$ for some $u\in\bF_q^\times$, so in this case $c^p = -I$ so $\Aut^v(x)$ has order 2. If $t\ne -2$, then $c$ is diagonalizable over $\ol{\bF_q}$, so $\langle c\rangle$ has nontrivial intersection with $\{\pm I\}$ if and only if $|c| = n_q(t)$ is even.
\end{proof}

Finally we prove Lemma \ref{lemma_free}. Along the way we will also count the number of $\rot_1$ orbits on $X^\circ_{-2}(p)$ for $p\ge 3$.

\begin{proof}[Proof of Lemma \ref{lemma_free}] Since $\Tr_*(\gamma_\infty) = \rot_1$, we must analyze the action of $\rot_1$ on the $X_t(q)$ for various $t\in\bF_q$. For any $a\in \bF_q, t\in\bF_q$, the action of $\rot_1$ visibly preserves the conics
$$C_1(a)_t := X_t(q)_{x = a} = \{(x,y,z)\in\bF_q^3\;|\; x = a\text{ and }y^2 + z^2 - ayz + (a^2-2-t) = 0\}\subset X_t(q)$$
where the $q$ is understood. Since $\rot_1$ is induced by an action of $\gamma_\infty\in\Aut(\Pi)$, it also preserves the subset
$$C_1(a)^*_t := X^*_t(q)\cap C_1(a)_t$$
corresponding to absolutely irreducible representations not of dihedral type, as well as the subset
$$C_1(a)^\circ_t := X^\circ_t(q)\cap C_1(a)_t$$
corresponding to surjective representations. Note that for $q\ge 3$, $\PSL_2(\bF_q)$ is not dihedral, so $C_1(a)^\circ_t\subset C_1(a)^*_t$ for $q\ge 3$. When $q = p$, Proposition \ref{prop_surjective_representation} described the conics $C_1(a)^\circ_{-2}\subset\bF_p^3$.


Here we analyze the action of $\rot_1$ on these conics, following \cite{BGS16arxiv} and \cite{MP18}. Homogenizing the equation $y^2+z^2 - ayz + (a^2-2-t) = 0$, by Chevalley-Warning we find that if $a^2\ne t+2\in\bF_q$ then $C_1(a)_t$ is nonempty. The discriminant of the corresponding ternary quadratic form is $(4-a^2)(a^2-2-t)$ and hence we find that $C_1(a)_t$ is a degenerate conic if and only if $a^2 = 4$ or $a^2 = t+2$ in $\bF_q$. Thus we will consider the values $a = \pm 2,\pm\sqrt{t+2}$ separately, where $\sqrt{t+2}\in\ol{\bF_q}$ is a square root of $t+2$. On each $C_1(a)_t$, $\rot_1$ acts as the linear transformation on the ambient affine $yz$-plane given by
$$\rot_1|_{C_1(a)_t} = \spmatrix{0}{1}{-1}{a}$$
Given a subset $Z\subset C_1(a)_t$, we say that $\rot_1$ \emph{acts freely on $Z$} if the cyclic group $\langle\spmatrix{0}{1}{-1}{a}\rangle$ acts freely on $Z$. We will analyze the action of $\rot_1$ according to the behavior of $a^2-4$. We will show that for any $q\ge 3$ and $a,t\in\bF_q$ with $t\ne 2$, $\rot_1$ acts freely on $C_1(a)^*_t$, so its orbits all have the same size $n_q(a)$ (Definition \ref{def_na}). For later use we will also count the number of $\rot_1$-orbits in the case where $q = p \ge 3$ and $t = -2$ (in which case $C_1(a)^*_{-2} = C_1(a)^\circ_{-2}$ by Proposition \ref{prop_surjective_representation}).


Since $C_1(a)^*_2 = C_1(a)^\circ_2 = \emptyset$, in the following we only consider the case when $t \ne 2\in\bF_q$. In particular, the exceptional cases $a^2 = 4, a^2 = t+2$ do not overlap. We will also restrict ourselves to the case $q\ge 3$ (equivalently, $\PSL_2(\bF_q)$ is not dihedral).


\begin{itemize}
\item[(1)] Suppose $a^2 - 4 = 0$ ($a = \pm 2$), $t\ne 2$, and $q\ge 3$ is odd. In this case we say $a$ is parabolic.

If $a = 2$, then $C_1(a)_t$ is given by $(y-z)^2 = t-2$, which is nonempty if and only if $t-2\in\bF_q^\times$ is a square. When this is the case it is the disjoint union of the two lines $y-z = \pm\sqrt{t-2}$. On $\bF_q^2$, $\rot_1$ acts via $\spmatrix{0}{1}{-1}{2}$ which is conjugate to $\spmatrix{1}{1}{0}{1}$. Thus one computes that it fixes pointwise the line $y = z$ in $\bF_q^2$ and acts freely with order $n_q(a) = p$ everywhere else. Since $t\ne 2$, $C_1(a)_t$ does not contain any fixed points, so $\rot_1$ acts on $C_1(a)_t$ freely with orbits of size $n_q(a) = p$. When $q = p\ge 3$ and $t = -2$, we obtain $2$ orbits (of size $p$) on $C_1(2)^*_{-2}$ when $p\equiv 1\mod 4$ and zero orbits if $p\equiv 3\mod 4$.


If $a = -2$, then $C_1(a)_t$ is given by $(y+z)^2 = t-2$, which is again nonempty if and only if $t-2\in\bF_q^\times$ is a square. When $t-2$ is a square, $C_1(a)_t$ is again a disjoint union of the lines $y+z = \pm\sqrt{t-2}$. On $\bF_q^\times$, $\rot_1$ acts via $\spmatrix{0}{1}{-1}{-2}\sim\spmatrix{-1}{1}{0}{-1}$ which acts by $v\mapsto -v$ on the line $y = -z$ and acts freely everywhere else, alternating between the two lines. Since $q$ is odd and $t\ne 2$, $\rot_1$ acts on $C_1(a)_t$ freely with order $n_q(a) = 2p$. When $q = p\ge 3$ and $t = -2$, we obtain one orbit (of size $2p$) on $C_1(-2)^*_{-2}$ when $p\equiv 1\mod 4$ and zero orbits if $p\equiv 3\mod 4$.


\item[(2)] Suppose $a^2 -4\in\bF_q^\times$ is a square, $t\ne 2$, and $q\ge 3$ is odd. In this case we say $a$ is hyperbolic. Hyperbolic $a$'s are in bijection with the set of polynomials of the form $T^2-aT+1$ with distinct roots in $\bF_q^\times$, and hence there are $\frac{q-3}{2}$ hyperbolic $a$'s.

In this case $\rot_1$ acts via $\spmatrix{0}{1}{-1}{a}$ which is diagonalizable with distinct eigenvalues $\omega, \omega^{-1}$ with $\omega = \frac{a\pm\sqrt{a^2-4}}{2}$ and $a = \omega + \omega^{-1}$. This matrix acts freely on $\bF_q^2 - \{(0,0)\}$, so $\rot_1$ acts freely on $C_1(a)_t - \{(a,0,0)\}\supset C_1(a)^*_t$ with orbits of size $n_q(a) = |\omega|$.


If $a^2 = t+2$ then $C_1(a)_t$ is the degenerate conic $(y-\omega z)(y-\omega^{-1}z) = 0$, which consists of two lines intersecting at the origin. If $a^2\ne t+2$ then $C_1(a)_t$ is a nondegenerate conic with two points at infinity, so $C_1(a)_t$ has $q-1$ points.


If $q = p\ge 3$ and $t = -2$ then the degenerate case $a^2 = t+2$ corresponds to $a = 0$, in which case $C_1(0)^*_{-2} = C_1(0)_{-2} - \{(0,0,0)\}$. In this case we obtain $\frac{2(p-1)}{4}$ orbits on $C_1(0)^*_{-2}$, each of size 4. This occurs if and only if $a^2 - 4 = -4\in\bF_p^\times$ is a square - equivalently $p\equiv 1\mod 4$. If $p\equiv 3\mod 4$, then $a = 0$ is not hyperbolic. For nondegenerate hyperbolic $a$'s, $|C_1(a)^*_{-2}| = p-1$ so $C_1(a)^*_{-2}$ contributes $\frac{p-1}{n_q(a)}$ orbits, each of size $n_q(a)$. Thus, if $q = p \ge 3$ and $t = -2$, the number of $\rot_1$ orbits on $C_1(a)^*_{-2}$ for hyperbolic $a$'s is:
$$\left(\text{Number of $\rot_1$ orbits on $\bigsqcup_{\substack{a\in\bF_p \\ \text{hyperbolic}}}C_1(a)^*_{-2}$}\right) = \left\{\begin{array}{rl}
\frac{2(p-1)}{4} + \sum_{\substack{d\mid p-1\\ d\ne 1,2,4}}\frac{\phi(d)}{2}\cdot\frac{p-1}{d}	& \text{if $p\equiv 1\mod 4$} \\
\sum_{\substack{d\mid p-1 \\d\ne 1,2}}\frac{\phi(d)}{2}\cdot\frac{p-1}{d} & \text{if $p\equiv 3\mod 4$}
\end{array}\right.$$
Here $d$ should be thought of as $n_q(a) = |\omega|$.

\item[(3)] Suppose $a^2-4\in\bF_q^\times$ is a non-square, $t\ne 2$, and $q\ge 3$ is odd. In this case we say $a$ is elliptic. Elliptic $a$'s are in bijection with the set of polynomials of the form $T^2-aT+1$ which are nonsplit in $\bF_q$ but have distinct roots in $\bF_{q^2}^\times$. These roots are precisely the elements of $\bF_{q^2}^\times$ which lie in the unique subgroup of order $q+1$ but not in the subgroup of order $q-1$, and hence there are $\frac{q-1}{2}$ elliptic $a$'s.

In this case $\rot_1$ acts via $\spmatrix{0}{1}{-1}{a}$ which is diagonalizable in $\bF_{q^2}$ with distinct conjugate eigenvalues $\omega,\omega^{-1}\in\bF_{q^2}^\times - \bF_q$, so $\omega^{-1} = \omega^q$, equivalently $\omega^{q+1} = 1$, and as usual $a = \omega+\omega^{-1}$. For the same reason as the hyperbolic case, we find $\rot_1$ acts freely on $C_1(a)_t - \{(0,0,0)\}\supset C_1(a)^*_t$ with orbits of size $n_q(a) = |\omega|$.


If $a^2 = t+2$ then over $\bF_{q^2}$, $C_1(a)_t$ is given by $(y - \omega z)(y-\omega^{-1}z) = 0$, but since $\omega\notin\bF_q^\times$, $C_1(a)_t$ is empty. If $a^2\ne t+2$ then $C_1(a)_t$ is a nondegenerate conic with $q+1$ points.


If $q = p\ge 3$ and $t = -2$, then the degenerate case $a^2 = t+2$ (equivalently $a = 0$, equivalently $|\omega| = 4$) is elliptic if and only if $a^2-4 = -4\in\bF_p^\times$ is a non-square, equivalently $p\equiv 3\mod 4$, in which case $C_1(0)^*_{-2}$ is empty. For every other elliptic $a$, $|C_1(a)^*_{-2}| = p+1$ so we have $\frac{p+1}{n_q(a)}$ $\rot_1$-orbits of size $n_q(a)$. Thus, if $q = p\ge 3$ and $t = -2$, the number of $\rot_1$-orbits on $C_1(a)^*_{-2}$ for elliptic $a$'s is:

$$\left(\text{Number of $\rot_1$ orbits on $\bigsqcup_{\substack{a\in\bF_p \\ \text{elliptic}}}C_1(a)^*_{-2}$}\right) = \left\{\begin{array}{rl}
\sum_{\substack{d\mid p+1 \\ d\ne 1,2}}\frac{\phi(d)}{2}\cdot\frac{p+1}{d} & \text{if $p\equiv 1\mod 4$} \\
\sum_{\substack{d\mid p+1 \\d\ne 1,2,4}}\frac{\phi(d)}{2}\cdot\frac{p+1}{d} & \text{if $p\equiv 3\mod 4$}
\end{array}\right.$$

Here $d$ should be thought of as $n_q(a) = |\omega|$.

\item[(4)] Finally assume $q\ge 3$ is even and $t\ne 2$ (equivalently $t\ne 0$).
\begin{itemize}
\item If $a = 0 = -2 = 2$, then $C_1(0)_t$ is given by the equation $y^2+z^2 = (y-z)^2 = t$. Here $\rot_1 = \spmatrix{0}{1}{-1}{0} = \spmatrix{0}{1}{1}{0}$ fixes points on the line $y = z$ and acts with order 2 on every other vector in $\bF_q^2$. Such fixed points lie in $C_1(0)_t$ only when $t = 0 = 2$, which we've excluded, so in this case $C_1(0)_t$ has $q$ points and $\rot_1$ acts freely on $C_1(0)_t$ with order $n_q(a) = 2 = p$.

\item If $a\in\bF_q^\times$, $\rot_1$ acts via $\spmatrix{0}{1}{-1}{a}$ with characteristic polynomial $T^2 - aT + 1$. Since $a\ne 0$ this matrix has distinct eigenvalues $\omega, \omega+a\in\bF_{q^2}^\times$ so again $\rot_1$ is diagonalizable over $\bF_q^\times$ and hence acts freely on $\bF_q^2 - \{(0,0)\}$ and hence also acts freely on $C_1(a)_t - \{(a,0,0)\}\supset C_1(a)^*_t$ with orbits of size $n_q(a)$.
\end{itemize}
\end{itemize}
The above discussion implies that for any prime power $q\ge 3$, $a\in\bF_q$, $t\ne 2\in\bF_q$, every orbit of $\rot_1$ on $C_1(a)_t - \{(a,0,0)\}\subset\bF_q^3$ has size $n_q(a)$ (Definition \ref{def_na}). In particular this holds for every $\rot_1$ orbit on $C_1(a)^*_t$ and $C_1(a)^\circ_t$.
\end{proof}


\begin{prop}\label{prop_empty_conics} For $q = p\ge 3$ a prime, $C_1(a)^*_{-2} = C_1(a)^\circ_{-2}$ if empty if and only if $a\in\{0,2,-2\}\subset\bF_p$ and $p\equiv 3\mod 4$.	 Moreover, we have
$$|\bX^*(p)| = |X_{-2}^\circ(p)| = |X_{-2}^*(p)| = \left\{\begin{array}{rl}
	p(p+3) & p\equiv 1\mod 4 \\
	p(p-3) & p\equiv 3\mod 4
\end{array}\right.$$
\end{prop}
\begin{proof} From the proof of Lemma \ref{lemma_free}, we see that $C_1(a)^*_{-2}$ is always nonempty if $a \notin \{0,\pm 2\}$. If $a = 0$ then we're in the degenerate case $a^2 = t+2$, in which case $C_1(a)^*_{-2}$ is empty if and only if $p\equiv 3\mod 4$. If $a = \pm 2$, then $a$ is parabolic and is again empty if and only if $p\equiv 3\mod 4$.


To compute the cardinality of $X_{-2}^*(p)$, if $p\equiv 1\mod 4$ then the cases $a = \pm2$ contribute $4p$ solutions. The degenerate hyperbolic case $a = 0$ contributes $2(p-1)$ solutions, and the remaining $\frac{p-5}{2}$ nondegenerate hyperbolic $a$'s contribute $(p-1)$ solutions each. Each of the $\frac{p-1}{2}$ elliptic $a$'s is nondegenerate and contributes $p+1$ solutions each, so in total we get $p(p+3)$ solutions in this case. If $p\equiv 3\mod 4$ then there are no parabolic $a$'s, and each of the $\frac{p-3}{2}$ hyperbolic $a$'s is nondegenerate. The case $a = 0$ is degenerate elliptic, giving zero solutions, leaving $\frac{p-3}{2}$ nondegenerate ellipitic $a$'s, so in this case we get $p(p-3)$ as desired.
\end{proof}

\begin{defn} For a natural number $n$, let
$$\Phi(n) := \sum_{d\mid n}\frac{\phi(d)}{d}$$	
\end{defn}

\begin{prop}\label{prop_rot_orbits} For $q = p \ge 3$, the number of $\rot_1$-orbits on $X_{-2}^*(p) = X_{-2}^\circ(p)$ is
$$\left|X^*_{-2}(p)/\rot_1\right| = \left\{\begin{array}{rl}
\frac{p-1}{2}\Phi(p-1) + \frac{p+1}{2}\Phi(p+1) + \frac{-5p+11}{4} & \text{if $p\equiv 1\mod4$} \\[5pt]
\frac{p-1}{2}\Phi(p-1) + \frac{p+1}{2}\Phi(p+1) + \frac{-7p-1}{4} & \text{if $p\equiv 3\mod4$}
\end{array}\right.$$	
\end{prop}
\begin{proof} Follows from the proof of Lemma \ref{lemma_free}.	
\end{proof}


\subsection{Congruences for $\SL_2(\bF_q)$-structures}\label{ss_congruences_for_SL2}
Let $q \ge 3$ be a prime power. Let $G\le\SL_2(\bF_q)$ be an absolutely irreducible subgroup which is not of dihedral type (Definition \ref{def_dihedral_type}). Then we may speak of the trace of elements of $G$. Let $t\ne 2\in\bF_q$, and let $\cAdm(G)_t\subset\cAdm(G)$ be the open and closed substack classifying $G$-covers whose Higman invariants have trace $t$ (Definition \ref{def_trace_invariant_of_subgroup}). Let $\cX\subset\cAdm(G)_t$ be a component, and let $X$ be its coarse scheme. Here we will apply Theorem \ref{thm_congruence} to establish congruences on the degree of the map $X\rightarrow\ol{M(1)}$. We want to bound the integers $d_\cX,m_\cX$ that appear in Theorem \ref{thm_congruence}.


First we bound $m_\cX$. Let $\cC\stackrel{\pi}{\lra}\cE\lra\cX$ be the universal admissible $\SL_2(\bF_q)$-cover over $\cX$ with reduced ramification divisor $\cR_\pi\subset\cC$.

\begin{prop}\label{prop_automorphism_groups} Let $q\ge 3$ be a prime power. Let $t\in\bF_q$ be a $q$-admissible trace. Then the automorphism groups of geometric points of $\cR_\pi$ are all killed by 24. In the following cases we can do slightly better:
\begin{itemize}
\item[(a)] If $q$ is even then the automorphism groups are killed by 12.
\item[(b)] If $q$ is odd, $t\ne -2$, and $n_q(t)$ is odd, then the automorphism groups are killed by 12.
\end{itemize}
In particular, in the language of Theorem \ref{thm_congruence}, we must have $m_\cX\mid 24$ for $q$ odd and $m_\cX\mid 12$ in cases (a) or (b) above.

\end{prop}
\begin{proof} Follows immediately from Theorem \ref{thm_representable}, noting that all automorphism groups in $\ol{\cM(1)}$ are killed by 12.
\end{proof}


Next we bound $d_\cX$. Let $\cR\subset\cR_\pi$ be a connected component, then we wish to bound the degree of $\cR/\cX$. By the definition of the Higman invariant, the stabilizer in $\SL_2(\bF_q)$ of a geometric point of $\cC$ lying over the zero section of $\cE$ is conjugate to $\langle c\rangle$. By Proposition \ref{prop_stacky_RRD}, we find that $\cR/\cX$ is Galois with Galois group a subgroup of $C_{\SL_2(\bF_q)}(\langle c\rangle)/\langle c\rangle$. We can explicitly compute these centralizers as follows.

\begin{prop}\label{prop_inertia_normalizers} Let $q = p^r\ge 3$ with $p$ prime. Let $t\ne 2\in\bF_q$, and let $c\in\SL_2(\bF_q) - \{\pm I\}$ be a noncentral element with trace $\tr(c) = t$. The centralizers $C_{\SL_2(\bF_q)}(\langle c\rangle)$ are as follows:
\begin{itemize}
\item[(a)] If $t^2 - 4 = 0$, then $|c| = 2p$ and $C_{\SL_2(\bF_q)}(\langle c\rangle) = \left\{\spmatrix{a}{b}{0}{a^{-1}}\;|\;a = \pm 1, b\in\bF_q\right\}\cong\bF_q\times \mu_2$. Thus $C_{\SL_2(\bF_q)}(\langle c\rangle)$ has order $2q$.
\item[(b)] If $t^2 - 4$ is a square in $\bF_q^\times$, then $|c|\mid q-1$ and $C_{\SL_2(\bF_q)}(\langle c\rangle)$ is cyclic of order $q-1$.
\item[(c)] If $t^2 - 4$ is a nonsquare in $\bF_q^\times$, then $|c|\mid q+1$ and $C_{\SL_2(\bF_q)}(\langle c\rangle)$ is cyclic of order $q+1$.
\end{itemize}
In cases (a) (resp. (b),(c)) we will say that $c$ is parabolic (resp. hyperbolic, elliptic). Moreover (a) and (c) can only occur if $q$ is odd. In particular, we find that if $c$ is parabolic (resp. hyperbolic, elliptic), then in the language of Theorem \ref{thm_congruence}, for any component $\cX\subset\cAdm(\SL_2(\bF_q))_t$, $d_\cX$ must divide $p^{r-1}$ (resp. $\frac{q-1}{|c|},\frac{q+1}{|c|}$). 
\end{prop}
\begin{proof} If $t^2 - 4 = 0$, then since $t\ne 2$, we must have $t = -2$, $q$ odd, and $|c| = 2p$. By Proposition \ref{prop_Higman_vs_trace}(c), $c$ is conjugate to $\spmatrix{-1}{u}{0}{-1}$ for some $u\in\bF_q^\times$. An explicit calculation shows that its normalizer is the group of matrices of the form $\spmatrix{a}{b}{0}{a^{-1}}\in\SL_2(\bF_q)$ with $a^2\in\bF_p^\times$, and its centralizer is the subgroup with $a^2 = 1$.
	

For (b) and (c), we first make the following observation: If $k$ is a field and $C\in\GL_d(k)$ is diagonalizable over $k$, then for any $A\in\GL_d(k)$, we have $ACA^{-1} = C^n$ if and only if $A$ sends $\lambda$-eigenvectors to $\lambda^n$-eigenvectors. Indeed, if $\lambda_v$ denotes the eigenvalue of an eigenvector $v\in k^d$, then
\begin{eqnarray*}
ACA^{-1} = C^n & \iff & ACA^{-1}v = C^nv = \lambda_v^n v \quad\text{ for all eigenvectors $v\in k^d$} \\
 & \iff & C(A^{-1}v) = \lambda_v^n(A^{-1}v) \quad\text{ for all eigenvectors $v\in k^d$}
\end{eqnarray*}

If $t^2 - 4$ is a square in $\bF_q^\times$, then $c$ is diagonalizable over $\bF_q$ with distinct eigenvalues. If $v,w$ is an eigenbasis, then $A$ centralizes $c$ if and only if $Av = \alpha v$, $Aw = \alpha^{-1}w$ for some $\alpha\in\bF_q^\times$, so the centralizer is cyclic of order $q-1$.  


If $t^2 - 4$ is a nonsquare in $\bF_q^\times$, then $c$ is not diagonalizable over $\bF_q$ but is diagonalizable over $\bF_{q^2}$ with distinct roots. By (b), we have $C_{\SL_2(\bF_{q^2})}(\langle c\rangle)$ is cyclic of order $q^2-1$. Let $\sigma$ denote the $q$-power Frobenius automorphism. Let $v\in\bF_{q^2}^2$ be an eigenvector for $c$, then we claim that $v,v^\sigma$ is an eigenbasis for $c$. If $v = (x,y)\ne (0,0)$, then $v^\sigma = av$ for $a\in\bF_{q^2}^\times$ if and only $x^q = ax, y^q = ay$, so $x^{q-1} = y^{q-1} = 1$, so $v\in\bF_q^2$, which contradicts the assumption that $c$ is not diagonalizable over $\bF_q$, so $\{v,v^\sigma\}$ is an eigenbasis. If $A\in\SL_2(\bF_{q})$ centralizes $c$ then we must have
\begin{itemize}
	\item $Av = \alpha v$ for some $\alpha\in\bF_{q^2}^\times$, and
	\item $Av^\sigma = (Av)^\sigma = (\alpha v)^\sigma = \alpha^q v^\sigma$.
\end{itemize}
Since $\det(A) = 1$, we must have $\alpha^{q+1} = 1$. Conversely, the relation $(Av)^\sigma = Av^\sigma$ implies that any matrix in $\GL_2(\ol{\bF_q})$ sending $(v,v^\sigma)$ to $(\alpha v,\alpha^q v^\sigma)$ where $\alpha^{q+1} = 1$ must lie in $C_{\SL_2(\bF_q)}(\langle c\rangle)$, so $C_{\SL_2(\bF_q)}(\langle c\rangle)$ is cyclic of order $q+1$.
\end{proof}

Plugging the above results into Theorem \ref{thm_congruence} gives us the following.

\begin{thm}\label{thm_congruence_on_degrees} Let $q = p^r\ge 3$ with $p$ prime. Let $G\le\SL_2(\bF_q)$ a subgroup which is not of dihedral type. Let $t\ne 2\in\bF_q$, and let $c\in\SL_2(\bF_q) - \{\pm I\}$ be a noncentral element with trace $\tr(c) = t$. Let $\cX\subset\cAdm(G)_t$ be a connected component with coarse scheme $X$. Then
\begin{itemize}
\item[(a)] If $t^2 - 4 = 0$ and $q = p$, then $|c| = 2p$ and
$$\deg(X\rightarrow\ol{M(1)}) \equiv 0 \mod 	p$$
\item[(b)] If $t^2 - 4$ is a square in $\bF_q^\times$, then
$$\deg(X\rightarrow\ol{M(1)}) \equiv 0 \mod \frac{|c|}{\gcd\left(|c|,\frac{2(q-1)}{|c|}\right)}$$
In particular, if $n\mid |c|$ is coprime to $\frac{2(q-1)}{|c|}$, then $\deg(X\rightarrow\ol{M(1)})\equiv 0\mod n$.
\item[(c)] If $t^2 - 4$ is a nonsquare in $\bF_q^\times$ then
$$\deg(X\rightarrow\ol{M(1)}) \equiv 0 \mod \frac{|c|}{\gcd\left(|c|,\frac{2(q+1)}{|c|}\right)}$$
In particular, if $n\mid |c|$ is coprime to $\frac{2(q+1)}{|c|}$, then $\deg(X\rightarrow\ol{M(1)})\equiv 0\mod n$.
\end{itemize}
If $q$ is even then $t^2 -4$ must be a square in $\bF_q^\times$ then (b) can be strengthened by replacing the $2(q-1)$ by $q-1$. 
\end{thm}
\begin{proof} Let $C\stackrel{\pi}{\lra}\cE\lra\cX$ be the universal family, and let $\cR\subset\cR_\pi$ be a component of the reduced ramification divisor. By Proposition \ref{prop_stacky_RRD}, the degree of $\cR/\cX$ divides the order of $C_{\SL_2(\bF_q)}(\langle c\rangle)/\langle c\rangle$. Thus the theorem follows directly from Proposition \ref{prop_automorphism_groups}, \ref{prop_inertia_normalizers}, and Theorem \ref{thm_congruence}. For the statement when $q$ is even, note that $2 = 0 = -2\in\bF_q$ is never a $q$-admissible trace for $q$ even, and every element of $\bF_q^\times$ is a square.
\end{proof}

Using the ``moduli interpretation'' for $X_t^*(q)$ (Theorem \ref{thm_moduli_interpretation}, Definition \ref{def_Xq}), these congruences can be transported to congruences on the cardinality of $\Aut^+(\Pi)$-orbits on the sets $X_t^*(q)$.

\begin{thm}\label{thm_congruence_on_orbits} Let $q = p^r\ge 3$ with $p$ prime. Let $t\ne 2\in\bF_q$. Let $n_q(t)$ denote the order of any non-central element of $\SL_2(\bF_q)$ with trace $t$ (c.f. Definition \ref{def_na} and Proposition \ref{prop_na}). Let $P = (X,Y,Z)$ be an $\bF_q$-point of the affine surface given by the equation
$$x^2 + y^2 + z^2 - xyz = 2+t$$
Suppose at least two of $\{X,Y,Z\}$ are nonzero. Let $\cO$ be the $\Aut^+(\Pi)$-orbit of $P$.
\begin{itemize}
\item[(a)] If $t^2-4 = 0$ and $q = p \ge 3$, then
$$|\cO|\equiv 0\mod p$$
\item[(b)] If $t^2-4$ is a square in $\bF_q^\times$, then
$$2|\cO|\equiv 0 \mod \frac{n_q(t)}{\gcd\left(n_q(t),\frac{2(q-1)}{n_q(t)}\right)}$$
\item[(c)] If $t^2-4$ is a nonsquare in $\bF_q^\times$, then
$$2|\cO|\equiv 0 \mod \frac{n_q(t)}{\gcd\left(n_q(t),\frac{2(q+1)}{n_q(t)}\right)}$$
\end{itemize}
If $q$ is even and $\cO\subset X_t^\circ(q)$ then $t^2-4$ must be a square in $\bF_q^\times$ and (b) can be strengthened by replacing $2|\cO|$ by $|\cO|$ and $2(q-1)$ by $q-1$. 
\end{thm}

\begin{proof} Since $t\ne 2$, by Proposition \ref{prop_dihedral_representations}, our assumption on $\{X,Y,Z\}$ implies that $P\in X^*_t(q)$ (i.e. $P$ is absolutely irreducible and not of dihedral type), so $\cO\subset X^*_t(q)$. By Galois theory (Proposition \ref{prop_monodromy_formulas}), the orbit $\cO$ corresponds to a connected component $\cM\subset\cM(G)_t/D(q,G)$ with degree over $\cM(1)$ equal to $|\cO|$, where $G$ is the image of an absolutely irreducible representation $\varphi : \Pi\rightarrow\SL_2(\bF_q)$ not of dihedral type. Let $M$ denote its coarse scheme. Then because $\gamma_{-I}$ acts trivially on $X^*(q)$ (Proposition \ref{prop_monodromy_formulas}), $|\cO|$ is also equal to the degree of $M\rightarrow M(1)$. We have maps
$$\cAdm^0(G)_t\stackrel{}{\lra}\cM(G)_t\lra\cM(G)_t/D(q,G)$$
where the first is rigidification by $Z(\SL_2(\bF_q))$ (Proposition \ref{prop_compactification}(b)), and the second is finite \'{e}tale of degree $|D(q,G)|$. At the level of coarse schemes, the maps induce
$$\Adm^0(G)_t\stackrel{\cong}{\lra} M(G)_t\lra M(G)_t/D(q,G)$$
The first map is an isomorphism, and the second is finite flat with degree at most $|D(q,G)|$. By Proposition \ref{prop_Dq} and \ref{prop_ai_images}, we always have $|D(q,G)|\le 2$. Thus if $X$ is any connected component of $\Adm^0(G)_t$ mapping to $M$, then $\deg(X/M)\le 2$, so in this case we have
$$|\cO| = \deg(\cM\rightarrow\cM(1)) = \deg(M\rightarrow M(1)) = \alpha\cdot\deg(X\rightarrow M(1))$$
where $\alpha\in\{1,\frac{1}{2}\}$. If $q$ is even and $G = \SL_2(\bF_q)$ (equivalently $\cO\subset X^\circ_t(q)$), then $D(q,G) = D(q)$ is trivial. Thus everything follows immediately from Theorem \ref{thm_congruence_on_degrees}.
\end{proof}

\begin{remark} When $G\le\SL_2(\bF_q)$ is a proper subgroup, the bound $d_\cX\mid |C_{\SL_2(\bF_q)}(\langle c\rangle)/\langle c\rangle|$ can be improved to $d_\cX\mid |C_G(\langle c\rangle)/\langle c\rangle|$. Using Proposition \ref{prop_ai_images}, this would give slight strengthenings of Theorems \ref{thm_congruence_on_degrees} and \ref{thm_congruence_on_orbits}.
\end{remark}

\subsection{Strong approximation for the Markoff equation}\label{ss_bgs_conjecture}
The Markoff surface is the affine surface $\bM$ defined by the equation
$$\bM : x^2 + y^2 + z^2 -3xyz = 0$$
This equation is called the Markoff equation. A Markoff triple is an positive integer solution to this equation, and a Markoff number is an integer that appears as a coordinate of a Markoff triple. Here we prove some results regarding $\bM$. It will be convenient to use an twisted version of $\bM$, which is isomorphic to $\bM$ over $\bZ[1/3]$, but has the benefit of admitting a moduli interpretation. Recall that taking trace coordinates induces an isomorphism $\Tr : X_{\SL_2}\rightiso\bA^3$, and $X_{\SL_2,-2}\subset X_{\SL_2}$ is the closed subscheme corresponding to representations with trace invariant $-2$ (Definition \ref{def_Xq}). Under the isomorphism $\Tr$, $X_{\SL_2,-2}$ is isomorphic to the affine surface over $\bZ$ given by
$$\bX : x^2 + y^2 + z^2 - xyz = 0$$
The map $\xi : (x,y,z)\mapsto (3x,3y,3z)$ defines a map
$$\xi : \bM\lra \bX$$
which induces an isomorphism of schemes $\xi : \bM_{\bZ[1/3]}\rightiso \bX_{\bZ[1/3]}$. Since $\bX(\bF_3) = \{(0,0,0)\}$, $\xi$ also induces a bijection on integral points
$$\xi : \bM(\bZ)\rightiso \bX(\bZ)$$
Let $\bX^*(p) := \bX(\bF_p) - \{(0,0,0)\}$, and let $\bM^*(p) := \bM(\bF_p) - \{(0,0,0)\}$. Then by Proposition \ref{prop_trace_fibration}, the $\Aut(\Pi)$-action on $\bX$ is generated by the automorphisms
$$\begin{array}{rcl}
R_3 : (x,y,z) & \mapsto & (x,y,xy-z) \\
\tau_{12} : (x,y,z) & \mapsto & (y,x,z) \\
\tau_{23} : (x,y,z) & \mapsto & (x,z,y)
\end{array}$$
Via $\xi$ we obtain an induced action of $\Aut(\Pi)$ on $\bM_{\bZ[1/3]}$ generated by $\tau_{12},\tau_{23}$, and $R_3' : (x,y,z)\mapsto (x,y,3xy-z)$, which extends to an action on $\bM$ given by the same equations.


In \cite{BGS16,BGS16arxiv}, Bourgain, Gamburd, and Sarnak study the action of $\Aut(\Pi)$ on $\bM(\bF_p)$ for primes $p$. They conjectured that for every prime $p$, $\Aut(\Pi)$ acts transitively on $\bM^*(p)$. One may check that for $p = 3$, $\Aut(\Pi)$ acts transitively on $\bM^*(3)$ and $\bX^*(3)$ (the latter is empty). Thus, their conjecture is equivalent to

\begin{conj}[{\cite{BGS16,BGS16arxiv}}]\label{conj_bgs} For all primes $p$, $\Aut(\Pi)$ acts transitively on $\bX^*(p)$.
\end{conj}

They were able to establish their conjecture for all but a sparse (though infinite) set of primes $p$:

\begin{thm}[{\cite[Theorem 2]{BGS16arxiv}}]\label{thm_bgs_sparse}
Let $\bE_\bgs$ denote the ``exceptional set'' of primes for which $\Aut(\Pi)$ fails to act transitively on $\bX^*(p)$. For any $\epsilon > 0$, 
$$\{p\in\bE_\bgs\;|\;p\le x\} = O(x^\epsilon).$$
\end{thm}

Moreover, they show that even if the conjecture were to fail, it cannot fail too horribly:
\begin{thm}[{\cite[Theorem 1]{BGS16arxiv}}]\label{thm_bgs_asymptotic} Fix $\epsilon > 0$. For every prime $p$, there is a $\Aut(\Pi)$-orbit $\cC(p)$ such that
$$|\bX^*(p) - \cC(p)| \le p^\epsilon\qquad\text{for large $p$}$$
whereas $|\bX^*(p)| = p(p+3)$ (resp. $p(p-3)$) if $p\equiv 1\mod 4$ (resp. $3\mod 4$).
\end{thm}

Recall that we have defined $X^*_{-2}(q)$ as the subset of $\bX(\bF_q)$ which correspond (via Theorem \ref{thm_moduli_interpretation}) to absolutely irreducible representations not of dihedral type (Definition \ref{def_Xq}). Thus Theorem \ref{thm_congruence_on_orbits}(a) implies:

\begin{thm}\label{thm_main_divisibility} Every $\Aut^+(\Pi)$-orbit on $\bX^*(p)$ has size divisible by $p$.	
\end{thm}
\begin{proof} By Proposition \ref{prop_surjective_representation}, for $p\ge 3$, $X^*_{-2}(p) = \bX^*(p)$, so this case follows from Theorem \ref{thm_congruence_on_orbits}(a). For $p = 2$, one can check the statement by hand.	
\end{proof}

Combined with Theorem \ref{thm_bgs_asymptotic}, this establishes Conjecture \ref{conj_bgs} for all but finitely many primes.

\begin{thm}\label{thm_bgs_conj} The exceptional set $\bE_\bgs$ of Theorem \ref{thm_bgs_sparse} is finite and explicitly bounded.
\end{thm}
\begin{proof} The finiteness follows from Theorem \ref{thm_congruence_on_orbits}(a) and Theorem \ref{thm_bgs_asymptotic}. The fact that one can explicitly bound the set $\bE_\bgs$ follows from the fact that the methods of \cite[Theorem 1]{BGS16} are effective. An explicit upper bound was obtained by Elena Fuchs (private communication).
\end{proof}

\begin{defn} Following \cite{BGS16}, we say that an affine variety $V$ over $\bZ$ satisfies \emph{strong approximation mod $n$} if the natural map $V(\bZ)\rightarrow V(\bZ/n\bZ)$ is surjective.	
\end{defn}

\begin{thm}\label{thm_main_applications} In light of Theorem \ref{thm_bgs_conj}, we have
\begin{itemize}
\item[(a)] For all but finitely many primes $p$, $\Aut^+(\Pi)$ acts transitively on $\bX^*(p)$ and $\bM^*(p)$
\item[(b)] For all but finitely many primes $p$, $M(\SL_2(\bF_p))_{-2}^\abs := M(\SL_2(\bF_p))_{-2}/D(p)$ is connected.
\item[(c)] For all but finitely many primes $p$, both $\bX$ and $\bM$ satisfy strong approximation mod $p$.
\end{itemize}
To be precise, (a), (b), (c) hold for every prime $p\notin\bE_\bgs$.
\end{thm}
\begin{proof} By the definition of $\bE_\bgs$, for every $p\notin\bE_\bgs$, $\Aut(\Pi)$ acts transitively on $\bX^*(p)$. Since the $\Aut(\Pi)$-action has a fixed point (namely $(3,3,3)$) and $\Aut^+(\Pi)\le\Aut(\Pi)$ has index 2, this also implies the transitivity of the $\Aut^+(\Pi)$ action on $\bX^*(p)$, which via the isomorphism $\xi$ also implies the transitivity of the action on $\bM^*(p)$ for $p\ne 3$. The case $p = 3$ can be checked by hand. Part (b) is the Galois-theoretic translation of part (a) (c.f. Proposition \ref{prop_monodromy_formulas}). Part (c) follows from the observation that the reduction map $\bX(\bZ)\rightarrow \bX(\bF_p)$ is $\Aut(\Pi)$-equivariant, and assuming (a), the orbits of $(0,0,0),(3,3,3)\in \bX(\bZ)$ map surjectively onto $\bX(\bF_p)/\Aut(\Pi)$ for $p\in\bE_\bgs$, $p\ne 3$, which shows that $\bX$ satisfies strong approximation in the sense described above, which also implies strong approximation for $\bM$ at every $p\ne 3$. It can be checked by hand that strong approximation also holds at $p = 3$, noting that $\bX^*(3)$ is empty.
\end{proof}

Thus, we have effectively reduced Conjecture \ref{conj_bgs} to a finite computation. In \cite{IL20}, de-Courcy-Ireland and Lee have verified the conjecture for all primes $p < 3000$, so one expects the final computation to give a positive solution to the conjecture. If the computation indeed verifies the conjecture, then it follows from Proposition \ref{prop_empty_conics} that there are no congruence constraints on Markoff numbers mod $p$ other than the ones first noted in \cite{Frob13}, namely that if $p\equiv 3\mod 4$ and $p\ne 3$ then a Markoff number cannot be $\equiv 0,\frac{\pm 2}{3}\mod p$.


It follows from the work of Meiri-Puder \cite{MP18} that the strong approximation property can be further extended to squarefree integers $n$ whose prime divisors avoid the finite set $\bE_\bgs$ and moreover satisfy the following condition: 
\begin{equation}\label{eq_property_P_p}
    \BMP(p) := \begin{array}{r}\text{The property that either $p\equiv 1 \mod 4$, or} \\
\text{the order of $\frac{3+\sqrt{5}}{2}\in\bF_{p^2}$ is at least $32\sqrt{p+1}$.}
\end{array}
\end{equation}
Thus $\BMP(p)$ is satisfied for all primes $p\equiv 1\mod 4$. Moreover by \cite[Proposition A.1]{MP18}, we find that $\BMP(p)$ holds for a density 1 set of primes. Their theorem is:

\begin{thm}\label{thm_composite_moduli} Let $p_1,\ldots,p_r\notin\bE_\bgs$ be distinct primes such that for each $i\in\{1,\ldots,r\}$, $p_i\notin\bE_\bgs$ and satisfies $\BMP(p_i)$. Let $n := p_1p_2\cdots p_r$. The Chinese remainder theorem implies that $\bX(\bZ/n\bZ) = \bX(\bZ/p_1\bZ)\times\cdots\times \bX(\bZ/p_r\bZ)$. Let $\bX^*(n) \subset\bX(\bZ/n\bZ)$ denote the solutions which do not reduce to the trivial solution $(0,0,0)\mod p_i$ for any $i$. Then we have
\begin{itemize}
\item[(a)] $\Aut(\Pi)$ acts transitively on $\bX^*(n)$.
\item[(b)] Both $\bX$ and $\bM$ satisfies strong approximation mod $n$.
\end{itemize}

\end{thm}
\begin{proof} Part (a) is \cite[Corollary 1.7]{MP18}, which by induction on $r$ implies that the $\Aut(\Pi)$-orbit space on $\bX(\bZ/n\bZ)$ is the product of the orbit spaces on $\bX(\bZ/p_i\bZ)$ for $i = 1,\ldots,r$. This gives strong approximation for $\bX$. The same argument gives strong approximation for $\bM$ nod $n$, noting that if $3\mid n$ then $\bM(\bF_3)$ has two singleton orbits $\{(0,0,0)\}$ and $\{(1,1,1)\}$.
\end{proof}


Finally, we note that Theorem \ref{thm_congruence_on_orbits} also gives a generalization of Theorem \ref{thm_main_divisibility} to the $\bF_q$-points of the Markoff-type equations $x^2 + y^2 + z^2 - xyz = k$. For example, it implies:
\begin{thm}\label{thm_general_congruence} Suppose $\omega\in\bF_{q^2}^\times$ satisfies $t := \omega + \omega^{-1}\in\bF_q - \{2,-2\}$. If $\ell$ is an odd prime and $\ord_\ell(|\omega|) = r$ and $\ord_\ell(q(q^2-1)) = r+s$, and $P = (X,Y,Z)$ is an $\bF_q$-point of 
$$x^2 + y^2 + z^2 - xyz = t+2$$
with at least two of $\{X,Y,Z\}$ nonzero, then the $\Aut^+(\Pi)$-orbit of $P$ has cardinality divisible by $\ell^{\max\{r-s,0\}}$. 
\end{thm}

Combined with results analogous to the Theorems \ref{thm_bgs_asymptotic} and \ref{thm_bgs_sparse} announced in \cite{BGS16}, we expect that the congruences obtained in Theorem \ref{thm_congruence_on_orbits} for general $t\in\bF_q$ will help resolve many cases of strong approximation for these more general Markoff-type equations as well.

\subsection{A genus formula for $M(\SL_2(\bF_p))_{-2,\Qbar}$; Finiteness of $\SL_2(\bF_p)$-structures of trace invariant $-2$}\label{ss_genus_formulas}

By Theorem \ref{thm_main_applications} we find that the stack
$$\cM_p := \cM(\SL_2(\bF_p))^\abs_{-2} := \cM(\SL_2(\bF_p))_{-2}/D(p)$$
is \emph{connected} for primes $p$ not lying in the finite set $\bE_\bgs$. In this section we will establish a genus formula for (the compactification) of its coarse scheme $M_p$, and show that for a density 1 set of primes, $M_p$ is a ``noncongruence modular curve''. By Proposition \ref{prop_Higman_vs_trace}(c'), the natural map $M(\SL_2(\bF_p))_{-2}\rightarrow M_p$ is a totally split cover, and hence this will also compute the genus of either of the two components of $M(\SL_2(\bF_p))_{-2}$.


For $p\ge 5$, let $\ol{M_p}$ denote the smooth compactification of the 1-dimensional scheme $M_p := M(\SL_2(\bF_p))^\abs_{-2}$ (By Proposition \ref{prop_Higman_vs_trace}, $M_p$ is empty for $p = 2,3$). By Theorem \ref{thm_basic_properties}(2), the forgetful map
$$\ff : \ol{M_p}\rightarrow\ol{M(1)} = \bP^1_\Qbar$$
is finite \'{e}tale over the complement of $j = 0,1728,\infty$, where as usual we use the $j$-invariant to identify $\ol{M(1)}$ with the projective $j$-line. Thus, if $p\notin\bE_\bgs$ (equivalently $\ol{M_p}$ is connected), then the genus of $\ol{M_p}$ can be computed by the Riemann-Hurwitz formula (see \cite[\S7.4, Theorem 4.16]{Liu02} or \cite[\S II, Theorem 5.9]{Sil09}). For this we need to know the degree of $\ff$, and the cardinalities of the ramified fibers $\ff^{-1}(0),\ff^{-1}(1728),\ff^{-1}(\infty)$. Using Proposition \ref{prop_monodromy_formulas}, all of these quantities can be calculated from the $\Aut^+(\Pi)$ action on $X_{-2}^*(p)$. For example, Proposition \ref{prop_empty_conics} implies that
$$\deg(\ff) = \left\{\begin{array}{rl}
p(p+3) & p\equiv 1\mod 4 \\
p(p-3) & p\equiv 3\mod 4
\end{array}\right.$$

The ramification above $j = 0,1728$ was computed in \cite[Proposition 3.3.2]{BBCL20}:

\begin{prop}\label{prop_0_1728} Let $p\ge 5$ be a prime.
\begin{itemize}
\item[(a)] There is a unique unramified point in $\ff^{-1}(0)$. Every other point is ramified with index 3, so
$$|\ff^{-1}(0)| = \left\{\begin{array}{ll}
	\frac{p(p+3)-3}{3} & p\equiv 1\mod 4 \\
	\frac{p(p-3)-3}{3} & p\equiv 3\mod 4
\end{array}\right.$$

\item[(b)] There are precisely two unramified points in $\pi^{-1}(1728)$ if $p\equiv 1,7\mod 8$, and no unramified points otherwise. Every other point is ramified with index 2. Thus we obtain
$$|\ff^{-1}(1728)| = \left\{\begin{array}{ll}
\frac{p^2+3p+2}{2} & \text{if $p\equiv 1\mod 8$} \\
\frac{p^2-3p}{2} & \text{if $p\equiv 3\mod 8$} \\
\frac{p^2+3p}{2} & \text{if $p\equiv 5\mod 8$} \\
\frac{p^2-3p+2}{2} & \text{if $p\equiv 7\mod 8$}
\end{array}\right.$$
\end{itemize}
\end{prop}

Recall that for $n\in\bN$, $\Phi(n) = \sum_{d\mid n}\frac{\phi(d)}{d}$. Theorem \ref{thm_basic_properties}\ref{part_fibers} (or Proposition \ref{prop_monodromy_formulas}) gives a combinatorial classification of the cusps. Using this, Proposition \ref{prop_rot_orbits} calculates the number of cusps of $\ol{M_p}$:

\begin{prop}\label{prop_number_of_cusps} Let $p\ge 5$ be a prime.
$$|\ff^{-1}(\infty)| = \left\{\begin{array}{ll}
	\frac{p-1}{2}\Phi(p-1) + \frac{p+1}{2}\Phi(p+1) + \frac{-5p+11}{4} & \text{if $p\equiv 1\mod 4$} \\[5pt]
	\frac{p-1}{2}\Phi(p-1) + \frac{p+1}{2}\Phi(p+1) + \frac{-7p-1}{4} & \text{if $p\equiv 3\mod 4$}
\end{array}\right.$$
\end{prop}

An elementary Riemann-Hurwitz calculation yields the following

\begin{thm}\label{thm_genus_formula} Let $p\ge 5$ be prime. Let $\ol{M}_p$ be the smooth compactification of $M_p$. The smooth compactification of $M(\SL_2(\bF_p))_{-2}$ is a disjoint union of two copies of $\ol{M}_p$. For $p\notin\bE_\bgs$, $\ol{M}_p$ is a smooth curve of genus

$$\genus(\ol{M}_p) = \frac{1}{12}p^2 - \frac{p-1}{4}\Phi(p-1) - \frac{p+1}{4}\Phi(p+1) + \epsilon(p)$$
where
$$\epsilon(p) = \left\{\begin{array}{ll}
	\frac{7}{8}p - \frac{29}{24} & \text{if $p\equiv 1\mod 8$}\\[5pt]
	\frac{5}{8}p + \frac{19}{24} & \text{if $p\equiv 3\mod 8$}\\[5pt]
	\frac{7}{8}p - \frac{17}{24} & \text{if $p\equiv 5\mod 8$}\\[5pt]
	\frac{5}{8}p + \frac{7}{24} & \text{if $p\equiv 7\mod 8$}
\end{array}\right.$$
In particular $\genus(\ol{M}_p)\sim\frac{1}{12}p^2$, and moreover for all $p\ge 5$, we have
$$\genus(\ol{M}_p)\ge \frac{1}{12}p^2 - \frac{1}{2}(p-1)^{3/2} - \frac{1}{2}(p+1)^{3/2} + \frac{1}{2}p$$

\end{thm}
\begin{proof} Everything but the lower bound follows from the above discussion. The lower bound follows from the upper bound for the divisor function $d(n) \le 2\sqrt{n}$, where $d(n)$ is the number of positive divisors of $n$, and noting that $\Phi(n)\le d(n)\le 2\sqrt{n}$, and $\epsilon(p)\ge \frac{1}{2}p$ for $p\ge 5$.
	
\end{proof}

\begin{thm}\label{thm_finiteness_of_SL2p_structures} For $p\ge 13$, $p\notin\bE_\bgs$, $\ol{M}_p$ is a smooth curve of genus $\ge 2$. For $p = 5,7,11$, $\ol{M}_p$ is a smooth curve of genus $0,0,1$ respectively. In particular, for $p\ge 13$, $p\notin\bE_\bgs$ and any number field $K$, only finitely many elliptic curves admit a $\SL_2(\bF_p)$-structure with ramification index $2p$.
\end{thm}
\begin{proof} By the classification of elements in $\SL_2(\bF_p)$ as parabolic, hyperbolic, or elliptic, only parabolic elements have order $2p$, and by Proposition \ref{prop_Higman_vs_trace}, this only happens when $p\ge 3$ and the trace is $-2$. The first part of the theorem follows from the lower bound for the genus obtained in Theorem \ref{thm_genus_formula}, together with some explicit computer-aided calculations. The finiteness follows from the first by Falting's theorem.
\end{proof}

Given a connected finite \'{e}tale cover $\ff : \cM\rightarrow\cM(1)_\bC$ inducing $M\rightarrow M(1)_\bC$ on coarse schemes, the analytic theory identifies $\cM$ with a quotient $[\cH/\Gamma]$ where $\cH$ is the Poincar\'{e} upper half plane, and $\Gamma\le\SL_2(\bZ)$ is a finite index subgroup acting by Mobius transformations \cite{Chen18}. The subgroup $\Gamma$ is uniquely determined by $\ff$ up to conjugation.

\begin{defn} Let $k\subset\bC$ be a subfield. We say that a finite \'{e}tale map $\ff : \cM\rightarrow\cM(1)_k$ is a congruence modular stack (and $M\rightarrow M(1)_k$ is a congruence modular curve) if $\cM_\bC$ is connected and the subgroup $\Gamma\le\SL_2(\bZ)$ associated to $\ff_\bC$ is a congruence subgroup of $\SL_2(\bZ)$ (that is, $\Gamma$ contains $\Ker(\SL_2(\bZ)\rightarrow\SL_2(\bZ/n))$ for some integer $n\ge 1$). Otherwise we say that it is a noncongruence modular stack (resp. curve).
\end{defn}

In \cite[Conjecture 4.4.1]{Chen18}, the author conjectured that the components of $\cM(G)$ for nonsolvable groups $G$ should all be noncongruence. In general the problem of determining whether the components of $\cM(G)$ are congruence or noncongruence can be quite difficult. From the results of \cite[\S5.5]{Chen18}, this problem is also central to the geometric approach to the \emph{unbounded denominators conjecture} for noncongruence modular forms \cite[Conjecture 5.2.2]{Chen18}. We end this section by showing that for a density 1 set of primes $p$, all components of $\cM(\PSL_2(\bF_p))_{-2}$ (and hence $\cM(\SL_2(\bF_p))_{-2}$) are noncongruence.


Let $\Pi,E,x_E$ be as in Situation \ref{situation_galois_theory}. We work over $\Qbar$. The action of $D(p)\cong\Out(\SL_2(\bF_p))$ on $\SL_2(\bF_p)$ descends via the characteristic quotient $h : \SL_2(\bF_p)\rightarrow\PSL_2(\bF_p)$ to an outer action on $\PSL_2(\bF_p)$, where it also acts via the full outer automorphism group of $\PSL_2(\bF_p)$. Thus this quotient induces a finite \'{e}tale surjection
$$\cM(h) : \cM(\SL_2(\bF_p))/D(p)\lra\cM(\PSL_2(\bF_p))/D(p)$$
and a surjection
$$h_* : \Epi(\Pi,\SL_2(\bF_p))/D(p)\lra\cM(\PSL_2(\bF_p))/D(p)$$
which can be identified with the map on fibers above $x_E$ induced by $\cM(h)$. For $t\in\bF_p$, let $\cM(\PSL_2(\bF_p))_t$ denote the image of $\cM(\SL_2(\bF_p))_t$, and we say that $\cM(\PSL_2(\bF_p))_t$ classifies $\PSL_2(\bF_p)$-structures of trace invariant $t$. Because $h$ is a central extension, there is no risk of confusion - any preimage of $x\in\cM(\PSL_2(\bF_p))_t$ in $\cM(\SL_2(\bF_p))$ must also have trace invariant $t$. Let
$$\bP_\MP := \{\text{primes }p\ge 5\;|\; p\notin\bE_\bgs \text{ and $p$ satisfies $\BMP(p)$ (see \eqref{eq_property_P_p})}\}$$
Then $\bP_\MP$ is a density 1 set of primes.

\begin{thm}[Meiri-Puder]\label{thm_MP} We work over $\Qbar$. For $p\in\bP_\MP$, the stack $\cM(\PSL_2(\bF_p))_{-2}^\abs = \cM(\PSL_2(\bF_p))_{-2}/D(p)$ is connected. Its degree over $\cM(1)$ is
$$d_p := \left\{\begin{array}{rl}
\frac{p(p+3)}{4} & p\equiv 1\mod 4 \\
\frac{p(p-3)}{4} & p\equiv 3\mod 4	
\end{array}\right.$$
and its monodromy group\footnote{Equivalently, the Galois group of its Galois closure.} over $\cM(1)$ is the full alternating or symmetric group on $d_p$:
$$\text{Mon}(\cM(\PSL_2(\bF_p))_{-2}^\abs\rightarrow\cM(1)) = \left\{\begin{array}{rl}
	S_{d_p} & \text{if }p\equiv 5,7,9,11\mod 16 \\
	A_{d_p} & \text{if }p\equiv 1,3,13,15\mod 16
\end{array}\right.$$
\end{thm}
\begin{proof} Since $p\in\bP_\MP$, $p\notin\bE_\bgs$, so the connectedness statement follows from the connectedness of $\cM(\SL_2(\bF_p))^\abs_{-2}$ (Theorem \ref{thm_main_applications}(b)). Let $V$ denote the order 4 group of automorphisms acting freely on $\bX^*(p)$ by negating two of the coordinates. Let $\bY^*(p) := \bX^*(p)/V$. The action of $\Gamma$ on $\bX^*(p)$ descends to an action on $\bY^*(p)$, and the theory of the character variety allows us to identify the quotient map $\bX^*(p)\rightarrow\bY^*(p)$ with the map
$$h_* : \Epi^\ext(\Pi,\SL_2(\bF_p))_{-2}/D(p)\rightarrow\Epi^\ext(\Pi,\PSL_2(\bF_p))_{-2}/D(p)$$
where the $\Gamma$-action is induced by the action of $\Out(\Pi)$ (Proposition \ref{thm_moduli_interpretation}). The sizes of $\bX^*(p)$ were computed in Proposition \ref{prop_empty_conics}. Thus the formula for $d_p$ follows from the freeness of the $V$-action. By \cite[Theorems 1.3, 1.4]{MP18}, the permutation image of $\Gamma$ on $\bY^*(p)$ is either the full alternating or symmetric group. Since alternating groups in degrees $\ge 5$ do not have nontrivial index 2 subgroups, the same is true for the permutation image of $\Out^+(\Pi)$ acting on $\Epi^\ext(\Pi,\PSL_2(\bF_p))_{-2}/D(p)$ (note $|\bY^*(p)|\ge 7$ for $p\ge 5$). The determination of exactly when one obtains the alternating (or symmetric group) is done in \cite[\S3.3]{BBCL20}.
\end{proof}

\begin{remark}\label{remark_MP_conjecture} Conjecturally, this Theorem holds for all primes $p\ge 5$ \cite[Conjecture 1.2]{MP18}	
\end{remark}

\begin{cor}\label{cor_noncongruence} For $p\in\bP_\MP$, the stack $\cM(\PSL_2(\bF_p))^\abs_{-2}$ is noncongruence, so the same is true of $\cM(\SL_2(\bF_p))^\abs_{-2}$.
\end{cor}
\begin{proof} First observe that a cover of a noncongruence modular stack is noncongruence, so it suffices to show that $\cM(\PSL_2(\bF_p))^\abs_{-2}$ is noncongruence. If $\cM(\PSL_2(\bF_p))^\abs_{-2}\rightarrow\cM(1)$ is congruence, then it fits into a factorization
$$\cM(n)\rightarrow\cM(\PSL_2(\bF_p))^\abs_{-2}\rightarrow\cM(1)$$
where $\cM(n)$ is connected (in fact it is a component of $\cM(\bZ/n\times\bZ/n)$) and the composition is Galois with Galois group $\SL_2(\bZ/n)$. It follows that the monodromy group of $\cM(\PSL_2(\bF_p))^\abs_{-2}/\cM(1)$ is a quotient of $\SL_2(\bZ/n)$. Writing $n = \prod_{i=1}^r q_i$ with each $q_i$ a prime power, we have $\SL_2(\bZ/n) = \prod_{i=1}^r\SL_2(\bZ/q_i)$, and hence the composition factors of $\SL_2(\bZ/n)$ are either abelian or of the form $\PSL_2(\bF_p)$ for primes $p\ge 5$, so the same must be true of the composition factors of the monodromy group of $\cM(\PSL_2(\bF_p))^\abs_{-2}/\cM(1)_\bC$, but this contradicts Theorem \ref{thm_MP}.
\end{proof}

Thus, Theorem \ref{thm_genus_formula} can be viewed as a noncongruence analog of Rademacher's conjecture, proved by Dennin \cite{Den75}, that there exist only finitely many congruence subgroups of a given genus. By contrast, it follows from Belyi's theorem that there are infinitely many noncongruence subgroups of every genus, so to obtain finiteness, one must restrict the types of noncongruence modular curves considered. Theorem \ref{thm_genus_formula} yields finiteness for the family $\{M(\SL_2(\bF_p))_{-2}^\abs\;|\; p\in\bP_\MP\}$.

\section{Appendix}
\subsection{Normalized coordinates for tame balanced actions on prestable curves}

The purpose of this section is to prove Proposition \ref{prop_normalized_coordinates} below, which shows that the \'{e}tale local picture of an admissible $G$-cover (Definition \ref{def_admissible} can be checked on fibers. The trickiest part is to check this at a node, where the calculation uses the explicit form of the projection maps onto isotypic subspaces (Lemma \ref{lemma_tame_decomposition}). 

\begin{remark}[Noetherian approximation]\label{remark_noetherian_approximation} It is sometimes useful to prove results about prestable curves by working over a Noetherian base. By standard Noetherian approximation arguments, we do not lose any generality in doing so. A precise statement we'll need is this. Let $C\rightarrow S$ be a prestable curve equipped with an effective Cartier divisor $R\subset C$ finite \'{e}tale over a quasicompact quasiseparated scheme $S$ and an $S$-linear action of a finite group $G$ preserving $R$. Then we may write $S$ as a limit $S = \lim_{i\in I} S_i$ with affine transition morphisms with each $S_i$ of finite type over $\bZ$ \cite[01ZA,07RN]{stacks}. In this case each map $S\rightarrow S_i$ is also affine \cite[01YX]{stacks}. Moreover, for some $i\in I$ the pair $(C,R)$ with $G$-action is the base change of a prestable curve $C_i\rightarrow S_i$ with divisor $R_i\subset C_i$ finite \'{e}tale over $S_i$ and $G$-action preserving $R_i$. The key fact is that the category of schemes of finite presentation over $S$ is the colimit of the categories of schemes of finite presentation over $S_i$ \cite[01ZM]{stacks}. For example, from this, one can find a map $R_i\rightarrow C_i$ over $S_i$ which pulls back to $R\rightarrow C$. By \cite[0C5F,081C,081F]{stacks}, we may assume that $C_i/S_i$ is a prestable curve, and that $R_i/S_i$ is finite \'{e}tale, and hence $R_i\rightarrow C_i$ must be the inclusion of an effective Cartier divisor. Similarly, viewing the $G$-action as being given as a collection of automorphisms satisfying certain properties, by \cite[01ZM]{stacks}, we may assume the $G$-action is also the pullback of an $S_i$-linear $G$-action on $(C_i,R_i)$.
\end{remark}


\begin{lemma}\label{lemma_tame_decomposition} Let $A$ be a ring such that $\Spec A$ is connected. Let $e\ge 1$ be an integer invertible in $A$. Suppose the finite \'{e}tale group scheme $\mu_{e,A} = \Spec A[x]/(x^e-1)$ is totally split over $A$. Let $M$ be an $A$-module, and let $G$ be a cyclic group generated by $g$, acting $A$-linearly on $M$. For a root of unity $\zeta\in\mu_e(A)$, let
$$p_\zeta : M\lra M\qquad\text{be given by}\qquad p_\zeta(m) = \frac{1}{e}\sum_{j=0}^{e-1}\zeta^{-j}g^j(m).$$
Then $p_\zeta(M) = M_\zeta := \{m\in M\;|\; gm = \zeta m\}$, and $\bigoplus_{\zeta\in\mu_e(A)} p_\zeta : M\rightarrow\bigoplus_{\zeta\in\mu_e(A)} M_\zeta$ is an isomorphism. Moreover, for each $\zeta\in\mu_e(A)$, formation of $M_\zeta$ defines an exact functor $\Mod_{A[G]}\rightarrow\Mod_{A[G]}$.
\end{lemma}
\begin{proof} It's easy to check that $p_\zeta$ maps into $M_\zeta$, and that the composition $M_\zeta\hookrightarrow M\stackrel{p_\zeta}{\rightarrow} M_\zeta$ is the identity. Thus, each $M_\zeta$ is a direct summand of $M$. Moreover, if $\zeta_e\in A$ is a primitive $e$-th root of unity, then looking at the map $A[x]/(x^e-1)\rightarrow A$ sending $x\mapsto \zeta_e$ shows that $\frac{1}{e}\sum_{\zeta\in\mu_e(A)}\zeta^{-j} = 1$ if $j\equiv 0\mod e$, and is zero otherwise. Thus, the map
\begin{eqnarray*}
\oplus_{\zeta\in\mu_e(A)}p_\zeta : M & \longrightarrow & \bigoplus_{\zeta\in\mu_e(A)} M_\zeta \\
m & \mapsto & \sum_{\zeta\in\mu_e(A)}p_\zeta(m) = \frac{1}{e}\sum_{\zeta\in\mu_e(A)}\sum_{j=0}^{e-1}\zeta^{-j}g^j(m) = \frac{1}{e}\sum_{j=0}^{e-1}\sum_{\zeta\in\mu_e(A)}\zeta^{-j}g^j(m)
\end{eqnarray*}
is the identity. This establishes the desired decomposition. The exactness of $M\mapsto M_\zeta$ is easy to check.
\end{proof}

\begin{lemma}[Henselization commutes with $G$-invariants]\label{lemma_henselization_commutes_with_invariants} Let $R$ be a ring equipped with an action of a finite group $G$ (no tameness assumptions). Let $\fm_R\subset R$ be a maximal ideal. Let $\fm_{R^G} := \fm_R\cap R^G$. Let $(R^h,\fm_R^h), ((R^G)^h,\fm_{R^G}^h)$ denote the henselizations of the pairs $(R,\fm_R), (R^G,\fm_{R^G})$ \cite[0A02]{stacks}. Then
\begin{enumerate}[label=(\alph*)]
\item $\fm_{R^G} := \fm_R\cap R^G$ is a maximal ideal of $R^G$,
\item $R^h,(R^G)^h$ are local rings with maximal ideals $\fm_R^h,\fm_{R^G}^h$,
\item Let $f : (R^G)^h\rightarrow R^h$ be the natural map induced by the morphism of pairs $(R^G,\fm_{R^G})\rightarrow (R,\fm_R)$. Then $f$ induces an isomorphism $R\otimes_{R^G}(R^G)^h\rightiso R^h$.
\item The natural map $(R^G)^h = (R^G)\otimes_{R^G} (R^G)^h\lra (R\otimes_{R^G}(R^G)^h)^G\cong (R^h)^G$ is an isomorphism.
\end{enumerate}
\end{lemma}
\begin{proof} Integral morphisms are universally closed, so we get (a). Henselization of a pair $(A,I)$ preserves the quotient $A/I$ \cite[0AGU]{stacks}, so $R^h/\fm_R^h\cong R/\fm_R, (R^G)^h/\fm_{R^G}^h\cong R^G/\fm_{R^G}$ are fields, so $\fm_R^h,\fm_{R^G}^h$ are maximal. On the other hand, $\fm_R^h,\fm_{R^G}^h$ must be contained in the Jacobson radical of their corresponding rings \cite[09XE]{stacks}, so $R^h, (R^G)^h$ are local, so we get (b). Since every $r\in R$ satisfies the polynomial $\prod_{g\in G}(T-gr)\in R^G[T]$, $R^G\ra R$ is integral, so (c) is \cite[0DYE]{stacks}. Finally, henselization is flat, so (d) follows from the fact that taking $G$-invariants commutes with flat base change \cite[Proposition A.7.1.3]{KM85}.
\end{proof}

\begin{prop}[Normalized coordinates for tame balanced actions on nodal curves]\label{prop_normalized_coordinates} Let $G$ be a finite group with order invertible on $S$. Let $C/S$ be a prestable curve equipped with a ($S$-linear) right action of $G$ acting faithfully on fibers. Let $k$ be an algebraically closed field and $\ol{p} : \Spec k\rightarrow C$ a geometric point with image $\ol{s}$ in $S$. Suppose the stabilizer $G_{\ol{p}}$ is cyclic of order $e$, and let $g\in G_{\ol{p}}$ be a generator. Let $A := \cO_{S,\ol{s}}$ and let $R := \cO_{C_A,\ol{p}}$ the strict local ring of $C_A$ at $\ol{p}$, equipped with the induced action of $G_{\ol{p}}$.
\begin{enumerate}[label=(\alph*)]
	\item\label{part_normal_form_at_smooth_point} Suppose the image of $\ol{p}$ is smooth inside $C_{\ol{s}}$. Then there is an $A$-algebra homomorphism
	$$\phi : A[z]\lra R$$
	such that
	\begin{itemize}
	\item $\phi$ induces an isomorphism between $R$ and the strict local ring of $A[z]$ at $(\fm_A,z)$.
	\item There is a primitive $e$th root of unity $\zeta_e\in A$ such that $\phi$ is $G_\ol{p}$-equivariant relative to the action of $G_\ol{p}$ on $A[z]$ given by $gz = \zeta_ez$.
	\end{itemize}
	\item\label{part_normal_form_at_node} Suppose the image of $\ol{p}$ is a node in the fiber $C_\ol{s}$. The normalization of $C_{\ol{s}}$ defines a decomposition of the Zariski cotangent space $T_{C_\ol{s},\ol{p}}^*$ into a sum of two 1-dimensional subspaces, called ``branches''. Suppose the action of $G_\ol{p}$ on $T_{C_\ol{s},\ol{p}}^*$ preserves this decomposition and has image contained in $\SL(T_{C_\ol{s},\ol{p}}^*)$. Then there is an element $a\in\fm_A$ and an $A$-algebra homomorphism
	$$\phi : A[z,w]/(zw-a)\lra R$$
	such that
	\begin{itemize}
	\item $\phi$ induces an isomorphism between $R$ and the strict local ring of $A[z,w]/(zw-a)$ at $(\fm_A,z,w)$.
	\item There is a primitive $e$th root of unity $\zeta_e\in A$ such that $\phi$ is $G_\ol{p}$-equivariant relative to the action of $G_\ol{p}$ on $A[z,w]/(zw-a)$ given by $gz = \zeta_ez$, $gw = \zeta_e^{-1}w$.
	\end{itemize}
\end{enumerate}	
In particular, the quotient map $\pi : C\rightarrow C/G$ satisfies conditions \ref{part_admissible_local_marking}, \ref{part_admissible_local_nodes}, and \ref{part_admissible_balanced} of the definition of an admissible $G$-cover (Definition \ref{def_admissible}).
\end{prop}
\begin{proof} By Noetherian approximation (Remark \ref{remark_noetherian_approximation}), we may assume that $S$ is of finite type over $\bZ$.\footnote{In particular every ring in the proof will be Noetherian. This makes it easier for us to work with completions.} Let $\kappa$ be the residue field of $A$, and let $\ol{R} := R\otimes_A\kappa$. 


We begin with part \ref{part_normal_form_at_smooth_point}. By the local picture at a smooth point \cite[054L]{stacks}, there is a map $A[x]\rightarrow R$ inducing an isomorphism between $R$ and the strict local ring of $A[x]$ at the maximal ideal $(\fm_A,x)$. We abuse notation and also let $x$ denote the image of $x$ under the map $A[x]\rightarrow R$. Let $\what{R}$ denote the completion of $R$, then $\what{R} \cong \what{A}\ps{x}$. The Zariski cotangent space of $C_\ol{s}$ at $\ol{p}$ is the 1-dimensional $\kappa$-vector space $\fm_{\ol{R}}/\fm_{\ol{R}}^2 = xR/(\fm_A,x^2)$. Since the $G$-action is faithful, $g$ acts on the cotangent space by multiplication by a primitive $e$th root of unity $\zeta_e$. Let $\ol{x}$ denote the image of $x$ in $\fm_\ol{R}/\fm_{\ol{R}}^2$. By Lemma \ref{lemma_tame_decomposition}, the surjection $xR\rightarrow\fm_\ol{R}/\fm_{\ol{R}}^2$ induces a surjection $(xR)_{\zeta_e}\rightarrow (\fm_{\ol{R}}/\fm_{\ol{R}}^2)_{\zeta_e} = \fm_{\ol{R}}/\fm_{\ol{R}}^2$. Let $x'\in (xR)_{\zeta_e}$ denote any preimage of $\ol{x}$, then we have $x' = x + m + fx^2$ for some $m\in\fm_AR$ and $f\in R$. Because $\what{A}\ps{x}$ is also the completion of $A[x]$ at $(\fm_A,x)$, the universal property of power series rings gives us a unique $\what{A}$-algebra map $\what{A}\ps{x}\rightarrow \what{A}\ps{x}$ sending $x\mapsto x' = x+m+fx^2$. This map is moreover an automorphism since it is the composition of the automorphisms $x\mapsto x+m$ and $x\mapsto x+fx^2 = x(1+fx)$. We will view this as giving an automorphism of $\what{R}$. Since $R$ is a colimit of \'{e}tale $A[x]$-algebras, there is an \'{e}tale $A[x]$-algebra $R_0$ such that $x'$ comes from an element of $R_0$. Then we have a commutative diagram

\[\begin{tikzcd}
	A[x]\ar[r]\ar[d,"x\mapsto z"'] & R_0\ar[r] & R\ar[r,hookrightarrow] & \what{R}\ar[d,"x\mapsto x'"] \\
	A[z]\ar[r,"z\mapsto x'"] & R_0\ar[r] & R\ar[r,hookrightarrow] & \what{R}
\end{tikzcd}\]
where the unlabeled maps are the obvious ones. Since $A$ is already strict henselian, $\what{R}$ is the completion of both $R_0$ and also of $A[x]$ at $(\fm_A,x)$. Thus we find that the map $A[z]\rightarrow R_0$ sending $z\mapsto x'$ induces an isomorphism on completions, and hence it is \'{e}tale \cite[\S4.3 Proposition 3.26]{Liu02}, so the composition $\phi : A[z]\rightarrow R_0\rightarrow R$ sending $z\mapsto x'$ identifies $R$ with the strict local ring of $A[z]$ at $(\fm_A,z)$. By definition of $x'$, $\phi$ is $G_\ol{p}$-equivariant relative to the action $gz = \zeta_e z$, as desired. Because $A$ has separably closed residue field, by Lemma \ref{lemma_henselization_commutes_with_invariants}(d), the map $R^G\rightarrow R$ is $G$-equivariantly isomorphic to the map of henselizations induced by $A[z^e] = A[z]^G\hookrightarrow A[z]$, which shows that $\pi : C\rightarrow C/G$ satisfies condition \ref{part_admissible_local_marking} of the definition of an admissible $G$-cover.


Next we address \ref{part_normal_form_at_node}. By the local picture at a node \cite[0CBY]{stacks}, we find that $R$ is the strict henselization of $A[x,y]/(xy-a)$ for some $a\in\fm_A$. Again we abuse notation and let $x,y$ also denote their images in $R$. The cotangent space of the $x$-branch is $T^*_x := x\ol{R}/(x\ol{R}\cap\fm_{\ol{R}}^2)$. Since $G_\ol{p}$ acts faithfully, it acts on $T^*_x$ by multiplication by a primitive $e$-th root of unity $\zeta_e$. Since $g$ preserves the branches of the node, it preserves the ideals $xR, yR$, so by Lemma \ref{lemma_tame_decomposition}, we obtain a surjection $(xR)_{\zeta_e}\rightarrow (T^*_x)_{\zeta_e} = T^*_x$. Let $x'\in (xR)_{\zeta_e}$ be any preimage of a basis of $T^*_x$, then we must have $x' = ux$ where $u\notin (\fm_AR,xR,yR) = \fm_R$, so $u\in R^\times$ is a unit. This implies that $u^{-1}y$ almost lies in $R_{\zeta_e^{-1}}$, in the sense that
\begin{equation}\label{eq_almost_semi_invariant}
uxu^{-1}y = xy = a = g(a) = g(xy) = g(ux u^{-1}y) = \zeta_e ux g(u^{-1}y)	
\end{equation}
Let $p_{\zeta_e^{-1}}$ be as in Lemma \ref{lemma_tame_decomposition}, then by \eqref{eq_almost_semi_invariant}, $y' := p_{\zeta_e^{-1}}u^{-1}y$ satisfies
$$x'y' = \frac{x'}{e}\sum_{j=0}^{e-1}\zeta_e^jg^j(u^{-1}y) = \frac{1}{e}\sum_{j=0}^{e-1}\zeta_e^juxg^j(u^{-1}y) = \frac{1}{e}\sum_{j=0}^{e-1}uxu^{-1}y = a$$
On the other hand, $y'$ and $u^{-1}y$ both map to the same basis element of $T^*_y := y\ol{R}/(y\ol{R}\cap\fm_{\ol{R}}^2)$, so $y'\equiv u^{-1}y\mod(\fm_AR, x^2R, y^2R)$. Arguing as in the smooth case, we find that there is an automorphism of $\what{R}$ sending $(x,y)\mapsto (x',y')$ (also see \cite[Proposition 2.1.1(ii)]{Wew99}), and such that the map $\phi : A[z,w]/(zw-a)\rightarrow R$ sending $(z,w)\mapsto (x',y')$ satisfies the desired properties. As in the smooth case, Lemma \ref{lemma_henselization_commutes_with_invariants}(d) implies that $R^G\hookrightarrow R$ is $G$-equivariantly isomorphic to the map on henselizations induced by $A[z^e,w^e]/(z^ew^e-a^e)\hookrightarrow A[z,w]/(zw-a)$, which shows that $\pi : C\rightarrow C/G$ satisfies conditions \ref{part_admissible_local_nodes} and \ref{part_admissible_balanced} of the definition of an admissible $G$-cover.
\end{proof}

\subsection{The normalizer of $\SL_n(\bF_q)$ in $\GL_n(\ol{\bF_q})$}
\begin{prop}\label{prop_normalizer} Let $q$ be a prime power, let $n\ge 1$ be an integer. Then
$$N_{\GL_n(\ol{\bF_q})}(\SL_n(\bF_q)) = \ol{\bF_q}^\times\cdot\GL_n(\bF_q) = \{uA \;|\; u\in\ol{\bF_q}^\times, A\in\GL_n(\bF_q)\}$$
Moreover, the same is true if $\SL_n(\bF_q)$ is replaced by $\GL_n(\bF_q)$.
\end{prop}
\begin{proof} Certainly every matrix of the form $uA$ with $u\in\ol{\bF_q}^\times, A\in\GL_n(\bF_q)$ normalizes $\SL_n(\bF_q)$, so it remains to show that any matrix which normalizes must take this form. Let $\phi\in\Aut(\GL_n(\ol{\bF_q}))$ denote the Frobenius automorphism defined on matrices by acting on coefficients via $a\mapsto a^q$. Then, for $A\in\GL_n(\ol{\bF_q})$, $A\in\GL_n(\bF_q)$ if and only if $\phi(A) = A$. Thus, for $B\in\SL_n(\bF_q)$, we have $ABA^{-1}\in\SL_n(\bF_q)$ if and only if $\phi(ABA^{-1}) = \phi(A)B\phi(A)^{-1} = ABA^{-1}$, which happens if and only if $A^{-1}\phi(A)$ centralizes $B$. Thus, $A$ normalizes $\SL_n(\bF_q)$ if and only if $A^{-1}\phi(A)$ centralizes $\SL_n(\bF_q)$, which by Schur's lemma happens if and only if $A^{-1}\phi(A)\in\ol{\bF_q}$ is a scalar. Equivalently, this is to say that $\phi(A) = uA$ for some unit $u\in\ol{\bF_q}^\times$. It remains to characterize the elements of $\SL_n(\ol{\bF_q})$ on which $\phi$ acts by multiplication by a unit $u\in\ol{\bF_q}^\times$. 


Suppose $A = (a_{ij}) \in\GL_n(\ol{\bF_q})$ satisfies $\phi(A) = uA$ with $u\in\ol{\bF_q}^\times$. Let $r\in\{a_{ij}\}$ be chosen so that $r\ne 0$. Then since $a_{ij}^q = ua_{ij}$, it follows that $\phi\left(\frac{1}{r}A\right) = \frac{1}{r}A$, so $\frac{1}{r}A\in\GL_n(\bF_q)$, so we may write
$$A = rA'\qquad \text{where $r\in\ol{\bF_q}^\times, A'\in\GL_n(\bF_q)$}$$
as desired.
\end{proof}

\subsection{Images of absolutely irreducible representations $\varphi : \Pi\rightarrow\SL_2(\bF_q)$}
Let $\Pi$ be a free group of rank 2. In \cite{Mac69}, Macbeath classified the possible images of absolutely irreducible representations $\varphi : \Pi\rightarrow\SL_2(\bF_q)$. Here, $\PSL_2(\bF_q)$ denotes the quotient of $\SL_2(\bF_q)$ by the subgroup of scalar matrices, and $\PGL_2(\bF_q)$ denotes the quotient of $\GL_2(\bF_q)$ by the subgroup of scalar matrices.

\begin{prop}\label{prop_ai_images} Let $q = p^r$ be a prime power. Let $\varphi : \Pi\rightarrow\SL_2(\bF_q)$ be an absolutely irreducible representation. Let $G := \varphi(\Pi)$, and let $\ol{G}$ be its image in $\PSL_2(\bF_q)$. Let $Z := Z(\SL_2(\bF_q))$ be the center, so $Z$ has order 2 for $q$ odd and is trivial for $q$ even. Then $G,\ol{G}$ must fall into one of the following categories:
\begin{itemize}
\item[(E1)] $\ol{G}\cong D_{2n}$ is dihedral ($n\ge 2$).
\item[(E2)] $\ol{G}\cong A_4$.
\item[(E3)] $\ol{G}\cong A_5$.
\item[(E4)] $\ol{G}\cong S_5$.
\item[(P1)] $\ol{G}$ is $\PGL_2(\bF_q)$-conjugate to $\PSL_2(\bF_{q'})$ for some $q'\mid q$. In this case $G$ is $\GL_2(\bF_q)$-conjugate to $\SL_2(\bF_{q'})$.
\item[(P2)] $\ol{G}$ is $\PGL_2(\bF_q)$-conjugate to $\PGL_2(\bF_{q'})$ for some $q'$ satisfying $q'^2\mid q$. In this case if $q$ is odd then $G$ is $\GL_2(\bF_q)$-conjugate to $\langle\SL_2(\bF_{q'}),\spmatrix{a}{0}{0}{a^{-1}}\rangle$ where $a\in\bF_{q'^2} - \bF_{q'}$ with $a^2\in\bF_{q'}$. If $q$ is even then $\PGL_2(\bF_{q'}) = \PSL_2(\bF_{q'})$ and $G$ must be as in case P1.
\end{itemize}
Moreover, we have
\begin{itemize}
\item[(a)] In each of the above cases, $G$ is a central extension of $\ol{G}$ by $Z$.	
\item[(b)] In each of the above cases except (E1) (dihedral), let $N := N_{\GL_2(\bF_q)}(G)$ be its normalizer and $C := C_{\GL_2(\bF_q)}(G)$ be its centralizer. Then $N/C$ has order at most 2.
\end{itemize}

\end{prop}
\begin{proof} The cases (E1)-(E4) are called \emph{exceptional}, and the cases P1 and P2 are called \emph{projective}. They are not necessarily mutually exclusive. We begin with (a). Since $Z$ is trivial for $q$ even, we may assume $q$ odd. In this case, an element $g\in\PSL_2(\bF_q)$ of even order must either be diagonalizable over $\bF_{q^2}$ or be conjugate to $\spmatrix{-1}{u}{0}{-1}$ for $u\in\bF_q^\times$. If $g$ is diagonalizable over $\bF_{q^2}$ of order $2k$, then $g^k = -I$, so $Z\subset G$. If $g$ is conjugate to $\spmatrix{-1}{u}{0}{-1}$, then $g^p = -I$, so $Z\subset G$.


For (b), we will proceed case by case. First suppose $q$ is even, so $Z = 1$ and $G \cong \ol{G}$. Then $N/C$ is naturally a subgroup of $\Out(G)$, but $\Out(A_4)\cong\Out(A_5)$ have order 2, and $\Out(S_5)$ is trivial. Finally, from Proposition \ref{prop_normalizer}, we know that $\GL_2(\bF_q)$ acts on $\SL_2(\bF_{q'})$ via $\GL_2(\bF_{q'})$, so in this case we also have $|N/C|\le 2$. Now suppose $q$ is odd, so $G$ is a central extension of $\ol{G}$ by $Z\cong\bZ/2\bZ$. Case (P1) proceeds exactly as in the case where $q$ is even. In case (P2), again we may assume $\ol{G} = \PGL_2(\bF_{q'})$ and $G = \langle \SL_2(\bF_{q'}),\spmatrix{a}{0}{0}{a^{-1}}\rangle$. We know that the image of $N$ in $\PGL_2(\bF_q)$ must normalize $\PGL_2(\bF_{q'})$, hence it normalizes its unique index 2 subgroup $\PSL_2(\bF_{q'})$, so $N$ must normalize $\SL_2(\bF_{q'})$, but then Proposition \ref{prop_normalizer} implies that $N$ acts via $\GL_2(\bF_{q'})$, so again we must have $|N/C|\le 2$.


In cases (E2)-(E4), for each of $\ol{G} = A_4,A_5,S_5$, we will compute the outer automorphism groups of central extensions of $\ol{G}$ by $Z$.. Using the exact sequence \cite[Exercise 6.1.5]{Weibel94}
$$0\rightarrow\Ext^1_\bZ(H_1(G,\bZ),A)\rightarrow H^2(G,A)\rightarrow\Hom(H_2(G,\bZ),A)\rightarrow 0$$
we find that $H^2(A_4,Z), H^2(A_5,Z), H^2(S_5,Z)$ have orders 2, 2, and 4 respectively. Using GAP we explicitly construct the associated central extensions, and find that in all cases except for the split extension $S_5\times Z$, the extension has an outer automorphism group of order 2. However, we claim that if $\ol{G}\cong S_5$, then $G$ cannot be isomorphic to $S_5\times Z$. Indeed, in this case there is an element $\ol{g}\in\ol{G}$ of order 4, but for $q$ odd $\ol{g}$ must be the image of a matrix $g\in\SL_2(\bF_q)$ which is diagonalizable over $\ol{\bF_q}$, so $g$ must have order 8, but $S_5\times Z$ does not contain any elements of order 8. This shows that if $\ol{G} = A_4,A_5,S_5$, then $\Out(G)\cong\bZ/2\bZ$, so $|N/C| \le 2$.


Finally, we show that $G,\ol{G}$ must fall into one of the categories listed. Note that since $\varphi$ is absolutely irreducible, by Lemma \ref{lemma_absolutely_irreducible}, we must have $\tr\varphi([a,b])\ne 2$, so $\varphi$ must be ``nonsingular'' in the terminology of the first proof of Theorem \ref{thm_moduli_interpretation}, and hence \cite[Theorem 2]{Mac69} implies that $\ol{G}$ cannot be an ``affine group'' in the sense of \cite[\S4]{Mac69}. Thus the classification in \cite[\S4]{Mac69} implies that $\ol{G}$ must fall into one of the categories described above. We can rule out the case $\ol{G} \cong D_{2}$ because in that case $G$ would be abelian. It remains to show that $G$ has the stated form in cases (P1) and (P2), but this follows from comparing cardinalities.	
\end{proof}

\subsection{\'{E}tale local rings of Deligne-Mumford stacks}\label{ss_dm_elr}

The goal of this section is to show that (the completion of) the \'{e}tale local ring of a geometric point of a Deligne-Mumford stack is canonically isomorphic to the universal deformation ring at that point. While this is certainly well-known to experts, the author does not know of a good reference. The discussion here parallels the development in the stacks project \cite[06G7]{stacks}, but we do \emph{not} make the assumption that $k$ is a \emph{finite} $\Lambda$-algebra (see below). We do this so we can easily talk about the universal deformation rings of geometric points of $\cM$.

\subsubsection{\'{E}tale local rings}
We begin with a discussion of \'{e}tale local rings.


\begin{defn}\label{def_\'{e}tale_local_ring} Let $\cM$ be a Deligne-Mumford stack whose diagonal is representable (by schemes)\footnote{This implies that if $f : U\rightarrow\cM$, $g : V\rightarrow\cM$ are any maps from schemes, then $U\times_\cM V$ is a scheme. I.e., $f,g$ is are representable (by schemes).}. Let $\Omega$ be a separably closed field, and let $x : \Spec \Omega\rightarrow\cM$ be a point. An \'{e}tale neighborhood of $x$ is a quadruple $(U,i,\tilde{x},\alpha)$ where $U$ is an affine scheme, $U,i,\tilde{x}$ form a diagram
\[\begin{tikzcd}
	\Spec \Omega\ar[r,"\tilde{x}"]\ar[rd,"x"] & U\ar[d,"i"] \\
	 & \cM
\end{tikzcd}\]
and $\alpha$ is an isomorphism $x\rightiso i\circ\tilde{x}$ in $\cM(\Spec \Omega)$. A morphism of neighborhoods $(U,i,\tilde{x},\alpha) \rightarrow (U',i',\tilde{x}',\alpha')$ is a pair $(f,\beta)$ where $f$ is a map $f : U\rightarrow U'$ and $\beta$ is an isomorphism $\beta : i\rightiso i'\circ f$ in $\cM(U)$ such that all 2-morphisms in the associated ``tetrahedron'' are compatible. Note that given $f$, if there exists a $\beta$ making $(f,\beta)$ into a morphism of neighborhoods, then $\beta$ is unique. Let $N_{\cM,x}$ denote the category of \'{e}tale neighborhoods of $x$. Because $\cM$ has representable diagonal, the same argument as in the schemes case shows that $N_{\cM,x}$ is cofiltered \cite[03PQ]{stacks}. The \'{e}tale local ring of $\cM$ at $x$ is

$$\cO_{\cM,x} := \colim_{(U,i,\tilde{x},\alpha)}\Gamma(U,\cO_U)$$
where the colimit runs over the category of \'{e}tale neighborhoods $N_{\cM,x}$. If $\cM$ is a scheme, then $\cO_{\cM,x}$ is the strict henselization of the local ring of the image of $x$ \cite[04HX]{stacks}. The inclusion functor $\Phi : \AffSch\hookrightarrow\Sch$ is right adjoint to the functor $X\mapsto \Spec\Gamma(X,\cO_X)$ \cite[01I1]{stacks}, so $\Phi$ preserves all limits. Thus $\Spec\cO_{\cM,x}$ is also the limit of all \'{e}tale neighborhoods of $x$. Let $\kappa(x)$ denote the residue field of $\cO_{\cM,x}$. For any \'{e}tale neighborhood $(U,i,\tilde{x},\alpha)$, $\cO_{\cM,x}$ fits into a canonical 2-commutative diagram
\[\begin{tikzcd}
	\Spec \Omega\ar[r]\ar[rr, bend left = 20,"{\lim\tilde{x}}"]\ar[rrd, bend right = 15, "x"'] & \Spec\kappa(x)\ar[r] & \Spec\cO_{\cM,x}\ar[d,"i_x"]\ar[r] & U\ar[ld,"i"] \\
	 & & \cM
\end{tikzcd}\]
\end{defn}

\begin{prop}\label{prop_elr} Let $\cM$ be a Deligne-Mumford stack with representable diagonal. Let $\Omega$ be a separably closed field, and $x : \Spec \Omega\rightarrow\cM$ a point. The canonical map $i_x : \Spec\cO_{\cM,x}\rightarrow\cM$ is formally \'{e}tale. If $(U,q,\tilde{x},\alpha)$ is an \'{e}tale neighborhood of $x$ and $u$ is the image of $\tilde{x}$, then there is a canonical isomorphism $\cO_{\cM,x}\cong\cO_{U,\tilde{x}}$ which induces an isomorphism between $\kappa(x)$ and the separable closure of $\kappa(u)$ inside $\Omega$.
\end{prop}
\begin{proof} First note that since $\cM$ is Deligne-Mumford, \'{e}tale neighborhoods \emph{exist}. Thus the map $i_x$ is a limit of \'{e}tale morphisms, so it is formally \'{e}tale. The strict henselization $\cO_{U,\tilde{x}}$ is defined as the (global sections) of the cofiltered limit of \'{e}tale neighborhoods of $\tilde{x} : \Spec \Omega\rightarrow U$. The category of \'{e}tale neighborhoods of $\tilde{x}$ embeds into the category of \'{e}tale neighborhoods of $x$, and since every \'{e}tale neighborhood of $x$ is refined by an \'{e}tale neighborhood of $\tilde{x}$, this map induces an isomorphism on limits, whence the isomorphism $\cO_{\cM,x}\cong\cO_{U,\tilde{x}}$. The final statement follows from the fact that the residue field of $\cO_{U,\tilde{x}}$ is the separable closure of $\kappa(u)$ inside $\Omega$.

\end{proof}

\subsubsection{\'{E}tale local rings vs universal deformation rings}
We work universally over a scheme $\bS$. Let $s : \Spec k\rightarrow\bS$ be a morphism with $k$ a field. Suppose it factors as $\Spec k\rightarrow\Spec\Lambda\subset\bS$ where $\Spec\Lambda\subset\bS$ is a Noetherian open affine subscheme. Let $\cC_\Lambda = \cC_{\Lambda,k}$ be the category of pairs $(A,\psi)$, where $A$ is an Artinian local $\Lambda$-algebra and $\psi : A/\fm_A\rightiso k$ is an isomorphism of $\Lambda$-algebras. A morphism $(A,\psi)\rightarrow (A',\psi')$ in $\cC_\Lambda$ is a local $\Lambda$-algebra homomorphism $f : A\rightarrow A'$ such that $\psi'\circ(f\mod\fm) = \psi$. The category $\cC_\Lambda$ has a final object, given by $(k,\id_k)$. Note that if $\Spec k\rightarrow\Spec\Lambda'\subset\bS$ is another factorization, the categories $\cC_\Lambda,\cC_{\Lambda'}$ are canonically isomorphic. Let
$$p : \cM\rightarrow(\Sch/\bS)$$
be an algebraic stack over $\bS$ and let $x_0 : \Spec k\rightarrow\cM$ be a morphism. By the 2-Yoneda lemma, we will identify $x_0$ with the object it defines in $\cM(\Spec k)$ \cite[04SS]{stacks}. 


A \emph{deformation} of $x_0$ over $(A,\psi)\in\cC_\Lambda$ is by definition a pair $(x,\varphi)$, where $x\in\cM(\Spec A)$, and $\varphi : x_0\rightarrow x$ is a morphism in $\cM$ such that $p(\varphi) : \Spec k\rightarrow \Spec A$ induces the isomorphism $\psi : A/\fm_A\rightiso k$. A morphism of deformations $(x,\varphi)\rightarrow(x',\varphi')$ is a map $f : x\rightarrow x'$ with $f\circ\varphi = \varphi'$. A morphism $f : (x,\varphi)\rightarrow (x',\varphi')$ of deformations over $(A,\psi)\in\cC_\Lambda$ is an \emph{$A$-isomorphism} if $f$ is an isomorphism and $p(f) = \id_A$.

\begin{defn} The \emph{deformation functor} for $x_0$ is the functor
\begin{eqnarray*}
F_{x_0} : \cC_\Lambda & \longrightarrow & \Sets \\
(A,\psi) & \mapsto & \{\text{deformations of $x_0$ over $(A,\psi)$}\}/\text{$A$-isomorphisms}
\end{eqnarray*}	
\end{defn}

\begin{remark} In a Deligne-Mumford stack, the diagonal (and hence inertia stack) is unramified (and hence formally unramified). This implies that a deformation over $(A,\psi)$ has no nontrivial $A$-automorphisms.
\end{remark}


If $f : \cM\rightarrow\cN$ is a morphism of algebraic stacks, then if $(x,\varphi)$ is a deformation of $x_0$ over $(A,\psi)$, then $(f(x),f(\varphi))$ is a deformation of $f(x_0)$ over $(A,\psi)$. This defines a morphism of functors $f_* : F_{x_0}\rightarrow F_{f\circ x_0}$.


Let $\what{\cC}_\Lambda$ be the category of pairs $(R,\psi)$, where $R$ is a Noetherian complete local $\Lambda$-algebra and $\psi : R/\fm_R\rightiso k$ is a $\Lambda$-algebra isomorphism. Morphisms are local homomorphisms respecting $\psi$'s. Thus, $\cC_\Lambda$ embeds as a full subcategory of $\what{\cC_\Lambda}$. Given $(R,\psi)\in\what{\cC}_\Lambda$, it gives rise to a functor $h_{R,\psi} : \cC_\Lambda\rightarrow\Sets$ defined by $h_R((A,\psi')) = \Hom_{\what{\cC}_\Lambda}((R,\psi),(A,\psi'))$. A functor $F : \cC_\Lambda\rightarrow\Sets$ is \emph{pro-representable} by $(R,\psi)\in\what{\cC}_\Lambda$ if there is an isomorphism $F\cong h_{R,\psi}$. If $F_{x_0}$ is pro-represented by $(R,\psi)$, then we say that $R$ is a \emph{universal deformation ring for $x_0$}.






\begin{prop}\label{prop_elr2udr} Let $\cM$ be a Deligne-Mumford stack with representable diagonal over a scheme $\bS$. Let $\Spec\Lambda\subset\bS$ a Noetherian open affine, and let $k$ be a $\Lambda$-algebra which is a field. Let $\pi : U\rightarrow\cM$ be a formally \'{e}tale morphism with $U$ a Noetherian scheme, let $\tilde{x}_0 : \Spec k\rightarrow U$ be a point with image $u\in U$, and let $x_0 := \pi\circ\tilde{x}_0$, with associated deformation functor $F_{x_0} : \cC_\Lambda\rightarrow\Sets$ as above. Let $R := \cO_{U,u}$. Assume that $\tilde{x}_0$ induces an isomorphism $\xi : R/\fm_R\rightiso k$. Let $\what{R} := \lim_n R/\fm_R^n$ be the completion, with induced map $\what{\xi} : \what{R}/\fm_{\what{R}}\rightiso k$. Then
\begin{itemize}
\item[(a)] The map $\Spec\what{R}\rightarrow\cM$ induces an isomorphism $\eta_{\what{R}} : h_{\what{R},\what{\xi}}\rightiso F_{x_0}$.
\item[(b)] Let $z_0 : \Spec\Omega\rightarrow\cM$ be a point with $\Omega$ separably closed, let $\kappa(z_0) \subset\Omega$ be the residue field of $\cO_{\cM,z_0}$, and let $z_0' : \Spec \kappa(z_0)\rightarrow\cM$ the corresponding map. Then the map $\Spec\cO_{\cM,z_0}\rightarrow\cM$ identifies $\what{\cO_{\cM,z_0}}$ with the universal deformation ring of $z_0'$.
\item[(c)] Let $\rho : V\rightarrow\cN$ be a formally \'{e}tale morphism from a Noetherian scheme $V$ to a Deligne-Mumford stack $\cN$ with representable diagonal, and suppose we are given a map $g : U\rightarrow V$ fitting into a commutative diagram
\[\begin{tikzcd}
	U\ar[r,"\pi"]\ar[d,"g"] & \cM\ar[d,"f"] \\
	V\ar[r,"\rho"] & \cN
\end{tikzcd}\]
Let $v := g(u)$, and suppose $g$ induces an isomorphism of residue fields $\kappa(u)\rightiso\kappa(v)$. Let $S := \cO_{V,v}$, so $g$ induces a map $S\rightarrow R$, which induces an isomorphism $\zeta : S/\fm_S\rightiso k$. Then the corresponding diagram
\[\begin{tikzcd}
	h_{\what{R},\what{\xi}}\ar[r,"\eta_{\what{R}}"]\ar[d] & F_{x_0}\ar[d,"f_*"] \\
	h_{\what{S},\what{\zeta}}\ar[r,"\eta_{\what{S}}"] & F_{f\circ x_0}
\end{tikzcd}\]
commutes, where $\eta_{\what{R}},\eta_{\what{S}}$ are as given in (a) and $h_{\what{R},\what{\xi}}\rightarrow h_{\what{S},\what{\zeta}}$ is induced by $g$.
\end{itemize}
\end{prop}

\begin{proof} Any map $a : (\what{R},\what{\xi})\rightarrow (A,\psi)$ with $(A,\psi)\in\cC_\Lambda$ defines a map $x_a : \Spec A\rightarrow \Spec R\rightarrow\cM$. Since $\cM$ is fibered in groupoids, there is a map $\varphi : x_0\rightarrow x_a$ in $\cM$ inducing the isomorphism $\psi : A/\fm_A\rightiso k$, and this map $\varphi$ is unique up to precomposing with $A$-isomorphisms. The pair $(x_a,\varphi)$ is thus a deformation of $x_0$ over $A$, and this defines a morphism of functors
$$\eta_{\what{R}} : h_{\what{R},\what{\xi}}\rightarrow F_{x_0}$$
Let $h_{R,\xi} : \cC_\Lambda\rightarrow\Sets$ be the functor sending $(A,\psi)$ to the set of local $\Lambda$-algebra homomorphisms $a : R\rightarrow A$ satisfying $\psi\circ (a\mod\fm_R) = \xi$. Since any local homomorphism $\what{R}\rightarrow A$ must factor through $\what{R}/\fm_{\what{R}}^n\cong R/\fm_R^n$ for some $n$, the natural map $h_{R,\xi}\rightarrow h_{\what{R},\what{\xi}}$ is an isomorphism. Thus it suffices to show that the map
$$\eta_R : h_{R,\xi}\longrightarrow F_{x_0}$$
is an isomorphism. Let $(x,\varphi)$ be a deformation of $x_0$ over $(A,\psi)\in\cC_\Lambda$, so that $p(\varphi) : \Spec k\rightarrow\Spec A$ induces $\psi : A/\fm_A\rightiso k$. The map $\varphi : x_0\rightarrow x$ induces a unique isomorphism $\alpha_\varphi : x_0\rightiso x\circ p(\varphi)$. The triple $(\tilde{x}_0, p(\varphi),\alpha_\varphi)$ defines a map $t : \Spec k\rightarrow \Spec R\times_\cM\Spec A$ making the following diagram commute (ignoring the dotted arrow for now)

\[\begin{tikzcd}
	\Spec k\ar[d,"p(\varphi)"']\ar[r,"t"]\ar[rr,bend left = 20,"\tilde{x}_0"] & \Spec R\times_\cM\Spec A\ar[r]\ar[d,"\pi_{u,A}"] & \Spec R\ar[d,"\pi_u"] \\
	\Spec A\ar[r,equals]\ar[ur,dashed] & \Spec A\ar[r,"x"] & \cM
\end{tikzcd}\]

Since $\pi$ is formally \'{e}tale, the same is true of $\pi_u$ and $\pi_{u,A}$ \cite[04EG]{stacks}. Thus, there is a unique dotted arrow making the diagram commute, and hence the corresponding map $\Spec A\rightarrow\Spec R$ gives an element of $h_{R,\xi}(A,\psi)$ inducing the deformation $(x,\varphi)$. The existence of the dotted arrow implies that $\eta_R$ is surjective, and the uniqueness implies that $\eta_R$ is injective. This proves (a). Part (b) follows from (a) by setting $U = \Spec\cO_{\cM,z_0}$, and (c) follows from the construction of the isomorphism $\eta_{\what{R}}$.
\end{proof}


\begin{prop}\label{prop_deformation_theoretic_criterion_of_etaleness} Let $\cM,\cN$ be locally Noetherian Deligne-Mumford stacks with representable diagonal. Let $f : \cM\rightarrow\cN$ be a quasi-finite morphism. The following are equivalent
\begin{itemize}
\item[(1)] $f$ is \'{e}tale.
\item[(2)] For every point $z_0 : \Spec\Omega\rightarrow\cM$ with $\Omega$ separably closed, the induced morphism of \'{e}tale local rings $(f_{z_0})_* : \cO_{\cN,f\circ z_0}\rightarrow \cO_{\cM,z_0}$ is an isomorphism.
\item[(3)] For every point $z_0 : \Spec\Omega\rightarrow\cM$ with $\Omega$ separably closed and $\kappa(z_0) = \Omega$ (c.f. Definition \ref{def_\'{e}tale_local_ring}), $f$ induces an isomorphism of deformation functors $f_* : F_{z_0}\rightiso F_{f\circ z_0}$.
\end{itemize}
\end{prop}
\begin{proof} First we show (1)$\iff$(2). Let $V\rightarrow\cN$ be an \'{e}tale neighborhood of $f\circ z_0$, and then $\cM_V := \cM\times_\cN V\rightarrow\cM$ is \'{e}tale and $\cM_V$ is also a Deligne-Mumford stack with representable diagonal. Let $u : U\rightarrow\cM_V$ be an \'{e}tale neighborhood of some lift of $z_0$ to $\cM_V$. Possibly shrinking $U,V$, we may assume that $u$ is quasi-finite. Let $\tilde{z_0} : \Spec\Omega\rightarrow U$ be a lift of $z_0$ to $U$. Then $f$ is \'{e}tale at $z_0$ if and only if $\cM_V\rightarrow V$ is \'{e}tale at $u(\tilde{z_0})$ if and only if the map $h : U\rightarrow V$ is \'{e}tale at the image of $\tilde{z_0}$ \cite[02KM]{stacks}. Since $h$ is quasi-finite, the induced map on local rings is of finite presentation, so $h$ is \'{e}tale at the image of $\tilde{z_0}$ if and only if the induced map of local rings is weakly \'{e}tale \cite[0CKP, 039L]{stacks}, but this is equivalent to $h$ inducing an isomorphism on \'{e}tale local rings at $\tilde{z_0}$ \cite[094Z]{stacks}. Since $U\rightarrow\cM$ and $V\rightarrow\cN$ are \'{e}tale, this map of \'{e}tale local rings is precisely the map $(f_{z_0})_*$, so (1)$\iff$(2).


Proposition \ref{prop_elr2udr} shows that (2)$\Rightarrow$(3). The converse follows from \cite[\S4.3, Proposition 3.26]{Liu02}. 
\end{proof}

\bibliography{references}

\end{document}